\documentclass[a4paper]{article}
\usepackage{preamble_notes}
\usepackage{longtable} 
\usepackage{booktabs}
\usepackage{authblk}

\title{Gaussian Mixture Model with unknown diagonal covariances via continuous sparse regularization}

\bgroup

\author[1]{Romane Giard}
\author[1,3]{, Yohann De Castro}
\author[2]{ and Clément Marteau}

{
\affil[1]{{Centrale Lyon, CNRS UMR 5208, Institut Camille Jordan, Écully, France.}}
\affil[2]{Univ. Claude Bernard, CNRS UMR 5208, Institut Camille Jordan, Villeurbanne, France.}
\affil[3]{Institut Universitaire de France (IUF)}
}
\date{version as of \today}
\egroup

\begin{document}

\maketitle

\begin{abstract}
This paper addresses the statistical estimation of Gaussian Mixture Models (GMMs) with unknown diagonal covariances from independent and identically distributed samples. We employ the Beurling-LASSO (BLASSO), a convex optimization framework that promotes sparsity in the space of measures, to simultaneously estimate the number of components and their parameters.

Our main contribution extends the BLASSO methodology to multivariate GMMs with component-specific unknown diagonal covariance matrices. This setting is significantly more flexible than previous approaches, which required known and identical covariances. We establish non-asymptotic recovery guarantees with nearly parametric convergence rates for component means, diagonal covariances, and weights, as well as for density prediction. 

A key theoretical contribution is the identification of an explicit separation condition on mixture components that enables the construction of non-degenerate dual certificates—essential tools for establishing statistical guarantees for the BLASSO. Our analysis leverages the Fisher-Rao geometry of the statistical model and introduces a novel semi-distance adapted to our framework, providing new insights into the interplay between component separation, parameter space geometry, and achievable statistical recovery. 
\end{abstract}

\medskip

\noindent Corresponding author: Romane Giard\\
\noindent E-mail address: \texttt{romane.giard@ec-lyon.fr}

\section{Introduction}
Gaussian Mixture Models (GMMs) are a cornerstone of statistical modeling, offering a flexible and powerful framework for representing complex data distributions. They are widely applied in diverse fields, including clustering and density estimation \citep{mclachlan_mm_2000,bouveyron_model-based_2019}, image processing \citep{hdmi}, bioinformatics and economics \citep[Section~8.4]{handbook_mm}.
Despite their ubiquity, the estimation of GMMs, particularly when the number of components and their covariance structures are unknown, presents significant statistical and computational challenges. 

The predominant method for GMM estimation is the Expectation-Maximization (EM) algorithm, which iteratively maximizes the log-likelihood. While the EM algorithm guarantees a non-decreasing likelihood and converges to a stationary point under some conditions \citep{mclachlan_em, wu_em, balakrishnan_em}, it faces significant practical challenges. The log-likelihood function is non-convex and multi-modal, making the algorithm sensitive to initialization and prone to converging to local maxima rather than the global optimum.
Crucially, it requires the number of mixture components to be specified in advance. We aim to overcome these limitations.

\medskip
This paper introduces a novel approach to the estimation of multivariate GMMs with unknown diagonal covariances, leveraging the Beurling-LASSO (BLASSO) methodology \citep{decastro_2012, candes_fernandez_granda_2014, blasso_duval_peyre, boyer2017adapting, blasso_poon, supermix}. The BLASSO provides a convex optimization framework in the space of measures that promotes sparsity, allowing the recovery of discrete measures from continuous observations. It offers a principled way to simultaneously estimate the number of components, their respective weights, means, and diagonal covariance matrices. It has not yet been applied to settings where each mixture component is allowed to have its own covariance structure. This extension makes it possible to handle significantly more realistic models.
For instance, in clustering applications, assuming identical covariances across components yields Voronoi partitions (linear decision boundaries), as in $k$-means. In contrast, allowing distinct covariance matrices leads to boundaries defined by quadratic equations \citep[Section 4.2]{ml_probabilistic_perspective}. 

\medskip
A pivotal contribution of our research is the identification of a separation condition for the mixture components. This condition is instrumental in the construction of so-called \emph{non-degenerate dual certificates} \citep{candes_fernandez_granda_2014, blasso_duval_peyre}, which play a key role in establishing estimation guarantees. Our analysis is grounded in the Fisher-Rao geometry of the statistical model, providing theoretical insights into the intricate relationships between component separation, the underlying geometry of the parameter space, and the achievable statistical recovery.

\subsection{Continuous sparse regression for Gaussian Mixture Models}
We denote by $\N(t,C)$ a multivariate Gaussian distribution in dimension $d\in\mathbb{N}^*$, with mean $t\in \R^d$ and covariance matrix $C\in \S_{++}^d$, where $\S_{++}^d$ denotes the space of positive-definite symmetric matrices of size $d \times d$. Given $u=(u_1,\ldots, u_d)\in(\mathbb{R}_+^*)^d$ a vector with positive entries, we denote by $\diag((u_{1})^2,\ldots, (u_{d})^2)$ the $d\times d$ diagonal matrix with diagonal entries $u_k^2$, where each $u_k$ is interpreted as a marginal standard deviation. The density of $\mathcal{N}\left(t,\diag((u_{1})^2,\ldots, (u_{d})^2)\right)$ is denoted by $\varphi_{(t,u)}$.

\medskip
In this paper, we observe a $n$-sample $X_1,\ldots,X_n \in \R^d$ drawn from a Gaussian Mixture distribution with diagonal covariances, defined by: 
\begin{equation}\label{eq:sample_distrib}
X_i \iid \sum_{j=1}^{s}a_j^0 \varphi_{(t_j^0,u_j^0)} \eqcolon f^0 
\quad 
\text{with}
\quad 
\sum_{j=1}^s a_j^0=1
\quad 
\text{and}
\quad 
a_j^0>0
\end{equation}
where $t_j^0\in\R^d$, $u_j^0=(u_{j,k}^0)_{k=1,\ldots,d}\in(\R_+^*)^d$.
Our goal is to estimate the number of \emph{components} $s$ (also called the sparsity index), the \emph{weights}~$a_j^0$ and the \emph{location-scale} parameters $(t_j^0, u_j^0)$ of each mixture component, indexed by $j\in \{1,\ldots,s\}$, from the observations $X_1,\ldots,X_n$.

\begin{remark}
\label{remark:identifiability}
 As a matter of fact, the Gaussian Mixture Model with unknown number of components, weights, means and covariance matrices:
 \[
 \bigg\lbrace \sum_{j=1}^p a_j \N(t_j,C_j) \: : \: p \geq 1\,, \: a_j>0\,, \: t_j \in \R^d\,, \: C_j \in \S_{++}^d
 \,, \: \forall i\neq j
 \,,\: (t_j,C_j) \neq (t_i,C_i) 
 \,, \: \sum_{j=1}^p a_j=1 \bigg\rbrace\] 
 is identifiable, \ie, if two distributions of this model are equal, then they have the same number of components, and the same components (up to a permutation) with the same associated weights. We refer to \citep{teicher_identifiability_gmm_1d} for the proof in dimension $d=1$, and \citep{yakowitz_identifiability_gmm_multidim} for the generalization in dimension $d \geq 1$. Under mild assumptions, an even stronger result holds for \emph{continuous} mixtures of Gaussian distributions, defined by a density on the space of mean and covariance matrix \citep{bruni_identifiability_gmm_continuous}.
\end{remark}

\medskip
Our estimation strategy relies on the Beurling-LASSO (BLASSO). This framework, introduced in \citep{decastro_2012,candes_fernandez_granda_2014}, has been successfully applied to various statistical estimation problems, particularly in the context of sparse signal recovery \citep{blasso_duval_peyre} and compressed sensing \citep{off-the-grid_cs}. It has been extended to various settings, including the estimation of Gaussian Mixture Models with known covariances \citep{supermix,off-the-grid_cs}. Algorithmic implementations are discussed in \citep{cpgd_chizat} and \citep{fastpart}. One key modeling idea of this approach is to lift the parameter space into the space of measures via the embedding
\[
(a_j,t_j,u_j)_{j=1}^p\mapsto \mu=\sum_{j=1}^p~a_j\delta_{(t_j,u_j)}\,.
\]
where $(t_j,u_j)_{j=1}^p$ are referred to as the \emph{particles} of $\mu$. Any discrete probability measure on $\R^d \times (0,+\infty)^d$ with finite support, of size $p$ for any $p\geq1$, describes a set of parameters. We stress that the parameter $p$, the number of components, is free and not prescribed.
The law of the $n$-sample $(X_i)_{i=1}^n$ (see \eqref{eq:sample_distrib}) is unambiguously described by the target parameters $(a^0_j,t^0_j,u^0_j)_{j=1}^s$ (by identifiability of the model, see Remark~\ref{remark:identifiability}), and hence unambiguously represented by the so-called \emph{target measure} $\mu^0\coloneq \sum_{j=1}^s~a_j^0\delta_{(t_j^0,u_j^0)}$.

\medskip
Our estimator will be defined as a solution of a minimization problem over the space of measures, constructed from the observations $X_1,\ldots,X_n$. It is designed to estimate the target measure $\mu^0$. In particular, we want our estimator to put a mass close to $a_j^0$ around the particle $(t_j^0,u_j^0)$. We will consider a convex loss of the form
\begin{equation}\label{eq:general_form_blasso}
 F_{n,\tau}(\mu)+\kappa R(\mu)
\end{equation}
where $F_{n,\tau}(\mu)$ is a \emph{data fidelity} term, $R(\mu)$ is a \emph{regularization} term enforcing sparsity, and $\kappa>0$ a tuning parameter. The data fidelity term compares a predicted density encoded by $\mu$ with an empirical approximation of the target density. This approximation is obtained from the empirical distribution of $(X_i)_{i=1}^n$ by convolution with a Gaussian kernel of covariance $\tau^2\mathrm{Id}_d$, depending on a \emph{smoothing parameter} $\tau$. This approach is standard in kernel density estimation \citep{tsybakov_kde}, but the main difference here is that $\tau$ will not be necessarily chosen to match some nonparametric rate. Its calibration will sometimes answer to an alternative purpose. 
The regularization term $R(\mu)$ is the total variation (TV) norm, analog of the $\ell^1$-norm for measures, and aims to concentrate the mass of our estimator in a few regions. We refer to Section \ref{section:statistical_modeling} and \eqref{eq:pb_BLASSO_gaussian_kernel_norm} for a complete description of the BLASSO procedure.

\subsection{Related works}

Gaussian mixture estimation ranges from clustering to parameter estimation. Our focus is primarily on the latter. Prior work has investigated a variety of settings under different modeling assumptions.

\paragraph{Gaussian location mixtures}
When covariances are known, \citep{heinrich_kahn} derived local asymptotic minimax rates for univariate mixtures. Generalizing to dimension $d$, \citep{doss2023optimal} established minimax rates under $W_1$ and Hellinger distances.
\citep{regev_well_separated_gaussians} further highlighted the gap between statistical and computational limits: while a mean separation of $\sqrt{\ln s}$ suffices for identifiability with a polynomial sample size, efficient algorithms typically require a separation of order $\sqrt{d}$.
Other approaches include Bayesian methods with priors on the number of components \citep{ohn_lin_2023_growing_number_components} or approximations by finite mixtures \citep{ma_yun_best_approx_gmm}.

\paragraph{Unknown covariances}
Intermediate settings with \emph{shared} unknown variance have been studied using the Denoised Method of Moments (DMM) \citep{wu_yang_dmm} in the univariate case, or EM \citep{bing2025} in dimension $d$ under a separation condition on the components.

The general case of \emph{varying} unknown covariances remains challenging. Restricted to two components, efficient estimators exist given sufficient separation \citep{hardt_price} or specific weight constraints \citep{wu_zhou_2_components}.

When the number of mixture components $s$ is unknown, \citep{ho_nguyen_aos_2016} analyzed the Maximum Likelihood Estimation (MLE) in over-parameterized settings, establishing slow polynomial convergence rates dependent on the excess number of components. \citep{guha_bayesian_mm} also studied this setting in a bayesian framework. In both works, the resulting convergence guarantees are pointwise in the sense that the bounds depend on the underlying target mixture.
Alternatively, \citep{belkin_polynomial_learning} proposed a method of moments (MoM) with polynomial sample complexity when $s$ is known; however, their grid-search approach remains exponentially hard in the dimension. Finally, \citep{optimalclustering} derived minimax bounds for clustering risks with anisotropic covariances using Lloyd's algorithm.

\paragraph{Positioning}
Our work departs from these approaches by adopting the Beurling-LASSO (BLASSO) framework. Unlike the EM algorithm \citep{bing2025} or MoM \citep{belkin_polynomial_learning}, this convex optimization approach does not require specifying $s$. While BLASSO is established for fixed covariances \citep{supermix}, extending it to multivariate GMMs with component-specific \emph{unknown} diagonal covariances presents a distinct theoretical challenge. We address this by introducing a geometry-adapted semi-distance and establishing non-asymptotic recovery guarantees for both parameter estimation and density prediction. Our results are not pointwise (except for the sparsity index estimation displayed in Section \ref{section:NDSC}); rather, they hold uniformly over classes of mixtures satisfying a prescribed separation condition.

\subsection{Contributions}
Our primary contribution is the extension of the BLASSO framework to estimate GMMs with unknown diagonal covariances. This setting introduces a key challenge: the associated kernel (see Section \ref{section:dual_certif}) is not translation-invariant, departing from many standard BLASSO applications. We address this by:
\begin{itemize}
 \item Introducing a reparameterization of the measures to work with a normalized kernel—an object essential for the theoretical analysis of the BLASSO, describing the correlation between 2 location-scale parameters $(t,u)$ and $(t',u')$. 
 \item Establishing recovery guarantees through the construction of non-degenerate dual certificates. This involves proving a modified version of the \emph{local positive curvature} assumption (LPC) from \citep[Assumption~$1$]{off-the-grid_cs} for our specific non translation-invariant kernel.
 \item Using a semi-distance naturally aligned with the kernel-induced geometry. This allows us to formulate more tractable conditions for the certificate construction, relaxing the reliance on the Fisher-Rao distance used in related works.
 \item Establishing conditions on the separation between the components of the mixture (\ie, the particles of~$\mu^0$) with respect to the semi-distance. These conditions depend on bounds on the variance, the sparsity $s$, the dimension $d$, and the smoothing parameter $\tau$.
 Higher sparsity levels or looser bounds on the variances necessitate a wider minimal separation between the components of $\mu^0$. See \eqref{eq:minimum_separation} for a precise condition.
\end{itemize}
We also provide novel prediction guarantees for the target density, achieving nearly parametric rates of convergence in different regularization regimes.

\paragraph{Informal results on error bounds}
We define $\mu_\omega^0$ as a reparameterized version of the target measure $\mu^0$,
\[
\mu_\omega^0=\sum_{j=1}^s \omega_j^0\, \delta_{(t_j^0,u_j^0)}
\quad \text{with} \quad \omega_j^0 = W(x_j^0)\,a_j^0\,,
\]
where $W$ is a positive function that will be specified later. 
According to the procedure displayed in Section \ref{section:statistical_modeling}, we estimate $\mu_\omega^0$ rather than $\mu^0$. Our estimator is given by the argument minimum of a function $J_W$ of the form \eqref{eq:general_form_blasso} over nonnegative measures with support restricted to a compact set $\X \subset
\R^d \times [u_{\min},+\infty)^d$ (see \eqref{eq:pb_BLASSO_gaussian_kernel_norm} for a formalized version). Here, $u_{\min}$ denotes a lower bound on the diagonal elements of the covariance matrix. 

Most of our results remain valid for approximate solutions, that is for any measure $\mu$ satisfying $J_W(\mu)\leq J_W(\mu_\omega^0)$, and not only the exact argument minimum solution of $J_W$. Accordingly, we will use the notation $\mu_{n,\omega}^\star$ when an exact solution is required, and $\hat{\mu}_{n,\omega}$ when an approximate solution suffices.

We evaluate the estimator by comparing its mass against the mass of $\mu_\omega^0$ within ``near regions'' $\X_j^{\rm near}(r_e)$ centered on the true parameters $(t_j^0,u_j^0)$, and in the complementary ``far region'' within $\X$. Near regions correspond to balls of radius $r_e$ for a semi-distance $\mathpzc{d}$ (given by \eqref{eq:def_semi_distance} below). Figure \ref{fig:regions_comparison_estimate_mu^0} illustrates this partitioning. The near region $\X_j^{\rm near}(r_e)$ depends only on its radius $r_e$, $x_j^0$, the dimension $d$, and the smoothing parameter $\tau$.
\begin{figure}[ht!]
\centering
\begin{tikzpicture}
 \definecolor{lightorange}{RGB}{255, 181, 112}
 \definecolor{darkblue}{RGB}{0, 0, 139}

 \draw[->] (-2.3, 0)--(5.5, 0) node[anchor=north] {$t$};
 \draw[->] (0, 0)--(0, 5) node[anchor=west] {$u$};

 \newcommand{\largecloudA}[2]{
 \begin{scope}[shift={(#1,#2)}, scale=0.9]
 \fill[lightorange]
 plot[smooth cycle, tension=1] coordinates {
 (0.4,0.7) (0.9,0.2)
 (0,-0.7) (-0.8,0.3)
 };
 \end{scope}
 }

 \newcommand{\largecloudB}[2]{
 \begin{scope}[shift={(#1,#2)}, scale=0.9]
 \fill[lightorange]
 plot[smooth cycle, tension=1] coordinates {
 (0.8,0.7)
 (-0.2,-0.7) (-0.7,-0.2) (-0.5,0.5)
 };
 \end{scope}
 }

 \newcommand{\largecloudC}[2]{
 \begin{scope}[shift={(#1,#2)}, scale=0.9]
 \fill[lightorange]
 plot[smooth cycle, tension=1] coordinates {
 (0.4,0.9) (0.9,0.8)
 (0.5,-0.2) (-0.7,-0.1)
 };
 \end{scope}
 }

 \largecloudA{-1.1}{1.5}
 \largecloudB{1.2}{3.1}
 \largecloudC{4.0}{2.5}

 \fill[lightorange]
 plot[smooth cycle, tension=1] coordinates {
 (-0.6,3.1) (-0.5,2.9) (-0.3,3.0)
 };
 \fill[lightorange]
 plot[smooth cycle, tension=1] coordinates {
 (3.0,2.3) (2.5,2.1) (2.9,2.0)
 };

 \begin{scope}
 \clip plot[smooth cycle, tension=0.5] coordinates {
 (-1.9,0.9) (-1.7,2.6) (0.2,4.1)
 (2.5,4.6) (5.0,3.4) (5.1,1.7) (2.0,0.4)
 }
 (-1.2,1.7) circle[radius=0.8]
 (1.2,3.2) circle[radius=0.8]
 (4.0,2.5) circle[radius=0.8];
 
 \fill[pattern=north east lines, pattern color=gray] 
 (-2.3,0) rectangle (5.6,5); 
\end{scope}

 \draw[thick] (-1.2,1.7) circle[radius=0.8];
 \fill[darkblue] (-1.2,1.7) circle[radius=2.5pt];

 \draw[thick] (1.2,3.2) circle[radius=0.8];
 \fill[darkblue] (1.2,3.2) circle[radius=2.5pt];
 \draw[dashed, -] (1.2,3.2)--(1.2,4.0) node[midway, right] {$r_e$};
 
 \draw[thick] (4.0,2.5) circle[radius=0.8];
 \fill[darkblue] (4.0,2.5) circle[radius=2.5pt];

 \draw[very thick, smooth, color=black]
 plot[smooth cycle, tension=0.5] coordinates {
 (-1.9,0.9) (-1.7,2.6) (0.2,4.1)
 (2.5,4.6) (5.0,3.4) (5.1,1.7) (2.0,0.4)
 };

 \node at (2.6,4.1) {$\X$};

 \coordinate (massLabel) at (-3.3,2.7);
 \node[above] at (massLabel) {\small Mass of $\mu_\omega^0$ (blue dots)};
 \draw[->, thick] (massLabel) to[out=270, in=170] (-1.25,1.75);

 \coordinate (regionLabel) at (-1.2,4.0);
 \node[above] at (regionLabel) {\small Near region};
 \draw[->, thick] (regionLabel) to[out=290, in=180] (0.4,3.3);

 \coordinate (estimateLabel) at (5.5,4.0);
 \node[above] at (estimateLabel) {\small Mass of $\hat{\mu}_{n,\omega}$ (orange)};
 \draw[->, thick] (estimateLabel) to[out=250, in=20] (4.6,3.2);
\end{tikzpicture}

\caption{Schematic representation for Gaussian mixture models in dimension $d=1$. Both parameters $u$ (standard deviation) and $t$ (mean) are in dimension $1$, resulting in a 2-dimensional plot in location-scale space~$(t,u)$. The discs represent the near regions, shown schematically: these regions correspond to balls defined with respect to a semi-distance, not the Euclidean distance. The hatched area corresponds to the far region. 
}
\label{fig:regions_comparison_estimate_mu^0}
\end{figure}

\medskip

Our estimator satisfies the following properties, given in expected value over $X_1,\ldots,X_n$.

\begin{theorem}[Recovery guarantees for the estimation of $\mu_\omega^0$, informal result] \label{th:informal_estimation}
Assume that the particles $(t_j^0,u_j^0)$ of $\mu^0$ are sufficiently separated, where the minimal separation constraint only depends on the dimension~$d$, the sparsity index $s$, bounds on the variance and choice of a smoothing parameter $\tau \leq u_{\min}$. Choosing as regularization parameter $\kappa=\frac{\sqrt{2}}{(2\pi)^{d/4} \tau^{d/2} \sqrt{n}}$, for any $r_e$ such that $0<r_e \leq r$ with $r$ a fixed constant depending on $d$, it holds that, omitting the dependence on $d$,
\begin{equation}
 \E{|\mu_\omega^0(\X_j^{\rm near}(r_e))- \hat{\mu}_{n,\omega}(\X_j^{\rm near}(r_e))|}\lesssim \frac{s}{r_e^2\sqrt{n} \tau^{d/2}} \quad \forall \, j=1,\ldots, s \label{eq:bound_informal_effective_near}
\end{equation} and 
\[\mathbb{E}{\bigg[\Big|\hat{\mu}_{n,\omega}
\Big(
 \X \setminus \bigcup_j \X_j^{\rm near}(r)
\Big)
\Big|\bigg]}\lesssim \frac{s}{\sqrt{n} \tau^{d/2}}\,. \]
\end{theorem}
\noindent
We refer to Theorems \ref{th:estimation_error} and \ref{th:main_theorem_certif}, Proposition \ref{prop:effective_near_regions} with Lemma \ref{lemma:tilde_varepsilon_3} for precise statements and related discussions. Note that the constant factor omitted in the bound \eqref{eq:bound_informal_effective_near} does not depend on the target mixture, only on the dimension $d$.
Our estimator allows us to locate the particles of the reparameterized target measure on fixed regions with a parametric convergence rate.
The radius $r_e$ can be chosen as desired (sufficiently small), leading to degraded bounds (the rate depends on $r_e^{-2}$).
We give in Corollary \ref{cor:control_renormalized_estimate} a variation of the previous result, for a direct estimator of $\mu^0$—and with decreasing size of regions.

\medskip
\begin{remark}
 Note that in the above result, $\kappa$ does not depend on $s$ (unknown in practice). We call it the \emph{$s$-agnostic} choice. Making the \emph{$s$-dependent} choice $\kappa=\frac{\sqrt{2}}{(2\pi)^{d/4} \tau^{d/2}\sqrt{sn}}$, we obtain a rate of $\frac{\sqrt{s}}{r_e^2 \sqrt{n}\tau^{d/2}}$ for the bound \eqref{eq:bound_informal_effective_near} (better dependence on $s$). See Remark \ref{remark:choice_kappa_cor_near_effective}.
\end{remark}

\medskip
We provide at the same time prediction error bounds. We achieve an almost parametric rate (up to a logarithmic factor) for the prediction of the target density~$f^0$, in two distinct regimes: under small regularization (Proposition \ref{prop:prediction_under_small_reg}) and with larger regularization (Theorem \ref{th:prediction_error}) when the assumption of Theorem \ref{th:informal_estimation} is verified. We present an informal version of these results. 
\begin{theorem}[Recovery guarantees for the prediction of $f^0$, informal results] Let $\tau=\frac{\sqrt{2}u_{\min}}{\sqrt{\ln n}}$. We construct an estimator $\hat{h}_n$ of $f^0$ from $\hat{\mu}_{n,\omega}$. We consider two regimes.
\begin{itemize}
 \item Assume that the particles $(t_j^0,u_j^0)$ of $\mu^0$ are sufficiently separated. Choosing $\kappa=\frac{\sqrt{2}(\ln n)^{d/4}}{(2\pi)^{d/4} (2u_{\min}^2)^{d/4} \sqrt{n}}$, omitting the dependence on $d$ and on bounds on the variance,
\[\E{\norm{\hat{h}_n - f^0}_{L^2(\R^d)}^2}\lesssim \frac{s(\ln n)^{d/2}}{n}\,.\]
 \item Without separation assumptions on $(t_j^0,u_j^0)_j$, choosing $\kappa=\frac{4(\ln n)^{d/2}}{(2\pi)^{d/2} (2u_{\min}^2)^{d/2} n}$, omitting the dependence on~$d$ and on bounds on the variance,
\[\E{\norm{\hat{h}_n - f^0}_{L^2(\R^d)}^2}\lesssim \frac{(\ln n)^{d/2}}{n}\,.\]
\end{itemize}
\end{theorem}

The regularization parameter $\kappa$ is allowed to depend on $n$. Its value has some importance on the performances of the procedure. It is calibrated according to the objective to be achieved: estimation, prediction or both. Under some conditions on $\mu^0$, large regularization provides good estimation and prediction. Small regularization allows the BLASSO to focus on minimizing the data fidelity term, yielding good prediction without any separation condition on the particles of $\mu^0$. For theoretical purposes, we will also briefly discuss an $s$-dependent choice of $\kappa$, which leads to improved bounds in the large regularization regime (the bound is of order $\frac{(\ln n)^{d/2}}{n}$ for $s$ reasonably small, see Remark \ref{remark:choice_kappa_pred} below). 

\medskip	
We finally derive an alternative result for $\mu_{n,\omega}^\star$, an exact solution to \eqref{eq:pb_BLASSO_gaussian_kernel_norm}, under some specific conditions. We provide below an informal result, formalized in Corollary \ref{cor:ndsc}. 

\begin{theorem}[Sparsity of the estimator for a large sample size, informal result]
 Let $0<\tau\leq u_{\min}$ (fixed). Assume that the particles of $\mu^0$ are sufficiently separated, where the minimal separation is the same as for Theorem \ref{th:informal_estimation}. Choosing $\kappa=\frac{\alpha \sqrt{\ln n}}{(2\pi)^{d/4} \tau^{d/2} \sqrt{n}}$ for any $\alpha>0$, there exists $n_0\in \mathbb{N}$ depending on $\mu^0, \X,\tau, \alpha$, such that if the sample size $n$ verifies $n\geq n_0$, then $\mu_{n,\omega}^\star$ is $s$-sparse with probability greater than $1-C_\Gamma n^{-\frac{\gamma_0^2\alpha^2 }{C_\Gamma^2}}$, with $C_\Gamma$ a universal positive constant and $\gamma_0$ depending on $d$.
Moreover, writing $\mu_{n,\omega}^\star=\sum_{j=1}^s \omega_j^\star \delta_{x_j^{\star}}$, omitting constants depending on $\X,\mu^0,\tau$ we have for all $j=1,\ldots,s$
\begin{equation}
 |a_j^0 -a_j^\star | \lesssim \alpha \sqrt{\frac{{\ln n}}{{n}}} \quad \text{and} \quad \mathpzc{d}((t_j^{\star},u_j^{\star}),(t_j^0,u_j^0)) \lesssim \alpha \sqrt{\frac{{\ln n}}{{n}}}\,,
\label{eq:th_sparsity_informal}
\end{equation}
where $\mathpzc{d}(\dotp,\dotp)$ is the semi-distance between location-scale parameters defined in \eqref{eq:def_semi_distance} below, and $a_j^\star=\frac{\omega_j^\star}{W(x_j^\star)}$ is an estimator of the true weight $a_j^0$.
\end{theorem}

The result displayed here provides a different flavor on the BLASSO performances. First, this bound holds provided the sample size is large enough (and under a separation condition). In such case, we first establish that the measure $\mu_{n,\omega}^\star$ has, with high probability, exactly the same sparsity index $s$ than the target $\mu^0$. Moreover, the bound \eqref{eq:th_sparsity_informal} provides theoretical guarantees on the estimation of the mixture parameters themselves instead of a control on far and near regions (Theorem~\ref{th:informal_estimation}).

\paragraph{Comparison with the state-of-the-art rates}
The natural benchmarks are minimax rates for the estimation of the mixing measure in
Wasserstein distance. For \emph{strongly identifiable} (in the second-order sense) families
with known number of components, the mixing measure can be estimated at the parametric rate
$n^{-1/2}$, degrading to $n^{-1/(4(p-s)+2)}$ when the fitted order $p$ overshoots the true
order $s$ \citep{heinrich_kahn}. These benchmarks are not available here, for two reasons.
First, the heteroscedastic location-scale Gaussian family is only \emph{weakly} identifiable
(the heat equation makes second-order location derivatives colinear with first-order scale
derivatives), and over-parameterized maximum likelihood then converges at rates
$n^{-1/(2\bar r)}$ with $\bar r\geq 4$ \citep{ho_nguyen_aos_2016}. Second, absent any
separation condition, parametric rates are unattainable for \emph{any} procedure: already for
location mixtures with known common covariance, the minimax $W_1$ rate is $n^{-1/(4k-2)}$ in
dimension one \citep{heinrich_kahn,wu_yang_dmm} and $\left(\frac{d}{n}\right)^{1/4}+n^{-1/(4k-2)}$ in dimension
$d$ \citep{doss2023optimal}, reaching $n^{-1/2}$ only for fully separated components
\citep{wu_yang_dmm}. 

Our results, established under an explicit separation condition, are consistent with this landscape. Theorem~\ref{th:informal_estimation} locates the mass of the
reparameterized target at a parametric rate in $n$, uniformly over the class and without
knowledge of $s$; extracting the individual weights
(Corollary~\ref{cor:control_renormalized_estimate}) incurs the slower $n^{-1/6}$, a byproduct
of the certificate-based analysis. For exact solutions and $n$ large enough, the sparsity
result above (formalized in Corollary~\ref{cor:ndsc}) recovers the number of components and
estimates weights and location-scale parameters at the near-parametric pointwise rate $\sqrt{\frac{\ln n}{n}}$—matching, up to logarithms, the benchmark for exactly-fitted strongly identifiable models,
although the family is weakly identifiable and $s$ is unknown. Finally, in contrast with
likelihood or moment-based procedures, which take the order (or an upper bound on it) as
input, our estimator is defined through a program that is \emph{convex} in the measure and
order-free. Practical solvers such as sliding Frank-Wolfe
\citep{sliding_frank_wolfe_denoyelle} or Conic Particle Gradient Descent \citep{cpgd_chizat}
come with convergence guarantees under additional assumptions; note that
Theorems~\ref{th:estimation_error} and~\ref{th:prediction_error} hold for \emph{any} measure
whose objective value does not exceed that of the reweighted target—a strictly weaker
requirement than solving \eqref{eq:pb_BLASSO_gaussian_kernel_norm} to optimality. Only the
results of Section~\ref{section:NDSC} require an exact solution.

\paragraph{Outline}
Section \ref{section:statistical_modeling} introduces the statistical framework and the BLASSO estimators $\hat{\mu}_{n,\omega}$, $\mu_{n,\omega}^\star$ used for recovering Gaussian Mixture Models. In Section \ref{section:estimation}, we establish recovery guarantees for the estimation of the target measure, relying on the existence of non-degenerate dual certificates. Section \ref{section:prediction} focuses on prediction guarantees for the density $f^0$, providing rates of convergence under different regularization regimes. In Section \ref{section:dual_certif}, we construct non-degenerate dual certificates by analyzing the kernel properties and deriving sufficient conditions on $\mu^0$. Finally, Section \ref{section:NDSC} demonstrates that, for sufficiently large sample sizes and under sufficient separation between components, the estimator $\mu_{n,\omega}^\star$ is, with high probability, a discrete measure. Section \ref{section:further_remarks} concludes with a discussion of potential extensions, algorithmic considerations, and open problems. It also comments on the separation condition required in our analysis. 
The main results are summarized in Tables \ref{table:results_summary_kappa_not_depending_on_s} and \ref{table:results_summary_kappa_depending_on_s} with corresponding assumptions, choices of smoothing parameter and regularization. Table \ref{table:notations} summarizes notation used throughout the paper. All proofs and technical results are gathered in the appendix at the end of the paper. To facilitate verification, we also provide a notebook (the associated Zenodo repository can be found at \citep{notebook}) implementing our calculations using a symbolic computation library.

\section{Statistical modeling}\label{section:statistical_modeling}
Our approach operates in the space of Radon measures, targeting the recovery of a sparse (\ie discrete) measure. Each particle of this discrete measure encodes the parameters of a Gaussian component, and its associated mass corresponds to the component's proportion in the mixture. In this section, we introduce the necessary notation, describe the statistical model, and present the BLASSO estimator.

\subsection{Radon measures}
We first start with some definitions allowing for a rigorous introduction of the BLASSO principle. 

For $A \subset \R^p$ locally compact, we denote $\C(A)$ the set of all continuous functions from $A$ to $\R$ and $\C_0(A)$ the set of continuous functions that vanish at infinity, \ie
\[\C_0(A)\coloneq \Big\{f \in \C(A) \: :\: \forall\, \varepsilon>0\,, \:  \exists\, C \text{ compact s.t.\ } |f|\leq \varepsilon \text{ on } A\setminus C \Big\} \,.\]
When $A$ is compact, $\C_0(A)=\C(A)$.
\begin{definition}[Radon measure on $A \subset \R^p$]
Let $A \subset \R^p$ be locally compact. The space of (real-valued) Radon measures $\M(A)$ is defined as the dual of $\left(\C_0(A), \norm{\dotp}_{\infty}\right)$. 
Its dual norm is the total variation norm:
\[\norm{\mu}_{\rm TV}=\sup_{\substack{\eta \in \C_0(A)\\ \norm{\eta}_{\infty}\leq 1}} \, \int_{A} \eta(x) \d\mu(x). \,\]
\end{definition}

\noindent
We denote $\M(A)^+$ the set of nonnegative measures: $\mu \in \M(A)^+$ if for all $\eta \in \C_0(A)$ such that $\eta \geq 0$, $\int_A \eta \, \d \mu \geq 0$. A discrete mixture measure, or a sparse measure, is a measure that can be expressed as a finite weighted sum of Dirac measures: \[\mu=\sum_{j=1}^{s} a_j \delta_{x_j}\]
 where $s \geq 1$, $a_1,\ldots,a_s \in \R$, $x_1,\ldots,x_s \in A$. Remark that $\norm{\mu}_{\rm TV}=\sum_{j=1}^s |a_j|$ when the $(x_j)_j$ are distinct.

Other details about the functional framework can be found in Appendix \ref{section:functional_framework}.

\subsection{Model and estimator}
\paragraph{Model} 
Our aim is to recover the target measure encoding the Gaussian Mixture distribution: 
\[
\mu^0=\sum_{j=1}^{s} a_j^0\delta_{x_j^0} 
\quad \text{where} 
\quad a_1^0,\ldots a_s^0>0\,, \; \sum_{j=1}^{s} a_j^0=1\,.
\] 
We assume that the location-scale parameters $(x_1^0,\ldots,x_s^0)$ are distinct points, each of the form $x_j^0=(t_j^0,u_j^0)$ with $t_j^0=(t_{j,k}^0)_{k=1}^d \in\R^d$ and $u_j^0~=~(u_{j,k}^0)_{k=1}^d \in [u_{\min},+\infty)^d$ where $u_{\min}$ is some positive lower bound on $(u_{j,k}^0)_k$. The parameters $t_j^0$ and $u_j^0$ represent respectively the mean and square root of the diagonal covariance of a Gaussian component with weight $a_j^0$. 

We denote $\varphi$ the standard Gaussian density in $\R$, and $\sigma \coloneq \F{\varphi}=e^{-\frac{\dotp^2}{2}}$ its associate Fourier transform.
We can rewrite $f^0=\Phi \mu^0$ where $\Phi:\M(\R^d \times [u_{\min},+\infty)^d) \longrightarrow L^2(\R^d)$ is the linear operator defined by \[\Phi\mu : z \in \R^d \mapsto \int_{\R^d \times [u_{\min},+\infty)^d} \prod_{k=1}^d \frac{1}{u_k} \varphi\left(\frac{z_k - t_k}{u_k}\right) \,\d \mu(t,u) \quad \forall \, \mu \in \M(\R^d \times [u_{\min},+\infty)^d)\,.\]

\medskip
Our aim is to recover $\mu^0$ from a $n$-sample $X_1,\ldots,X_n \iid f^0$ (see \eqref{eq:sample_distrib}), considering that the weights $\{a_j^0\}_j$, the location-scale parameters $\{x_j^0\}_j$ and the sparsity index $s$ are unknown. The empirical density associated with our sample $X_1,\ldots,X_n$ is defined as 
\[\hat{f}_n\coloneq\frac{1}{n}\sum_{i=1}^{n}\delta_{X_i}\,.
\] 
This empirical measure is smoothed, hence allowing for a comparison with any prediction $\Phi\mu$. We will use a Gaussian convolution with smoothing parameter $\tau>0$, introducing 
\[
\lambda: z \in \R^d \mapsto \frac{e^{-\frac{\norm{z}_2^2}{2\tau^2}}}{(2 \pi \tau^2)^{d/2}} \quad \text{and} \quad \Lambda \coloneq \F{\lambda}=e^{-\frac{\tau^2 \norm{\dotp}_2^2}{2}}.
\]
Then, we set
\[
L \circ
\hat{f}_n=\lambda *\hat{f}_n \quad \text{and} \quad L\circ f=\lambda * f \quad \forall f \in L^2(\R^d)\,.
\]
The term $L \circ \hat{f}_n$ is related to Kernel Density Estimation (KDE, \citep{tsybakov_kde}). The choice of $\tau$ is important if we want to estimate $f^0$ in addition to $\mu^0$, and will be discussed in Section \ref{section:prediction}. Our setting differs in this respect from super-resolution \citep{candes_fernandez_granda_2014}, where the analogous parameter $\lambda_c$ is imposed by the experimental conditions (namely, the frequency cut-off in that paper).

\paragraph{Hilbert space for the data fidelity term} The Hilbert space $\L$ in which we compare the observation and the prediction is the RKHS associated with $\lambda$. We define it using Mercer's theorem:
\[
\L \coloneq \left\lbrace f: \R^d \rightarrow \R\; : \; f \in L^2(\R^d) \text { s.t. }\|f\|_{\L}^2=\int_{\R^d} \frac{|\F{f}(\xi)|^2}{\F{\lambda}(\xi)} \d \xi<+\infty\right\rbrace,
\]
with dot product
\begin{equation}
\label{eq:dot_product_L}
 \forall f, g \in \L, \quad\innerprod{f}{g}_{\L}=\frac{1}{(2\pi)^d}\int_{\R^d} \frac{\overline{\F{f}(\xi)} \F{g}(\xi)}{\F{\lambda}(\xi)} \d \xi\,.
\end{equation}

\paragraph{Problem} We will search the target measure over the space of nonnegative Radon measures $\M(\X)^+$ where $\X$ is a compact set of $\R^d \times [u_{\min},+\infty)^d$. We will restrict the possible values of the covariance even further later.

\medskip
The BLASSO problem constructed in \citep{supermix} can be generalized to our setting. We may consider the following optimization problem:
\begin{equation}
\label{eq:pb_BLASSO_gaussian_kernel_not_norm}
 \underset{\mu \in \M(\X)^+}{\arg\min} \; J(\mu) \, 
 \quad \text{where} 
 \quad J(\mu) \coloneq\frac{1}{2}\norm{L\circ\hat{f}_n-L\circ \Phi \mu}_{\L}^2+\kappa\norm{\mu}_{\rm TV}\,, \;\; \kappa>0 \,.
 \tag{$\bar{\mathcal{P}}$}
\end{equation}
The data fidelity term $\frac{1}{2}\norm{L\circ\hat{f}_n-L\circ \Phi \mu}_{\L}^2$ compares smooth versions of the empirical density to any candidate for the predicted density function $\Phi \mu$. The regularization $\kappa\norm{\mu}_{\rm TV}$ promotes the sparsity of the solution. The choice of $\kappa$ will be discussed in Section \ref{section:estimation}: a balance must be found between the data fidelity term and the regularization term.

\medskip
However, we do not work with the loss $J$ defined by \eqref{eq:pb_BLASSO_gaussian_kernel_not_norm}.
Indeed, the model we consider is more complex than that of \citep{supermix}. Specifically, the addition of diagonal covariances in the parameterization of $\mu$ renders the problem \eqref{eq:pb_BLASSO_gaussian_kernel_not_norm} unsuitable for the use of standard proof techniques from the BLASSO literature \citep{blasso_poon}. 
The main difficulty stems from the associated kernel $\bar K$ describing the correlation between 2 features, \ie 
\[\bar K(x,x')\coloneq \innerprod{L\circ\Phi \delta_x}{L\circ\Phi \delta_x'}_\L \quad \forall \, x,x' \in \X\,.\] Unlike in the classical BLASSO setting, this kernel is not normalized: its diagonal value $\bar K(x,x)$ is not constant, but depends on the scale parameter $u$ through $\varphi_{(t,u)}$. By contrast, in the standard location model each component is a translate $\varphi(\dotp - t)$ of a fixed feature map, so that the correlation kernel is \emph{translation-invariant} and, in particular, satisfies $K(x,x)=1$ for all $x$. This normalization plays a important role in the BLASSO analysis, as it underlies the interpolation arguments used in the construction of dual certificates (see Section \ref{section:dual_certif}). 

\paragraph{Reparameterization}
To address this, we introduce a reparameterization that restores the normalization of the kernel. This is achieved by introducing a positive weighting function defined by
\begin{equation}\label{eq:def_W}
 W(x)\coloneq \prod_{k=1}^d (2\pi)^{-1/4}(2u_k^2+\tau^2)^{-1/4} \quad \forall \, x=((t_1,\ldots,t_d),(u_1,\ldots,u_d)) \in \R^d \times [u_{\min},+\infty)^d\,.
\end{equation} 
Setting $\Psi\delta_{(t,u)}\coloneq \frac{1}{W(t,u)} L\circ \Phi\delta_{(t,u)}$, we have
\[
  \norm{\Psi\delta_{(t,u)}}_{\L}=\frac{1}{W(t,u)} \norm{L\circ\Phi\delta_{(t,u)}}_{\L}=1 \quad\forall \,(t,u)\in\X,
\]
so that the correlation kernel $K_{\rm norm}=\innerprod{\Psi\delta_{(t,u)}}{\Psi\delta_{(t',u')}}_\L$ satisfies $K_{\rm norm}((t,u),(t,u))=1$ and is locally concave on the diagonal, which is precisely the hypothesis under which certificates exist \citep{off-the-grid_cs}. 

We point out that dealing with normalized designs is standard in high-dimensional statistics, up to a reparameterization of the target. In our framework, the operator $\Psi$ plays the role of a continuous design. 

Without loss of generality, we therefore consider from now on the following BLASSO problem:
\begin{equation}
\tag{$\mathcal{P}_{\kappa}$}
 \underset{\mu \in \M(\X)^+}{\min} \; J_{W}(\mu) \quad \text{where} \quad J_{W}(\mu)\coloneq \frac{1}{2}\norm{L\circ\hat{f}_n-L\circ \Phi \left(\frac{\mu}{W}\right)}_{\L}^2+\kappa\norm{\mu}_{\rm TV} \,,
 \label{eq:pb_BLASSO_gaussian_kernel_norm}
\end{equation}
with $\kappa>0$ and where we define, for all $\mu\in\M(\R^d\times [u_{\min},+\infty)^d)$, $\frac{\mu}{W}$ as the measure such that 
\[
\forall \eta \in \C_0(\R^d \times [u_{\min},+\infty)^d)\,, 
\quad \int_{\R^d \times [u_{\min},+\infty)^d} \eta \, \d \left(\frac{\mu}{W}\right) = \int_{\R^d\times [u_{\min},+\infty)^d} \frac{1}{W(x)}\eta(x) \, \d\mu(x)\,.
\] 
We can show that the optimization problem \eqref{eq:pb_BLASSO_gaussian_kernel_norm} has a solution (see the appendix, Proposition \ref{prop:existence_solution_blasso}).

\paragraph{Estimator} According to the problem \eqref{eq:pb_BLASSO_gaussian_kernel_norm}, we consider an estimator $\mu_{n,\omega}^\star$ defined as 
\begin{equation}\label{eq:mu_n_omega_star}
\mu_{n,\omega}^\star \in \underset{\mu \in \M(\X)^+}{\arg\min} \; J_{W}(\mu) \,.
\end{equation}
It is not a direct estimate of $\mu^0$, but rather of its reparameterized version, defined as the weighted measure
\begin{equation}
\mu_\omega^0 = \sum_{j=1}^s \omega_j^0 \delta_{x_j^0} \quad \text{with} \quad \omega_j^0 = W(x_j^0) a_j^0 \quad \forall j\in \lbrace 1,\dots, s \rbrace. 
\label{eq:mu_omega}
\end{equation}
Most of our results also hold for an \emph{approximate solution}, and we will use the notation $\hat{\mu}_{n,\omega}$ for our estimator when it suffices that
\begin{equation}\label{eq:hat_mu_n_omega}
\hat{\mu}_{n,\omega}\in \big\{ \mu \in \M(\X)^+\: : \: J_{W}(\mu) \leq J_W(\mu_\omega^0) \big\}\,.
\end{equation}
Remark that $\mu_{n,\omega}^\star$ satisfies the relaxed condition \eqref{eq:hat_mu_n_omega}.
This flexibility allows us to obtain the guarantees presented in Sections \ref{section:estimation} and~\ref{section:prediction} without requiring the exact minimization of $J_W$—this is computationally advantageous. While algorithmic details are beyond the scope of this paper, they will be briefly discussed in Section \ref{section:further_remarks} (see also \citep{cpgd_chizat,fastpart}). 
In contrast to other sections, the results of Section~\ref{section:NDSC} specifically require an \emph{exact solution} $\mu_{n,\omega}^\star$.

\begin{remark}
Historically, the BLASSO method has been applied to recover a sparse target measure that is not necessarily a probability measure.
As $\mu^0$ is a probability measure in our setting, one may instead consider the constrained problem 
\begin{equation} \label{eq:pb_BLASSO_constrained}
\tag{$\mathcal{P}_{C}$}
\underset{\mu \in \M(\X)^+, \: \norm{\mu}_{\rm TV}=1}{\arg\min} \; \frac{1}{2}\norm{L\circ\hat{f}_n-L\circ \Phi \mu}_{\L}^2 \,.
\end{equation}
Such an approach deserves some comments.
First, the unconstrained programs \eqref{eq:pb_BLASSO_gaussian_kernel_norm} and \eqref{eq:pb_BLASSO_gaussian_kernel_not_norm} belong to a family of convex potentials studied in optimization on the space of measures. A line of work, \eg \citep{cpgd_chizat,fastpart}, shows convergence of various gradient descent strategies (particle, Wasserstein gradient flows, stochastic gradient), while such results do not exist for the constrained program \eqref{eq:pb_BLASSO_constrained} to the best of our knowledge. 
Second, for the program that we deal with (see \eqref{eq:pb_BLASSO_gaussian_kernel_norm}), the reparameterized target $\mu_\omega^0$ does not verify $\norm{\mu_\omega^0}_{\rm TV}=1$, and enforcing the constraint $\norm{\frac{\hat{\mu}_{n,\omega}}{W}}_{\rm TV}=1$ is an obstacle to establishing guarantees.
\end{remark}

\section{Estimation guarantees}\label{section:estimation}
In this section, we give guarantees for the recovery of $\mu_\omega^0$ and $\mu^0$ using the estimator $\hat{\mu}_{n,\omega}$ introduced in \eqref{eq:hat_mu_n_omega}. We provide bounds on the difference in mass assigned by the estimator and the target measure over relevant regions of the space. Our guarantees hold under the existence of so-called non-degenerate \emph{dual certificates}, which are the key objects for analyzing the properties of the BLASSO estimator \citep{candes_fernandez_granda_2014,blasso_duval_peyre, supermix, off-the-grid_cs}. In a general context, they are functions interpolating the signs of the particles of the sparse target measure, with some prescribed smoothness and shape constraints. We investigate the construction of such objects in Section \ref{section:dual_certif}. 

\subsection{Non-degenerate dual certificates} In this section, we define non-degenerate certificates and introduce their properties, closely related to the optimization problem at hand.
These certificates are linked to the \emph{feature map}, which is defined as the linear operator \[\Psi: \M(\R^d \times [u_{\min},+\infty)^d) \rightarrow \L \: , \: \mu \mapsto L \circ \Phi\frac{\mu}{W}\,.\] 
Remark that the loss function of \eqref{eq:pb_BLASSO_gaussian_kernel_norm} can be rewritten as $J_W: \mu \in \M(\X)^+ \mapsto \frac{1}{2} \norm{L \circ \hat{f}_n - \Psi \mu}_\L^2 + \kappa \norm{\mu}_{\rm TV}$.
The adjoint operator of $\Psi$ restricted to $\M(\X)$ verifies 
\begin{equation}\label{eq:scalar_product_Psi*}
 \innerprod{\Psi\mu}{p}_{\L}=\innerprod{\Psi^*p}{\mu}_{\C(\X),\M(\X)}=\int \Psi^*p \, \d \mu \quad \forall\, p \in \L 
\end{equation}
for all $\mu \in \M(\X)$. 
In particular, for all $x \in \X$, 
 $\left[\Psi^*p\right] (x)=\innerprod{p}{\Psi\delta_{x}}_{\L}$.

\medskip
Dual certificates are continuous functions of the form $\Psi^* p$ for some $p \in \L$. Our estimation guarantees are based on controls of these certificates on regions defined by $\mu^0$ and some appropriate distance. In \citep{supermix} the Euclidean distance is used. The latter is however not adapted to the geometry of our problem, as we deal with unknown covariances. In particular, the associated kernel is not translation-invariant, see Section~\ref{section:dual_certif}. \citep{off-the-grid_cs} work alternatively with the Fisher-Rao distance, which is however impractical to use in our context. In fact, interactions between the kernel and the Fisher-Rao appear to be quite intricate to manage. A key aspect of our contribution is to introduce greater flexibility in the control of certificates. We relax the assumptions from \citep{off-the-grid_cs} which are challenging to verify in our context, by expressing controls using a specific distance on regions defined by a different semi-distance.
More specifically, we use the semi-distance $\mathpzc{d}$ defined by 
\begin{equation}\label{eq:def_semi_distance}
 \mathpzc{d}(x,x')^2= \sum_{k=1}^d \left(\frac{(t_k-t_k')^2}{{u_k'}^2+u_k^2+\tau^2} + \ln\left(\frac{{u_k'}^2+u_k^2+\tau^2}{\sqrt{2u_k^2+\tau^2}\sqrt{2{u_k'}^2+\tau^2}}\right)\right) \quad \forall\, x,x'\in \R^d \times [u_{\min},+\infty)^d\,.
\end{equation} 
This semi-distance is symmetric, nonnegative, verifies $\mathpzc{d}(x,x')=0 \iff x=x'$ for all $x, x'\in \R^d \times [u_{\min},+\infty)^d$, but the triangle inequality does not hold. 

We use this semi-distance because it is closely tied to our problem and to the feature map. Section \ref{section:dual_certif} makes this connection with the optimization problem \eqref{eq:pb_BLASSO_gaussian_kernel_norm} precise.
Moreover, $\mathpzc{d}$ rescales the separation between means by the corresponding variances, which is statistically natural: two Gaussian components with identical means but large variances are effectively ``closer'' than those with small variances.

\medskip
Controls of the certificates also depend on some distance $\mathfrak{d}_{\mathfrak{g}}$ on $\X$, left unspecified in this section. We will see in Section \ref{section:dual_certif} that the Fisher-Rao distance is well-suited for this problem. To establish the existence of non-degenerate dual certificates, we ensure compatibility between the semi-distance and $\mathfrak{d}_\mathfrak{g}$ (see Section \ref{section:dual_certif}). Up to the semi-distance $\mathpzc{d}$, the elements displayed in this section are quite generic in the BLASSO literature and have been discussed in various contexts \citep{blasso_duval_peyre,supermix,blasso_poon}. 

First, we introduce the so-called \emph{far} and \emph{near} regions.

\begin{definition}[Near and far regions] \label{def:near_far_regions}
 For $r>0$, we define the near regions associated with each particle $x_j^0$, $j\in\{1,\ldots,s\}$ as \[\X_j^{\rm near}(r)\coloneq \{x \in \X \: :\: \mathpzc{d}(x,x_j^0)\leq r\}\] and the far region \[\X^{\rm far}(r)\coloneq \X \setminus \X^{\rm near}(r) \quad \text{where} \quad \X^{\rm near}(r) \coloneq \bigcup_j \X_j^{\rm near}(r)\,.\]
\end{definition}
\noindent
The jth near region identifies the points in $\X$ that are close to $x_j^0$, the jth particle of $\mu^0$. To ensure accurate recovery, we require the existence of a (global) dual certificate $\eta$ that satisfies non-degeneracy conditions. These conditions are adapted below from \citep[Definition 2]{off-the-grid_cs}, with modifications to account for the definition of near and far regions, all based on the semi-distance $\mathpzc{d}$ rather than the Fisher-Rao distance. In addition, the restriction to nonnegative measures in \eqref{eq:pb_BLASSO_gaussian_kernel_norm} relaxes the conditions required for the negative part of the global dual certificate.

\begin{definition}[Global non-degenerate certificate] \label{def:non_degenerate_certificate}
Let $\varepsilon_0,\varepsilon_2,r \in \R_+^*$. A function $\eta \in {\mathrm{Im}}(\Psi^*)$ is a $(\varepsilon_0,\varepsilon_2,r)$-non-degenerate certificate associated with the measure $\mu^0$ if 
\begin{enumerate}
 \item $\eta(x_j^0)=1$ for all $j=1,\ldots,s$,
 \item $\eta(x)\leq 1-\varepsilon_0$ for all $x\in \X^{\rm far}(r)$,
 \item $\eta(x)\leq 1-\varepsilon_2 \mathfrak{d}_{\mathfrak{g}}(x,x_j^0)^2$ for all $x\in \X_j^{\rm near}(r)$, $j\in\{1,\ldots,s\}$.
\end{enumerate}
\end{definition}

\noindent
A global non-degenerate certificate is thus a function whose regularity is prescribed (it must be in $\mathrm{Im}(\Psi^*)$), which interpolates (item 1 of Definition \ref{def:non_degenerate_certificate}) and localizes (items 2 and 3) the particles of the target $\mu^0$.

\medskip

We will make use of additional non-degenerate certificates, localizing the jth particle of $\mu^0$.
\begin{definition}[Local non-degenerate certificates]\label{def:local_non_degenerate_certificate}
Let $\tilde \varepsilon_0,\tilde \varepsilon_2,r\in \R_+^*$. For $j\in \{1,\ldots,s\}$, $\eta_j \in {\mathrm{Im}}(\Psi^*)$ is a $(\tilde\varepsilon_0,\tilde\varepsilon_2,r)$-non-degenerate certificate for the jth near region if 
\begin{enumerate}
 \item $\eta_j(x_j^0)=1$ and $\eta_j(x_i^0)=0$ for all $i\in \{1,\ldots,s\}$ such that $i \neq j$,
 \item $|\eta_j(x)|\leq 1-\tilde\varepsilon_0$ for all $x\in \X^{\rm far}(r)$,
 \item $|1-\eta_j(x)|\leq \tilde\varepsilon_2 \mathfrak{d}_{\mathfrak{g}}(x,x_j^0)^2$ for all $x\in \X_j^{\rm near}(r)$,
 \item $|\eta_j(x)|\leq \tilde\varepsilon_2 \mathfrak{d}_{\mathfrak{g}}(x,x_i^0)^2$ for all $x\in \X_i^{\rm near}(r)$, for all $i\in \{1,\ldots,s\}$ such that $i \neq j$. 
\end{enumerate}
\end{definition}

\medskip
Note that the near and far regions along with $\Psi^*$ depend on the parameter $\tau$ appearing in $\lambda$. The distance~$\mathfrak{d}_\mathfrak{g}$ will also depend on it. For the estimation of $\mu^0$, the choice of $\tau$ is related via $(x_1^0,\ldots,x_s^0)$ to the possibility of constructing non-degenerate certificates (see Section \ref{section:dual_certif}), and influences the convergence rate (see Theorem \ref{th:estimation_error}). For the estimation of the density $f^0$, this choice is more critical: we suggest possible values for~$\tau$ in Section \ref{section:prediction}.

\subsection{Error bounds}
We extend here the results of \citep{supermix} to our setting, where both means and covariances of the mixture model are unknown. Providing estimation guarantees requires to bound the error we made by approximating $L\circ f^0$ from our random observations (see Lemma \ref{lemma:control_noise}).
We give results taken in expected value over $X_1, \ldots, X_n \iid f^0$.

\begin{lemma}[Control of the noise. {\citep[Lemma~3]{supermix_supp}}]\label{lemma:control_noise}
We define the so-called noise term $\Gamma_n$ as
\begin{equation}\label{eq:def_noise}
 \Gamma_n \coloneq L\circ\hat{f}_n - L \circ \Phi \mu^0\,. 
\end{equation}
We have
\[\E{\norm{\Gamma_n}_{\L}^2} \leq\frac{4 \int_{\R^d} \Lambda}{(2\pi)^d n}=\frac{4}{(2\pi)^{d/2}\tau^d n} \eqcolon \rho_n^2\,.\]
Moreover, universal constants $C_{\Gamma},\tilde C_{\Gamma}>0$ exist such that
\[\E{\norm{\Gamma_n}_{\L}^4} \leq \tilde C_\Gamma \rho_n^4\] 
and
\[\forall \rho >0\,, \quad \norm{\Gamma_n}_{\L}^2\leq \rho \frac{C_{\Gamma}^2 \int_\R \Lambda}{n(2\pi)^d}= \rho \frac{C_{\Gamma}^2 }{n\tau^d(2\pi)^{d/2}}\quad \text{with probability greater than } 1-C_{\Gamma}e^{-\rho} \,.\]
\end{lemma}
\noindent
The proof is given in Appendix \ref{section:proof_control_noise}. The constants $C_\Gamma, \tilde C_\Gamma$ in the previous lemma can be made explicit (see \citep[Theorem~3.1]{houdre_reynaud_u_stats}).

\medskip
The control of any estimator $\hat{\mu}_{n,\omega}$ satisfying \eqref{eq:hat_mu_n_omega} also involves the non-degenerate dual certificates introduced in Definitions \ref{def:non_degenerate_certificate} and \ref{def:local_non_degenerate_certificate} above.
 Our recovery guarantees are based on the following assumption.
\begin{assumption} \label{assumption:existence_non-degenerate_certif}
 There exists $\eta=\Psi^* p$ a global $(\varepsilon_0,\varepsilon_2,r)$-non-degenerate certificate associated with $\mu^0$, and local $(\tilde \varepsilon_0, \tilde \varepsilon_2, r)$-non-degenerate certificates $\eta_j=\Psi^* p_j$ ($j \in \{1,\ldots,s\}$) for each near region. Moreover, there exists $c_p>0$ such that $\norm{p}_\L^2 \leq c_p s$ and $\norm{p_j}_\L^2 \leq c_p$ for $j=1,\ldots,s$.
\end{assumption}
\noindent
The dependence on $s$ for $\norm{p}_L^2$ is quite natural (related to our construction of certificates, see Proposition~\ref{prop:norm_certif} in the appendix) and already appears in the literature, \cf \citep{blasso_poon}. 
In Section \ref{section:dual_certif}, we show that Assumption \ref{assumption:existence_non-degenerate_certif} holds under some conditions on $\X$ and $\mu^0$ by explicitly constructing certificates. In particular, we consider $\X \subset \R^d \times [u_{\min}, u_{\max}]^d$ and we impose a minimal separation between the particles of $\mu^0$ depending on $s, u_{\min}, u_{\max}, d, \tau$. We provide fixed values for $r, \varepsilon_i, \tilde \varepsilon_i, c_p$, depending only on the dimension $d$. For a precise statement, see Theorem \ref{th:main_theorem_certif} below. 

\medskip
Now, we have all the ingredients to provide our first results concerning the performances of our estimator. We present a control of the mass of $\hat{\mu}_{n,\omega}$ on the near and far regions (Definition \ref{def:near_far_regions}) in the next theorem, whose proof is displayed in Appendix \ref{section:proof_th_estimation_error}.
\begin{theorem}[Estimation error]
 \label{th:estimation_error}
 Assume that Assumption \ref{assumption:existence_non-degenerate_certif} holds.
Setting $\kappa=\frac{\rho_n}{\sqrt{c_p}}$ in \eqref{eq:pb_BLASSO_gaussian_kernel_norm}, we have the following controls on $\hat{\mu}_{n,\omega}$.
\begin{enumerate}
 \item Control over the mass of the estimator on the far region:
\[\E{\hat{\mu}_{n,\omega} (\X^{\rm far}(r))} \leq \frac{ \sqrt{c_p}}{2 \varepsilon_0}\rho_n(1+ \sqrt{s})^2 \, .\]
\item Accuracy of the mass reconstruction: for any $j \in \{1,\ldots,s\}$,
\[\E{\left|\omega_j^0-\hat{\mu}_{n,\omega}(\X_j^{\rm near}(r))\right|} \leq 2\sqrt{c_p}\rho_n(1 + \sqrt{s}) +\max\left\{\frac{1-\tilde \varepsilon_0}{\varepsilon_0},\frac{\tilde\varepsilon_2}{\varepsilon_2} \right\} \frac{\sqrt{c_p}}{2}\rho_n(1+ \sqrt{s})^2\,.\]
\label{eq:local-soft-thresholding}
 \item Stability of the mass: 
 \[ - 2\sqrt{c_ps}\rho_n\left(1+ \sqrt{s}\right) \leq \E{\norm{\hat{\mu}_{n,\omega}}_{\rm TV}}- \norm{\mu_{\omega}^0}_{\rm TV} \leq \frac{\sqrt{c_p}}{2}\rho_n\,.\] 
 \label{eq:soft-thresholding}
\end{enumerate}
\end{theorem}

\noindent
We present below several remarks on this result and its proof.

\begin{remark}[A Bregman divergence approach]
The non-degenerate certificates from Assumption \ref{assumption:existence_non-degenerate_certif} allow us to derive recovery guarantees from the control of the so-called Bregman divergence, defined for $\eta \in \C(\X)$ by 
\begin{equation}\label{eq:def_bregman_divergence}
 D_{\eta}(\hat{\mu}_{n,\omega},\mu_{\omega}^0)=\norm{\hat{\mu}_{n,\omega}}_{\rm TV}-\norm{\mu_{\omega}^0}_{\rm TV}-\int_{\X}\eta\, \d\left(\hat{\mu}_{n,\omega}-\mu_{\omega}^0 \right)\,.
\end{equation}
In particular, if $\eta$ is a global non-degenerate certificate, we get $\int_\X \eta \,\d\mu_{\omega}^0=\norm{\mu_{\omega}^0}_{\rm TV}$. In such a case, the Bregman divergence is nonnegative: $D_{\eta}(\hat{\mu}_{n,\omega},\mu_{\omega}^0) = \int_\X (1- \eta ) \, \d \hat{\mu}_{n,\omega}\geq 0$.
The proof of Theorem \ref{th:estimation_error} is based on lower and upper bounds on the Bregman divergence with $\eta$ from Assumption \ref{assumption:existence_non-degenerate_certif}. First, using that $\eta=\Psi^* p$ with the Cauchy-Schwarz inequality, we get \[\left|\int_\X \eta \, \d(\hat{\mu}_{n,\omega}-\mu_{\omega}^0)\right|=|\innerprod{p}{\Psi(\hat{\mu}_{n,\omega}-\mu_{\omega}^0)}_\L|\leq \norm{p}_\L \norm{\Psi(\hat{\mu}_{n,\omega}-\mu_{\omega}^0)}_\L \,.\]
We control this quantity by using the inequality $J_W(\hat{\mu}_{n,\omega})\leq J_W(\mu_{\omega}^0)$ (recall \eqref{eq:hat_mu_n_omega}), along with the control on the noise term established in Lemma \ref{lemma:control_noise}. This leads to the following upper bound on the Bregman divergence: \[\E{D_{\eta}\left(\hat{\mu}_{n,\omega}, \mu_{\omega}^0\right)} \leq \frac{\sqrt{c_p}}{2}\rho_n(1+ \sqrt{s})^2\,.\]
At the same time, since $\eta$ satisfies Definition \ref{def:non_degenerate_certificate}, we get
\[D_{\eta}\left(\hat{\mu}_{n,\omega}, \mu_\omega^0\right)= \int (1- \eta) \, \d \hat{\mu}_{n,\omega} \geq \varepsilon_0 \hat{\mu}_{n,\omega}(\X^{\rm far}(r))+ \varepsilon_2 \sum_{j=1}^{s} \int_{\X_j^{\rm near}(r)} \mathfrak{d}_{\mathfrak{g}}(x,x_j^0)^2 \, \d \hat{\mu}_{n,\omega}(x)\] from which we deduce the control of the estimator on the far region. The local dual certificates enable us to retrieve the control on the near regions, using again the lower bound on the Bregman divergence.
\end{remark}

\medskip
\begin{remark}[Choice of $\kappa$, dependence on $s$]\label{remark:choice_kappa_estim}
 The regularization parameter $\kappa$ is chosen as $\frac{\rho_n}{\sqrt{c_p}}$ in Theorem \ref{th:estimation_error}. In particular, it does not depend on the sparsity index $s$ of $\mu^0$, unknown in practice—we call this choice \emph{$s$-agnostic}. As a result, the user does not require prior knowledge of the target measure to select the regularization parameter: the proposed value can be directly used to solve the BLASSO numerically. The \emph{$s$-dependent} choice $\kappa=\frac{\rho_n}{\sqrt{s c_p}}$ results in better rates for the estimation (\ie, for the bounds presented in Theorem \ref{th:estimation_error})—linear on~$\sqrt{s}$ rather than on $s$ (see Equation \eqref{eq:estim_kappa_s_dependent} in Appendix). In all cases, the estimation error depends on $s$; recovering a mixture with a larger number of components incurs a higher estimation cost.
To conclude this discussion, we stress that these possible choices for $\kappa$ are proposed according to theoretical considerations. For practical applications, this regularization parameter can be calibrated via cross-validation.
\end{remark}

\medskip
\begin{remark}[Soft-thresholding effect]
    It is known that the LASSO estimator has a soft-thresholding effect, \eg \citep{friedman2007pathwise}. It suggests that the $\ell_1$-regularization is biased in the sense that each of the $s$ weight components is under-estimated by an additive factor proportional to the regularization term $\kappa$. Having~$s$ components, we expect this bias to be of the order of $s\kappa$. Our result in \eqref{eq:soft-thresholding} (Theorem \ref{th:estimation_error}) is aligned with this comment since 
    \[ \Big|\E{\norm{\hat{\mu}_{n,\omega}}_{\rm TV}}- \norm{\mu_{\omega}^0}_{\rm TV}\Big|=O(s\kappa)\,,
    \]
    up to a constant that may depend on $c_p$.
\end{remark}

\medskip

\begin{remark}[Choice of $\tau$]\label{remark:choice_tau_estim} Although this section focuses on the estimation of $\mu^0$, an additional objective that can be pursued in parallel is the estimation of the associated density $f^0$—a task we refer to as \emph{prediction}. This prediction objective influences the choice of the smoothing parameter. As we will see in Section~\ref{section:prediction}, when $\X \subset \R^d \times [u_{\min},+\infty)^d$ the choice $\tau=\sqrt{2} \frac{u_{\min}}{\sqrt{\ln n}}$ results in almost parametric convergence rates for both the estimation of $f^0$ (Theorem \ref{th:prediction_error}) and the estimation of $\mu_\omega^0$. In fact, under Assumption \ref{assumption:existence_non-degenerate_certif}, Theorem \ref{th:estimation_error} entails that with $\tau=\sqrt{2} \frac{u_{\min}}{\sqrt{\ln n}}$, setting $\kappa=\frac{\rho_n}{\sqrt{c_p}}$, keeping only the dependence on $n$ and $s$ we have $\E{\left|\omega_j^0-\hat{\mu}_{n,\omega}(\X_j^{\rm near}(r))\right|} \lesssim \frac{s(\ln n)^{d/4}}{\sqrt{n}}$ for all $j=1,\ldots,s$.

If we do not need to predict the density $f^0$, we can alternatively take $0<\tau\leq u_{\min}$ fixed. This leads to a rate of $\frac{s}{\sqrt{n}}$ for the estimation. Note that Assumption \ref{assumption:existence_non-degenerate_certif} is more difficult to check when $\tau$ is large, \cf Theorem~\ref{th:main_theorem_certif}.
Also note that the near and far regions depend on $\tau$ through the semi-distance $\mathpzc{d}$.
\end{remark}

\subsection{Effective near regions}
Theorem \ref{th:estimation_error} describes the performances of the BLASSO estimator in our setting. It provides control over the proximity between $\hat\mu_{n,\omega}$ and $\mu_\omega^0$ on associated far and near regions. However, our initial target is $\mu^0$ instead of $\mu_\omega^0$. Using Theorem \ref{th:estimation_error} and \eqref{eq:mu_omega}, we can easily give a bound for $\E{\left|a_j^0-\frac{\hat{\mu}_{n,\omega}}{W(x_j^0)}(\X_j^{\rm near}(r))\right|}$, but we cannot access $W(x_j^0)$ without knowledge of $\mu^0$. In this context, a fair estimator of $\mu^0$ is the renormalized measure $\frac{\hat{\mu}_{n,\omega}}{W}$. Providing recovery guarantees for $\frac{\hat{\mu}_{n,\omega}}{W}$ requires to control $W$ on the near regions.

\medskip
Another restrictive aspect of Theorem \ref{th:estimation_error} is that we consider controls on regions with a fixed radius $r$ (which is directly related to the dual certificate). We would like to locate the mass of $\hat{\mu}_{n,\omega}$ more precisely.

However, we can overcome these initial limitations by providing controls of the estimator on 
\begin{equation}\label{eq:def_effective_near_regions}
 \X_j^{\rm near}(r_e) \coloneq \{x \in \X \: :\: \mathpzc{d}(x,x_j^0) \leq r_e\} 
\end{equation}
for $r_e \leq r$. Such regions are called effective near regions, and have been introduced by \citep{decastro_gribonval_jouvin}. 
To extend controls on $\X_j^{\rm near}(r)$ to controls on $\X_j^{\rm near}(r_e)$, the choice of $\mathfrak{d}_{\mathfrak{g}}$ plays an important role. It should be indeed compatible with our semi-distance $\mathpzc{d}$ in a sense which is made precise in Proposition \ref{prop:effective_near_regions} below.

\begin{proposition}\label{prop:effective_near_regions}
Let $r>0$. Assume that Assumption \ref{assumption:existence_non-degenerate_certif} holds.
Assume that there exists $\tilde\varepsilon_3>0$ such that for all $0<r_e \leq r$, for all $j\in\{1\ldots,s\}$, we have 
\begin{equation}\label{eq:assumption_tilde_varepsilon_3}
\mathfrak{d}_{\mathfrak{g}}(x_j^0,x)^2 \geq \frac{r_e^2}{\tilde\varepsilon_3} \quad \forall\,x\in \X_j^{\rm near}(r) \setminus \X_j^{\rm near}(r_e)\,.
\end{equation}
Then, for $\kappa=\frac{\rho_n}{\sqrt{c_p}}$ and for any $0<r_e\leq r$,
\begin{equation}\label{eq:bound_omega_effective_near}
 \E{|\omega_j^0 -\hat{\mu}_{n,\omega}(\X_j^{\rm near}(r_e))| }\leq 2\sqrt{c_p} \rho_n(1+\sqrt{s}) + \max\left\{ \frac{1-\tilde \varepsilon_0}{\varepsilon_0} , \frac{1}{\varepsilon_2}\left(\frac{\tilde\varepsilon_3}{r_e^2}+ \tilde\varepsilon_2\right) \right\} \rho_n\frac{\sqrt{c_p}}{2}(1+\sqrt{s})^2 \,.
 \end{equation}
\end{proposition}
\noindent
The proof is similar to that of Theorem \ref{th:estimation_error}, and is presented in Appendix \ref{section:proof_prop_effective_near}.

\medskip
We check the assumption \eqref{eq:assumption_tilde_varepsilon_3} in Lemma \ref{lemma:local_control_metric_eps3_dim_d} (in Appendix), for $\mathfrak{d}_\mathfrak{g}$ chosen as the Fisher-Rao distance. For the specific choice of $r$ which is made in Section \ref{section:dual_certif}, (depending on $d$), we show in Lemma \ref{lemma:tilde_varepsilon_3} that $\tilde \varepsilon_3$ can be chosen as a constant not depending on $d$. Note that, as we decrease the size of near region $r_e$ towards zero, the bound \eqref{eq:bound_omega_effective_near} grows as $\frac{s\rho_n}{r_e^2}$, omitting the dependence on $c_p$,  $\varepsilon_0$, $\varepsilon_2$, $\tilde \varepsilon_0$, $\tilde \varepsilon_2$ and $\tilde \varepsilon_3$. 
\begin{remark}
We can make choices for $r_e$ that depend on $n$ in Proposition \ref{prop:effective_near_regions}. We give the associated bound for $\E{|\omega_j^0 -\hat{\mu}_{n,\omega}(\X_j^{\rm near}(r_e))| }$ (see \eqref{eq:bound_omega_effective_near}), that holds under the assumptions of Proposition \ref{prop:effective_near_regions} and with $n$ large enough such that $r_e\leq r$. We only keep the dependence on $n,\tau$ and $s$.
Let $\alpha>0$.
\begin{itemize}
 \item For $r_e=(\ln n)^{-\alpha}$, \[\E{|\omega_j^0 -\hat{\mu}_{n,\omega}(\X_j^{\rm near}((\ln n)^{-\alpha}))|} \lesssim \frac{s(\ln n)^{2\alpha}}{\sqrt{n}\tau^{d/2}}\,.\] 
 \item For $r_e=n^{-\alpha}$, $\alpha< \frac{1}{4}$, \[\E{|\omega_j^0 -\hat{\mu}_{n,\omega}(\X_j^{\rm near}(n^{-\alpha}))|} \lesssim \frac{s}{n^{\frac{1}{2}- 2\alpha} \,\tau^{d/2}}\,.\] 
\end{itemize}
These choices provide an overview of possible convergence rates, and show the trade-off between localization and control of the mass.
\end{remark}

\medskip
The above proposition provides, as Theorem \ref{th:estimation_error}, guarantees on the estimation on $\mu_\omega^0$ rather than on $\mu^0$. However, by combining Proposition \ref{prop:effective_near_regions} and a control of the function $W$ on the effective regions, we can derive guarantees for the estimation of $\mu^0$ by $\hat{\mu}_n\coloneq \frac{\hat{\mu}_{n,\omega}}{W}$, as displayed in the following corollary. 
\begin{corollary}\label{cor:control_renormalized_estimate}
 We work under the same assumption as Proposition \ref{prop:effective_near_regions}. We choose $\kappa=\frac{\rho_n}{\sqrt{c_p}}$ and $r_e=n^{-1/6}$. Assume that $0<r_e \leq r$. Omitting the dependence on $d$, $c_p$, $\varepsilon_i$, $\tilde \varepsilon_i$, $r$ we have 
 \[\E{\left|a_j^0-\frac{\hat{\mu}_{n,\omega}}{W}(\X_j^{\rm near}(n^{-1/6}))\right|} \lesssim \left(s\tau^{-d/2}W(x_j^0)^{-1}+a_j^0\right) n^{-1/6}\,.\]
\end{corollary}
\noindent
The proof can be found in Appendix \ref{section:proof_cor_effective_near}. The difficulty that appears here is that providing control on $\frac{\hat{\mu}_{n,\omega}}{W}$ requires local control of the function $W$, which slightly deteriorates the bound for the estimation of the weights $(a_j^0)_{j=1}^s$. The rate $n^{-1/6}$ provides the best compromise for the size of the effective near regions $r_e$ when dealing with this upper bound.

\begin{remark}\label{remark:choice_kappa_cor_near_effective}
 Making the $s$-dependent choice of regularization $\kappa=\frac{\rho_n}{\sqrt{c_p s}}$, a modified version of Proposition \ref{prop:effective_near_regions} yields, keeping only the dependence on $r_e, \tau, s, n$, \[\E{\left|\omega_j^0-\hat{\mu}_{n,\omega}(\X_j^{\rm near}(r_e))\right|} \lesssim \frac{\sqrt{s}}{\tau^{d/2}\sqrt{n}r_e^2}\,,\] \cf \eqref{eq:prop_effective_near_kappa_s_dep} in Appendix. With this choice of $\kappa$, the result of Corollary \ref{cor:control_renormalized_estimate} becomes \[\E{\left|a_j^0-\frac{\hat{\mu}_{n,\omega}}{W}(\X_j^{\rm near}(n^{-1/6}))\right|} \lesssim \left(\sqrt{s}\tau^{-d/2}W(x_j^0)^{-1}+a_j^0\right) n^{-1/6}\,,\] see \eqref{eq:cor_control_renormalized_estimate_kappa_s_dep} in Appendix.
\end{remark}

\noindent
Remark that all the results given in this section are controls in expected value. Controls with high probability could also be given, using Lemma \ref{lemma:control_noise}. Also note that the target $\mu^0$ is fixed (does not vary with $n$).

\section{Prediction}\label{section:prediction}
In Section \ref{section:estimation}, we provided controls for the estimator $\hat{\mu}_{n,\omega}$.
In this section, we look at the prediction $\Phi \hat{\mu}_n=\Phi\frac{\hat{\mu}_{n,\omega}}{W}$ made by the BLASSO of the target density $\Phi \mu^0=f^0$, and we derive bounds for the so-called prediction error $\norm{\Phi(\hat{\mu}_n-\mu^0)}_{L^2}^2$. This is a question that has been quite overlooked in the BLASSO literature. We can however mention \citep{butucea_prediction} where the prediction error is investigated in a different context and for a slightly different problem.

\medskip
We establish in this section almost parametric rates for the bound on the prediction error, in two distinct regimes characterized by the value of the regularization parameter $\kappa$. With small regularization, no assumptions on $\mu^0$ are needed (Section \ref{section:control_prediction_small_reg}).
With stronger regularization, we show in Section \ref{section:control_prediction_strong_reg} that the BLASSO achieves a good prediction when dual certificates exist (namely when Assumption \ref{assumption:existence_non-degenerate_certif} holds).

\subsection{Prediction under small regularization}\label{section:control_prediction_small_reg}
We can achieve good prediction results with small regularization, using control of the low frequencies of the estimated density (using Lemma \ref{lemma:control_noise}) and of its high frequencies, resulting from an appropriate choice of $\tau$. These controls require an upper bound on the variance.
We hence assume from now on that $\X \subset \R^d \times [u_{\min},u_{\max}]^d$.

\begin{proposition}[Prediction error under small regularization]\label{prop:prediction_under_small_reg} Assume that $\X \subset \R^d \times [u_{\min},u_{\max}]^d$.
Choosing $\tau= \frac{\sqrt{2}u_{\min}}{\sqrt{\ln n}}$ and $\kappa=\rho_n^2$, keeping only the dependence on $n$ we have 
\[\E{\norm{\Phi(\hat{\mu}_n - \mu^0)}_{L^2}^2} \lesssim \frac{(\ln n)^{d/2}}{n} \,.\]
\end{proposition}
\noindent
We do not make any assumption on the existence of dual certificates in the above proposition. The proof is given in Appendix \ref{section:proof_prop_prediction_under_small_reg}. This result shows that in the small regularization regime corresponding to $\kappa=\rho_n^2$, we obtain a quasi-parametric rate (up to a log factor) for the estimation of the density $f^0$.

\begin{remark}[Choice of $\tau$, bounds on the variances]
Keeping $\tau$ fixed does not lead to a good prediction rate: to control the high frequencies of the predicted density with Lemma \ref{lemma:control_high_freq} in Appendix, we need to lower bound the variances of $\X$ with a parameter not depending on $n$, while $\tau$ must decrease as $n$ grows. 
In Proposition \ref{prop:prediction_under_small_reg}, we also use an upper bound on the variances in $\X$ to control $\E{\norm{\hat{\mu}_n}_{\rm TV}^2}$. 
\end{remark}

\begin{remark}[Comparison with Kernel Density Estimation]
Note that $L \circ \hat{f}_n$ is itself a kernel density estimator of the density (see \citep{tsybakov_kde}). However, the resulting rate of convergence deteriorates very quickly with the dimension: it can be shown that with $\X \subset \R^d \times [u_{\min},+\infty)^d$, setting $\tau=\frac{1}{\sqrt{\ln n}n^{\frac{1}{4+d}}}$, omitting the dependence on $d$ we have 
 \[\E{\norm{L \circ \hat{f}_n- \Phi \mu^0}_{L^2}^2} \lesssim \frac{(\ln n)^{d/2}}{n^{\frac{4}{d+4}}u_{\min}^{d+4}}\,.\] 
 The proof is in Appendix \ref{section:proof_prediction_kde}. This bound can be explained by the fact that the Gaussian kernel is not appropriate for the estimation of so-called super-smooth target densities. Nevertheless, as displayed in Proposition~\ref{prop:prediction_under_small_reg}, an almost parametric rate is obtained for $\Phi \hat \mu_n$, which is based on a regularized version of $L \circ \hat{f}_n$.
\end{remark}

\begin{remark}[On the $(\ln n)^{d/2}$ prediction factor]\label{rem:logn}
The logarithmic factor originates in the construction of the observation: the embedding smooths the empirical measure at scale $\tau$, and the choice
$\tau \asymp \frac{u_{\min}}{\sqrt{\ln n}}$ balances the high-frequency bias of the embedded density
against the low-frequency control of the noise, yielding the power $d/2$. Whether such a factor
is intrinsic depends on the class. For location mixtures with a \emph{known} upper bound $k$ on
the number of components and known common isotropic covariance, the minimax rate of density
estimation is $\sqrt{d/n}$ in Hellinger distance, with no logarithmic factor
\citep{doss2023optimal}. That benchmark is not directly comparable to ours: it takes $k$ as
input whereas our procedure is order-free, it concerns homoscedastic location mixtures, and it
is stated in Hellinger distance whereas our bound is in $L^2$ (for scales bounded below,
$\norm{p-q}_{L^2}^2 \lesssim u_{\min}^{-d}\, h^2(p,q)$, so Hellinger guarantees are the
stronger ones). In the order-free regime, by contrast, logarithmic factors are present in all
known bounds: the NPMLE for $d$-variate Gaussian location mixtures attains a squared Hellinger
risk of order $(\ln n)^{d+1}/n$ \citep{SahaGuntuboyina2020}, and minimax lower bounds for
univariate Gaussian mixtures with unrestricted mixing already exhibit a logarithmic factor
\citep{Kim2014}. Our log-power $d/2$ is milder than the NPMLE's $d+1$; we do not claim it is
optimal.
\end{remark}

\paragraph{Choice of regularization}
Proposition \ref{prop:prediction_under_small_reg} does not take advantage on the fact that $\mu^0$ is a discrete measure. This is not surprising, as the considered regularization is very small. 
The choice $\kappa=\rho_n^2$ is not suited for the estimation of the target measure $\mu^0$ (see Section \ref{section:estimation}). In fact, the result displayed in Proposition \ref{prop:prediction_under_small_reg} does not take into account the trade-off we would like to have between making a good prediction and guaranteeing a good estimation of the target measure $\mu^0$.
In particular, choosing $\kappa = \rho_n^2$, we have no control over the proximity between $\norm{\mu_\omega^0}_{\rm TV}$ and $\norm{\hat{\mu}_{n,\omega}}_{\rm TV}$ as $n$ grows: we only know that $\E{\norm{\hat{\mu}_{n,\omega}}_{\rm TV}} \leq \norm{\mu_\omega^0}_{\rm TV} + \frac{1}{2}$ (as $\kappa \E{\norm{\hat{\mu}_{n,\omega}}_{\rm TV}} \leq \frac{\rho_n^2}{2}+ \kappa\norm{\mu_\omega^0}_{\rm TV}$), which is enough to control the high frequencies of $\Phi\hat{\mu}_n$. 
For an estimation purpose, we need more regularization (\ie a larger $\kappa$). This is exactly what has been done in Theorem \ref{th:estimation_error} where $\kappa=\frac{\rho_n}{\sqrt{c_p}}$.

\begin{remark}
This property of the BLASSO estimator, \ie obtaining a good prediction rate under small regularization, highlights a strong difference from the LASSO framework. In the LASSO setting \citep{lasso1996,tibshirani_notes_lasso}, regularization plays a greater role in controlling the prediction error: to get a parametric rate, we need to use bounds on the estimation error.
\end{remark}

\subsection{Prediction under large regularization} \label{section:control_prediction_strong_reg}
Provided that non-degenerate dual certificates exist, one can obtain appropriate bounds for both estimation and prediction, with $\kappa$ chosen as in Theorem \ref{th:estimation_error}. The following theorem presents the prediction bound in this regime.

\begin{theorem} \label{th:prediction_error}
 Assume that $\X \subset \R^d \times [u_{\min},u_{\max}]^d$. 
 Let $\tau^2=\frac{2u_{\min}^2}{\ln n}$. If Assumption \ref{assumption:existence_non-degenerate_certif} holds, choosing $\kappa=\frac{\rho_n}{\sqrt{c_p}}$, keeping only the dependence on $s$ and $n$ we have
 \[\E{\norm{\Phi\hat{\mu}_n - \Phi\mu^0}_{L^2(\R^d)}^2}\lesssim \frac{s (\ln n)^{d/2}}{n}\,.\]
\end{theorem}
\noindent
The proof is given in Appendix \ref{section:proof_th_prediction_error}. The upper bound on the prediction error displayed here is slightly larger than in Proposition \ref{prop:prediction_under_small_reg}. In particular, this bound depends linearly on the sparsity index $s$. We refer for instance to \citep{butucea_prediction} for a similar bound in a different context. We recall that this specific choice for~$\kappa$ allows us to control the estimation performances of the estimator (see Theorem \ref{th:estimation_error}).
\begin{remark}\label{remark:choice_kappa_pred}
 Similarly to the estimation task (see Remark \ref{remark:choice_kappa_estim}), a choice of regularization depending on $s$ leads to a better rate. With $\kappa=\frac{\rho}{\sqrt{c_p s}}$, we get $\E{\norm{\Phi\hat{\mu}_n - \Phi\mu^0}_{L^2(\R^d)}^2}\lesssim \frac{(\ln n)^{d/2}}{n}$ for $s=O(n (\ln n)^{d/2})$. We refer to Equation \eqref{eq:remark_choice_kappa_pred} in Appendix \ref{section:proof_th_prediction_error}.
\end{remark}

\section{Dual certificates}\label{section:dual_certif} 
The proofs of Theorem \ref{th:estimation_error} and Proposition \ref{prop:effective_near_regions} rely on the existence of non-degenerate dual certificates associated with $\mu^0$ (Assumption \ref{assumption:existence_non-degenerate_certif}). In this section, we explain how to construct such objects and the assumptions needed.

\paragraph{Connection with previous works} 
In the general framework of BLASSO, the objective is to recover a sparse target measure from the observed signal. A key analytical tool in this setting is the dual certificate—a smooth function associated with the underlying feature map of the problem at hand. These certificates identify the positions of the target particles by interpolating the signs of the target measure.

Dual certificates can be traced back to super-resolution \citep{candes_fernandez_granda_2014} and minimal extrapolation \citep{decastro_2012}. They are related to the dual solutions of the BLASSO when the observation is noiseless, \cf \citep{blasso_duval_peyre} (in our framework, the noiseless observation corresponds to $\E{L\circ\hat{f_n}}=L\circ f^0$). Our particular construction of dual certificates is inspired from the BLASSO literature (\eg, ``vanishing derivatives pre-certificate'' in \citep{blasso_duval_peyre}, ``limit certificate'' in \citep{off-the-grid_cs}). Although our setting differs—we observe a sample drawn from some mixture distribution—the construction of the dual certificates follows the same principles: they depend solely on the target measure $\mu^0$ and the feature map $\Psi$, and not on the observed data nor the regularization. Let us mention that other constructions exist for translation-invariant kernels (\eg, \emph{pivot certificates} in \citep{supermix,decastro_gribonval_jouvin}).

The existence of certificates requires assumptions on $\mu^0$, mainly on the separation between its particles. In \citep{off-the-grid_cs}, precise conditions are proposed, based on controls on the kernel. In this section, we adapt and check these assumptions.

\paragraph{Contributions} 
Our main contribution is the construction of dual certificates for a BLASSO problem involving a kernel that is not translation-invariant. 

In Section \ref{section:statistical_modeling}, we have adapted the data fitting term to obtain a normalized kernel, and introduced reparameterized measures (see \eqref{eq:pb_BLASSO_gaussian_kernel_norm}).
It is however challenging to check that $K_{\rm norm}$ satisfies the Local Positive Curvature assumption \citep[Assumption 1]{off-the-grid_cs} with the Fisher-Rao distance. We therefore consider regions for controls defined instead by the semi-distance $\mathpzc{d}$ introduced in \eqref{eq:def_semi_distance}. Even so, controls of the certificates are expressed with the Fisher-Rao distance, since they are carried out using Taylor expansions along the Fisher-Rao geodesics.

A technical contribution is the derivation of bounds and local controls for the Riemannian derivatives of the normalized Gaussian kernel.
Additionally, we establish the compatibility between the Fisher–Rao metric and the semi-distance. It is essential to control the Fisher–Rao geodesics within the balls with respect to $\mathpzc{d}$ in order to make use of the local bounds on the kernel.
Combining these elements, we derive sufficient conditions on the target $\mu^0$ that guarantee the existence of non-degenerate dual certificates. The next subsections are dedicated to the proof of the following theorem.

\begin{theorem}[Main result: existence of certificates under a minimal separation]\label{th:main_theorem_certif}
Assume that $\X$ is a compact set of $\R^d \times [u_{\min},u_{\max}]^d$. Let $s\geq 2$, $0<\tau \leq u_{\min}$ and $\{x_j^0\}_{j=1}^s \subset \X$. We set $\mathfrak{d}_\mathfrak{g}$ as the Fisher-Rao distance (associated with the metric $\mathfrak{g}$ defined by \eqref{eq:def_fisher_rao_metric}).
If 
\begin{equation}
\min_{i \neq j} \mathpzc{d}(x_i^0, x_j^0) \geq \max \left\{\frac{\sqrt{u_{\max}^2 + \frac{1}{4}\frac{0.3025^2}{d}(2u_{\max}^2+\tau^2)}}{u_{\min}} \left(\Delta+\frac{0.3025}{\sqrt{d}}\right) \: , \: 2\frac{u_{\max}}{u_{\min}} \Delta \right\} + \sqrt{d\ln \left(\frac{u_{\max}^2}{u_{\min}^2}\right)}
\label{eq:minimum_separation}
\end{equation}
where $\Delta=2\sqrt{11.9+3\ln(d+6.62)+\ln(s-1)}$,
then Assumption \ref{assumption:existence_non-degenerate_certif} holds with $r=\frac{0.3025}{\sqrt{d}}$, $(\varepsilon_0,\varepsilon_2)=\left(\frac{0.03911}{d}, 0.06158\right)$, $(\tilde \varepsilon_0,\tilde \varepsilon_2)=\left(\frac{0.03911}{d},\frac{\sqrt{4d^2+10d}}{2}+0.004106\right)$ and $c_p=2$.
\end{theorem}
\noindent
This result stems from Theorems \ref{th:existence_certificate_under_positive_curvature_assumption} and \ref{th:check_local_curvature_assumption} below. The calculation of the parameters is detailed in \citep[Section VII.3]{notebook}.

Theorem \ref{th:main_theorem_certif} ensures the existence of non-degenerate dual certificates when the particles $\{x_j^0\}_{j=1}^s$ are sufficiently separated. This separation, involving both means and covariances of the Gaussian components, is expressed via the semi-distance $\mathpzc{d}$, rather than the Fisher-Rao distance.

Note also that we give explicit constants, but a less constructive approach could be considered, based on the continuity of $K_{\rm norm}$ and its derivatives.

\subsection{Geometrical framework} 
\paragraph{Kernel}
Building on the work of \citep{blasso_duval_peyre} and \citep{off-the-grid_cs}, we consider dual certificates of the form 
\begin{equation} \label{eq:eta_general_form}
 \eta_{\alpha,\beta} = \sum_{j=1}^{s} \alpha_j K_{\rm norm}(x_j^0,\dotp) + \sum_{j=1}^{s} \beta_j^T \nabla_1 K_{\rm norm}(x_j^0,\dotp)
\end{equation}
where $\alpha_j \in \R$, $\beta_j \in \R^{2d}$ for all $j\in\{1,\ldots,s\}$, and $K_{\rm norm}$ is the real-valued kernel defined, for all $x, x'\in\R^d\times[u_{\min},+\infty)^d$, by
\begin{align}
 K_{\rm norm}(x,x')= \innerprod{\Psi\delta_x}{\Psi\delta_{x'}}_{\L} =\prod_{k=1}^d (2u_k^2+\tau^2)^{1/4} (2{u_k'}^2+\tau^2)^{1/4}\frac{e^{-\frac{(t_k-t_k')^2}{2(u_k^2+{u_k'}^2+\tau^2)}}}{(u_k^2+{u_k'}^2+\tau^2)^{1/2}}\,. \label{eq:kernel_K_norm}
\end{align}
The kernel $K_{\rm norm}$ is normalized according to the definition of $W$ (involved in $\Psi$).

The gradient of $K_{\rm norm}$ with respect to its first variable is written as $\nabla_1 K_{\rm norm}$ ($\nabla_2 K_{\rm norm}$ denotes the gradient with respect to the second variable).

\begin{remark}[Smoothness of the feature map]\label{remark:smoothness_Psi_delta_x}
 Note that $K_{\rm norm}$ is $C^{\infty}$ on $(\R^d \times [u_{\min},+\infty)^d)^2$. By iteratively applying \citep[Lemma 4.34]{steinwart}, it comes that
\[\left(x \mapsto \Psi\delta_x\right) \in \C^\infty(\R^d \times [u_{\min},+\infty)^d,\L)\] and that $\innerprod{\partial_1 \Psi\delta_x}{\partial_2 \Psi\delta_{x'}}_\L=\partial_1 \partial_2 K_{\rm norm}(x,x')$, where $\partial_1$ (resp.\ $\partial_2$) denotes here any derivative \wrt to $x$ (resp.~$x'$).
\end{remark}

The next lemma shows the relevance of constructing $\eta$ of the form \eqref{eq:eta_general_form}: such functions belong to ${\mathrm{Im}}(\Psi^*)$.
\begin{lemma}\label{lemma:eta_c_infty_im_psi_star}
 Let $\alpha \in \R^s$, $\beta \in \R^{2d\times s}$. The function $\eta_{\alpha,\beta}$ introduced in \eqref{eq:eta_general_form} verifies, for all $x \in \R^d \times [u_{\min},+\infty)^d$, \[\eta_{\alpha,\beta}(x)=\innerprod{p_{\alpha,\beta}}{\Psi\delta_{x}}_\L \quad \text{with} \quad p_{\alpha,\beta}=\sum_{j=1}^s \alpha_j \Psi\delta_{x_j^0}+ \sum_{j=1}^s \beta_j \nabla_{x}\left(\Psi\delta_{x_j^0}\right) \in \L \,.\]
 Furthermore, $\eta_{\alpha,\beta} \in \C^\infty(\R^d \times [u_{\min},+\infty)^d)$ and in particular, $\restr{\eta_{\alpha,\beta}}{\X} \in {\mathrm{Im}}(\Psi^*)$.
\end{lemma}
\noindent
The proof is an immediate consequence of Remark \ref{remark:smoothness_Psi_delta_x} and of the definition of $K_{\rm norm}$ in \eqref{eq:kernel_K_norm}. 

The construction of a non-degenerate (global) dual certificate follows several steps: first we construct a function $\eta$ of the form \eqref{eq:eta_general_form}, choosing $\alpha_j$, $\beta_j$ such that $\eta(x_j^0)=1$ and $\nabla \eta (x_j^0)=0$. This amounts to solve some linear system (see \eqref{eq:Upsilon} in Appendix).
Then we use Taylor expansions on geodesics associated with the distance $\mathfrak{d}_\mathfrak{g}$ to control $\eta$ on the near and far regions, in the spirit of \citep{off-the-grid_cs}. The Fisher-Rao distance appears to be well suited for this purpose.

\paragraph{Fisher-Rao metric} We work in a Riemannian geometry framework, and we use the Fisher-Rao metric induced by $K_{\rm norm}$ (\cf \citep[Lemma 1]{off-the-grid_cs}), defined for $x\in \R^d \times [u_{\min},+\infty)^d$ by 
\begin{equation}
\label{eq:def_fisher_rao_metric}
 \mathfrak{g}_x\coloneq\nabla_1\nabla_2K_{\rm norm}(x,x)=\diag \left(\frac{1}{2u_1^2+\tau^2},\ldots, \frac{1}{2u_d^2+\tau^2},\frac{2u_1^2}{(2u_1^2+\tau^2)^2} ,\ldots, \frac{2u_d^2}{(2u_d^2+\tau^2)^2} \right)\,.
\end{equation}
We will use the associated norm, defined by
\[\norm{v}_x=\sqrt{v^T \mathfrak{g}_x v} \quad \forall \: v \in \R^{2d}\,, \;x \in \R^d \times [u_{\min}, +\infty)^d \,. \]
Details about this metric are provided in Appendix \ref{section:properties_fisher_rao}, together with a description of the associated geodesics. We recall that a geodesic for the metric $\mathfrak{g}$ between $x,x' \in \R^d \times [u_{\min},+\infty)^d$ is a piecewise continuously differentiable function $\gamma:[0,1] \rightarrow \R^d \times [u_{\min},+\infty)^d$ such that $\gamma(0)=x, \gamma(1)=x'$, minimizing the quantity $\int_0^1 \norm{\dot\gamma(y)}_{\gamma(y)}^2\, \d y$. In dimension $d=1$, our Fisher-Rao geodesics share the same paths as those of the Poincaré half-plane model (\cf Lemma \ref{lemma:form_geodesics_dim_1} in Appendix): they are portions of straight lines parallel to $\{t=0\}$ and semicircles whose origin is on $\{u=0\}$.
The notation $\mathfrak{d}_\mathfrak{g}$ refers to the associated distance. It is defined by $\mathfrak{d}_\mathfrak{g}(x,x')^2=\int_0^1 \norm{\dot\gamma(y)}_{\gamma(y)}^2\, \d y$ with $\gamma$ the geodesic between $x,x'$. This Fisher-Rao distance is not practical to define near and far regions since it is not directly linked to the correlation between 2 features calculated with the kernel. This motivates the use of another distance on $\X$ to express the recovery results: we use the semi-distance defined by \begin{equation}\label{eq:link_semi_dist_kernel}
 \mathpzc{d}(x,x')=\sqrt{-2 \ln (K_{\rm norm}(x,x'))}\,,
\end{equation}
whose expression is given by \eqref{eq:def_semi_distance}.

\paragraph{Riemannian derivatives}
Following the work of \citep{off-the-grid_cs}, we will show that the dual certificates we construct are non-degenenerate (see Definitions \ref{def:non_degenerate_certificate} and \ref{def:local_non_degenerate_certificate}) by controlling the Riemannian derivatives of $K_{\rm norm}$.
For details about this framework, see \citep[p.263]{off-the-grid_cs}. 

Riemannian derivatives involve the Christoffel symbols associated with $\mathfrak{g}$. We will use the notation $\Gamma^{t_k},\Gamma^{u_k}$, and refer to \eqref{eq:christoffel_symbols} in Appendix \ref{section:christoffel_symbols} for a precise definition.

\begin{definition}\label{def:riemannian_derivatives}
Let $\psi \in \C^2(\R^d \times [u_{\min}, +\infty]^d)$. Let $x,x' \in \R^d \times [u_{\min}, +\infty]^d$.
\\\underline{Riemannian Hessian:} we define \[H^{\mathfrak{g}} \psi(x)=\nabla^2 \psi(x)-\sum_{k=1}^d\Gamma^{t_k} \partial_{t_k}\psi(x)-\sum_{k=1}^d\Gamma^{u_k} \partial_{u_k}\psi(x)\,.\]
\underline{Covariant derivatives:} Let $v,v' \in \R^{2d}$.
We define \[D_0[\psi](x)=\psi(x)\,, \quad D_1[\psi](x)[v]=v^T \nabla \psi(x)\,, \quad D_2[\psi](x)[v,v'] = v^T H^{\mathfrak{g}} \psi(x) v'\,.\]
\underline{Operator norms:} For $j\in\{0,1,2\}$, we define the operator norm \[\norm{D_j[\psi](x)}_x \coloneq \sup_{\substack{V=[v_1,\ldots,v_j]\in (\R^{2d})^j \\ \forall l=1,\ldots,j\,, \: \norm{v_l}_x \leq 1}} D_j[\psi](x)[V]\,.\]
\underline{Kernel derivatives and associated operators:} Let $i,j\in\{0,1,2\}$.
We define the covariant derivative of the kernel of order $i$ with respect to the first variable $x$ and of order $j$ with respect to the second variable $x'$ by \[[Q]K_{\rm norm}^{(ij)}(x,x')[V]=\innerprod{D_i[\Psi](x)[Q]}{D_j[\Psi](x')[V]}_\L \quad \forall\, Q \in (\R^{2d})^i\,, V \in (\R^{2d})^j\,.\]
The associated operator norm is \[\norm{K_{\rm norm}^{(ij)}(x,x')}_{x,x'}\coloneq \sup_{\substack{Q=[Q_1,\ldots,Q_i]\in (\R^{2d})^i\,,\,V=[V_1,\ldots,V_j]\in (\R^{2d})^j\\ \forall l\, \norm{Q_l}_x\,, \, \norm{V_l}_{x'} \leq 1}} [Q]K_{\rm norm}^{(ij)}(x,x')[V] \,.\]
\end{definition}
\noindent
Simplified expressions are given in Appendix (Lemma \ref{lemma:simplified_expressions_operator_norms_kernel}).

\subsection{Construction of dual certificates}\label{section:construct_dual_certif}
In what follows, we give some results allowing us to adapt and check the hypothesis of \citep[Theorem~2]{off-the-grid_cs}, which shows that the (global) non-degenerate certificate exists under some conditions on the kernel.
The corresponding proof can be adapted (see for instance \citep[Section 6.7]{off-the-grid_cs}) to show the existence of additional certificates for the near regions, under the same conditions.

\paragraph{Local positive curvature assumption} The following definition allows us to adapt \citep[Assumption 1]{off-the-grid_cs} to our framework. In particular, we require some smoothness and structural properties for the kernel that enable the construction of dual certificates, as stated below.

\begin{definition}[Kernel of local positive curvature with parameters $s$, $\Delta$, $r$, $\bar\varepsilon_0$ and $\bar\varepsilon_2$] \label{def:positive_curvature} Let $K$ be a real-valued normalized kernel of positive type, in $C^3((\R^d \times [u_{\min},+\infty)^d)^2)$.
It is said to satisfy the local positive curvature assumption (LPC) if the following holds:
	\begin{enumerate}
		\item For all $i,j\in{0,1,2}$ with $i+j\leq 3$, 
		\begin{equation*}
		\sup_{x,x'\in \R^d \times [u_{\min},+\infty)^d}\norm{K^{(ij)}(x,x')}_{x,x'}\leq B_{ij} <+\infty\,.
		\end{equation*}
	For $i=0,1,2$, we denote $B_i\coloneq 1+ B_{0i}+ B_{1i}$. 
		\item There exists $r>0$ such that $K$ has strictly positive curvature constants $\bar\varepsilon_0$ and~$\bar\varepsilon_2$ with 
		\begin{gather*}
		 K(x,x')\leq 1-\bar\varepsilon_0\,,\quad \forall x,x'\in\R^d \times [u_{\min},+\infty)^d \textnormal{ s.t. }\mathpzc{d}(x,x')\geq r
		\,,\\
    - K^{(02)}(x,x')[v,v]\geq \bar\varepsilon_2  \|v\|_{x'}^2\,,\quad \forall v\in\R^{2d}\,,  \quad \forall x,x'\in{\R^d \times [u_{\min},+\infty)^d}\:\: \text{s.t.}\:\: \mathpzc{d}(x,x')< r  \,.   
		\end{gather*}
	\item There exist $s\geq2$, $\Delta(s)>0$ such that for all $\{x_l\}_{l=1}^s\in\mathcal S_{\Delta(s)}$,
		\begin{equation*}
           \sum_{l=2}^{s}\norm{K^{(ij)}(x_1,x_l)}_{x_1,x_l}\leq \frac{1}{64}\min\bigg(\frac{\bar\varepsilon_0(r)}{B_0},\frac{\bar\varepsilon_2(r)}{B_2}\bigg) \quad\forall\, (i,j)\in\{0,1\}\times\{0,1,2\}
		\,,
		\end{equation*}
		where $\mathcal S_\Delta\coloneq\left\{\{x_l\}_{l=1}^{{s}}\subset \R^d \times [u_{\min},+\infty)^d \,:\,\min_{m\neq l}\mathpzc{d}(x_m,x_l)\geq \Delta\right\}$.
	\end{enumerate}
Such kernel is said to verify the LPC with parameters $s$, $\Delta(s)$, $r$, $\bar\varepsilon_0(r)$ and $\bar\varepsilon_2(r)$.
\end{definition}

\begin{remark} This definition differs on some points from \citep[Assumption 1]{off-the-grid_cs}. First, since we solve \eqref{eq:pb_BLASSO_gaussian_kernel_norm} for nonnegative measures, we do not have to control the negative part of the certificates, and thus do not require $r<B_{02}^{-1/2}$.
Moreover, we use controls on regions defined by the semi-distance $\mathpzc{d}$ instead of $\mathfrak{d}_\mathfrak{g}$. We need to provide bounds on the kernel on $\R^d \times [u_{\min},+\infty)^d$ and not $\X$, because a Fisher-Rao geodesic between 2 points of $\X$ could be outside $\X$. The use of a semi-distance, which better describes the correlation between 2 features (as expressed by the kernel), also leads to difficulties. We need to ensure that $\mathpzc{d}$ is compatible with the Fisher-Rao distance. This is detailed in the paragraph below.
\end{remark}

\paragraph{Compatibility between the Fisher-Rao metric and the semi-distance} For $x \in \R^d \times [u_{\min},+\infty)^d$ and $R>0$, we denote \[B_{\mathpzc{d}}(x,R)\coloneq\{x' \in \R^d \times [u_{\min},+\infty)^d \: :\: \mathpzc{d}(x,x') \leq R \}\] and \[\mathring B_{\mathpzc{d}}(x,R)\coloneq\{x' \in \R^d \times [u_{\min},+\infty)^d \: :\: \mathpzc{d}(x,x') <R\}\,. \]
For $A \subset \R^d \times [u_{\min},+\infty)^d$ and $x_0\in A$, we define the set containing all points from geodesics connecting $x_0$ with points of $A$, \[\mathcal{G}_{x_0}(A)\coloneq \{\gamma(y) \: : \: \gamma \text{ is a Fisher-Rao geodesic}\,, \: y\in [0,1]\,, \: \gamma(0)=x_0\, ,\gamma(1)\in A\}\,.\]
\begin{lemma}[Fisher-Rao geodesics remain within balls \wrt the semi-distance]\label{lemma:geodesics_within_near_regions}
Let $r>0$ and $x_0 \in \R^d \times [u_{\min},+\infty)^d$. Then \[\mathcal{G}_{x_0}\left(B_{\mathpzc{d}}(x_0,r)\right) \subset B_{\mathpzc{d}}(x_0,r) \,.\]
\end{lemma} 
\noindent
The proof is detailed in Appendix \ref{section:proof_lemma_geodesics_within_near_regions}.
This lemma allows us to use the controls we have on the balls $B_{\mathpzc{d}}(x_j^0,r)$ on the paths of the geodesics between $x_j^0$ and a point in $B_{\mathpzc{d}}(x_j^0,r)$. It entails that $\mathpzc{d}$ and $\mathfrak{d}_\mathfrak{g}$ are in some sense compatible. 

To control the certificates, we also require that certain balls (with respect to the semi-distance) around the particles be disjoint. 
We restrict the possible values for the variances to establish this property: we need upper and lower bounds to establish a ``pseudo-quasi triangle inequality'' for $\mathpzc{d}$ (see Appendix \ref{section:proof_lemma_minimum_separation}). 
\begin{lemma}[Separation of the balls \wrt the semi-distance]\label{lemma:minimum_separation}
Let $r,\Delta>0$. Assume that $\X \subset \R^d \times [u_{\min},u_{\max}]^d$.
 If
 \begin{equation*}
 \min_{i \neq j} \mathpzc{d}(x_i^0, x_j^0) \geq \max \left\{\frac{\sqrt{u_{\max}^2 + \frac{1}{4}r^2(2u_{\max}^2+\tau^2)}}{u_{\min}} (\Delta+r) \: , \: 2\frac{u_{\max}}{u_{\min}} \Delta \right\} + \sqrt{d\ln \left(\frac{u_{\max}^2}{u_{\min}^2}\right)}\eqcolon \Delta_\tau\, , 
 \end{equation*}
 then the balls $\mathring B_{\mathpzc{d}}(x_j^0,\Delta) \cap \X$ are disjoint, and for all $j \neq i$, $\mathcal{G}_{x_j^0}(B_{\mathpzc{d}}(x_j^0,r) \cap \X)$ does not intersect $\mathring B_{\mathpzc{d}}(x_i^0,\Delta)$.
\end{lemma}
\noindent
The proof is given in Appendix \ref{section:proof_lemma_minimum_separation}. We stress that this lemma leads to the separation condition in Theorem \ref{th:main_theorem_certif}. It is crucial to manage the kernel on the far and near regions according to the local positive curvature assumption. 

\paragraph{Non-degeneration of certificates under a minimal separation}
Using the compatibility between the semi-distance and the Fisher-Rao metric, and controls on $K_{\rm norm}$, we can show the existence of non-degenerate certificates under some condition on the minimal separation between the particles of $\mu^0$. This is what the following theorem entails.
\begin{theorem}\label{th:existence_certificate_under_positive_curvature_assumption}
 Assume that $\X \subset \R^d \times [u_{\min},u_{\max}]^d$ and that $K_{\rm norm}$ satisfies the LPC (Definition \ref{def:positive_curvature}) with parameters $s, \Delta, r, \bar\varepsilon_0, \bar \varepsilon_2$. 
 
 If $\{x_j^0\}_{j=1}^s \subset \X$ satisfies $\min_{i \neq j} \mathpzc{d}(x_i^0, x_j^0) \geq \Delta_\tau$ (defined in Lemma \ref{lemma:minimum_separation}), then there exists a $\left(\varepsilon_0=\frac{7}{8}\bar \varepsilon_0, \varepsilon_2=\frac{15}{32}\bar \varepsilon_2, r\right)$-global non-degenerate certificate $\eta$ of the form \eqref{eq:eta_general_form}.

 Under the same assumptions, for all $j=1,\ldots,s$ there exists a certificate $\eta_j$ for the jth near region of the form \eqref{eq:eta_general_form}, of parameters $\left(\tilde \varepsilon_0=\frac{7}{8}\bar \varepsilon_0,\tilde \varepsilon_2=\frac{B_{02}+\bar\varepsilon_2/16}{2}, r\right)$. 

 Moreover, $\restr{\eta}{\X}=\Psi^*p$ and $\restr{\eta_j}{\X}=\Psi^*p_j$ where $\norm{p}_\L\leq \sqrt{2s}$ and $\norm{p_j}_\L\leq \sqrt{2}$.
\end{theorem}
\noindent
The proof is based on an adaptation of \citep[Theorem 2]{off-the-grid_cs}, and is given in Appendix \ref{section:proof_existence_certificate_under_positive_curvature_assumption}.

It remains to prove that $K_{\rm norm}$ satisfies the LPC—this is the purpose of the following theorem, whose proof is provided in Appendix \ref{section:controls_kernel}. We give general bounds for $K_{\rm norm}$ in dimension $d\geq 1$, and we provide a tighter constant for $\Delta(s)$ in the case $d=1$.
\begin{theorem}\label{th:check_local_curvature_assumption}
 Let $s\geq 2$, $d\geq 1$. Assume that $\X \subset \R^d \times [u_{\min},u_{\max}]^d$ and that $\tau \leq u_{\min}$. 
Then $K_{\rm norm}$ satisfies the LPC with parameters $s$, $r=\frac{0.3025}{\sqrt{d}}$, $\bar\varepsilon_2(r)=0.13139$, $\bar\varepsilon_0(r)=\frac{0.0894}{2d}$, and 
\[
\Delta(s)= 2\sqrt{ 11.9+3\ln(d+6.62)+\ln(s-1)}\,.
\] Moreover, we can take $B_{02}=\sqrt{4d^2+10d}$ (item 1 of Definition \ref{def:positive_curvature}).
\end{theorem}

\begin{remark}
 When $d=1$, the minimal separation consition can be improved: $K_{\rm norm}$ satisfies the LPC with the same parameters $r, \bar\varepsilon_0, \bar\varepsilon_2$ as in Theorem \ref{th:check_local_curvature_assumption}, with $\Delta(s)=2\sqrt{13.88 + \ln(s-1)}$ (Appendix \ref{section:controls_kernel}).
\end{remark}
\begin{remark}
 Note that the assumption on the separation between particles given in Theorem \ref{th:existence_certificate_under_positive_curvature_assumption} depends on $\tau$ through the semi-distance. We write $\mathpzc{d}=\mathpzc{d}_{\tau}$ to emphasize this dependency. At first sight, the separation condition thus appears to depend on $n$ when choosing $\tau^2=\frac{2u_{\min}^2}{\ln n}$.
 However, as $\tau \in \R^+ \mapsto \mathpzc{d}_\tau(x,x')$ is decreasing and $\tau \in \R^+ \mapsto \Delta_\tau$ is increasing, if $\{x_j^0\}_{j=1}^s$ satisfies $\min_{i \neq j} \mathpzc{d}_{\tau_1}(x_i^0, x_j^0) \geq \Delta_{\tau_1}$ for some $\tau_1>0$, then it also satisfies $\min_{i \neq j} \mathpzc{d}_{\tau_2}(x_i^0, x_j^0)~\geq~\Delta_{\tau_2}$ for $0<\tau_2 \leq \tau_1$.
This motivates the introduction of the following assumption:
\begin{equation} \label{eq:min_sep_not_depending_on_tau}
 \underset{x_i^0 \neq x_j^0}{\min} \mathpzc{d}_{\tau_{\max}}(x_i^0,x_j^0) \geq \Delta_{\tau_{\max}}\quad \text{and} \quad \forall j=1,\ldots,s\,, \; x_j^0 \in \X \subset \R^d \times [u_{\min},u_{\max}]^d\,, \quad \tau_{\max} \leq u_{\min} \,.
\end{equation}
The recovery guarantees are not expressed using this assumption, but it is useful when thinking about the convergence of the solution towards $\mu_\omega^0$ as $n\to \infty$. The target is fixed, but for a good prediction $\tau$ must decrease with the number of observations. The constraint \eqref{eq:min_sep_not_depending_on_tau} provides in this context a condition on $\mu^0$ that does not depend on $n$.
\end{remark}

\paragraph{Controls on the effective near regions} Theorem \ref{th:check_local_curvature_assumption} sets the size of near regions we consider for the certificates. We can lower-bound the Fisher-Rao distance with the semi-distance on these regions, allowing us to extend the control of the estimator to near regions of smaller size (Proposition \ref{prop:effective_near_regions}). This result is key to provide recovery guarantees for the proximity of $\mu^0$ and the renormalized estimator $\frac{\hat{\mu}_{n,\omega}}{W}$ (\cf Corollary \ref{cor:control_renormalized_estimate}). Its proof can be found in Appendix \ref{section:proof_lemma_tilde_varepsilon_3}.
\begin{lemma}\label{lemma:tilde_varepsilon_3}
 Assume that $x,x_0 \in\R^d \times [u_{\min},+\infty)^d$. If $r_e \leq \mathpzc{d}(x,x_0) \leq r=\frac{0.3025}{\sqrt{d}}$, then $\mathfrak{d}_\mathfrak{g}(x,x_0)^2 \geq \frac{r_e^2}{\tilde \varepsilon_3}$ with $\tilde \varepsilon_3=2.84$. 
\end{lemma}

\section{Sparsity of the solution under large sample sizes}
\label{section:NDSC}
The certificates constructed in Section \ref{section:construct_dual_certif} allow us to obtain the non-asymptotic recovery guarantees presented in Section \ref{section:estimation}. These results are stated in terms of control of the mass in the far and near regions. Nevertheless, they say nothing about the sparsity of the estimator.

\medskip
In this section, we present results of a different nature. Our estimator is here an exact solution $\mu_{n,\omega}^\star$, verifying \eqref{eq:mu_n_omega_star}. Following \citep{blasso_duval_peyre}, we show that under the same separation condition as in Theorem \ref{th:main_theorem_certif}, and for large sample sizes, $\mu_{n,\omega}^\star$ is sparse and has exactly 1 particle in each near region with high probability. We stress that in the following, $0<\tau\leq u_{\min}$ is fixed (does not decrease as $n \to \infty$). 

\paragraph{Non Degenerate Source Condition}
Our analysis is again based on non-degenerate dual certificates (Definition \ref{def:non_degenerate_certificate}). The global certificate $\eta$ constructed in Theorem \ref{th:main_theorem_certif} is said to satisfy the Non Degenerate Source Condition \eqref{eq:NDSC}:
\begin{equation}\label{eq:NDSC} \tag{NDSC} \forall \, x\in \X\,, \quad \eta(x) \leq 1\, , \quad \eta(x)=1 \iff x\in \{x_j^0\}_j\,, \quad \nabla^2 \eta(x_j^0)\prec 0 \quad \forall\, j=1,\ldots,s\,.\end{equation}
Detailed versions of these properties are established in Lemma \ref{lemma:eta_ndsc}.

This certificate is the vanishing derivative pre-certificate of \citep[Section 4]{blasso_duval_peyre}. We can then show (see Lemma \ref{lemma:unique_sol_minimal_certif_ndsc} in Appendix) that $\eta$ corresponds to a minimal norm certificate \citep[Proposition 7]{blasso_duval_peyre}: we have $\restr{\eta}{\X}=\Psi^*p_{0,0}$ where 
\begin{equation}
\label{eq:def_p_00}
 p_{0,0}\coloneq \underset{p \in \L}{\arg\min} \left\lbrace \norm{p}_{\L}\; : \; \Psi^*p(x) \leq 1 \: \forall x\in\X\,, \: \Psi^*p(x_j^0)=1  \: \forall j=1,\ldots,s \right\rbrace \,.\tag{$p_{0,0}$} 
\end{equation}
The definitions \eqref{eq:NDSC} and \eqref{eq:def_p_00} slightly differ from those appearing in \citep[Definition 1 and Proposition 7]{blasso_duval_peyre}. Since \eqref{eq:pb_BLASSO_gaussian_kernel_norm} is solved over the space of nonnegative measures (rather than signed measures), there is no need to control the negative part of the global certificate. Hence the condition $|\eta(x)| \leq 1$ for all $x \in \X$ in \citep{blasso_duval_peyre} is replaced by $\restr{\eta}{\X} \leq 1$ (see Appendix \ref{section:proof_lemma_eta_ndsc}).

The minimal norm property allows us to control the sparsity index of $\mu_{n,\omega}^\star$ for large sample sizes, leading to the following theorem.

\begin{theorem}\label{th:NDSC}
Under the assumptions of Theorem \ref{th:main_theorem_certif}, there exists $\kappa_0>0$ (depending on $\mu^0$, $\X$ and $\tau$) and $\gamma_0>0$ depending on $d$ such that for all $\kappa\leq \kappa_0$ and if $\norm{\Gamma_n}_{\L} \leq \gamma_0 \kappa$, then $\mu_{n,\omega}^\star$ is $s$-sparse and has exactly 1 particle in each $\X_j^{\rm near}(r)$, with $r=\frac{0.3025}{\sqrt{d}}$.
 Moreover, writing $\mu_{\kappa,b}=\sum_{j=1}^s \omega_j^{\kappa,b} \delta_{x_j^{\kappa,b}}$ with $x_j^{\kappa,b} \in \X_j^{\rm near}(r)$, and denoting $a_j^{\kappa,b}= \frac{\omega_j^{\kappa,b}}{W(x_j^{\kappa,b})}$ we have \[|a_j^{\kappa,b}-a_j^0| \lesssim_{\X, \mu^0, \tau} \kappa \quad \text{and} \quad \mathpzc{d}(x_j^0, x_j^{\kappa,b}) \lesssim_{\X, \mu^0, \tau} \kappa \quad \forall \, j=1,\ldots, s\,.\]
\end{theorem}
\noindent
The proof can be found in Appendix \ref{section:proof_result_ndsc}.

This result should be seen in conjunction with Lemma \ref{lemma:control_noise}. The latter indicates that the assumption $\norm{\Gamma_n}_{\L}~\leq~\gamma_0 \kappa$ is verified with high probability for $n$ large enough. 
We detail this in the following corollary. 

\begin{corollary} \label{cor:ndsc}Let $c_\kappa >0$. Assume that the conditions given in Theorem \ref{th:main_theorem_certif} hold.
Let $n \geq \frac{c_\kappa^2}{\kappa_0^2 (2\pi)^{d/2} \tau^d}$ and $\kappa=\frac{c_\kappa}{(2\pi)^{d/4} \tau^{d/2} \sqrt{n}}$. With probability greater than $1-C_\Gamma e^{-\left(\frac{\gamma_0 c_\kappa}{C_\Gamma}\right)^2}$, $\mu_{n,\omega}^\star=\sum_{j=1}^s \omega_j^\star \delta_{x_j^{\star}}$ where $\omega_j^\star>0$ and $x_j^{\star} \in \X_j^{\rm near}(r)$ for all $j=1,\ldots, s$. Moreover, \begin{equation}\label{eq:bound_omega_ndsc}
 |a_j^0 -a_j^\star | \lesssim_{\X,\mu^0,\tau} \frac{c_\kappa}{\sqrt{n}} \quad \text{and} \quad \mathpzc{d}(x_j^\star,x_j^0)\lesssim_{\X,\mu^0,\tau} \frac{c_\kappa}{\sqrt{n}} \,.
\end{equation} 
\end{corollary}
\noindent
The proof is given in Appendix \ref{section:proof_cor_ndsc}. The result is established for a generic value of $c_\kappa$, but specific choices verifying $c_{\kappa,n}=\underset{n \to + \infty}{o}(\sqrt{n})$ can be considered. For instance, choosing $c_{\kappa,n}=\alpha \sqrt{\ln n}$ with $\alpha>0$, there exists $n_{0,\alpha}\in \mathbb{N}$ depending on $ \X, \mu^0, \tau, \alpha$ such that for all $n\geq n_{0,\alpha}$ and with probability at least $1-C_\Gamma n^{-\frac{\gamma_0^2 \alpha^2}{C_\Gamma^2}}$, $\mu_{n,\omega}^\star=\sum_{j=1}^s \omega_j^\star \delta_{x_j^{\star}}$ where for all $j=1,\ldots,s$, 
\begin{equation}
|a_j^0 -a_j^\star | \lesssim_{\X, \mu^0, \tau} \alpha \frac{\sqrt{\ln n}}{\sqrt{n}} \quad \text{and} \quad \mathpzc{d}(x_j^{\star},x_j^0) \lesssim_{\X, \mu^0, \tau} \alpha \frac{\sqrt{\ln n}}{\sqrt{n}}\,.
\label{eq:NDSC_casparticulier}
\end{equation}
Compared to Theorem \ref{th:estimation_error}, Inequality \eqref{eq:NDSC_casparticulier} yields more classical results in parametric estimation. We indeed obtain parametric rates of convergence (up to a logarithmic factor) for the estimation of both the true weights and the location-scale parameters. Moreover, the estimator has a sparsity index exactly matching that of the target measure $\mu^0$. These bounds hold for sufficiently large sample sizes. We also stress that these results are obtained under a specific tuning of the parameters $\kappa$ and $\tau$, which differs from that used in the previous sections.

\section{Discussion and open problems}\label{section:further_remarks}

In this paper, we have described some theoretical properties of the BLASSO in the framework of Gaussian mixtures. In particular, we have considered the case where each underlying component has an unknown diagonal covariance. At this stage, further improvements and extensions are still possible and are discussed below.

\paragraph{Algorithmic considerations} Recall that Theorems \ref{th:estimation_error} and \ref{th:prediction_error} apply with any $\hat{\mu}_{n,\omega}$ that verifies $J_W(\hat{\mu}_{n,\omega})\leq J_W(\mu_\omega^0)$: our estimator is not necessarily the solution of \eqref{eq:pb_BLASSO_gaussian_kernel_norm}, but an approximate solution. So we can restrict \eqref{eq:pb_BLASSO_gaussian_kernel_norm} to any subset of $\M(\X)^+$ containing $\mu_\omega^0$. 

This remark is interesting from an algorithmic point of view. We do not know how to solve \eqref{eq:pb_BLASSO_gaussian_kernel_norm} numerically on the entire measure space $\M(\X)^+$, due to the absence of a parameterization of this space.
\citep{hardy_thesis} studied a version of the BLASSO over the space of $K$-sparse measures (discrete measures with less than $K$ particles).
This problem, although not convex, is closer to a realistic algorithmic framework, as we can parameterize the space of $K$-sparse measures. If $K\geq s$, $\mu_\omega^0$ belongs to this space.

Several algorithms have been proposed to solve the BLASSO on sparse measures, such as the sliding Frank-Wolfe algorithm \citep{sliding_frank_wolfe_denoyelle} or the Conic Particle Gradient Descent (CPGD) method \citep{cpgd_chizat}. The latter provides a fairly general algorithmic framework together with convergence rate guarantees. CPGD has been previously studied in the context of location mixture models in \citep{supermix, fastpart}. It is therefore natural to explore its applicability to the more challenging location–scale framework considered in this work. This investigation is part of our ongoing research. We present here a preliminary illustration from this line of work, relying on so-called ``natural gradient descent'' combined with ``conic retraction''.
Figure~\ref{fig:cpgd_training} illustrates the convergence behavior of the algorithm when targeting a mixture of two bivariate Gaussian components (see Figure~\ref{fig:target_mixture}). We have access to a sample of size $n=400$, and start from an initialization with three particles. The method successfully identifies the correct number of components, with one particle effectively disappearing as its associated weight converges to zero.

\begin{figure}[ht]
    \centering
    \includegraphics[width=0.42\linewidth]{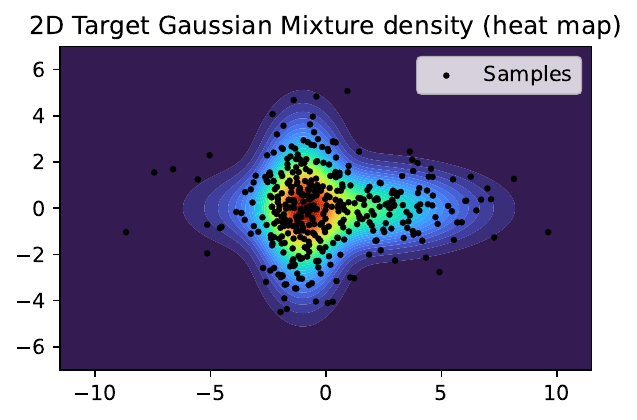}
    \caption{We have access to $n=400$ samples of the target mixture. The 2 target components have parameters $(a_1^0,a_2^0) = \left( \frac{1}{2}, \frac{1}{2}\right)$, $(t_1^0,t_2^0) = ((-1,0), (1,0))$, $(u_1^0,u_2^0) = ((1,2), (3,1))$. The corresponding separation with $\tau=0.8$ is $\mathpzc{d}(x_1^0,x_2^0) \simeq 1.1$ while $\Delta$ from Theorem \ref{th:check_local_curvature_assumption} is around $8.6$.
    }
    \label{fig:target_mixture}
\end{figure}

\begin{figure}[ht]
    \centering
    \includegraphics[width=1.0\linewidth]{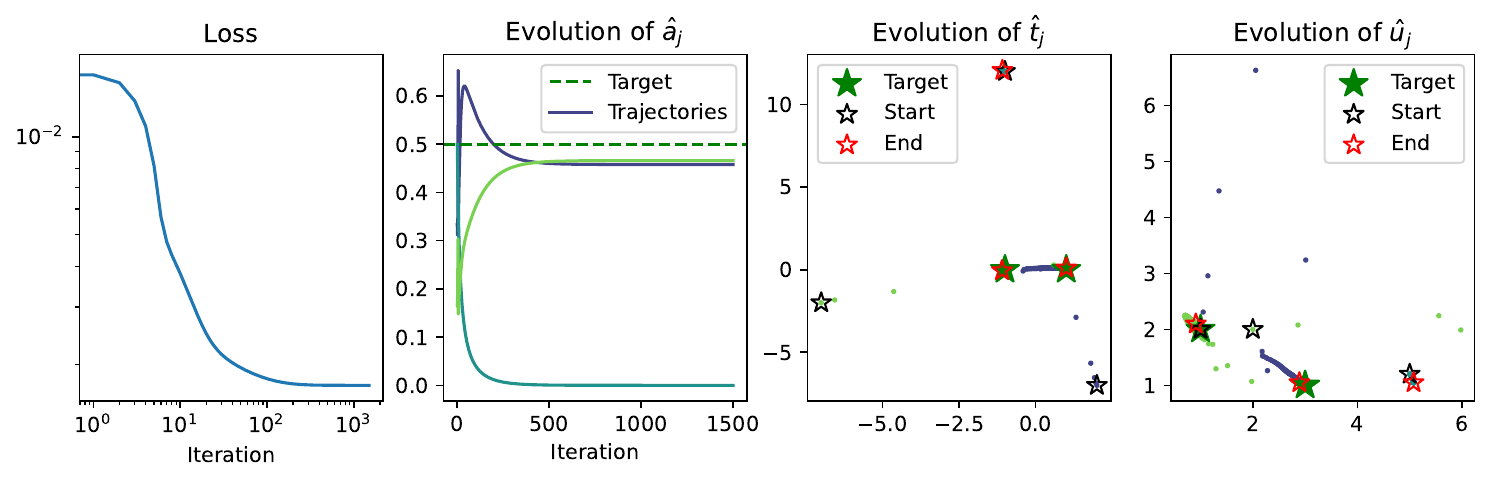}
    \caption{We take $\tau=0.8, \kappa=0.01$ and initialize CPGD with 3 particles. From left to right: the evolution of the loss during training; the evolution of the weights $\hat{a}$; of the parameters $\hat{t}$; and $\hat{u}$. The target is represented by green full stars, the initialization by empty black stars, the last estimate by empty red stars. The other iterates are represented by small dots, of different color for each particle.}
    \label{fig:cpgd_training}
\end{figure}

\paragraph{Non-diagonal covariance matrices}
In this contribution, we are only dealing with the case where the covariances of the mixture are diagonal.
Note that we can easily extend our model and results to the case where all covariances share the same (known) orientation, \ie we can diagonalize them in the same basis. However, we are not yet able to process general covariances with varying or unknown orientations. Indeed, the Fisher-Rao metric seems impractical to work with, mainly because it might be tricky to prove the compatibility between the associated distance and the semi-distance, and obtain the desired controls on the kernel. A possible outcome could be to consider an alternative metric.

\paragraph{Minimal separation condition}
Our results are established under a separation condition for the mixture components (Theorem \ref{th:main_theorem_certif}). In \citep{supermix}, the analysis is carried out for location Gaussian mixture models with a common, fixed covariance matrix across components. In this pure location setting, the minimal separation between components is allowed to shrink to zero as the sample size $n$ increases. This favorable behavior stems from the translation-invariant structure of the kernel arising in the pure location setting. Such kernels have been investigated in \citep{decastro_gribonval_jouvin}: this structure enables the construction of certificates based on a ``pivot'' kernel, whose decay can increase when dealing with closer particles. 

In contrast, we consider a substantially more complex location–scale model, where component-specific covariances are unknown and allowed to vary across components. In this setting, the associated kernel loses translation invariance and does not admit a practical pivot construction. The decay of $K_{\rm norm}$ does not allow us to consider $\Delta \to 0$ (for $(t,u)\neq (t',u') \in \R \times [u_{\min},+\infty)$, $e^{-\frac{(t-t')^2}{2(u^2+u'^2+\tau^2)}}$ cannot be as close to 0 as wanted by changing $\tau$).

\paragraph{Technical side notes}
We have used the semi-distance $\mathpzc{d}$ to define the near and far regions, and the Fisher-Rao metric to control the certificates with Taylor expansions. These are somewhat arbitrary or improvable choices. The semi-distance $\mathpzc{d}$ seems appropriate to control the kernel, because it measures precisely how 2 points are spaced for the kernel. But the downside is that it does not satisfy the triangle inequality. As a consequence, we need to take $\Delta_\tau$ much larger than $\Delta$. The Fisher-Rao metric allows us to retrieve quite easily global bounds for the kernel (\ie evaluations of the bounds $B_{ij}$, see item 1 of Definition \ref{def:positive_curvature}), but we could have used the Euclidean metric—although it leads to controls that depend on bounds on the variance.

\newpage
\section*{Summary of the main results}

The target measure is $\mu^0=\sum_{j=1}^s a_j^0\delta_{(t_j^0,u_j^0)} \in \M(\X)$. 
We assume in the following that $\X$ is a compact set of $\R^d \times [u_{\min},u_{\max}]^d$, and that $n\geq 2$, $s\geq 2$, $0<\tau \leq u_{\min}$. 
With the exception of prediction with small regularization, the results reported in the following tables require the following assumption. Keep in mind that $\mathpzc{d}$ depends on $\tau$.
\begin{assumption}\label{assumption:estimation}
 \[\min_{i \neq j} \mathpzc{d}(x_i^0, x_j^0) \geq \max \left\{\frac{\sqrt{u_{\max}^2 +\frac{1}{4} \frac{0.3025^2}{d}(2u_{\max}^2+\tau^2)}}{u_{\min}} \left(\Delta+\frac{0.3025}{\sqrt{d}}\right) \: , \: 2\frac{u_{\max}}{u_{\min}} \Delta \right\} + \sqrt{d\ln \left(\frac{u_{\max}^2}{u_{\min}^2}\right)}\] where $\Delta=2\sqrt{11.9+3\ln(d+6.62)+\ln(s-1)}$.
\end{assumption}

Some constants are omitted in the bounds, this is expressed using $\lesssim_c$ when the dependence on $c$ is not taken into account.

\begin{longtable}{>{\raggedright\arraybackslash}m{9em}>{\raggedright\arraybackslash}m{8em}>{\raggedright\arraybackslash}m{27em}} \caption{Recovery guarantees, $\kappa$ not depending on $s$} \label{table:results_summary_kappa_not_depending_on_s}\\
\toprule
 & \textbf{Parameters} & \textbf{Bounds} \\
 \midrule[\heavyrulewidth]
\endhead

\bottomrule
\endfoot

\textbf{Estimation} (Prop.~\ref{prop:effective_near_regions}, Thm.~\ref{th:main_theorem_certif}, Lem.~\ref{lemma:tilde_varepsilon_3},
Cor.~\ref{cor:control_renormalized_estimate})
& $\kappa = \frac{\sqrt{2}}{(2\pi\tau^2)^{d/4} \sqrt{n}}$
& For $0 < r_e \leq \frac{0.3025}{\sqrt{d}}$, $\mathbb{E}\left[ \left|\hat{\mu}_{n,\omega}(\X_j^{\rm near}(r_e)) - \mu_{\omega}^0(\X_j^{\rm near}(r_e)) \right| \right] \lesssim_d \frac{s}{r_e^2 \sqrt{n} \tau^{d/2}}$ \newline 
For $n^{-1/6} \leq \frac{0.3025}{\sqrt{d}}$, \newline $\mathbb{E}\left[ \left| a_j^0 - \frac{\hat{\mu}_{n,\omega}}{W}(\X_j^{\rm near}(n^{-1/6})) \right| \right] \lesssim_d \left(\frac{s}{\tau^{d/2}W(x_j^0)} + a_j^0\right)n^{-1/6}$ \\
\midrule

\textbf{Sparsity} (Cor.~\ref{cor:ndsc})
&$c_{\kappa,n}>0$,\newline $c_{\kappa,n}~=~o(\sqrt{n})$\newline $\kappa = \frac{c_{\kappa,n}}{(2\pi \tau^2)^{d/4} \sqrt{n}}$ 
& For $n \geq n_0$ (depends on $\mu^0, \X, \tau, (c_{\kappa,n})_n$), \newline with prob.\ $\geq 1 - C_\Gamma e^{-\left( \frac{\gamma_0 c_{\kappa,n}}{C_\Gamma} \right)^2}$:
 $\mu_{n,\omega}^\star = \sum_{j=1}^s \omega_j^\star \delta_{x_j^\star}$ and \newline
 $\mathpzc{d}(x_j^\star, x_j^0)^2 \lesssim_{\X, \mu^0, \tau} \frac{c_{\kappa,n}}{\sqrt{n}}$, \quad $|a_j^0 - a_j^\star| \lesssim_{\X, \mu^0, \tau} \frac{c_{\kappa,n}}{\sqrt{n}}$ \\
\midrule

\textbf{Prediction with small regularization}~(Prop.~\ref{prop:prediction_under_small_reg})
& $\tau = \frac{\sqrt{2} u_{\min}}{\sqrt{\ln n}}$ \newline $\kappa = \frac{4 (\ln n)^{d/2}}{(4\pi u_{\min}^2)^{d/2} n}$
& $\mathbb{E}\left[ \| \Phi \frac{\hat{\mu}_{n,\omega}}{W} - \Phi \mu^0 \|_{L^2}^2 \right] \lesssim_{u_{\min}, u_{\max}, d} \frac{(\ln n)^{d/2}}{n}$ \\
\midrule

\textbf{Prediction with a good estimator} (Thm.~\ref{th:prediction_error})
& $\tau = \frac{\sqrt{2} u_{\min}}{\sqrt{\ln n}}$ \newline $\kappa~=~\frac{\sqrt{2}(\ln n)^{d/2}}{(4\pi u_{\min}^2)^{d/4} \sqrt{n}}$
& $\mathbb{E}\left[ \| \Phi \frac{\hat{\mu}_{n,\omega}}{W} - \Phi \mu^0 \|_{L^2}^2 \right] \lesssim_{u_{\min}, u_{\max}, d} \frac{s (\ln n)^{d/2}}{n}$ \\
\end{longtable}

\vspace{-0.1cm}
\begin{longtable}{>{\raggedright\arraybackslash}m{9em}>{\raggedright\arraybackslash}m{8em}>{\raggedright\arraybackslash}m{27em}}
\caption{Recovery guarantees, $\kappa$ depending on $s$} \label{table:results_summary_kappa_depending_on_s} \\
\toprule
 & \textbf{Parameters} & \textbf{Bounds} \\
 \midrule[\heavyrulewidth]
\endhead

\bottomrule
\endfoot

\textbf{Estimation, v2} (Rk.~\ref{remark:choice_kappa_cor_near_effective}, Thm.~\ref{th:main_theorem_certif}, Lem.~\ref{lemma:tilde_varepsilon_3})
& $\kappa = \frac{\sqrt{2}}{(2\pi\tau^2)^{d/4} \sqrt{sn}}$
& For $0 < r_e \leq \frac{0.3025}{\sqrt{d}}$, \newline $\mathbb{E}\left[ \left|\hat{\mu}_{n,\omega}(\X_j^{\rm near}(r_e)) - \mu_{\omega}^0(\X_j^{\rm near}(r_e)) \right| \right] \lesssim_d \frac{\sqrt{s}}{r_e^2 \sqrt{n} \tau^{d/2}}$ \newline 
For $n^{-1/6} \leq \frac{0.3025}{\sqrt{d}}$, \newline $\mathbb{E}\left[ \left| a_j^0 - \frac{\hat{\mu}_{n,\omega}}{W}(\X_j^{\rm near}(n^{-1/6})) \right| \right] \lesssim_d \left(\frac{\sqrt{s}}{\tau^{d/2}W(x_j^0)} + a_j^0\right)n^{-1/6}$ \\
\midrule

\textbf{Prediction with a good estimator, v2} (Rk.~\ref{remark:choice_kappa_pred})
& $\tau = \frac{\sqrt{2} u_{\min}}{\sqrt{\ln n}}$ \newline $\kappa = \frac{\sqrt{2}}{(2\pi \tau^2)^{d/4} \sqrt{sn}}$
& If $s=O(n (\ln n)^{d/2})$, $\mathbb{E}\left[ \| \Phi \frac{\hat{\mu}_{n,\omega}}{W} - \Phi \mu^0 \|_{L^2}^2 \right] \lesssim_{u_{\min}, u_{\max}, d} \frac{(\ln n)^{d/2}}{n}$ \\
\end{longtable}

\newpage

\section*{Notation}\label{section:notation}
\renewcommand{\arraystretch}{1.4}

\begin{longtable}{p{0.17\textwidth}p{0.78\textwidth}}
\caption{Table of notations} \label{table:notations}\\
\toprule
\multicolumn{2}{c}{\textbf{Global notation}} \\
\midrule
$\X$ & compact of $\R^d\times [u_{\min},+\infty)^d$ with $u_{\min},u_{\max}>0$ \\
$\M(\X)^+$ & nonnegative Radon measures on $\X$ \\
$W$ & reparameterization function. For $x=((t_1,\ldots,t_d),(u_1,\ldots,u_d))~\in~\R^d~\times~[u_{\min},+\infty)^d$, $W(x)= \prod_{k=1}^d (2\pi)^{-1/4}(2u_k^2+\tau^2)^{-1/4}$ \\
$\mu^0,\mu_\omega^0$ & resp.\ the target probability measure $\sum_{j=1}^{s}a_j^0 \delta_{x_j^0}$ where $x_j^0=(t_j^0,u_j^0) \in \X$; its reparameterized version $W\mu^0=\sum_{j=1}^s \omega_j^0 \delta_{x_j^0}$ \\
$\varphi$, $\sigma$, $\Phi$ & resp.\ the function $z\in\R\mapsto \frac{e^{-\frac{z^2}{2}}}{\sqrt{2\pi}}$; its Fourier transform; the operator $\displaystyle\mu~\mapsto~\left(z \in \R^d \mapsto \int_{\R^d \times [u_{\min},+\infty)^d} \prod_{k=1}^d \frac{1}{u_k} \varphi\left(\frac{z_k - t_k}{u_k}\right) \,\d \mu(t,u)\right)$ \\
$X_1,\ldots,X_n$ & i.i.d.\ observations in $\R^d$, drawn from $f^0=\Phi \mu^0$ \\
$\mathbb{E}$ & expected value \wrt $X_1,\ldots,X_n$ \\
$\hat{f}_n$ & empirical density $\frac{1}{n} \sum_{i=1}^{n} \delta_{X_i}$ \\
$\kappa$ & regularization constant \\
$\hat{\mu}_{n,\omega},\hat{\mu}_n, \mu_{n,\omega}^\star$ & resp.\ a measure of $\M(\X)^+$ such that $J_W(\hat{\mu}_{n,\omega})~\leq~J_W(\mu_{\omega}^0)$ (see \eqref{eq:pb_BLASSO_gaussian_kernel_norm}); the measure $\frac{\hat{\mu}_{n,\omega}}{W}$; an exact solution of \eqref{eq:pb_BLASSO_gaussian_kernel_norm} \\
$\Gamma_n, \rho_n^2$ & resp.\ the noise $L\circ \hat{f}_n - L\circ f^0$; a bound on $\E{\norm{\Gamma_n}_{\L}^2}$ equal to $\frac{4}{(2\pi)^{d/2}\tau^d n}$ \\\\
\midrule
\multicolumn{2}{c}{\textbf{Kernel, differential geometry}} \\
\midrule
$\tau$, $\lambda$, $\Lambda$, $L$ & resp.\ the smoothing parameter $\tau$; the function $z\in\R^d\mapsto \frac{e^{-\frac{\norm{z}_2^2}{2\tau^2}}}{(2\pi\tau^2)^{d/2}}$; its Fourier transform; the operator $f \mapsto \lambda * f$ \\
$\L$ & RKHS associated with $\lambda$ (scalar product given by \eqref{eq:dot_product_L}) \\
$\Psi, K_{\rm norm}$ & resp.\ the feature map $\mu\in \M(\R^d \times [u_{\min},+\infty)^d) \mapsto L\circ \Phi\frac{\mu}{W}$; the normalized kernel $(x,x')\in \R^d \times [u_{\min},+\infty)^d \mapsto \innerprod{\Psi \delta_x}{\Psi \delta_{x'}}_\L$ \\
$\mathpzc{d}$ & semi-distance $(x,x')\mapsto \sqrt{-2\ln(K_{\rm norm}(x,x'))}$ (expression given by \eqref{eq:def_semi_distance}) \\
$\X_j^{\rm near}(r),\X^{\rm far}(r)$ & resp.\ the jth near region of radius $r$, $\{x\in \X \, : \, \mathpzc{d}(x,x_j^0)\leq r\}$; the far region $\X~\setminus~\bigcup_{j=1}^s~\X_j^{\rm near}(r)$ \\
$\eta$ & dual certificate for $\mu^0$, of the form $\Psi^* p$ with $p\in \L$ \\
$D_{\eta}(\hat{\mu}_{n,\omega},\mu_{\omega}^0)$ & Bregman divergence $\norm{\hat{\mu}_{n,\omega}}_{\rm TV}-\norm{\mu_{\omega}^0}_{\rm TV}-\int_{\X}\eta\, \d\left(\hat{\mu}_{n,\omega}-\mu_{\omega}^0 \right)$ \\
$\mathfrak{g},\mathfrak{d}_\mathfrak{g}$ & resp.\ the Fisher-Rao metric (see \eqref{eq:def_fisher_rao_metric}); the associated distance \\
$\gamma,\tilde \gamma$ & geodesic for the Fisher-Rao metric (parameterized on $[0,1]$ and by arc-length respectively) \\\\
\bottomrule
\end{longtable}

\newpage
\printbibliography

\newpage

\appendix
{
\begin{center}
\LARGE
    Gaussian Mixture Model with unknown diagonal covariances via continuous sparse regularization\\ --- \\ Appendix
\end{center}

\bigskip
{\ }
}
\section{Functional framework}\label{section:functional_framework}
The BLASSO operates on the space of Radon measures. In this section, we provide definitions of the operators on measures used throughout the paper.

Let $A \subset \R^d$.
In the following, $\mathcal{L}^\infty(A)$ is the set of bounded functions from $A$ to $\R$; $L^1(A)$ (resp. $L^2(A)$) the set of functions whose absolute value (resp. square) have a finite integral. The notions of convolution and Fourier transform can be extended to measures. 

\begin{definition}[Convolution $g* \mu$]
Let $A \subset \R^p$. The convolution between $g\in \mathcal{L}^\infty(A)$ and $\mu \in \M(A)$ is defined by \[g * \mu \coloneq \int_{A} g(\dotp - x) \d \mu(x) \,.\]
\end{definition}

\begin{definition}[Fourier transform over $L^1(A)$ and $\M(A)$]
Let $A \subset \R^p$. We use the Fourier transform defined for $g\in L^1(A)$ by \[\F{g}(\xi)\coloneq \int_{A} e^{-i \innerprod{\xi}{x}} g(x) \, \d x \,.\]
We also use its extension to $\M(A)$:
\[\F{\mu}(\xi)\coloneq \int_{A} e^{-i \innerprod{\xi}{x}} \, \d\mu(x) \quad \forall \, \mu \in \M(A)\,.\]
\end{definition}

\noindent
The standard properties of convolution and Fourier transform apply, such as $\F{g*\mu}=\F{g}\F{\mu}$ for all $g \in L^1(A) \cap \mathcal{L}^\infty(A)$ and all $\mu \in \M(A)$ (see, for example, \citep[Part 2]{rudin_functional_analysis}).

\section{Existence of a solution to the BLASSO}
\begin{proposition}\label{prop:existence_solution_blasso}
 The problem \eqref{eq:pb_BLASSO_gaussian_kernel_norm} has a solution.
\end{proposition}
\begin{proof}
\underline{$J_{W}$ is lower semi-continuous on $\M(\X)^+$ for the weak* convergence:} 
The TV norm is lower semi-continuous for the weak* convergence. It then suffices to show that $\mu \mapsto L \circ \Phi \frac{\mu}{W}$ is weak* to weak continuous. 
Let $\mu_j \overset{*}{\rightharpoonup} \mu$ in $\M(\X)$. We want to show that for $f\in\L$, \[\innerprod{L\circ \Phi \frac{\mu_j}{W}}{f}_{\L}\rightarrow~\innerprod{L\circ \Phi \frac{\mu}{W}}{f}_{\L}\,.\]
Using \citep[Lemma 4.29]{steinwart}, as \[K_{\rm norm}: (x,x') \in (\R^d \times [u_{\min},+\infty)^d)^2 \mapsto \innerprod{L\circ \Phi \frac{\delta_x}{W}}{L\circ \Phi \frac{\delta_{x'}}{W}}_\L\] is continuous (see \eqref{eq:kernel_K_norm}), we have \[
\left(x \mapsto L\circ \Phi \frac{\delta_x}{W}\right) \in \C(\X,\L)\,.
\]
Hence \[\left(x\mapsto \innerprod{L\circ \Phi \frac{\delta_x}{W}}{f}_{\L}\right) \in \C(\X)\,.\] 
As $\M(\X)=\C(\X)^*$, it comes \[\int \innerprod{L\circ \Phi \frac{\delta_x}{W}}{f}_{\L}\, \d \mu_j(x)\rightarrow \int\innerprod{L\circ \Phi \frac{\delta_x}{W}}{f}_{\L}\, \d \mu(x)\,,\] which concludes this part of the proof.
\\\underline{Conclusion:} The result follows noticing that $J_{W}(\mu)\underset{\norm{\mu}_{\rm TV}\to \infty}{\longrightarrow} \infty$: we can restrict the problem to a closed ball of $\M(\X)^+$. By the Banach–Alaoglu theorem, this ball is weakly* compact (a ball of $\M(\X)$ is weakly* compact, and $\M(\X)^+$ is a weakly* closed subspace of $M(\X)$). We deduce the existence of a minimizer of $J_W$ on $\M(\X)^+$.
\end{proof}

\section{Proof of Lemma \ref{lemma:control_noise}}\label{section:proof_control_noise}
\underline{Expected value of $\norm{\Gamma_n}_\L^2$:}
 Note that, according to the definition of $\Gamma_n$,
\[
\Gamma_n = \frac{1}{n} \sum_{i=1}^{n}\left(L \circ \delta_{X_i} - L \circ f^0 \right)\,.
\]
For all $h \in \L$, for all $i \in \{1,\ldots, n\}$, we have $\mathbb{E}\left[\innerprod{L \circ \delta_{X_i}}{h}_\L\right]=\innerprod{L \circ f^0}{h}_\L$. This entails that $\E{\innerprod{\Gamma_n}{h}_\L}=0$. 
Defining \[Z_i\coloneq L \circ \delta_{X_i}- L\circ f^0\] (where we recall that $f^0=\Phi \mu^0$), observe that $Z_1,\ldots,Z_n$ are i.i.d. Since for all $i \in \{1,\ldots, n\}$, \[\norm{Z_i}_\L^2=\frac{1}{(2\pi)^d} \int_{\R^d} \Lambda |\F{\delta_{X_i}-f^0}|^2\leq \frac{4 \int_{\R^d} \Lambda}{(2 \pi)^d}\,,\] we have 
\begin{align*}
 \norm{\Gamma_n}_\L^2& = \frac{1}{n^2}\sum_{i=1}^n \norm{Z_i}_\L^2 + \frac{1}{n^2} \sum_{i \neq j} \innerprod{Z_i}{Z_j}_\L \,, \\
 & \leq \frac{4 \int_{\R^d} \Lambda}{(2 \pi)^d n} + \frac{1}{n^2} \sum_{i \neq j} \innerprod{Z_i}{Z_j}_\L \,.
\end{align*}
We deduce that $\E{\norm{\Gamma_n}_\L^2}\leq \frac{4 \int_{\R^d} \Lambda}{(2 \pi)^d n}$ using that $\mathbb{E}[Z_i]=0$ for all $i\in \lbrace 1,\dots, n \rbrace$.
\\\underline{Control in probability:} The control in probability of $\norm{\Gamma_n}_\L^2$ comes from \citep[Lemma 3]{supermix_supp}, and stems from results on U-processes (see \citep[Proposition 2.3]{u_process}). In particular, we have 
\begin{equation}
\forall \rho >0\,, \quad \P{\norm{\Gamma_n}_{\L}^2> \rho \frac{C_{\Gamma}^2 }{n\tau^d(2\pi)^{d/2}}} \leq C_{\Gamma}e^{-\rho} \, ,
\label{eq:deviation_Gamma_n}
\end{equation}
for some positive constant $C_\Gamma>0$. 
\\\underline{Expected value of $\norm{\Gamma_n}_\L^4$:} Using \eqref{eq:deviation_Gamma_n}, it comes that
\begin{align*}
 \E{\norm{\Gamma_n}_\L^4}&= \int_0^\infty \P{\norm{\Gamma_n}_\L^4> x} \, \d x \,,\\
 &=\int_0^\infty \P{\norm{\Gamma_n}_\L^2> \sqrt{x}} \, \d x \,,\\
 &\leq \int_0^\infty C_\Gamma e^{-\sqrt{x} \frac{n \tau^d (2\pi)^{d/2}}{C_\Gamma^2}}\, \d x \,,
\end{align*}
so \[\E{\norm{\Gamma_n}_\L^4} \leq C_\Gamma\frac{C_\Gamma^4}{n^2 \tau^{2d} (2\pi)^{2d}} \int_0^\infty e^{-\sqrt{x}} \, \d x = \frac{2C_\Gamma^5}{n^2 \tau^{2d} (2\pi)^{2d}} \,.\]
This proves the desired result.

\section{Proof of Theorem \ref{th:estimation_error}}\label{section:proof_th_estimation_error}
The following proof is standard when dealing with the BLASSO procedure. Its main arguments can be found in \citep{supermix}; we adapt their proof to our setting. It is based on bounds on the Bregman divergence introduced in \eqref{eq:def_bregman_divergence}. With $W$ defined by \eqref{eq:def_W}, we denote in the following $\hat{\mu}_{n}=\frac{\hat{\mu}_{n,\omega}}{W}$. We also recall from \eqref{eq:mu_omega} that $\mu^0 = \frac{\mu_{\omega}^0}{W}$.
\\\underline{Upper bound on the Bregman divergence:} Since $J_{W}(\hat{\mu}_{n,\omega})\leq J_W(\mu_\omega^0)$, we get 
\begin{equation}\label{eq:basic_ineq_mu_minimizer_v0}
 \norm{L \circ \hat{f}_n-L \circ \Phi \hat{\mu}_{n}}_{\L}^2+2\kappa \norm{\hat{\mu}_{n,\omega}}_{\rm TV} 
\leq\norm{L\circ \hat{f}_n-L \circ \Phi \mu^0}_{\L}^2+2\kappa \norm{\mu_{\omega}^0}_{\rm TV}
\end{equation}
which can be rewritten as
\begin{equation}
\norm{L \circ \hat{f}_n-L \circ \Phi \hat{\mu}_n}_{\L}^2+2\kappa D_{\eta}\left(\hat{\mu}_{n,\omega}, \mu_\omega^0\right) +2\kappa \int_{\X} \eta \, \d(\hat{\mu}_{n,\omega}-\mu_\omega^0) \leq\norm{\Gamma_n}_{\L}^2 \, ,
\label{eq:basic_ineq_mu_minimizer}
\end{equation}
for any dual certificate $\eta$ satisfying the requirements of Assumption \ref{assumption:existence_non-degenerate_certif}. As $\eta=\Psi^*p$, recalling \eqref{eq:scalar_product_Psi*} we have \[\int_{\X} \eta\, \d(\hat{\mu}_{n,\omega}-\mu_\omega^0) =\innerprod{p}{L\circ \Phi(\hat{\mu}_n - \mu^0)}_{\L}\,.\] So using Cauchy-Schwarz inequality, \eqref{eq:basic_ineq_mu_minimizer} leads to 
 \[
 \norm{L \hat{f}_n-L \circ \Phi \hat{\mu}_n}_{\L}^2+ 2\kappa D_{\eta}\left(\hat{\mu}_{n,\omega}, \mu_\omega^0\right)-2\kappa\norm{p}_{\L}\norm{L\circ\Phi\mu^0-L \circ \Phi \hat{\mu}_n}_{\L} \leq \norm{\Gamma_n}_{\L}^2\,.
 \] The triangle inequality \[\norm{L\circ\Phi\mu^0-L \circ \Phi \hat{\mu}_n}_{\L}\leq \norm{\Gamma_n}_{\L}+\norm{L \hat{f}_n-L \circ \Phi \hat{\mu}_n}_{\L}
 \]
 then leads to
 \begin{equation*}
 \left(\norm{L \circ\hat{f}_n-L \circ \Phi \hat{\mu}_n}_{\L}-\kappa \norm{p}_{\L}\right)^2+ 2\kappa D_{\eta}\left(\hat{\mu}_{n,\omega}, \mu_\omega^0\right) \leq \left(\norm{\Gamma_n}_{\L}+\kappa\norm{p}_{\L}\right)^2 \,.
 \end{equation*}
As the Bregman divergence is positive, we deduce from the previous inequality that
 \begin{equation}\label{eq:ineq_mu_minimizer}
D_{\eta}\left(\hat{\mu}_{n,\omega}, \mu_\omega^0\right) \leq \frac{\norm{\Gamma_n}_{\L}^2}{2\kappa}+\frac{\kappa}{2} \norm{p}_{\L}^2+ \norm{\Gamma_n}_{\L}\norm{p}_{\L} \quad \text{and} \quad \norm{L \circ\hat{f}_n-L \circ \Phi \hat{\mu}_n}_{\L} \leq \norm{\Gamma_n}_{\L}+2\kappa\norm{p}_{\L}\,.
\end{equation}
The bound in expected value on the Bregman divergence follows by applying Lemma \ref{lemma:control_noise} and using $\norm{p}_\L \leq \sqrt{c_p s}$. We have 
\begin{equation}\label{eq:ineq_expected_value_bregman}
 \E{D_{\eta}\left(\hat{\mu}_{n,\omega}, \mu_{\omega}^0\right)} \leq \frac{\rho_n^2}{2\kappa}+\kappa \frac{c_p}{2} s+ \rho_n \sqrt{c_ps}\,.
\end{equation}
Taking $\kappa=\frac{\rho_n}{\sqrt{c_p}}$ gives 
\begin{equation} \label{eq:expected_value_divergence_bregman_kappa_agnostic}
 \E{D_{\eta}\left(\hat{\mu}_{n,\omega}, \mu_{\omega}^0\right)} \leq \frac{\sqrt{c_p}}{2}\rho_n(1+ \sqrt{s})^2\,.
\end{equation}
\underline{Lower bound on the Bregman divergence:} We use the controls of the non-degenerate certificate $\eta$ on the near and far regions. Since $\hat{\mu}_{n,\omega}$ is nonnegative, according to Definition \ref{def:non_degenerate_certificate} we have 
\begin{equation} \label{eq:ineq_bregman_epsilon}
D_{\eta}\left(\hat{\mu}_{n,\omega}, \mu_\omega^0\right)= \int (1- \eta) \, \d \hat{\mu}_{n,\omega} \geq \varepsilon_0 \hat{\mu}_{n,\omega}(\X^{\rm far}(r))+ \varepsilon_2 \sum_{j=1}^{s} \int_{\X_j^{\rm near}(r)} \mathfrak{d}_{\mathfrak{g}}(x,x_j^0)^2 \, \d \hat{\mu}_{n,\omega}(x)\,.
\end{equation}
Combining \eqref{eq:expected_value_divergence_bregman_kappa_agnostic} and \eqref{eq:ineq_bregman_epsilon}, we deduce the bound for the far region (item 1 of Theorem \ref{th:estimation_error}).
To control the mass of the estimator on the jth near region, we make use of the local non-degenerate certificate $\eta_j$ (Definition \ref{def:local_non_degenerate_certificate}).
We have, for all $j=1,\ldots,s$, 
\begin{align}
 |\omega_j^0 -\hat{\mu}_{n,\omega}(\X_j^{\rm near}(r)) |
 &= \left| \omega_j^0 - \int \eta_j \, \d \hat{\mu}_{n,\omega} + \int \eta_j \, \d \hat{\mu}_{n,\omega} - \int_{\X_j^{\rm near}(r)} \d \hat{\mu}_{n,\omega} \right| \,, \notag \\
 &\leq \left|\int \eta_j \, \d (\mu_\omega^0-\hat{\mu}_{n,\omega})\right| +\int_{\X^{\rm near}(r) \setminus \X_j^{\rm near}(r)} |\eta_j| \, \d \hat{\mu}_{n,\omega} + \int_{\X_j^{\rm near}(r)} |1-\eta_j| \,\d \hat{\mu}_{n,\omega} \notag \\
 &\quad +\int_{\X^{\rm far}(r)} |\eta_j|\d \hat{\mu}_{n,\omega} \,, \notag \\
 \begin{split}
    &\leq |\innerprod{p_j}{L\circ\Phi \mu^0-L\circ\Phi \hat{\mu}_n}_{\L}| + \tilde\varepsilon_2 \sum_{l=1}^{s}\int_{\X_l^{\rm near}(r)} \mathfrak{d}_{\mathfrak{g}}(x,x_l^0)^2 \, \d \hat{\mu}_{n,\omega}(x) \\
    &\quad + (1-\tilde\varepsilon_0)\hat{\mu}_{n,\omega}(\X^{\rm far}(r))  \,, 
 \end{split}\label{eq:intermediate_control_weights} \\
 &\leq \norm{p_j}_{\L}(2 \norm{\Gamma_n}_{\L} + 2\kappa \norm{p}_{\L}) + \max\left\{\frac{1-\tilde \varepsilon_0}{\varepsilon_0},\frac{\tilde\varepsilon_2}{\varepsilon_2} \right\}D_{\eta}\left(\hat{\mu}_{n,\omega}, \mu_\omega^0\right) \,, \notag
\end{align}
where for the last inequality, we have used \eqref{eq:ineq_bregman_epsilon}, and \[|\innerprod{p_j}{L\circ\Phi \mu^0-L\circ\Phi \hat{\mu}_n}_{\L}| \leq \norm{p_j}_{\L} \left(\norm{\Gamma_n}_{\L}+ \norm{L \circ \hat{f}_n - L \circ\Phi \hat{\mu}_n}_{\L} \right) \leq \norm{p_j}_{\L} \left(2\norm{\Gamma_n}_{\L}+ 2\kappa \norm{p}_{\L}\right) \] together with \eqref{eq:ineq_mu_minimizer}.
As $\norm{p_j}_\L \leq \sqrt{c_p}$, $\norm{p}_\L \leq \sqrt{c_p s}$ and $\kappa=\frac{\rho_n}{\sqrt{c_p}}$, we finally have \[ |\omega_j^0 -\hat{\mu}_{n,\omega}(\X_j^{\rm near}(r)) | \leq 2\sqrt{c_p}(\norm{\Gamma_n}_{\L}+\rho_n \sqrt{s}) +\max\left\{\frac{1-\tilde \varepsilon_0}{\varepsilon_0},\frac{\tilde\varepsilon_2}{\varepsilon_2} \right\}D_{\eta}\left(\hat{\mu}_{n,\omega}, \mu_\omega^0\right)\,.\] The result in expected value follows from $\E{\norm{\Gamma_n}_\L}\leq \rho_n$ (Lemma \ref{lemma:control_noise} used with Jensen's inequality) and \eqref{eq:expected_value_divergence_bregman_kappa_agnostic}.

For the stability of the mass (item 3), we use a similar but simpler reasoning. As $\eta \leq 1$ and $\hat{\mu}_{n,\omega}$ is nonnegative, we have 
\begin{align*}
 \norm{\mu_\omega^0}_{\rm TV}- \norm{\hat{\mu}_{n,\omega}}_{\rm TV}&= \int \eta \, (\d\mu_\omega^0-\d \hat{\mu}_{n,\omega} )+ \int(\eta-1)\, \d \hat{\mu}_{n,\omega}\,, \\
 &\leq \int \eta \, (\d\mu_\omega^0-\d \hat{\mu}_{n,\omega} ) \,, \\
 &= \innerprod{p}{L\circ\Phi \mu^0-L\circ\Phi \hat{\mu}_n}_{\L}\,, \\
 &\leq \norm{p}_{\L}(2 \norm{\Gamma_n}_{\L} + 2\kappa \norm{p}_{\L})\,.
\end{align*}
Using \eqref{eq:basic_ineq_mu_minimizer_v0}, we also have $\norm{\hat{\mu}_{n,\omega}}_{\rm TV}\leq \norm{\mu_\omega^0}_{\rm TV} + \frac{1}{2 \kappa} \norm{\Gamma_n}_{\L}^2$. We can conclude by taking the expectation in these inequalities and using Lemma \ref{lemma:control_noise}.
\\\underline{With an $s$-dependent choice of regularization:} Choosing $\kappa=\frac{\rho_n}{\sqrt{c_p s}}$, \eqref{eq:ineq_expected_value_bregman} gives $\E{D_\eta(\hat{\mu}_{n,\omega},\mu_{\omega}^0)} \leq 2\rho_n \sqrt{c_p s}$. It comes that 
\begin{equation}\label{eq:estim_kappa_s_dependent}
 \E{|\omega_j^0 -\hat{\mu}_{n,\omega}(\X_j^{\rm near}(r)) |} \leq 2\sqrt{c_p}\rho_n (1+ \sqrt{s}) +\max\left\{\frac{1-\tilde \varepsilon_0}{\varepsilon_0},\frac{\tilde\varepsilon_2}{\varepsilon_2} \right\}2\rho_n \sqrt{c_p s}\,.   
\end{equation}

\section{Basic inequalities}
The following lemma gives inequalities useful when dealing with the semi-distance $\mathpzc{d}$. We will use them in various proofs.
\begin{lemma}[Basic inequalities]\label{lemma:basic_inequalities}
Let $a,b\in\R_+^{*}$, $c\geq 1$. Then
 \begin{equation}\label{eq:control_close_variances}
 \frac{a^2+b^2}{2ab} \leq c \iff a \in \left[b(c-\sqrt{c^2 -1}), b(c+\sqrt{c^2 -1})\right]\end{equation}
 and \begin{equation}\label{eq:control_close_variances_2}
 \frac{a^2+b^2}{2ab} \leq c \implies \frac{|a^2-b^2|}{a^2+b^2} \leq \sqrt{c^2-1}
 \end{equation}
 along with 
 \begin{equation}\label{eq:control_close_variances_3}
 \frac{a^2+b^2}{2ab} \leq c \implies \frac{2a^2}{a^2+b^2} \leq c+\sqrt{c^2-1} \,.
 \end{equation}
\end{lemma}
\begin{proof}
Let $a,b \in \R_+^{*}$ and $c\geq 1$.
 \\\underline{Proof of \eqref{eq:control_close_variances}:} We have 
 \[\frac{a^2+b^2}{2ab} \leq c \iff a^2+b^2-2cab\leq 0\,.\] As $4c^2b^2-4b^2=4b^2(c^2-1)\geq 0$, the roots of this polynomial in $a$ are 
 \[\frac{2cb \pm 2b\sqrt{c^2-1}}{2}=b(c \pm \sqrt{c^2-1})\] from which we deduce \eqref{eq:control_close_variances}.
 \\\underline{Proof of \eqref{eq:control_close_variances_2}:}
 Using \eqref{eq:control_close_variances} and $2ab\leq a^2+b^2$, we get 
 \begin{align*}
 \frac{|a^2-b^2|}{a^2+b^2} &\leq \frac{|a-b||a+b|}{2ab} \,,\\
 &\leq \frac{ \min\{a,b\}(1+c+\sqrt{c^2-1}) \max\{a,b\}(1-(c-\sqrt{c^2-1}))}{2ab} \,, \\
 &= \frac{(1+c+\sqrt{c^2-1})(1+\sqrt{c^2-1}-c)}{2} \,, \\
 &= \sqrt{c^2-1} \,.
 \end{align*}
 \underline{Proof of \eqref{eq:control_close_variances_3}:}
 Using again \eqref{eq:control_close_variances} with $2ab\leq a^2+b^2$, we have 
 \begin{align*}
 \frac{2a^2}{a^2+b^2} &\leq \frac{2ab(c+\sqrt{c^2-1})}{2ab} \,,\\
 &\leq c+\sqrt{c^2-1}\,.
 \end{align*}
\end{proof}

\section{Proofs related to the control of the estimator on the effective near regions}
\subsection{Proof of Proposition \ref{prop:effective_near_regions}}\label{section:proof_prop_effective_near}
The proof is similar to that of Theorem \ref{th:estimation_error} in Section \ref{section:proof_th_estimation_error}. Let $j\in \lbrace 1,\dots, s \rbrace$ and $\eta_j$ a corresponding local non-degenerate dual certificate satisfying the requirements of Assumption \ref{assumption:existence_non-degenerate_certif}. According to Definition \ref{def:local_non_degenerate_certificate}, for all $x\in \X_j^{\rm near}(r)$, we have $|\eta_j|\leq 1+\tilde \varepsilon_2 \mathfrak{d}_{\mathfrak{g}}(x,x_j^0)^2$. We deduce that
\begin{align*}
 |\omega_j^0 -\hat{\mu}_{n,\omega}(\X_j^{\rm near}(r_e)) |&= \left| \omega_j^0 - \int \eta_j \, \d \hat{\mu}_{n,\omega} + \int \eta_j \, \d \hat{\mu}_{n,\omega} - \int_{\X_j^{\rm near}(r_e)} \d \hat{\mu}_{n,\omega} \right| \,, \\
 &\leq \left|\int \eta_j \, \d (\mu_\omega^0-\hat{\mu}_{n,\omega})\right| +\int_{\X^{\rm near}(r) \setminus \X_j^{\rm near}(r)} |\eta_j| \, \d \hat{\mu}_{n,\omega} + \int_{\X_j^{\rm near}(r) \setminus \X_j^{\rm near}(r_e)} |\eta_j| \, \d \hat{\mu}_{n,\omega} \\ & \quad + \int_{\X_j^{\rm near}(r_e)} |1-\eta_j| \,\d \hat{\mu}_{n,\omega} +\int_{\X^{\rm far}(r)} |\eta_j|\d \hat{\mu}_{n,\omega} \,, \\
 &\leq \norm{p_j}_{\L}(2 \norm{\Gamma_n}_{\L} + 2\kappa \norm{p}_{\L}) + \tilde\varepsilon_2 \sum_{l\neq j}\int_{\X_l^{\rm near}(r)} \mathfrak{d}_{\mathfrak{g}}(x,x_l^0)^2 \, \d \hat{\mu}_{n,\omega} + (1-\tilde\varepsilon_0)\hat{\mu}_{n,\omega}(\X^{\rm far}(r))\\
 & \quad + \tilde\varepsilon_2\int_{\X_j^{\rm near}(r_e)} \mathfrak{d}_{\mathfrak{g}}(x,x_j^0)^2 \, \d \hat{\mu}_{n,\omega} +\int_{\X_j^{\rm near}(r)\setminus \X_j^{\rm near}(r_e)} (1+\tilde\varepsilon_2\mathfrak{d}_{\mathfrak{g}}(x,x_j^0)^2) \, \d \hat{\mu}_{n,\omega} \,. 
\end{align*} 
From \eqref{eq:assumption_tilde_varepsilon_3} we get $1 \leq \frac{\tilde\varepsilon_3}{r_e^2} \mathfrak{d}_{\mathfrak{g}}(x_j^0,x)^2$ for all $x\in \X_j^{\rm near}(r) \setminus \X_j^{\rm near}(r_e)$, so $1+\tilde \varepsilon_2 \mathfrak{d}_{\mathfrak{g}}(x_j^0,x)^2 \leq \left(\frac{\tilde\varepsilon_3}{r_e^2}+ \tilde \varepsilon_2 \right)\mathfrak{d}_{\mathfrak{g}}(x_j^0,x)^2$.
Using again \eqref{eq:ineq_bregman_epsilon}, we deduce that 
\begin{equation}\label{eq:prop_control_effective_near}
 |\omega_j^0 -\hat{\mu}_{n,\omega}(\X_j^{\rm near}(r_e))| \leq \norm{p_j}_{\L}(2 \norm{\Gamma_n}_{\L} + 2\kappa \norm{p}_{\L}) + \max\left\{ \frac{1-\tilde \varepsilon_0}{\varepsilon_0}, \frac{1}{\varepsilon_2}\left(\frac{\tilde \varepsilon_3}{r_e^2}+ \tilde\varepsilon_2\right) \right\}D_{\eta}(\hat{\mu}_{n,\omega}, \mu_\omega^0) \,. 
\end{equation}We can conclude the proof using the controls on $\E{\norm{\Gamma_n}_{\L}}$, $\norm{p_j}_\L$, $\norm{p}_\L$, $\E{D_{\eta}(\hat{\mu}_{n,\omega}, \mu_\omega^0)}$ stemming from our assumptions along with Lemma \ref{lemma:control_noise} and \eqref{eq:expected_value_divergence_bregman_kappa_agnostic}. Choosing $\kappa=\frac{\rho_n}{\sqrt{c_p}}$,
\[|\omega_j^0 -\hat{\mu}_{n,\omega}(\X_j^{\rm near}(r_e))| \leq 2\sqrt{c_p}\rho_n (1+ \sqrt{s}) + \max\left\{ \frac{1-\tilde \varepsilon_0}{\varepsilon_0}, \frac{1}{\varepsilon_2}\left(\frac{\tilde \varepsilon_3}{r_e^2}+ \tilde\varepsilon_2\right) \right\}\frac{\sqrt{c_p}}{2} \rho_n (1+\sqrt{s})^2\,.\]
\underline{With an $s$-dependent choice of regularization:} Choosing $\kappa=\frac{\rho_n}{\sqrt{c_p s}}$, using \eqref{eq:prop_control_effective_near} and the control of $\E{D_{\eta}(\hat{\mu}_{n,\omega}, \mu_\omega^0)}$ stemming from \eqref{eq:ineq_expected_value_bregman} we get \begin{align}
 \E{|\omega_j^0 -\hat{\mu}_{n,\omega}(\X_j^{\rm near}(r_e))|} & \leq 4\sqrt{c_p} \rho_n + \max\left\{ \frac{1-\tilde \varepsilon_0}{\varepsilon_0}, \frac{1}{\varepsilon_2}\left(\frac{\tilde \varepsilon_3}{r_e^2}+ \tilde\varepsilon_2\right) \right\}2\rho_n\sqrt{c_p s} \,, \notag \\
 &\lesssim \frac{\sqrt{s}}{\tau^{d/2} \sqrt{n} r_e^2} \label{eq:prop_effective_near_kappa_s_dep}
\end{align}
keeping only the dependence on $s,\tau,r_e,n$.

\subsection{Proof of Corollary \ref{cor:control_renormalized_estimate}}\label{section:proof_cor_effective_near}
Let $0<r_e\leq r$.
We first prove that 
\begin{equation}\label{eq:final_result_bound_reparameterized_estimate_effective_near}
\left|\frac{\hat{\mu}_{n,\omega}}{W}(\X_j^{\rm near}(r_e))-a_j^0\right| \leq (1+H(r)r_e) W(x_j^0)^{-1}\left|\omega_j^0-\hat{\mu}_{n,\omega}(\X_j^{\rm near}(r_e))\right| + a_j^0 H(r)r_e\end{equation} where 
\begin{equation}
 \label{eq:def_H(r)_effective}
 H(r)=\frac{(e^{r^2} + \sqrt{e^{2r^2}-1})^{d/2}-1}{r}\,.
\end{equation}
We use the triangle inequality \[\left|a_j^0-\frac{\hat{\mu}_{n,\omega}}{W}(\X_j^{\rm near}(r_e))\right| \leq \underbrace{\left|a_j^0-\frac{\hat{\mu}_{n,\omega}(\X_j^{\rm near}(r_e))}{W(x_j^0)}\right|}_{\eqcolon A}+ \underbrace{\left|\frac{\hat{\mu}_{n,\omega}(\X_j^{\rm near}(r_e))}{W(x_j^0)}-\frac{\hat{\mu}_{n,\omega}}{W}(\X_j^{\rm near}(r_e))\right|}_{\eqcolon B} \,.\]
\underline{Control of $B$:} We recall the definitions of effective near regions \eqref{eq:def_effective_near_regions} and of $W$ (see \eqref{eq:def_W}). We have
\[\frac{\hat{\mu}_{n,\omega}}{W}(\X_j^{\rm near}(r_e))=\int_{\X_j^{\rm near}(r_e)} \frac{W(x_j^0)}{W(x)} \, \d \frac{\hat{\mu}_{n,\omega}}{W(x_j^0)}(x)\] so \[\left|\frac{\hat{\mu}_{n,\omega}}{W}(\X_j^{\rm near}(r_e))-\frac{\hat{\mu}_{n,\omega}(\X_j^{\rm near}(r_e))}{W(x_j^0)}\right| \leq \max \left\lbrace \sup_{x \in \X_j^{\rm near}(r_e)} \frac{W(x_j^0)}{W(x)} -1 \,, \,1- \inf_{x \in \X_j^{\rm near}(r_e)} \frac{W(x_j^0)}{W(x)} \right \rbrace \times \frac{\hat{\mu}_{n,\omega}(\X_j^{\rm near}(r_e))}{W(x_j^0)} \,.\]
For $x\in\X_j^{\rm near}(r_e)$, as $\prod_{k=1}^d \frac{(u_{j,k}^0)^2+u_k^2+\tau^2}{\sqrt{2( u_{j,k}^0)^2+\tau^2}\sqrt{2u_k^2+\tau^2}}\leq e^{r_e^2}$ and as each term of this product is greater than $1$, we have $\frac{(u_{j,k}^0)^2+u_k^2+\tau^2}{\sqrt{2( u_{j,k}^0)^2+\tau^2}\sqrt{2u_k^2+\tau^2}}\leq e^{r_e^2}$ for all $k=1,\ldots,d$. Using \eqref{eq:control_close_variances}, this implies
\begin{equation*}\sqrt{u_k^2+\frac{\tau^2}{2}} \in \left[\sqrt{(u_{j,k}^0)^2+\frac{\tau^2}{2}} (e^{r_e^2} - \sqrt{e^{2r_e^2}-1}), \sqrt{( u_{j,k}^0)^2+\frac{\tau^2}{2}} (e^{r_e^2} + \sqrt{e^{2r_e^2}-1})\right] \; \forall \,k=1,\ldots,d\,, \end{equation*}
from which we deduce that
\begin{align*}
 \max \left\lbrace \sup_{x \in \X_j^{\rm near}(r_e)} \frac{W(x_j^0)}{W(x)} -1 \,, \,1- \inf_{x \in \X_j^{\rm near}(r_e)} \frac{W(x_j^0)}{W(x)} \right \rbrace & \leq \max\left\lbrace(e^{r_e^2} + \sqrt{e^{2r_e^2}-1})^{d/2}-1 \,, \, 1- (e^{r_e^2} -\sqrt{e^{2r_e^2}-1})^{d/2} \right\rbrace \,,\\ &= (e^{r_e^2} + \sqrt{e^{2r_e^2}-1})^{d/2}-1 
\end{align*}
where we used that $(e^{r_e^2} + \sqrt{e^{2r_e^2}-1})^{d/2}\geq 1$ to establish that \[1-(e^{r_e^2} -\sqrt{e^{2r_e^2}-1})^{d/2}=1-(e^{r_e^2} +\sqrt{e^{2r_e^2}-1})^{-d/2} \leq (e^{r_e^2} + \sqrt{e^{2r_e^2}-1})^{d/2}-1 \,.\]
Hence
\begin{equation}\label{eq:intermediate_result_bound_estimate_effective_near}
 \left|\frac{\hat{\mu}_{n,\omega}}{W}(\X_j^{\rm near}(r_e))-\frac{\hat{\mu}_{n,\omega}(\X_j^{\rm near}(r_e))}{W(x_j^0)}\right| \leq \left( (e^{r_e^2} + \sqrt{e^{2r_e^2}-1})^{d/2}-1\right) \frac{\hat{\mu}_{n,\omega}(\X_j^{\rm near}(r_e))}{W(x_j^0)} \,.
\end{equation}
\underline{Control of A and proof of \eqref{eq:final_result_bound_reparameterized_estimate_effective_near}:}
As \[\left|a_j^0-\frac{\hat{\mu}_{n,\omega}(\X_j^{\rm near}(r_e))}{W(x_j^0)}\right| =W(x_j^0)^{-1}\left|\omega_j^0-\hat{\mu}_{n,\omega}(\X_j^{\rm near}(r_e))\right|\,,\] from \eqref{eq:intermediate_result_bound_estimate_effective_near} we get
\[\left|\frac{\hat{\mu}_{n,\omega}}{W}(\X_j^{\rm near}(r_e))-a_j^0\right| \leq (e^{r_e^2} + \sqrt{e^{2r_e^2}-1})^{d/2} W(x_j^0)^{-1}|\omega_j^0 - \hat{\mu}_{n,\omega}(\X_j^{\rm near}(r_e))| + a_j^0 \left( (e^{r_e^2} + \sqrt{e^{2r_e^2}-1})^{d/2}-1\right)\,.\]
We conclude the proof of \eqref{eq:final_result_bound_reparameterized_estimate_effective_near} by noticing that $h:r_e \in \R^+ \mapsto (e^{r_e^2} + \sqrt{e^{2r_e^2}-1})^{d/2}$ is convex (for all $d \in \mathbb{N}^*$), hence for $r_e\leq r$ we have \[(e^{r_e^2} + \sqrt{e^{2r_e^2}-1})^{d/2} \leq h(0)+ \frac{h(r)-h(0)}{r}r_e=1+ \frac{h(r)-1}{r}r_e\,.\]
\underline{Conclusion:}
 Taking $r_e=n^{-\alpha}$ with $\alpha>0$, \eqref{eq:final_result_bound_reparameterized_estimate_effective_near} gives \[\E{\left|a_j^0-\frac{\hat{\mu}_{n,\omega}}{W}(\X_j^{\rm near}(n^{-\alpha}))\right|} \leq W(x_j^0)^{-1} \E{\left|\omega_j^0-\hat{\mu}_{n,\omega}(\X_j^{\rm near}(n^{-\alpha}))\right|} \left(1+ H(r)n^{-\alpha}\right) + a_j^0 H(r)n^{-\alpha}\] where $H(r)$ is defined by \eqref{eq:def_H(r)_effective}.
 Proposition \ref{prop:effective_near_regions} gives $\E{\left|\omega_j^0-\hat{\mu}_{n,\omega}(\X_j^{\rm near}(n^{-\alpha}))\right|} \lesssim \frac{s}{\tau^{d/2} \sqrt{n} r_e^2}=\frac{s}{\tau^{d/2}n^{1/2-2\alpha}}$. We choose $r_e=n^{-1/6}$ to balance the terms.
 \\\underline{With $\kappa=\frac{\rho_n}{\sqrt{c_p s}}$:} Choosing $\kappa=\frac{\rho_n}{\sqrt{c_p s}}$, using \eqref{eq:prop_effective_near_kappa_s_dep} we get for $n^{-1/6} \leq r$
\begin{equation}\label{eq:cor_control_renormalized_estimate_kappa_s_dep}
 \E{\left|a_j^0-\frac{\hat{\mu}_{n,\omega}}{W}(\X_j^{\rm near}(n^{-1/6}))\right|} \lesssim \left(W(x_j^0)^{-1} \sqrt{s}\tau^{-d/2} +a_j^0\right)n^{-1/6}\,.
 \end{equation}

\section{Proofs related to guarantees on the prediction}
We will use the following lemma to go from controls of $\norm{L\circ \Phi(\hat{\mu}_n - \mu^0)}_\L^2$ to controls on $\norm{\Phi(\hat{\mu}_n - \mu^0)}_{L^2(\R^d)}^2$.
\begin{lemma}[Control of the high frequencies] \label{lemma:control_high_freq}
 Assume that $\X \subset \R^d \times [u_{\min},+\infty)^d$. Let $\tau>0$. We work with $\Lambda(\xi)= e^{-\frac{1}{2}\tau^2 \norm{\xi}_2^2}$. Let $\mu_1, \mu_2 \in~\M(\X)$. Then 
 \[\norm{\Phi(\mu_1 - \mu_2)}_{L^2}^2 \leq e \norm{L\circ \Phi (\mu_1 - \mu_2)}_{\L}^2+ \frac{2}{(2\pi)^d}(\norm{\mu_1}_{\rm TV}^2+\norm{\mu_2}_{\rm TV}^2) \frac{\tau^d d^{d/2}}{2^{d/2} u_{\min}^{2d}}e^{-2\frac{u_{\min}^2}{\tau^2}}\,.\] 
\end{lemma}
\begin{proof}
 Let $\Tau>0$. First write \[\norm{\Phi(\mu_1 - \mu_2)}_{L^2}^2=\frac{1}{(2\pi)^d}\int_{\left[-\frac{1}{\Tau},\frac{1}{\Tau} \right]^d} |\F{\Phi(\mu_1 - \mu_2)}|^2+\frac{1}{(2\pi)^d}\int_{\R^d \setminus \left[-\frac{1}{\Tau},\frac{1}{\Tau} \right]^d} |\F{\Phi(\mu_1 - \mu_2)}|^2\,.\] Then, remark that 
 \begin{align*}
 \frac{1}{(2\pi)^d}\int_{\left[-\frac{1}{\Tau},\frac{1}{\Tau} \right]^d} |\F{\Phi(\mu_1 - \mu_2)}|^2 &= \frac{1}{(2\pi)^d}\int_{\left[-\frac{1}{\Tau},\frac{1}{\Tau} \right]^d} \frac{\Lambda}{\Lambda}|\F{\Phi(\mu_1 - \mu_2)}|^2 \,, \\
 &\leq e^{\frac{d\tau^2}{2 \Tau^2}} \frac{1}{(2\pi)^d}\int_{\left[-\frac{1}{\Tau},\frac{1}{\Tau} \right]^d} \Lambda |\F{\Phi(\mu_1 - \mu_2)}|^2 \,, \\
 &\leq e^{\frac{d\tau^2}{2 \Tau^2}} \norm{L \circ\Phi(\mu_1 - \mu_2)}_\L^2 \,.
 \end{align*}
Concerning the high frequencies of $\Phi(\mu_1 - \mu_2)$, recall that $u_1,\ldots,u_d \geq u_{\min}$ for all $((t_1,\ldots,t_d),(u_1,\ldots,u_d)) \in \X$.
 Hence
 \begin{align*}
 \frac{1}{(2\pi)^d}\int_{\R^d \setminus \left[-\frac{1}{\Tau},\frac{1}{\Tau} \right]^d } |\F{\Phi(\mu_1 - \mu_2)}|^2 &\leq \frac{2}{(2\pi)^d}(\norm{\mu_1}_{\rm TV}^2+\norm{\mu_2}_{\rm TV}^2) \int_{\R^d \setminus \left[-\frac{1}{\Tau},\frac{1}{\Tau} \right]^d } e^{-u_{\min}^2 \norm{\xi}_2^2} \, \d\xi \,, \\
 &\leq \frac{2}{(2\pi)^d}(\norm{\mu_1}_{\rm TV}^2+\norm{\mu_2}_{\rm TV}^2) \left(\frac{\Tau}{u_{\min}^2}e^{-\frac{u_{\min}^2}{\Tau^2}}\right)^d
 \end{align*}
 using 
 \begin{equation}
 \label{eq:ineq_high_freq_gauss_dim_1}
 \int_{\R \setminus \left[-\frac{1}{\Tau},\frac{1}{\Tau} \right] } e^{-u_{\min}^2 z^2} \, \d z =2 \frac{1}{\Tau}\int_{\left[1 ,+\infty\right)} e^{-u_{\min}^2 \frac{z^2}{\Tau^2}} \, \d z \leq 2 \frac{1}{\Tau}\int_{\left[1 ,+\infty\right)} z e^{-u_{\min}^2 \frac{z^2}{\Tau^2}} \, \d z = \frac{\Tau}{u_{\min}^2}e^{-\frac{u_{\min}^2}{\Tau^2}}\,.
 \end{equation}

 So \[\norm{\Phi(\mu_1 - \mu_2)}_{L^2(\R^d)}^2 \leq e^{\frac{d\tau^2}{2 \Tau^2}} \norm{L\circ \Phi (\mu_1 - \mu_2)}_{\L}^2+ \frac{2}{(2\pi)^d}(\norm{\mu_1}_{\rm TV}^2+\norm{\mu_2}_{\rm TV}^2) \left(\frac{\Tau}{u_{\min}^2}e^{-\frac{u_{\min}^2}{\Tau^2}} \right)^d\,.\]
Taking $\Tau=\frac{\tau \sqrt{d}}{\sqrt{2}}$,
we get \[\norm{\Phi(\mu_1 - \mu_2)}_{L^2}^2 \leq e \norm{L\circ \Phi (\mu_1 - \mu_2)}_{\L}^2+ \frac{2}{(2\pi)^d}(\norm{\mu_1}_{\rm TV}^2+\norm{\mu_2}_{\rm TV}^2) \frac{\tau^d d^{d/2}}{2^{d/2} u_{\min}^{2d}}e^{-2\frac{u_{\min}^2}{\tau^2}}\,.\]
\end{proof}

\subsection{Proof of Proposition \ref{prop:prediction_under_small_reg}}\label{section:proof_prop_prediction_under_small_reg}
We do not make any assumption on the existence of dual certificates in this proof.
From $J_W(\hat{\mu}_{n,\omega}) \leq J_W(\mu_\omega^0)$, we have
 \begin{align*}
 \norm{L\circ \Phi(\hat{\mu}_n-\mu^0)}_\L^2 &\leq 2\norm{L\circ \hat{f}_n- L\circ \Phi\hat{\mu}_n}_\L^2+ 2 \norm{L\circ \hat{f}_n- L\circ\Phi \mu^0}_\L^2 \,, \\
 &\leq 4\norm{\Gamma_n}_\L^2 + 4\kappa \norm{\mu_\omega^0}_{\rm TV} \,.
 \end{align*}
Combining this inequality with Lemma \ref{lemma:control_high_freq}, we get \[\norm{\Phi(\hat{\mu}_n - \mu^0)}_{L^2}^2 \leq 
 e \left(4\norm{\Gamma_n}_\L^2 + 4\kappa \norm{\mu_\omega^0}_{\rm TV} \right)+ \frac{2}{(2\pi)^d}(\norm{\mu^0}_{\rm TV}^2+\norm{\hat{\mu}_n}_{\rm TV}^2) \frac{\tau^d d^{d/2}}{2^{d/2} u_{\min}^{2d}}e^{-2\frac{u_{\min}^2}{\tau^2}} \,.\] 
Remark that for $\mu \in \M(\X)$ we have $\norm{\frac{\mu}{W}}_{\rm TV} \leq \norm{\mu}_{\rm TV} \sup_\X \frac{1}{W} \leq (2\pi)^{d/4}(2u_{\max}^2+\tau^2)^{d/4} \norm{\mu}_{\rm TV}$, and in the same way $\norm{\mu}_{\rm TV} \leq (2\pi)^{-d/4}(2u_{\min}^2+\tau^2)^{-d/4} \norm{\frac{\mu}{W}}_{\rm TV}$. As $J_W(\hat{\mu}_{n,\omega}) \leq J_W(\mu_\omega^0)$ implies that $\norm{\hat{\mu}_{n,\omega}}_{\rm TV}\leq\frac{1}{2 \kappa} \norm{\Gamma_n}_\L^2+\norm{\mu_\omega^0}_{\rm TV}$, we deduce that 
\begin{align} 
 \frac{2}{(2\pi)^d}(\norm{\hat{\mu}_n}_{\rm TV}^2+\norm{\mu^0}_{\rm TV}^2) & \leq \frac{2}{(2\pi)^d} \left( (2\pi)^{d/2}(2u_{\max}^2+\tau^2)^{d/2}\left(\frac{1}{2 \kappa^2}\norm{\Gamma_n}_\L^4 + 2\norm{\mu_\omega^0}_{\rm TV}^2\right)+ \norm{\mu^0}_{\rm TV}^2\right) \,, \notag\\
 &\leq \frac{(2\pi)^{-d/2}(2u_{\max}^2+\tau^2)^{d/2} }{\kappa^2} \norm{\Gamma_n}_\L^4+ \frac{2}{(2 \pi)^d}\left( 2\left(\frac{2u_{\max}^2+ \tau^2}{2u_{\min}^2+ \tau^2}\right)^{d/2} +1\right)\norm{\mu^0}_{\rm TV}^2 \,, \notag\\
 & \leq \frac{(2\pi)^{-d/2}(2u_{\max}^2+\tau^2)^{d/2} }{\kappa^2} \norm{\Gamma_n}_\L^4+ \frac{2}{(2 \pi)^d}\left( 2\left(\frac{u_{\max}}{u_{\min}}\right)^{d} +1\right)\norm{\mu^0}_{\rm TV}^2 \label{eq:bound_sum_norm_tv_mu_0_hat_mu}
\end{align}
using that $\tau \in \R^+\mapsto \frac{2u_{\max}^2+ \tau^2}{2u_{\min}^2+ \tau^2}$ is decreasing. Hence
\begin{multline*}
 \norm{\Phi(\hat{\mu}_n - \mu^0)}_{L^2}^2 \leq e \left(4\norm{\Gamma_n}_\L^2 + 4\kappa (2\pi)^{-d/4}(2u_{\min}^2+\tau^2)^{-d/4}\norm{\mu^0}_{\rm TV} \right) \\
 + \left(\frac{(2\pi)^{-d/2}(2u_{\max}^2+\tau^2)^{d/2} }{\kappa^2} \norm{\Gamma_n}_\L^4+ \frac{2}{(2 \pi)^d}\left( 2\left(\frac{u_{\max}}{u_{\min}}\right)^{d} +1\right)\norm{\mu^0}_{\rm TV}^2 \right) \frac{\tau^d d^{d/2}}{2^{d/2} u_{\min}^{2d}}e^{-2\frac{u_{\min}^2}{\tau^2}} \,.
\end{multline*} 
With Lemma \ref{lemma:control_noise}, choosing $\tau= \frac{\sqrt{2}u_{\min}}{\sqrt{\ln n}}$ and $\kappa=\rho_n^2$ it comes 
\begin{align*}
\E{\norm{\Phi(\hat{\mu}_n - \mu^0)}_{L^2}^2} &\leq 4e\rho_n^2 \left(1+ (2\pi)^{-d/4}(\sqrt{2}u_{\min})^{-d/2}\left(1+ \frac{1}{\ln n} \right)^{-d/4}\norm{\mu^0}_{\rm TV} \right) \\
&\;\; + \left(\tilde C_\Gamma \pi^{-d/2}\left( u_{\max}^2 + \frac{u_{\min}^2}{\ln n}\right)^{d/2} + 2(2 \pi)^{-d} \norm{\mu^0}_{\rm TV}^2 \left(2\left(\frac{u_{\max}}{u_{\min}} \right)^d+1\right)\right) \frac{d^{d/2}}{(\ln n)^{d/2} u_{\min}^{d} n} \,,\\
& \lesssim \frac{(\ln n)^{d/2}}{n} \,.
\end{align*}

\subsection{Prediction with Kernel Density Estimation}\label{section:proof_prediction_kde}
\begin{lemma}
 With $\X \subset \R^d \times [u_{\min},+\infty)^d$, setting $\tau=\frac{1}{\sqrt{\ln n}n^{\frac{1}{4+d}}}$, omitting the dependence on $d$ we have 
 \[\E{\norm{L \circ \hat{f}_n- \Phi \mu^0}_{L^2}^2} \lesssim \frac{(\ln n)^{d/2}}{n^{\frac{4}{d+4}}u_{\min}^{d+4}}\,.\]
\end{lemma}
\begin{proof}
\underline{Control of $\norm{L \circ \Phi \mu^0 - \Phi \mu^0}_{L^2}^2$:}
As $|\F{\Phi \mu^0}(\xi)|^2 \leq \norm{\mu^0}_{\rm TV}^2 e^{-u_{\min}^2 \norm{\xi}_2^2}$ for $\xi \in \R^d$ and using \eqref{eq:ineq_high_freq_gauss_dim_1}, it comes that for $\Tau>0$,
\begin{align*}
 \norm{L \circ \Phi \mu^0 - \Phi \mu^0}_{L^2}^2&= \frac{1}{(2\pi)^d} \int_{\R^d} |\Lambda(\xi)-1|^2 |\F{\Phi \mu^0}(\xi)|^2 \,\d\xi \,, \\
 &\leq \frac{1}{(2\pi)^d} \int_{\R^d \setminus \left[ -\frac{1}{\Tau},\frac{1}{\Tau} \right]^d} |\F{\Phi \mu^0}(\xi)|^2 \,\d\xi + \frac{1}{(2\pi)^d} \frac{d^2 \tau^4}{4\Tau^4} \int_{ \left[ -\frac{1}{\Tau},\frac{1}{\Tau} \right]^d} |\F{\Phi \mu^0}(\xi)|^2 \,\d\xi\,, \\
 &\leq \frac{\norm{\mu^0}_{\rm TV}^2}{(2\pi)^d}\left(\frac{\Tau^d}{u_{\min}^{2d}}e^{-\frac{du_{\min}^2}{\Tau^2}} +\frac{2^d}{\Tau^d}\frac{\tau^4d^2}{4\Tau^4} \right)\,,
\end{align*}
where we used that for $\xi \in \R^d$, $|\Lambda(\xi)-1|^2=|e^{-\frac{\tau^2}{2} \norm{\xi}_2^2}-1|\leq 1$ and \[|\Lambda(\xi)-1|^2\leq |e^{-\frac{d \tau^2}{2\Tau^2}}-1|\leq \frac{d^2 \tau^4}{4\Tau^4} \quad \forall \, \xi \in \left[ -\frac{1}{\Tau},\frac{1}{\Tau} \right]^d \,.\]
\underline{Control of $\E{\norm{L \circ \hat{f}_n- \Phi \mu^0}_{L^2}^2}$:}
Using Lemma \ref{lemma:control_noise}, we have 
\[\E{\norm{L \circ \hat{f}_n- \Phi \mu^0}_{L^2}^2} \leq 2\rho_n^2 + 2\norm{L \circ \Phi \mu^0 - \Phi \mu^0}_{L^2}^2 \leq \frac{8}{(2\pi)^{d/2}\tau^d n}+\frac{2\norm{\mu^0}_{\rm TV}^2}{(2\pi)^d}\left(\frac{\Tau^d}{u_{\min}^{2d}}e^{-\frac{du_{\min}^2}{\Tau^2}} +\frac{2^d}{\Tau^d}\frac{\tau^4d^2}{4\Tau^4} \right)\,. \]
To balance these terms, we choose $\Tau=\frac{u_{\min} \sqrt{d}}{\sqrt{\ln n}}$ and $\tau=\frac{1}{\sqrt{\ln n}n^{\frac{1}{4+d}}}$.
It comes \[\E{\norm{L \circ \hat{f}_n- \Phi \mu^0}_{L^2}^2} \leq \frac{8(\ln n)^{d/2}}{(2 \pi)^{d/2} n^{\frac{4}{d+4}}} + \frac{2\norm{\mu^0}_{\rm TV}^2}{(2\pi)^d}\left(\frac{d^{d/2}}{u_{\min}^{d} (\ln n)^{d/2} n} +\frac{2^d (\ln n)^{d/2}}{4d^{d/2}u_{\min}^{d+4} n^{\frac{4}{d+4}}} \right)\,.\]
\end{proof}

\subsection{Proof of Theorem \ref{th:prediction_error}}\label{section:proof_th_prediction_error}
 Equation \eqref{eq:ineq_mu_minimizer} gives
 \[
 \norm{L \circ \Phi(\hat{\mu}_n -\mu^0)}_\L \leq 2\norm{\Gamma_n}_\L+2\kappa \norm{p}_\L 
 \]
 from which we deduce, using Lemma \ref{lemma:control_noise} and $\norm{p}_\L\leq \sqrt{c_p s}$, that 
 \[
 \E{\norm{L \circ \Phi(\hat{\mu}_n -\mu^0)}_\L^2} \leq 4(\rho_n+ \kappa \sqrt{c_p s})^2
 \,.\]
 To go from a control of $\E{\norm{L \circ \Phi(\hat{\mu}_n -\mu^0)}_\L^2}$ to a bound on $\E{\norm{\Phi\hat{\mu}_n - \Phi\mu^0}_{L^2(\R^d)}^2}$, we use \eqref{eq:bound_sum_norm_tv_mu_0_hat_mu} with Lemmas \ref{lemma:control_high_freq} and \ref{lemma:control_noise}. We get 
 \begin{align*}
 \E{\norm{\Phi\hat{\mu}_n - \Phi\mu^0}_{L^2(\R^d)}^2} &\leq e \E{\norm{L \circ \Phi(\hat{\mu}_n -\mu^0)}_\L^2} +\frac{2}{(2\pi)^d}\left(\E{\norm{\hat{\mu}_n}_{\rm TV}^2}+\norm{\mu^0}_{\rm TV}^2\right) \frac{\tau^d d^{d/2}}{2^{d/2} u_{\min}^{2d}}e^{-2\frac{u_{\min}^2}{\tau^2}}\,,\\
 &\leq 4 e(\rho_n+ \kappa \sqrt{c_p s})^2 \\
 &\quad +\bigg(
 \frac{(2\pi)^{d/2}(2u_{\max}^2+\tau^2)^{d/2} \tilde C_\Gamma \rho_n^4}{\kappa^2} + 2\bigg( 2\bigg(\frac{u_{\max}}{u_{\min}}\bigg)^{d} +1 \bigg)\norm{\mu^0}_{\rm TV}^2 \bigg) \frac{(2\pi)^{-d}\tau^d d^{d/2}}{2^{d/2} u_{\min}^{2d}}e^{-2\frac{u_{\min}^2}{\tau^2}} \,, \\
 & \leq 4 e(\rho_n+ \kappa \sqrt{c_p s})^2 \\
 &\quad + \bigg(4\bigg(\frac{u_{\max}^2}{u_{\min}^2}+\frac{1}{\ln n}\bigg)^{d/2} \tilde C_\Gamma \frac{(\ln n)^{d/2} \rho_n^2}{n\kappa^2} + 2\bigg( 2\bigg(\frac{u_{\max}}{u_{\min}}\bigg)^{d} +1\bigg)\norm{\mu^0}_{\rm TV}^2\bigg) \frac{(2\pi)^{-d} d^{d/2}}{ u_{\min}^{d} (\ln n)^{d/2} n} \,.
 \end{align*}
 With the choice $\kappa=\frac{\rho_n}{\sqrt{c_p}}$, it comes 
 \begin{align*}
  \E{\norm{\Phi\hat{\mu}_n - \Phi\mu^0}_{L^2(\R^d)}^2} &\leq 4e\rho_n^2(1+\sqrt{s})^2 \\
  &\quad + \bigg( 4\bigg(\frac{u_{\max}^2}{u_{\min}^2}+\frac{1}{\ln n}\bigg)^{d/2} \tilde C_\Gamma c_p\frac{(\ln n)^{d/2}}{n} + 2\norm{\mu^0}_{\rm TV}^2\bigg( 2\bigg(\frac{u_{\max}}{u_{\min}}\bigg)^{d} +1\bigg) 
 \bigg)\frac{(2\pi)^{-d}d^{d/2}}{ u_{\min}^{d} (\ln n)^{d/2} n}\,.   
 \end{align*}
 This concludes the proof of Theorem \ref{th:prediction_error}.
 With the choice $\kappa=\frac{\rho_n}{\sqrt{c_p s}}$, we get \begin{align}
 \E{\norm{\Phi\hat{\mu}_n - \Phi\mu^0}_{L^2(\R^d)}^2} &\leq 16e\rho_n^2+ \bigg( 4\bigg(\frac{u_{\max}^2}{u_{\min}^2}+\frac{1}{\ln n}\bigg)^{d/2} \tilde C_\Gamma c_p\frac{s (\ln n)^{d/2}}{n} + 2\norm{\mu^0}_{\rm TV}^2\left( 2\bigg(\frac{u_{\max}}{u_{\min}}\right)^{d} +1\bigg) 
 \bigg)\frac{(2\pi)^{-d}d^{d/2}}{ u_{\min}^{d} (\ln n)^{d/2} n} \,, \notag \\
 & \lesssim \bigg( \frac{s}{n (\ln n)^{d/2}} +1 \bigg) \frac{(\ln n)^{d/2}}{n}\label{eq:remark_choice_kappa_pred}
 \end{align}
 keeping only the dependence on $n$ and $s$.

\section{Properties of the Fisher-Rao metric}\label{section:properties_fisher_rao}
We present properties associated with the Fisher-Rao metric $\mathfrak{g}$, defined at point $x \in \R^d \times [u_{\min} ,+\infty)^d$ by $\mathfrak{g}_x=\nabla_1\nabla_2 K_{\rm norm}(x,x)$ (see also \eqref{eq:def_fisher_rao_metric}). Note that this metric depends on the smoothing parameter $\tau>0$ through $K_{\rm norm}$. We recall the definition of the Riemannian norm: for $v\in \R^{2d}$ and $x\in \R^d \times [u_{\min} ,+\infty)^d$, we define $\norm{v}_x=\sqrt{v^T \mathfrak{g}_x v}$.
\subsection{Christoffel symbols}\label{section:christoffel_symbols}
The non-zero Christoffel symbols associated with $\mathfrak{g}$ are
\begin{align*}
 \Gamma^{t_k}{}_{u_k t_k}=\Gamma^{t_k}{}_{t_k u_k}&=\frac{-2u_k}{2u_k^2+\tau^2}\,,\\
 \Gamma^{u_k}{}_{t_k t_k}&=\frac{1}{u_k} \,,\\
 \Gamma^{u_k}{}_{u_k u_k}&=\frac{\tau^2-2u_k^2}{u_k(2u_k^2+\tau^2)} 
\end{align*}
with $k= 1, \ldots,d$ (see \citep[Section I.1]{notebook}).
We define \begin{equation}\label{eq:christoffel_symbols}\Gamma^{t_k}=\begin{pmatrix}
 (\Gamma^{t_k}{}_{t_l t_m})_{1\leq l,m\leq d} & (\Gamma^{t_k}{}_{t_l u_m})_{1\leq l,m\leq d}
 \\ (\Gamma^{t_k}{}_{u_l t_m})_{1\leq l,m\leq d} & (\Gamma^{t_k}{}_{u_l u_m})_{1\leq l,m\leq d}
\end{pmatrix} \quad \text{and} \quad \Gamma^{u_k}=\begin{pmatrix}
 (\Gamma^{u_k}{}_{t_l t_m})_{1\leq l,m\leq d} & (\Gamma^{u_k}{}_{t_l u_m})_{1\leq l,m\leq d}
 \\ (\Gamma^{u_k}{}_{u_l t_m})_{1\leq l,m\leq d} & (\Gamma^{u_k}{}_{u_l u_m})_{1\leq l,m\leq d}
\end{pmatrix} \,.\end{equation}
\subsection{Geodesics and geodesic distance}\label{section:geodesics_geodesic_distance}
The next lemmas provide the parameterization of the geodesics associated with the Fisher-Rao metric. We denote $\tilde \gamma$ a geodesic parameterized by arc length connecting the points $x=\tilde \gamma(0), x'=\tilde \gamma(l) \in \R^d \times [u_{\min},+\infty)^d$. The parameter $l$ is the geodesic distance between $x$ and $x'$, denoted by $\mathfrak{d}_{\mathfrak{g}}(x,x')$.
We also denote $\gamma: y\in [0,1]\mapsto \tilde\gamma(ly)$ (it is the geodesic such that $x=\gamma(0)$, $x'=\gamma(1)$) and $\dot \gamma$ its derivative.

We do not use the formula of the geodesic distance $\mathfrak{d}_\mathfrak{g}$ in this paper, but we give it in the next lemmas for information.

\begin{lemma}[Geodesics of the Poincaré half-plane model]\label{lemma:geodesics_poincare_half_plane}
The Poincaré half-plane is $\{x=(t,u) \in \R\times \R_+^*\}$, on which we consider the metric defined by $\mathfrak{h}_x = \begin{pmatrix}
 \frac{1}{u^2}& 0 \\
 0 & \frac{1}{u^2}
\end{pmatrix}$ for all $x=(t,u) \in\R\times \R_+^*$. The associated norm is defined by $\norm{v}_x=\sqrt{v^T \mathfrak{h}_x v}$ for $v\in\R^2$. 

The Poincaré geodesics are circular arcs whose origin is on the axis $\{u=0\}$ and straight vertical lines (parallel to $\{t=0\}$).

A Poincaré geodesic parameterized by arc-length, denoted by $\tilde h=(\tilde h_t,\tilde h_u)$, is of the form \[\tilde h: y\in [0,l] \mapsto \left(\frac{\tanh(C_2+y)}{C_1}+C_3, \frac{1}{\cosh(C_2+y) |C_1|}\right)\] (semicircle) or 
\[\tilde h: y\in [0,l] \mapsto (C_3, |C_1| e^{y}) \quad \text{or} \quad \tilde h: y\in [0,l] \mapsto (C_3, |C_1| e^{-y})\] (straight line), where $C_1 \in \R^*$, $C_2,C_3\in \R$, $l \in \R^+$.

Moreover, writing $\mathfrak{d}_{\mathfrak{h}}$ the Poincaré distance, \[\mathfrak{d}_{\mathfrak{h}}(x,x')=\ln\left( \frac{\sqrt{(t-t')^2+(u+u')^2}+\sqrt{(t-t')^2+(u-u')^2}}{\sqrt{(t-t')^2+(u+u')^2}-\sqrt{(t-t')^2+(u-u')^2}}\right) \quad \forall\, x,x' \in \R \times \R_+^*\,.\]
\end{lemma}
\begin{proof}
The fact that the Poincaré geodesics are semicircles whose origin is on the axis $\{u=0\}$ and straight lines parallel to $\{t=0\}$ is well-known (see for instance \citep[Theorem 4.2.1]{stahl_poincare}).
The formula for $\mathfrak{d}_\mathfrak{h}$ can be found in \citep[Theorem 7.2.1]{beardon_geometry}.

We can check that the parameterizations given for the geodesics verify the geodesic equations
\begin{equation*}
 \begin{cases}
 \frac{(\dot{\tilde h}_t)^2+ (\dot{\tilde h}_u)^2}{\tilde h_u ^2}=1 \\
 \ddot{\tilde h}_t - 2\frac{\dot{\tilde h}_t \dot{\tilde h}_u}{\tilde h_u}=0 \\
 \ddot{\tilde h}_u - \frac{(\dot{\tilde h}_u)^2}{\tilde h_u}+ \frac{(\dot{\tilde h}_t)^2}{\tilde h_u} =0
 \end{cases}\,,
\end{equation*}
as done in \citep[Section I.2]{notebook}.
We found all the geodesics, because all the portions of Poincaré semicircles and straight lines can be obtained with appropriate choices of $C_1,C_2,C_3, l$.
\end{proof}

\begin{lemma}[Geodesics for $d=1$]\label{lemma:form_geodesics_dim_1}
Let $x,x' \in \R \times [u_{\min},+\infty)$.
\begin{itemize}
 \item If $t=t'$, the geodesic is of the form 
\begin{equation}\label{eq:form_geod_straight_lines}
 \tilde \gamma(y)=\left(c_3,\sqrt{\frac{c_1^2}{2}e^{\sqrt{8}y}-\frac{\tau^2}{2}}\right) \quad \forall \, y \in [0,l] \quad \text{or} \quad \tilde \gamma(y)=\left(c_3,\sqrt{\frac{c_1^2}{2}e^{-\sqrt{8}y}-\frac{\tau^2}{2}}\right) \quad \forall \, y \in [0,l]
\end{equation}
where $c_1,c_3\in \R$. It is a portion of a straight line parallel to $\{t=0\}$.
\item If $t\neq t'$, the geodesic is of the form 
\begin{equation}\label{eq:form_geod_semicircles}
 \tilde{\gamma}(y)=\left( c_{3} + \frac{\sqrt{2} \tanh{\left(\frac{c_{2}}{2} + \sqrt{2} y \right)}}{2 c_{1}}, \ \sqrt{- \frac{\tau^2}{2} + \frac{1 - \tanh^{2}{\left(\frac{c_{2}}{2} + \sqrt{2} y \right)}}{2 c_{1}^{2}}}\right) \quad \forall \, y \in [0,l] 
\end{equation}
where $c_1\neq 0$, $c_2,c_3\in \R$. It is a portion of a semicircle with center $(c_3,0)$ and radius \begin{equation}\label{eq:formula_radius_semicircle}
 \sqrt{\frac{1}{2c_1^2}-\frac{\tau^2}{2}}=\sqrt{\frac{1}{4}\left( \frac{-(t-t')^2+u^2-u'^2}{t'-t}\right)^2+u^2}\,.
\end{equation}
\end{itemize}
Moreover, the Fisher-Rao distance between $x$ and $x'$ is \[\mathfrak{d}_\mathfrak{g}(x,x')=
\sqrt{2} 
\ln\left( 
 \frac{\sqrt{(t-t')^2 + 
 \left(
 \sqrt{u^2+\frac{\tau^2}{2}} -\sqrt{u'^2+\frac{\tau^2}{2}}
 \right)^2} 
 + \sqrt{(t-t')^2 + 
 \left(
 \sqrt{u^2+\frac{\tau^2}{2}} +\sqrt{u'^2+\frac{\tau^2}{2}}
 \right)^2 }}{\sqrt{2}(2u^2+\tau^2)^{1/4}(2u'^2+\tau^2)^{1/4}}
\right)\,.\]
\end{lemma}
\begin{proof}
\underline{Link with the Poincaré half-plane model:}
Recall that the variational formulation of a (Fisher-Rao) geodesic $\gamma=(\gamma_t, \gamma_u)$ connecting $x,x'$ is \[\inf_{\gamma(0)=x, \gamma(1)=x'} \int_0^1 \left(\dot\gamma_t(y)^2\frac{1}{2\gamma_u(y)^2 +\tau^2}+ \dot\gamma_u(y)^2\frac{2\gamma_u(y)^2}{(2\gamma_u(y)^2 +\tau^2)^2}\right) \, \d y\,.\]
We use the change of variable $h=(h_t,h_u)=\left(\gamma_t, \sqrt{\gamma_u^2+\frac{\tau^2}{2}} \right)$. Noticing that $(\dot h_u)^2=\frac{\dot\gamma_u(y)^2\gamma_u(y)^2}{\gamma_u^2+\frac{\tau^2}{2}}$, it comes that the variational formulation is equivalent to the problem \[\inf\left \lbrace \frac{1}{2}\int_0^1 \left( (\dot h_t(y))^2\frac{1}{h_u(y)^2}+ (\dot h_u(y))^2\frac{1}{h_u(y)^2} \right) \, \d y \: : \: h(0)=\left(t,\sqrt{u^2 + \frac{\tau^2}{2}}\right), h(1)=\left(t',\sqrt{u'^2 + \frac{\tau^2}{2}}\right) \right \rbrace\,,\]
and we recognize the Poincaré metric tensor ($\frac{\d t^2+\d u^2}{u^2}$) in this formulation. So $h$ is the geodesic for the Poincaré half-plane metric connecting $\left(t,\sqrt{u^2 + \frac{\tau^2}{2}}\right)$ and $\left(t',\sqrt{u'^2 + \frac{\tau^2}{2}}\right)$.
\\\underline{Geodesic distance:}
In particular, using the formula for $\mathfrak{d}_{\mathfrak{h}}$ in Lemma \ref{lemma:geodesics_poincare_half_plane}, we have 
\begin{align*}
 \mathfrak{d}_\mathfrak{g}(x,x')&=\frac{1}{\sqrt{2}}\mathfrak{d}_\mathfrak{h}\left(\left(t,\sqrt{u^2 + \frac{\tau^2}{2}}\right),\left(t',\sqrt{u'^2 + \frac{\tau^2}{2}}\right)\right) \,, \\
 &= \frac{1}{\sqrt{2}} \ln\left( \frac{\sqrt{(t-t')^2 + \left(\sqrt{u^2+\frac{\tau^2}{2}} +\sqrt{u'^2+\frac{\tau^2}{2}}\right)^2 }+\sqrt{(t-t')^2 + \left(\sqrt{u^2+\frac{\tau^2}{2}} -\sqrt{u'^2+\frac{\tau^2}{2}}\right)^2}}{\sqrt{(t-t')^2 + \left(\sqrt{u^2+\frac{\tau^2}{2}} +\sqrt{u'^2+\frac{\tau^2}{2}}\right)^2 }-\sqrt{(t-t')^2 + \left(\sqrt{u^2+\frac{\tau^2}{2}} -\sqrt{u'^2+\frac{\tau^2}{2}}\right)^2}}\right) \,,\\
 &=\sqrt{2} \ln\left( \frac{\sqrt{(t-t')^2 + \left(\sqrt{u^2+\frac{\tau^2}{2}} -\sqrt{u'^2+\frac{\tau^2}{2}}\right)^2} + \sqrt{(t-t')^2 + \left(\sqrt{u^2+\frac{\tau^2}{2}} +\sqrt{u'^2+\frac{\tau^2}{2}}\right)^2 }}{\sqrt{2}(2u^2+\tau^2)^{1/4}(2u'^2+\tau^2)^{1/4}}\right) \,.
\end{align*}
\underline{parameterization by arc-length:} We saw that $\left(\gamma_t, \sqrt{\gamma_u^2+\frac{\tau^2}{2}} \right)$ is a geodesic for the Poincaré half-plane model. Lemma \ref{lemma:geodesics_poincare_half_plane} gives its parameterization by arc-length, $\tilde h$. 
For the $\mathfrak{h}$-norm, $\norm{\dot{\tilde h}_t(y), \dot{\tilde h}_u(y)}_{\tilde h(y)}=1$.
We deduce that for the $\mathfrak{g}$-norm, writing $g(y)=\left(\tilde h_t(y), \sqrt{\tilde h_u(y)^2 -\frac{\tau^2}{2}}\right)$, we have $\norm{\dot g_t(y), \dot g_u(y)}_{g(y)}=\frac{1}{\sqrt{2}}$. So defining $\tilde \gamma(y)=g(\sqrt{2} y)$, it comes that $\norm{\dot{\tilde \gamma}_t(y), \dot{\tilde \gamma}_u(y)}_{\gamma(y)}=1$ for the $\mathfrak{g}$-norm: $\tilde \gamma$ is a geodesic parameterized by arc-length. The geodesics we provide are of this form.
\\\underline{Radius of the semicircle connecting $x$ and $x'$:} The parameterization of the semicircle gives \[(\tilde \gamma_t-c_3)^2+\tilde \gamma_u^2=\frac{1}{2c_1^2}-\frac{\tau^2}{2}\] which is the square of the radius of the semicircle. We also have \[\frac{1}{2c_1^2}-\frac{\tau^2}{2}=(t-c_3)^2+u^2=(t'-c_3)^2+u'^2\,,\] from which we deduce that $c_3=\frac{t'^2+u'^2-(t^2+u^2)}{2(t'-t)}$ along with $\displaystyle\frac{1}{2c_1^2}-\frac{\tau^2}{2}=\frac{1}{4}\left( \frac{-(t-t')^2+u^2-u'^2}{t'-t}\right)^2+u^2$.
\end{proof}

We can extend this result to higher dimensions. By abuse of notation, $\mathfrak{d}_{\mathfrak{g}}(x_k,x_k')$ will refer to the Fisher-Rao distance in dimension 1 between $x_k$ and $x_k'$, for all $k\in\{1,\ldots,d\}$. The notation $\mathfrak{g}_{x_k}$ follows the same principle.

\begin{lemma}[Geodesics in dimension $d\geq 1$]\label{lemma:form_geodesics_dim_d}
 Let $x,x' \in \R^d \times [u_{\min},+\infty)^d$. 
 For $k=1,\ldots,d$, we denote $\bar \gamma_k$ the geodesic in dimension 1, parameterized by arc length connecting $x_k=(t_k,u_k)$ and $x_k'=(t_k',u_k')$. 
 The geodesic connecting $x$ and $x'$ is of the form \[\tilde \gamma=(\tilde \gamma_{t_1}, \ldots ,\tilde \gamma_{t_d}, \tilde \gamma_{u_1}, \ldots ,\tilde \gamma_{u_d}) \quad \text{where} \quad (\tilde \gamma_{t_k}(y),\tilde \gamma_{u_k}(y))=\bar \gamma_k(\sqrt{g_k} y) \quad \text{with} \quad g_k\geq 0 \quad \text{and} \quad \sum_{i=1}^d g_i=1 \,.\]
 Moreover, \[\mathfrak{d}_{\mathfrak{g}}( x, x')=\sqrt{\sum_{k=1}^d \mathfrak{d}_{\mathfrak{g}}(x_k,x_k')^2}\,.\]
\end{lemma}
\begin{proof}
 For all $k\in\{1,\ldots,d\}$, we denote $\mathfrak{d}_k=\mathfrak{d}_{\mathfrak{g}}(x_k,x_k')^2$ and $\tilde \gamma_k(y)= \bar \gamma_k \left(\frac{\sqrt{\mathfrak{d}_k}}{ \sqrt{\sum_{j} \mathfrak{d}_j}} y \right)$.
Then, for all $v\in \R^{2d}$ and $x\in \R^d \times [u_{\min},+\infty)^d$, we have 
\[
\norm{v}_x^2 = v^T \mathfrak{g}_x v = \sum_{k=1}^d (v_k, v_{k+d}) \mathfrak{g}_{x_k} (v_k, v_{k+d})^T,
\]
where $(v_k, v_{k+d})$ represents the $k$-th component of $v$ in $\R^d \times [u_{\min},+\infty)^d$. Using this, the geodesic $\tilde \gamma$ connecting $x$ and $x'$ can be expressed as:
\[
\tilde \gamma = (\tilde \gamma_{t_1}, \ldots, \tilde \gamma_{t_d}, \tilde \gamma_{u_1}, \ldots, \tilde \gamma_{u_d}),
\]
where each $(\tilde \gamma_{t_k}, \tilde \gamma_{u_k})$ corresponds to the geodesic $\tilde \gamma_k$ in dimension 1 connecting $x_k = (t_k, u_k)$ and $x_k' = (t_k', u_k')$. This ensures that $\tilde \gamma$ is the geodesic connecting $x$ and $x'$ in $\R^d \times [u_{\min},+\infty)^d$.

\end{proof}

\subsection{Compatibility with the semi-distance}

\subsubsection{Proof of Lemma \ref{lemma:geodesics_within_near_regions}}\label{section:proof_lemma_geodesics_within_near_regions}
Let $r>0$ and $x_0 \in \R^d \times [u_{\min}, +\infty)^d$. Let $\gamma$ be a geodesic between $x_0=\gamma(0)$ and $x=\gamma(1) \in \R^d \times [u_{\min}, +\infty)^d$ for the metric $\mathfrak{g}$. 

\paragraph{Lower bound on the variance}
We write $\gamma=(\gamma_{t_1},\ldots,\gamma_{t_d} ,\gamma_{u_1},\ldots,\gamma_{u_d})$. Recall that $(\gamma_{t_k},\gamma_{u_k})$ is a portion of a straight line parallel to $\{t_k=0\}$ or of a semicircle with center on $\{u_k=0\}$. As $x_0, x_1 \in \R^d \times [u_{\min}, +\infty)^d$, we deduce that $u_{\min} \leq \gamma_{u_k}(y)$ for all $k=1,\ldots,d$, $y \in [0,1]$.

\paragraph{The function $y\in [0,1]\mapsto \mathpzc{d}(x_0,\gamma(y))$ is non-decreasing}
We can reduce the problem to the case $d=1$. Indeed, for each $k \in \{1, \ldots, d\}$, $\gamma_k \coloneq (\gamma_{t_k}, \gamma_{u_k})$ is the geodesic in dimension 1 connecting $x_{0,k} = \gamma_k(0)$ and $x_k = \gamma_k(1)$ (see Lemma~\ref{lemma:form_geodesics_dim_d}). Furthermore, $\mathpzc{d}(x_0, \gamma(y))^2 = \sum_{k=1}^d \mathpzc{d}(x_{0,k}, \gamma_k(y))^2$, where, by abuse of notation, $\mathpzc{d}(x_{0,k}, \gamma_k(y))$ denotes the semi-distance in dimension 1 between $x_{0,k}$ and $\gamma_k(y)$. Hence, if the function $y \in [0,1] \mapsto \mathpzc{d}(x_{0,k}, \gamma_k(y))$ is increasing for all $k$, then $y \in [0,1] \mapsto \mathpzc{d}(x_0, \gamma(y))$ is also increasing.

Until the end of the proof, we therefore concentrate our attention on the case $d=1$.
Let $x_0, x \in \R \times [u_{\min}, +\infty)$, and let $\gamma$ and $\tilde{\gamma}$ be the geodesic connecting $x_0$ and $x$, parameterized by $[0,1]$ and by arc-length, respectively. Proving that $y \in [0,1] \mapsto \mathpzc{d}(x_0, \gamma(y))$ is increasing is equivalent to proving that $h: y \in [0, l] \mapsto \mathpzc{d}(x_0, \tilde{\gamma}(y))^2$ is increasing, where $l~=~ \mathfrak{d}_\mathfrak{g}(x_0, x)$.
We will consider alternatively the cases where the geodesic $\gamma$ is a straight line and a semicircle. It suffices to show that $h'\geq 0$ in both cases.
\\\underline{Proof for straight lines, $d=1$:} See \citep[Section II.2]{notebook}. Using the form of the geodesic given in Lemma \ref{lemma:form_geodesics_dim_1}, we can show that if $\tilde \gamma$ is a straight line, \[h(y)=\ln(\cosh(\sqrt{2}y))\,.\] We obtain this formula by deducing from $\tilde \gamma(0)=(t_0,u_0)$ in \eqref{eq:form_geod_straight_lines} that $c_1^2=2u_0^2+\tau^2$ (the two formulas of \eqref{eq:form_geod_straight_lines} give the same result). This function is non-decreasing on $\R^+$.
\\\underline{Proof for semicircles, $d=1$:} See \citep[Section II.1]{notebook}.
If $\tilde \gamma$ is a semicircle, then
\[h(y)=\ln{\left(\frac{\cosh^{2}{\left(\frac{c_{2}}{2} \right)} + \cosh^{2}{\left(\frac{c_{2}}{2} + \sqrt{2} y \right)}}{2 \cosh{\left(\frac{c_{2}}{2} \right)} \cosh{\left(\frac{c_{2}}{2} + \sqrt{2} y \right)}} \right)} + \frac{\sinh^{2}{\left(\sqrt{2} y \right)}}{\cosh^{2}{\left(\frac{c_{2}}{2} \right)} + \cosh^{2}{\left(\frac{c_{2}}{2} + \sqrt{2} y \right)}}\,.\] We obtain this formula by deducing from $\tilde \gamma(0)=(t_0,u_0)$ in \eqref{eq:form_geod_semicircles} that $c_3=t_0-\frac{\sqrt{2}\tanh\left(\frac{c_2}{2}\right)}{2c_1}$ and $c_1=\pm \frac{1}{\sqrt{\tau^2 + 2u_0^2}\cosh\left( \frac{c_2}{2}\right)}$.
Let $y\geq 0$.
We have $h(y)=A(y)+B(y)$ with
\[
A(y)= \frac{\sinh^{2}{\left(\sqrt{2} y \right)}}{\cosh^{2}{\left(\frac{c_{2}}{2} \right)} + \cosh^{2}{\left(\frac{c_{2}}{2} + \sqrt{2} y \right)}}\,.
\]and 
\[B(y)=\ln{\left(\frac{\cosh^{2}{\left(\frac{c_{2}}{2} \right)} + \cosh^{2}{\left(\frac{c_{2}}{2} + \sqrt{2} y \right)}}{2 \cosh{\left(\frac{c_{2}}{2} \right)} \cosh{\left(\frac{c_{2}}{2} + \sqrt{2} y \right)}} \right)}\,.\]
Then, \[A'(y)=\frac{\sqrt{2} \sinh(\sqrt{2}y) \left( \cosh(c_2+\sqrt{2}y) +\cosh(\sqrt{2}y)(2+\cosh(c_2))\right)}{\left(\cosh^2\left(\frac{c_2}{2}\right)+\cosh^2\left(\frac{c_2}{2}+\sqrt{2}y\right)\right)^2} \geq 0\] and \[B'(y)= \frac{\sqrt{2}\sinh\left(\frac{c_2}{2}+\sqrt{2}y\right)\left(\cosh^2\left(\frac{c_2}{2}+\sqrt{2}y\right) - \cosh^2\left(\frac{c_2}{2}\right)\right) }{\cosh\left(\frac{c_2}{2}+\sqrt{2}y\right) \left( \cosh^2\left(\frac{c_2}{2}\right)+\cosh^2\left(\frac{c_2}{2}+\sqrt{2}y\right)\right)} \,.\] 
If $c_2 \geq 0$, then as $y\geq 0$, $\cosh^2\left(\frac{c_2}{2}+\sqrt{2}y\right) \geq \cosh^2\left(\frac{c_2}{2}\right)$ so $B' \geq 0$, hence $\mathpzc{d}(x_0,\tilde{\gamma}(y))^2$ is increasing on $\mathbb{R}^+$. 
We now deal with the case $c_2<0$.
Remark that $B'(y) > 0$ if and only if $y \notin \left[\frac{-c_2}{2 \sqrt{2}} , \frac{-c_2}{\sqrt{2}}\right]$. 
We want to show that $A'+B' \geq 0$ on the interval $\left[\frac{-c_2}{2 \sqrt{2}} , \frac{-c_2}{\sqrt{2}}\right]$. 
First, using that for $y\in \left[\frac{-c_2}{2 \sqrt{2}} , \frac{-c_2}{\sqrt{2}}\right]$, we have 
\[\cosh^2\left(\frac{c_2}{2}+\sqrt{2}y\right) \leq \cosh^2\left(\frac{c_2}{2}\right)\,, \quad \sinh(\sqrt{2}y) \geq \sinh\left(\frac{-c_2}{2}\right)\,,\] \[\cosh(\sqrt{2}y) \geq \cosh\left(\frac{c_2}{2}\right)\quad \text{and} \quad 2+\cosh(c_2) \geq 2 \cosh^2\left(\frac{c_2}{2}\right)\,,\] we obtain a lower bound for $A'$ on $\left[\frac{-c_2}{2 \sqrt{2}} , \frac{-c_2}{\sqrt{2}}\right]$:
\begin{align*}
 A'(y) &\geq \frac{\sqrt{2} \sinh\left(\frac{-c_2}{2}\right)\left( \cosh(c_2+\sqrt{2}y) +\cosh(\sqrt{2}y)(2+\cosh(c_2))\right)}{2 \cosh^2\left(\frac{c_2}{2}\right)\left(\cosh^2\left(\frac{c_2}{2}\right)+\cosh^2\left(\frac{c_2}{2}+\sqrt{2}y\right)\right)} \,, \\
 & \geq \frac{\sqrt{2} \sinh\left(\frac{-c_2}{2}\right)\cosh\left(\frac{c_2}{2}\right)(2+\cosh(c_2))}{2 \cosh^2\left(\frac{c_2}{2}\right)\left(\cosh^2\left(\frac{c_2}{2}\right)+\cosh^2\left(\frac{c_2}{2}+\sqrt{2}y\right)\right)} \,, \\
 & \geq \frac{\sqrt{2} \sinh\left(\frac{-c_2}{2}\right)\cosh\left(\frac{c_2}{2}\right)}{\cosh^2\left(\frac{c_2}{2}\right)+\cosh^2\left(\frac{c_2}{2}+\sqrt{2}y\right)} \,.
\end{align*}
It follows that 
\[A'(y)+B'(y) \geq \frac{\sqrt{2} \cosh\left(\frac{c_2}{2}+\sqrt{2}y\right)\sinh\left(\frac{-c_2}{2}\right)\cosh\left(\frac{c_2}{2}\right) + \sqrt{2} \sinh\left(\frac{c_2}{2}+\sqrt{2}y\right)\left(\cosh^2\left(\frac{c_2}{2}+\sqrt{2}y\right) - \cosh^2\left(\frac{c_2}{2}\right)\right) }{\left(\cosh^2\left(\frac{c_2}{2}\right)+\cosh^2\left(\frac{c_2}{2}+\sqrt{2}y\right) \right) \cosh\left(\frac{c_2}{2}+\sqrt{2}y\right)} \,.\]
The positivity of $A'+B'$ follows, since 
\begin{align*}
 & \cosh\left(\frac{c_2}{2}+\sqrt{2}y\right)\sinh\left(\frac{-c_2}{2}\right)\cosh\left(\frac{c_2}{2}\right) + \sinh\left(\frac{c_2}{2}+\sqrt{2}y\right)\left(\cosh^2\left(\frac{c_2}{2}+\sqrt{2}y\right) - \cosh^2\left(\frac{c_2}{2}\right)\right) \\
 &=\frac{1}{4} \left( \sinh\left(\frac{3c_2}{2}+3\sqrt{2}y\right)-2\sinh\left(\frac{3c_2}{2}+\sqrt{2}y\right)-\sinh\left(\frac{c_2}{2}+\sqrt{2}y\right) \right) \,,\\
 & \geq \frac{1}{2} \left(\sinh\left(\frac{c_2}{2}+\sqrt{2}y\right)-\sinh\left(\frac{3c_2}{2}+\sqrt{2}y\right) \right) \,, \\
 & \geq 0\,,
\end{align*}
where we used
\begin{align*}
 \sinh\left(\frac{3c_2}{2}+3\sqrt{2}y\right) &= \sinh\left(\frac{c_2}{2}+\sqrt{2}y\right)\cosh(c_2+2\sqrt{2}y)+\cosh\left(\frac{c_2}{2}+\sqrt{2}y\right)\sinh(c_2+2\sqrt{2}y) \,, \\
 &\geq \sinh\left(\frac{c_2}{2}+\sqrt{2}y\right)+ \sinh(c_2+2\sqrt{2}y) \,, 
\end{align*}
\[\sinh\left(c_2+2\sqrt{2}y\right)=2\sinh\left(\frac{c_2}{2}+\sqrt{2}y\right)\cosh\left(\frac{c_2}{2}+\sqrt{2}y\right) \geq 2\sinh\left(\frac{c_2}{2}+\sqrt{2}y\right)\] and $c_2<0$ together with the fact that $\sinh$ is increasing.

\medskip
We have just proved the following statement: if $\mathpzc{d}(x_0,x)\leq r$, then \[\forall\,y\in [0,1], \quad \gamma(y) \in \{x'\in \R^d \times [u_{\min}, +\infty)^d\: :\: \mathpzc{d}(x_0,x')\leq r\}\,.\] This concludes the proof of Lemma \ref{lemma:geodesics_within_near_regions}. 

\subsubsection{Proof of Lemma \ref{lemma:minimum_separation}}\label{section:proof_lemma_minimum_separation}
Let $r,\Delta>0$. Assume that $\X \subset \R^d \times [u_{\min},u_{\max}]^d$. The semi-distance $\mathpzc{d}$ is defined by \eqref{eq:def_semi_distance}.
\\\underline{Control of the variance of the geodesics:} Let $x_j^0 \in \X$. For $x \in B_{\mathpzc{d}}(x_j^0,r) \cap \R^d \times [u_{\min},u_{\max}]^d$, we have $|t_k-t_{j,k}^0| \leq r\sqrt{2u_{\max}^2+\tau^2}$ for all $k\in\{1,\ldots,d\}$. 
We deduce that \[B_{\mathpzc{d}}(x_j^0,r) \cap (\R^d \times [u_{\min},u_{\max}]^d )\subset \left(\bigtimes_{k=1}^d [t_{j,k}^0 \pm r\sqrt{2u_{\max}^2+\tau^2}] \right) \times [u_{\min},u_{\max}]^d\] where $\bigtimes_{k=1}^d$ denotes the $d$-ary Cartesian product.

For any $k\in\{1,\ldots,d\}$, the semicircle $\gamma_k=(\gamma_{t_k},\gamma_{u_k})$ (geodesic in dimension 1) connecting the points $(t_{j,k}^0 - r \sqrt{2u_{\max}^2+\tau^2},u_{\max})$ and $(t_{j,k}^0,u_{\max})$ satisfies $\gamma_{u_k}(y)^2 \leq u_{\max}^2 + \frac{1}{4}r^2(2u_{\max}^2+\tau^2)$. This follows because the square of the radius of the semicircle is $u_{\max}^2 + \frac{1}{4}r^2(2u_{\max}^2+\tau^2)$ (see Lemma \ref{lemma:form_geodesics_dim_1}). The same holds for the semicircle connecting $(t_{j,k}^0 + r \sqrt{2u_{\max}^2+\tau^2},u_{\max})$ and $(t_{j,k}^0,u_{\max})$.

Any geodesic connecting $(t_{j,k}^0,u_{j,k}^0)$ with a point of $[t_{j,k}^0 - r\sqrt{2u_{\max}^2+\tau^2}\, ,\, t_{j,k}^0 + r\sqrt{2u_{\max}^2+\tau^2}]\times [u_{\min},u_{\max}]$ is below these geodesics.

To see why it is the case, one can note that 2 geodesics in our model have at most 1 intersection point, or their union is a geodesic. So a geodesic between points of $[t_{j,k}^0 - r\sqrt{2u_{\max}^2+\tau^2}\, ,\, t_{j,k}^0]\times [u_{\min},u_{\max}]$ cannot intersect twice the semicircle connecting $(t_{j,k}^0 - r \sqrt{2u_{\max}^2+\tau^2},u_{\max})$ and $(t_{j,k}^0 , u_{\max})$. An analogous statement holds when considering $[t_{j,k}^0\, , \, t_{j,k}^0 + r\sqrt{2u_{\max}^2+\tau^2}]\times [u_{\min},u_{\max}]$.

Using Lemma \ref{lemma:form_geodesics_dim_d}, we deduce that
\[
\mathcal{G}_{x_j^0}(B_{\mathpzc{d}}(x_j^0,r) \cap (\R^d \times [u_{\min},u_{\max}]^d)) \subset \R^d \times \left[u_{\min}, \sqrt{u_{\max}^2 + \frac{1}{4} r^2(2u_{\max}^2+\tau^2)}\right]^d\,.
\]
\underline{Pseudo-quasi triangle inequality:} Let $x, x', \tilde x \in \R^d \times [u_{\min},\tilde u_{\max}]^d$, where $\tilde u_{\max}$ is not necessarily equal to $u_{\max}$.
We have \[\frac{u_k^2+u_k'^2+\tau^2}{\sqrt{2u_k^2+\tau^2}\sqrt{2u_k'^2+\tau^2}}\leq \frac{2\tilde u_{\max}^2+\tau^2}{2u_{\min}^2+\tau^2}\leq\frac{\tilde u_{\max}^2}{u_{\min}^2} \]
using that $\tau \in \R^+ \mapsto \frac{2\tilde u_{\max}^2+\tau^2}{2u_{\min}^2+\tau^2}$ is decreasing (because $u_{\min}\leq \tilde u_{\max}$). 
Moreover, using the triangle inequality for the $\ell^2$-norm in dimension $d$ and that $u_k^2+\tilde u_k^2+\tau^2$ and $u_k'^2+\tilde u_k^2+\tau^2$ are smaller than $\frac{\tilde u_{\max}^2}{u_{\min}^2}(u_k^2+ u_k'^2+\tau^2)$, we get 
\begin{align*}
 \sqrt{\sum_{k=1}^d \frac{(t_k-t_k')^2}{u_k^2+u_k'^2+\tau^2}} &\leq \sqrt{\sum_{k=1}^d \frac{(t_k-\tilde t_k)^2}{u_k^2+ u_k'^2+\tau^2}}+\sqrt{\sum_{k=1}^d \frac{(t_k'-\tilde t_k)^2}{u_k^2+u_k'^2+\tau^2}} \,, \\
 & \leq \frac{\tilde u_{\max}}{u_{\min}}\left(\sqrt{\sum_{k=1}^d \frac{(t_k-\tilde t_k)^2}{u_k^2+ \tilde u_k^2+\tau^2}}+\sqrt{\sum_{k=1}^d \frac{(t_k'-\tilde t_k)^2}{u_k'^2+\tilde u_k^2+\tau^2}}\right) \,, \\
 & \leq\frac{\tilde u_{\max}}{u_{\min}}\left( \mathpzc{d}(x,\tilde x) + \mathpzc{d}(x', \tilde x)\right)\,.
\end{align*}
 Hence \[\mathpzc{d}(x,x') \leq \frac{\tilde u_{\max}}{u_{\min}} \left(\mathpzc{d}(x,\tilde x)+\mathpzc{d}(x',\tilde x)\right) + \sqrt{d \ln \left(\frac{\tilde u_{\max}^2}{u_{\min}^2}\right)}\,.\]
 So $\mathpzc{d}(x,\tilde x)\geq \frac{u_{\min}}{\tilde u_{\max}}\left(\mathpzc{d}(x,x') - \sqrt{d\ln \left(\frac{\tilde u_{\max}^2}{u_{\min}^2}\right)} \right)-\mathpzc{d}(x',\tilde x)$.
\\\underline{Conclusion:}
Let $x_j^0,x_i^0 \in \X$. Firstly, replacing $\tilde u_{\max}$ by $u_{\max}$ in the pseudo-quasi triangle inequality, for $\tilde x\in B_{\mathpzc{d}}(x_j^0,\Delta) \cap \X$, if $\frac{u_{\min}}{u_{\max}}\left(\mathpzc{d}(x_i^0,x_j^0) - \sqrt{d\ln \left(\frac{u_{\max}^2}{u_{\min}^2}\right)} \right)-\mathpzc{d}(\tilde x,x_j^0)> \Delta$ we have $\mathpzc{d}(x_i^0,\tilde x) >\Delta$. Note also that $\mathpzc{d}(\tilde x,x_j^0)< \Delta$, hence it suffices that $\mathpzc{d}(x_j^0,x_i^0)\geq 2\frac{u_{\max}}{u_{\min}} \Delta + \sqrt{d\ln \left(\frac{u_{\max}^2}{u_{\min}^2}\right)}$ for the open balls $\mathring B_{\mathpzc{d}}(x_j^0,\Delta) \cap \X$ to be disjoint.

Secondly, for $\tilde x\in \mathcal{G}_{x_j^0}(B_{\mathpzc{d}}(x_j^0,r) \cap \X)$, taking $\tilde u_{\max}=\sqrt{u_{\max}^2 + \frac{1}{4}r^2(2u_{\max}^2+\tau^2)}$ we have $\mathpzc{d}(x_i^0,\tilde x) \geq \Delta$ as soon as $\frac{u_{\min}}{\tilde u_{\max}}\left(\mathpzc{d}(x_i^0,x_j^0) - \sqrt{d\ln \left(\frac{\tilde u_{\max}^2}{u_{\min}^2}\right)} \right) \geq \Delta+r$.
We used that $\mathcal{G}_{x_j^0}(B_{\mathpzc{d}}(x_j^0, r)) \subset B_{\mathpzc{d}}(x_j^0, r)$, so $\mathpzc{d}(\tilde x, x_j^0)\leq r$ (see Lemma~\ref{lemma:geodesics_within_near_regions}).
This condition can be rewritten as $\mathpzc{d}(x_i^0,x_j^0)\geq \frac{\tilde u_{\max}}{u_{\min}} (\Delta+r) + \sqrt{d\ln \left(\frac{u_{\max}^2}{u_{\min}^2}\right)}$.
\medskip
We proved that if
 \begin{equation*}
 \min_{i \neq j} \mathpzc{d}(x_i^0, x_j^0) \geq \max \left\{\frac{\sqrt{u_{\max}^2 + \frac{1}{4}r^2(2u_{\max}^2+\tau^2)}}{u_{\min}} (\Delta+r) \: , \: 2\frac{u_{\max}}{u_{\min}} \Delta \right\} + \sqrt{d\ln \left(\frac{u_{\max}^2}{u_{\min}^2}\right)}\eqcolon \Delta_\tau \,, 
 \end{equation*}
 then the open balls $\mathring B_{\mathpzc{d}}(x_j^0,\Delta) \cap \X$ are disjoint, and for all $j \neq i$, the ball $\mathcal{G}_{x_j^0}(B_{\mathpzc{d}}(x_j^0,r) \cap \X)$ does not intersect $\mathring B_{\mathpzc{d}}(x_i^0,\Delta)$.

\subsubsection{Local control with the semi-distance}

\begin{lemma}[Local control of $\mathfrak{d}_\mathfrak{g}$ with the semi-distance, lower bound, $d=1$]\label{lemma:local_control_metric_eps3_dim_1}
Let $x_0, x \in \R\times [u_{\min},+\infty)$. If $\mathpzc{d}(x_0,x) \leq r$, then
 \[\mathfrak{d}_{\mathfrak{g}}(x_0,x)^2 \geq \frac{\mathpzc{d}(x_0,x)^2}{\tilde\varepsilon_3}\] for $\tilde\varepsilon_3\geq 1+\frac{1}{R(r)}$ with $R(r)$ defined by \eqref{eq:def_R(r)_eps3} below.
\end{lemma}
\begin{proof}
Let $x_0=(t_0,u_0), x=(t,u) \in \R\times [u_{\min},+\infty)$ such that $r_e \coloneq \mathpzc{d}(x_0,x) \leq r$. Without loss of generality, we assume that $u_0 \leq u$.

We examine the case where $u$ is ``far'' from $u_0$, and then the case where $t$ is far from $t_0$ and $u,u_0$ are close.
\\\underline{Bounds on $B_{\mathpzc{d}}(x_0,r)$, exhibition of 2 cases:} Since $\mathpzc{d}(x_0,x)^2=\frac{(t_0-t)^2}{u^2+u_0^2+\tau^2} + \ln\left( \frac{u^2+u_0^2+\tau^2}{\sqrt{2u^2+\tau^2}\sqrt{2u_0^2+\tau}}\right) = r_e^2$, we have that for all $0<w<1$, 
\begin{equation}
\label{eq:inter_formula_1st_case_proof_control_lower_bound}
\frac{(t_0-t)^2}{u^2+u_0^2+\tau^2} \geq w r_e^2 \end{equation} (first case) or \begin{equation*}
\frac{u^2+u_0^2+\tau^2}{\sqrt{2u^2+\tau^2}\sqrt{2u_0^2+\tau^2}} \geq e^{(1-w)r_e^2}\end{equation*} (second case). The second case can be rewritten as \begin{equation}
\label{eq:inter_formula_2nd_case_proof_control_lower_bound}
\sqrt{u^2+\frac{\tau^2}{2}} \geq \sqrt{u_0^2+\frac{\tau^2}{2}} (e^{(1-w)r_e^2} + \sqrt{e^{2(1-w)r_e^2}-1})\end{equation}(see \eqref{eq:control_close_variances}).
\\\underline{Assuming the first case holds and the second case does not:} Provided that \eqref{eq:inter_formula_1st_case_proof_control_lower_bound} holds, then $t_0 \neq t$ and the geodesic connecting $x$ and $x_0$ is a semicircle (see Lemma \ref{lemma:form_geodesics_dim_1}).
We suppose additionally that we are not in the second case, namely that \[\sqrt{u^2+\frac{\tau^2}{2}} \leq \sqrt{u_0^2+\frac{\tau^2}{2}} (e^{(1-w)r_e^2} + \sqrt{e^{2(1-w)r_e^2}-1})\,.\] 
In particular, using \eqref{eq:control_close_variances_2} we have \begin{equation}\label{eq:inter_formula_bound_u_1st_case_proof_control_lower_bound}
 \frac{|u_0^2-u^2|}{u_0^2+u^2+\tau^2}\leq \sqrt{e^{2(1-w)r_e^2}-1}\,. 
\end{equation}
Denoting $l$ the arc length (\ie $l=\mathfrak{d}_{\mathfrak{g}}(x,x_0)$), using \eqref{eq:form_geod_semicircles} and \eqref{eq:formula_radius_semicircle} we have 
\begin{equation}\frac{|\tanh(\frac{c_2}{2}+\sqrt{2}l)-\tanh(\frac{c_2}{2})|}{\sqrt{2}} = |c_1||t-t_0|\label{eq:tanh_t-t_0} \end{equation}
 for some $c_2$ and $c_1$ verifying \begin{equation*}
 \frac{1}{2c_1^2}-\frac{\tau^2}{2}=\frac{1}{4}\left(\frac{-(t-t_0)^2 +u_0^2-u^2}{t-t_0}\right)^2+u_0^2\,.
\end{equation*}
According to \eqref{eq:inter_formula_1st_case_proof_control_lower_bound} and \eqref{eq:inter_formula_bound_u_1st_case_proof_control_lower_bound},
\begin{align}
\frac{1}{|c_1||t-t_0|} &=\sqrt{\frac{1}{2}\left(-1+ \frac{u_0^2-u^2}{(t-t_0)^2} \right)^2+\frac{2u_0^2+\tau^2}{(t-t_0)^2}} \,, \notag \\
&=\sqrt{\frac{1}{2}+ \frac{1}{2}\frac{(u_0^2-u^2)^2}{(t-t_0)^4} +\frac{u_0^2+u^2+\tau^2}{(t-t_0)^2}} \,, \label{eq:c_1_t-t_0}\\
&\leq \sqrt{\frac{1}{2}+\frac{e^{2(1-w)r_e^2}-1}{2w^2r_e^4}+ \frac{1}{wr_e^2}} \,, \notag \\
& \leq \frac{1}{\sqrt{w}r_e}\sqrt{\frac{r^2}{2}+\frac{1-w}{w}\frac{e^{2r^2}-1}{2r^2}+ 1} \notag
\end{align}
where we used that $y\mapsto e^{2y}-1$ is convex along with $(1-w)r_e^2 \leq r^2$, and $w r_e^2 \leq r^2$.

So as $\tanh$ is 1-Lipschitz, 
\begin{align}
 \mathfrak{d}_{\mathfrak{g}}(x,x_0)=l&\geq |c_1||t-t_0|\,, \notag \\
 &\geq \frac{r_e\sqrt{w}}{\sqrt{1+\frac{r^2}{2} + \frac{1-w}{w} \frac{e^{2r^2}-1}{2r^2}}}\,. \label{eq:bound_1st_case_proof_control_lower_bound}
\end{align}
\underline{Assuming the second case holds:} If \eqref{eq:inter_formula_2nd_case_proof_control_lower_bound} holds, we must consider the 2 possible geodesics (\cf Lemma \ref{lemma:form_geodesics_dim_1}). 
\\If $t=t_0$, the geodesic is of the form $\left( c_{3} , \ \sqrt{- \frac{\tau^2}{2} +\frac{c_{1}^2}{2} e^{ \sqrt{8} y} }\right)$. As $\tilde \gamma(0)=x_0$ we have $c_1^2=2u_0^2+\tau^2$. Since $\tilde \gamma_u(l)=u$, we have $e^{\sqrt{8}l}=\frac{2u^2+\tau^2}{2u_0^2+\tau^2}$ from which we deduce that \[\mathfrak{d}_{\mathfrak{g}}(x_0,x)=l=\frac{1}{\sqrt{2}}\ln \left(\frac{\sqrt{u^2+\frac{\tau^2}{2}}}{\sqrt{u_0^2+\frac{\tau^2}{2}}} \right)\geq \frac{1}{\sqrt{2}} \ln \left(e^{(1-w)r_e^2} + \sqrt{e^{2(1-w)r_e^2}-1} \right)\,.\] 
If $t\neq t_0$, $\tilde \gamma_u$ is of the form $\sqrt{- \frac{\tau^2}{2} + \frac{1}{2 \cosh^{2}{\left(\frac{c_{2}}{2} + \sqrt{2} y \right)} c_{1}^{2}}}$. Using that $\tilde \gamma_u(0)=u_0$ and $\tilde \gamma_u(l)=u$, we get $\frac{\cosh^{2}{\left(\frac{c_{2}}{2}\right)}}{\cosh^{2}{\left(\frac{c_{2}}{2} + \sqrt{2} l \right)}}=\frac{2u^2+\tau^2}{2u_0^2+\tau^2}$, from which we deduce using \eqref{eq:inter_formula_2nd_case_proof_control_lower_bound} that \[\frac{\cosh{\left(\frac{c_{2}}{2}\right)}}{\cosh{\left(\frac{c_{2}}{2} + \sqrt{2} l \right)}} \geq e^{(1-w)r_e^2} + \sqrt{e^{2(1-w)r_e^2}-1}\,.\] As \[\frac{\cosh\left(\frac{c_{2}}{2}\right)}{\cosh\left(\frac{c_{2}}{2} + \sqrt{2} l \right)}= \cosh(\sqrt{2} l) - \frac{\sinh\left(\frac{c_{2}}{2}\right)\sinh(\sqrt{2} l)}{\cosh\left(\frac{c_{2}}{2}\right)} \leq \cosh(\sqrt{2} l) + |\sinh(\sqrt{2} l)|=e^{|\sqrt{2} l|}\,,\]
it comes that $e^{\sqrt{2}l} \geq e^{(1-w)r_e^2} + \sqrt{e^{2(1-w)r_e^2}-1}$ leading again to \[l \geq \frac{1}{\sqrt{2}}\ln\left( e^{(1-w)r_e^2} + \sqrt{e^{2(1-w)r_e^2}-1}\right)\,.\] 
As $\frac{1}{\sqrt{2}}\ln\left( e^{y^2} + \sqrt{e^{2y^2}-1}\right) \geq y $ for all $y\in\R$, we get for this second case 
\begin{equation}
\label{eq:bound_2nd_case_proof_control_lower_bound}
\mathfrak{d}_{\mathfrak{g}}(x,x_0) \geq \sqrt{1-w}r_e \,.
\end{equation}
\underline{Conclusion:} We now choose $0<w<1$ depending only on $r$ to balance the bounds of the two cases (namely, \eqref{eq:bound_1st_case_proof_control_lower_bound} and \eqref{eq:bound_2nd_case_proof_control_lower_bound}). Writing $X=\frac{1-w}{w}$, we want to pick $w$ such that
\begin{align*}
 &\sqrt{1-w}=\frac{\sqrt{w}}{\sqrt{1+\frac{r^2}{2} + \frac{1-w}{w} \frac{e^{2r^2}-1}{2r^2}}} \\
 \iff &\frac{e^{2r^2}-1}{2r^2}X^2 + \left(1+\frac{r^2}{2}\right)X -1=0 \,.
\end{align*}
This polynomial has exactly 1 positive root 
\begin{equation}\label{eq:def_R(r)_eps3}
 R(r)=\frac{-\left(1+\frac{r^2}{2}\right)+ \sqrt{\left(1+\frac{r^2}{2}\right)^2+2\frac{e^{2r^2}-1}{r^2}}}{\frac{e^{2r^2}-1}{r^2}} \,.
\end{equation}
We can take $w_r=\frac{1}{R(r)+1}$. Then \[\mathfrak{d}_\mathfrak{g}(x_0,x)\geq \sqrt{1-w_r}r_e= \sqrt{\frac{R(r)}{1+R(r)}} r_e\,.\]
\end{proof}

\begin{lemma}[Local control of $\mathfrak{d}_\mathfrak{g}$ with the semi-distance, lower bound, $d\geq 1$]\label{lemma:local_control_metric_eps3_dim_d}
Let $x_0, x \in \R^d\times [u_{\min},+\infty)^d$. If $\mathpzc{d}(x_0,x) \leq r$, then
 \[\mathfrak{d}_{\mathfrak{g}}(x_0,x)^2 \geq \frac{\mathpzc{d}(x_0,x)^2}{\tilde\varepsilon_3}\] for $\tilde\varepsilon_3\geq 1+ \frac{1}{R(r)}$ with $R(r)$ defined by \eqref{eq:def_R(r)_eps3}.
\end{lemma}
\begin{proof}
Let $x_0, x \in \R^d\times [u_{\min},+\infty)^d$ such that $\mathpzc{d}(x_0,x) \leq r$. For all $k\in \{1,\ldots,d\}$, $\mathpzc{d}(x_{0,k},x_k)\leq r$. Lemma \ref{lemma:local_control_metric_eps3_dim_1} gives $\mathfrak{d}_\mathfrak{g}(x_k,x_{0,k})^2\geq \frac{\mathpzc{d}(x_{0,k},x_k)^2}{1+\frac{1}{R(r)}}$. We conclude using that \[\mathfrak{d}_\mathfrak{g}(x,x_0)^2=\sum_{k=1}^d \mathfrak{d}_\mathfrak{g}(x_k,x_{0,k})^2 \geq \sum_{k=1}^d \frac{\mathpzc{d}(x_{0,k},x_k)^2}{1+\frac{1}{R(r)}}= \frac{\mathpzc{d}(x_0,x)^2}{1+\frac{1}{R(r)}}\,.\]
\end{proof}

\begin{lemma}[Local control of $\mathfrak{d}_\mathfrak{g}$ with the semi-distance, upper bound, $d=1$]\label{lemma:local_control_metric_eps4_dim_1}
Let $x, x_0\in \R \times [u_{\min},+\infty)$. If $\mathpzc{d}(x,x_0)<\sqrt{2}$, then $\mathfrak{d}_{\mathfrak{g}}(x,x_0)^2 \leq F( \mathpzc{d}(x,x_0))$ ($F$ defined by \eqref{eq:def_F(r_e)} below). 
\end{lemma}
\begin{proof}
Let $x, x_0\in \R \times [u_{\min},+\infty)$. Without loss of generality, we can assume that $u_0 \leq u$.
We denote $l=\mathfrak{d}_{\mathfrak{g}}(x,x_0)$ and $r_e=\mathpzc{d}(x,x_0)$. The proof is close to that of Lemma \ref{lemma:local_control_metric_eps3_dim_1}.
\\\underline{First case:} If $t\neq t_0$, the geodesic connecting $x$ and $x_0$ is a semicircle. Recalling \eqref{eq:tanh_t-t_0} and \eqref{eq:c_1_t-t_0}, we have 
\[\frac{|\tanh(\frac{c_2}{2}+\sqrt{2}l)-\tanh(\frac{c_2}{2})|}{\sqrt{2}|c_1|} = |t-t_0|\] for some $c_2$ and $c_1$ verifying 
\begin{align*}
 \frac{1}{|c_1||t-t_0|}
&=\sqrt{\frac{1}{2}+ \frac{1}{2}\frac{(u_0^2-u^2)^2}{(t-t_0)^4} +\frac{u_0^2+u^2+\tau^2}{(t-t_0)^2}} \,, \\ 
&\geq \frac{1}{r_e}\,.
\end{align*}
Hence \begin{equation}
\label{eq:inter_formula_tanh_proof_control_upper_bound}
\left|\tanh\left(\frac{c_2}{2}+\sqrt{2}l\right)-\tanh\left(\frac{c_2}{2}\right)\right| \leq \sqrt{2}r_e\,.
\end{equation}
Furthermore, using $\frac{u_0^2+u^2+\tau^2}{\sqrt{2u^2+\tau^2}\sqrt{2u_0^2+\tau^2}}\leq e^{r_e^2}$ together with \eqref{eq:control_close_variances} and recalling that $u_0\leq u$, we have
\begin{equation}
\label{eq:inter_formula_variances_proof_control_upper_bound}
1\leq \frac{\cosh{\left(\frac{c_{2}}{2}\right)}}{\cosh{\left(\frac{c_{2}}{2} + \sqrt{2} l \right)}}=\sqrt{\frac{2u^2+\tau^2}{2u_0^2+\tau^2}} \leq e^{r_e^2} + \sqrt{e^{2r_e^2}-1}\,.\end{equation}
Using \eqref{eq:inter_formula_variances_proof_control_upper_bound} along with the equality $\frac{\cosh (A+ B)}{\cosh(A)}=\cosh(B)+\tanh(A)\sinh(B)$, we deduce that
\[\cosh\left(\sqrt{2} l \right)+\tanh\left(\frac{c_{2}}{2}\right)\sinh\left(\sqrt{2} l \right) \leq 1\] and that 
\[\cosh\left(\sqrt{2} l \right)-\tanh\left(\frac{c_{2}}{2}+\sqrt{2} l\right)\sinh\left(\sqrt{2} l \right) \leq e^{r_e^2}+\sqrt{e^{2r_e^2}-1}\,.\]
It comes \[2\cosh\left(\sqrt{2} l \right)+\left(\tanh\left(\frac{c_{2}}{2}\right) -\tanh\left(\frac{c_{2}}{2}+\sqrt{2} l\right)\right)\sinh\left(\sqrt{2} l \right) \leq e^{r_e^2}+\sqrt{e^{2r_e^2}-1}+1\,,\] so
\[2\cosh\left(\sqrt{2} l \right)- \sinh\left(\sqrt{2} l \right)\sqrt{2} r_e \leq e^{r_e^2}+\sqrt{e^{2r_e^2}-1}+1\] where we have used \eqref{eq:inter_formula_tanh_proof_control_upper_bound}.
As $\cosh \geq 0$ and $\frac{\sinh(\sqrt{2}l)}{\cosh(\sqrt{2}l)}\leq 1$, we get \[\cosh(\sqrt{2}l) \left(1- \frac{1}{\sqrt{2}}r_e \right) \leq \frac{1}{2}\left(e^{r_e^2}+\sqrt{e^{2r_e^2}-1}+1\right)\,.\]
For $r_e <\sqrt{2}$ we get \begin{equation}
\label{eq:1st_bound_proof_control_upper_bound}
l^2 \leq \frac{1}{2} \arcosh\left( \frac{\frac{1}{2}\left(e^{r_e^2}+\sqrt{e^{2r_e^2}-1}+1\right)}{1- \frac{1}{\sqrt{2}}r_e}\right)^2 \,.
\end{equation}
\underline{Second case:} If $t_0=t$, using \eqref{eq:form_geod_straight_lines} we get $e^{\sqrt{8}l}=\frac{2u^2+\tau^2}{2u_0^2+\tau^2} \leq (e^{r_e^2}+\sqrt{e^{2r_e^2}-1})^2$ so 
\begin{equation} \label{eq:2nd_bound_proof_control_upper_bound}
 l^2 \leq \frac{1}{2}\ln(e^{r_e^2}+\sqrt{e^{2r_e^2}-1})^2=\frac{1}{2}\arcosh(e^{r_e^2})^2\,.
 \end{equation}
\underline{Comparison of the bounds found, conclusion:} 
Let $0\leq r_e< \sqrt{2}$. As $1- \frac{1}{\sqrt{2}}r_e \leq 1$ and $\sqrt{e^{2r_e^2}-1}\geq e^{r_e^2}-1$, the bound found for the first case (see \eqref{eq:1st_bound_proof_control_upper_bound}) is larger than the one found for the second case (see \eqref{eq:2nd_bound_proof_control_upper_bound}).

So if $\mathpzc{d}(x,x_0)<\sqrt{2}$, $\mathfrak{d}_\mathfrak{g}(x,x_0)^2\leq F(\mathpzc{d}(x,x_0))$ with
\begin{equation}\label{eq:def_F(r_e)}
 F: y \in \R^+\mapsto \frac{1}{2} \arcosh\left( \frac{\frac{1}{2}\left(e^{y^2}+\sqrt{e^{2y^2}-1}+1\right)}{1- \frac{1}{\sqrt{2}}y}\right)^2\,.
\end{equation}
\end{proof}

\begin{lemma}[Local control of $\mathfrak{d}_\mathfrak{g}$ with the semi-distance, upper bound, $d\geq1$]\label{lemma:local_control_metric_eps4_dim_d}
Let $x, x_0\in \R^d \times [u_{\min},+\infty)^d$. If $\mathpzc{d}(x,x_0)< \sqrt{2}$, then $\mathfrak{d}_{\mathfrak{g}}(x,x_0)^2 \leq d F(\mathpzc{d}(x,x_0))$ ($F$ defined by \eqref{eq:def_F(r_e)}). 
\end{lemma}
\begin{proof}
 Let $x,x_0\in \R^d \times [u_{\min},+\infty)^d$ such that $\mathpzc{d}(x,x_0)=r_e< \sqrt{2}$. As for all $k=1,\ldots,d$, $\mathpzc{d}(x_k,x_{0,k})\leq \mathpzc{d}(x,x_0)$,
by Lemma \ref{lemma:local_control_metric_eps4_dim_1} it comes that 
\begin{align*}
 \mathfrak{d}_{\mathfrak{g}}(x,x_0)^2&=\sum_{k=1}^d \mathfrak{d}_\mathfrak{g}(x_k,x_{0,k})^2 \,\\
 &\leq \sum_{k=1}^d F(\mathpzc{d}(x_k,x_{0,k}))\,,\\
 &\leq d F(\mathpzc{d}(x,x_{0}))
\end{align*}
where we used that $F$ is non-decreasing.
\end{proof}

\subsection{Proof of Lemma \ref{lemma:tilde_varepsilon_3}}
\label{section:proof_lemma_tilde_varepsilon_3}

 We use Lemma \ref{lemma:local_control_metric_eps3_dim_d}. It comes that we can choose $\tilde \varepsilon_3=1+\frac{1}{R\left(\frac{0.3025}{\sqrt{d}}\right)}$ with $R$ defined by \eqref{eq:def_R(r)_eps3} in the appendix. We aim to find a more interpretable parameter. 
 Remark that $r \in \R^+ \mapsto \frac{e^{2r^2}-1}{r^2}$ is increasing, along with $\frac{e^{2r^2}-1}{r^2}\geq 2$. We deduce that for all $0<r\leq 0.3025$,
 $R(r)\geq \frac{-\left(1+\frac{0.3025^2}{2}\right)+\sqrt{5}}{\frac{e^{2\times 0.3025^2}-1}{0.3025^2}}$.
 So $1+\frac{1}{R\left(\frac{0.3025}{\sqrt{d}}\right)}\leq 2.84$ (see \citep[Section III]{notebook}).

\section{Proof of Theorem \ref{th:existence_certificate_under_positive_curvature_assumption}}\label{section:proof_existence_certificate_under_positive_curvature_assumption}

\paragraph{Construction of the certificates by solving a linear system}
We give explicit formulas for $\eta$, $\eta_j$ of Theorem \ref{th:existence_certificate_under_positive_curvature_assumption}. These certificates are of the form \eqref{eq:eta_general_form}. In the following, we write $\eta=\eta_{\alpha,\beta}$ and $\eta_j=\eta_{\alpha^j,\beta^j}$.
Recalling Definitions \ref{def:non_degenerate_certificate} and \ref{def:local_non_degenerate_certificate}, we want that for all $j=1,\ldots,s$, $\eta_{\alpha,\beta}(x_j^0)=1$ and $\nabla\eta_{\alpha,\beta}(x_j^0)=0_{2d}$. We also want $\eta_{\alpha^j,\beta^j}(x_j^0)=1$, $\eta_{\alpha^j,\beta^j}(x_l^0)=0$ for all $l\neq j$, and $\nabla\eta_{\alpha^j,\beta^j}(x_l^0)=0_{2d}$.

These constraints translate to linear systems. Writing $e_j$ the vector of size $s$ containing a $1$ at position $j$, and zeros elsewhere, along with $1_s=\sum_{j=1}^s e_j$, we want to solve
\begin{align}
\begin{split} \label{eq:Upsilon}
 &\Upsilon \begin{pmatrix}
 \alpha \\
 \beta
 \end{pmatrix}=\begin{pmatrix}
 1_{s} \\
 0_{s2d}
 \end{pmatrix} \eqcolon \bm u_s\quad \text{and} \quad \Upsilon \begin{pmatrix}
 \alpha^j \\
 \beta^j
 \end{pmatrix}=\begin{pmatrix}
 e_j \\
 0_{s2d}
 \end{pmatrix}\eqcolon\bm u_s^j \quad \forall \, j=1,\ldots,s
 \\&\text{where} \quad \Upsilon=\begin{pmatrix}
 (K_{\rm norm}(x_i^0,x_j^0))_{i,j=1,\ldots,s} & (\nabla_1 K_{\rm norm}(x_i^0,x_j^0))_{i,j=1,\ldots,s}^T \\
 (\nabla_2 K_{\rm norm}(x_i^0,x_j^0))_{i,j=1,\ldots,s} & (\nabla_1 \nabla_2 K_{\rm norm}(x_i^0,x_j^0))_{i,j=1,\ldots,s} 
\end{pmatrix}\in \R^{s(1+2d)\times s(1+2d)}\,.
\end{split}
\end{align}
The following lemma entails that these linear systems admit a solution provided that a minimal separation condition holds.
\begin{lemma}\label{lemma:invertibility_Upsilon}
 Let $\{x_j^0\}_{j=1}^s \subset \R^d \times [u_{\min}, +\infty)^d$. Assume that $K_{\rm norm}$ verifies Definition \ref{def:positive_curvature}. If $\min_{i\neq j} \mathpzc{d}(x_i^0,x_j^0)~\geq~\Delta$, then $\Upsilon$ (see \eqref{eq:Upsilon}) is invertible. 
 
 Moreover, $\begin{pmatrix}
 \alpha \\
 \beta
 \end{pmatrix}=\Upsilon^{-1}\bm u_s$ and
$\begin{pmatrix}
 \alpha^j \\
 \beta^j
\end{pmatrix}=\Upsilon^{-1}\bm u_s^j$ are well-defined and verify, for all $j=1,\ldots,s$, 
\begin{align}
\begin{split}\label{eq:bounds_alpha_beta_certif}
 &\norm{\alpha}_\infty \vee \norm{\alpha^j}_\infty \leq \frac{1}{1-2h} \,,\\
&\max_{l=1,\ldots,s} \norm{\beta_l}_{x_l^0} \vee \max_{l=1,\ldots,s} \norm{\beta_l^j}_{x_l^0}\leq 4h\\
\text{and}\quad &\norm{\alpha-1_s}_\infty \vee |\alpha_j^j-1| \vee \max_{l \neq j}|\alpha_l^j|\leq \frac{2h}{1-2h}\,,
\end{split}
\end{align}
where $h=\frac{1}{64}\min\left(\frac{\bar\varepsilon_0(r)}{B_0},\frac{\bar\varepsilon_2(r)}{B_2}\right)$. 
\end{lemma}
\noindent
The proof can be found in \citep[pp. 269-270]{off-the-grid_cs}. This proof explicitly requires to handle normalized kernels. This motivates the introduction of $W$ in \eqref{eq:pb_BLASSO_gaussian_kernel_norm}.

\paragraph{Controlling the certificates via the kernel} We provide controls on the certificates we constructed in Lemma \ref{lemma:invertibility_Upsilon} on the far and near regions, in order to show that they are non-degenerate (Definitions \ref{def:non_degenerate_certificate} and \ref{def:local_non_degenerate_certificate}). To do so, we use Taylor expansions on Fisher-Rao geodesics along with bounds on the kernel stemming from the LPC (Definition \ref{def:positive_curvature}).
\begin{lemma} \label{lemma:adaptation_lemma2_poon}
Let $r>0$ and $x_0 \in \X \subset \R^d \times [u_{\min},+\infty)^d$. We define \[\X_0^{\rm near}(r)\coloneq \{ x' \in \X \: : \: \mathpzc{d}(x',x_0)\leq r \}\,.\]
Assume that there exists $\bar\varepsilon_2>0$
such that, for all $x \in \mathcal{G}_{x_0}(\X_0^{\rm near}(r))$ and $v\in \R^{2d}$, \[-K_{\rm norm}^{(02)}(x_0,x)[v,v] \geq \bar\varepsilon_2 \norm{v}_{x}^2 \quad \text{and} \quad \norm{K_{\rm norm}^{(02)}(x_0,x)}_{x} \leq B_{02} \,.\]
Let $\eta : \R^d \times [u_{\min},+\infty)^d \rightarrow \R$ be a $\C^2$ function. Then the following holds:
\begin{enumerate}[label=(\roman*)]
 \item If $\eta(x_0)=0$, $\nabla \eta(x_0)=0$ and $\norm{D_2[\eta](x)}_{x}\leq \delta$ for all $x\in \mathcal{G}_{x_0}(\X_0^{\rm near}(r))$, then $|\eta(x)| \leq \delta \mathfrak{d}_{\mathfrak{g}}(x,x_0)^2$ for all $x\in \X_0^{\rm near}(r)$.
 \item Let $a\in \{-1,1\}$. If $\eta(x_0)=a$, $\nabla \eta(x_0)=0$ and $\norm{aD_2[\eta](x)-K_{\rm norm}^{(02)}(x_0,x)}_{x}\leq \delta$ for all $x\in \mathcal{G}_{x_0}(\X_0^{\rm near}(r))$, for some $\delta< \bar\varepsilon_2$, then \[1-\frac{B_{02}+\delta}{2}\mathfrak{d}_{\mathfrak{g}}(x,x_0)^2\leq a\eta(x) \leq 1- \frac{\bar\varepsilon_2- \delta}{2} \mathfrak{d}_{\mathfrak{g}}(x,x_0)^2 \quad \forall \, x\in \X_0^{\rm near}(r)\,.\]
\end{enumerate} 
\end{lemma}
\begin{proof}
 This proof is based on \citep[Lemma 2]{off-the-grid_cs}. We only show $(ii)$, as the proof for $(i)$ is similar.
 \\Let $x\in \X_0^{\rm near}(r)$. We denote by $\gamma$ the geodesic for the metric $\mathfrak{g}$ between $x_0$ and $x$, parameterized between $0$ and $1$. We refer to Section \ref{section:geodesics_geodesic_distance} for a description of the geodesic properties and related notations. By definition, $\gamma(y)\in \mathcal{G}_{x_0}(\X_0^{\rm near}(r))$ for all $y \in [0,1]$. Hence, for all $y \in [0,1]$, \[\norm{aD_2[\eta](\gamma(y))-K_{\rm norm}^{(02)}(x_0,\gamma(y))}_{\gamma(y)}\leq \delta \quad \text{and} \quad -K_{\rm norm}^{(02)}(x_0,\gamma(y))[\dot \gamma(y),\dot \gamma(y)] \geq\bar\varepsilon_2 \norm{\dot \gamma(y)}_{\gamma(y)}^2\,,\] from which we deduce that \[aD_2[\eta](\gamma(y))[\dot \gamma(y),\dot \gamma(y)] \leq (-\bar\varepsilon_2+\delta)\norm{\dot \gamma(y)}_{\gamma(y)}^2\,.\]
By a Taylor expansion, we get
\begin{align*}
 a\eta(x)&=a\eta(x_0)+ a\nabla\eta(x_0)^T \dot\gamma(0) + \int_0^1 (1-y) aD_2[\eta](\gamma(y))[\dot\gamma(y),\dot\gamma(y)] \, \d y \,, \\
 &=1+\int_0^1 (1-y) aD_2[\eta](\gamma(y))[\dot\gamma(y),\dot\gamma(y)] \, \d y \,, \\
 &\leq 1-(\bar\varepsilon_2 - \delta)\int_0^1 (1-y) \norm{\dot\gamma(y)}_{\gamma(y)}^2 \, \d y \,,\\
 &=1-\frac{\bar\varepsilon_2- \delta}{2} \mathfrak{d}_{\mathfrak{g}}(x,x_0)^2 \,.
\end{align*}
Using the same reasoning, we also have $a\eta(x) \geq 1 - \frac{B_{02}+ \delta}{2}\mathfrak{d}_{\mathfrak{g}}(x,x_0)^2$.

Note that we do not have (unlike in \citep[Lemma 2]{off-the-grid_cs}) the bound $a\eta(x) \geq -1+ \frac{\bar\varepsilon_2 - \delta}{2} \mathfrak{d}_{\mathfrak{g}}(x,x_0)^2$. This is not a problem in our framework: as we work with nonnegative measures, we do not need to control the negative part of the certificate.
\end{proof}

\begin{theorem}\label{th:proof_non_degenerescence_certif}
 Let $\{x_j^0\}_{j=1}^s \subset \X \subset \R^d \times [u_{\min},u_{\max}]^d$. Assume that $K_{\rm norm}$ verifies Definition \ref{def:positive_curvature}. \\If $\min_{i\neq j} \mathpzc{d}(x_i^0,x_j^0) \geq \Delta_\tau$ (defined in Lemma \ref{lemma:minimum_separation}), then the certificates constructed in Lemma \ref{lemma:invertibility_Upsilon} are non-degenerate. The global certificate $\eta_{\alpha,\beta}$ is $(\frac{7}{8}\bar \varepsilon_0,\frac{15}{32}\bar \varepsilon_2, r)$-non-degenerate and the local certificate $\eta_{\alpha^j,\beta^j}$ is $(\frac{7}{8}\bar \varepsilon_0,\frac{B_{02}+ \bar \varepsilon_2/16}{2}, r)$-non-degenerate. 
\end{theorem}
\begin{proof}
This proof is largely based on \citep[pp. 270-241]{off-the-grid_cs}.
 First remark that Lemmas \ref{lemma:invertibility_Upsilon} and \ref{lemma:minimum_separation} hold under these assumptions. We denote $h= \frac{1}{64}\min\left(\frac{\bar\varepsilon_0(r)}{B_0},\frac{\bar\varepsilon_2(r)}{B_2}\right)$.
 \\\underline{Control on the far region:} Let $x \in \X^{\rm far}(r)$. As the open balls $\mathring B_{\mathpzc{d}}(x_i^0,\Delta)\cap \X$ are disjoint (Lemma \ref{lemma:minimum_separation}), there exists at most one index $j$ such that $\mathpzc{d}(x,x_j^0)<\Delta$ and for all $i\neq j$, $\mathpzc{d}(x,x_i^0)\geq \Delta$.
 So using bounds displayed in Lemma \ref{lemma:invertibility_Upsilon} (see \eqref{eq:bounds_alpha_beta_certif}),
 \begin{align*}
 |\eta_{\alpha,\beta}(x)|&=\left|\alpha_j K_{\rm norm}(x_j^0,x) +\sum_{j\neq i} \alpha_i K_{\rm norm}(x_i^0,x) + \beta_j^T K_{\rm norm}^{(10)}(x_j^0,x) + \sum_{j\neq i}\beta_i^T K_{\rm norm}^{(10)}(x_i^0,x)\right|\,,\\
 &\leq \norm{\alpha}_\infty \left(|K_{\rm norm}(x_j^0,x)| + \sum_{j\neq i} |K_{\rm norm}(x_i^0,x)| \right)+ \max_i \norm{\beta_i}_{x_i^0} \left(\norm{K_{\rm norm}^{(10)}(x_j^0,x)}_{x_j^0} + \sum_{j\neq i} \norm{K_{\rm norm}^{(10)}(x_i^0,x)}_{x_i^0} \right)\,,\\
 & \leq \frac{1}{1-2h}(1-\bar \varepsilon_0 + h)+ 4h(B_{10}+h) \,, \\
 & \leq 1- \frac{\bar \varepsilon_0-3h}{1-2h}+ 4h(B_{10}+h) \,, \\
 & \leq 1-\bar \varepsilon_0 +3h + 4h(B_{10}+h) \,, \\
 & \leq 1-\bar \varepsilon_0 +3\frac{\bar \varepsilon_0}{64} +4\frac{\bar \varepsilon_0}{64} +4\frac{\bar \varepsilon_0}{64^2} \,, \\
 & \leq 1 - \frac{7}{8}\bar \varepsilon_0\,.
 \end{align*}
 We can apply the same reasoning to show that $|\eta_{\alpha^j,\beta^j}(x)| \leq 1 - \frac{7}{8}\bar \varepsilon_0$.
 \\\underline{Controls on the near regions:} 
 Let $x \in \mathcal{G}_{x_j^0}(\X_j^{\rm near}(r))$. We have \[D_2[\eta_{\alpha,\beta}](x)= K^{(02)}_{\rm norm}(x_j^0,x)+(\alpha_j-1)K^{(02)}_{\rm norm}(x_j^0,x) + \sum_{i \neq j}\alpha_i K^{(02)}_{\rm norm}(x_i^0,x) + [\beta_j] K^{(12)}_{\rm norm}(x_j^0,x)+\sum_{i \neq j}[\beta_i] K^{(12)}_{\rm norm}(x_i^0,x)\,.\]
 As $\mathcal{G}_{x_j^0}(B_{\mathpzc{d}}(x_j^0,
 r)\cap \X)$ is disjoint from $\mathring B_{\mathpzc{d}}(x_i^0,\Delta)$ for $i \neq j$ (Lemma \ref{lemma:minimum_separation}), we have $\mathpzc{d}(x,x_i^0)\geq \Delta$ for all $i \neq j$. Using Lemma \ref{lemma:invertibility_Upsilon}, it comes
 \begin{align*}
 \norm{D_2[\eta_{\alpha,\beta}](x)-K^{(02)}_{\rm norm}(x_j^0,x)}_{x} &\leq |\alpha_j - 1| B_{02} + \norm{\alpha}_\infty h + \max_i\norm{\beta_i}_{x_i^0} (B_{12} + h)\,,\\
 & \leq \frac{2h}{1-2h} B_{02} + \frac{h}{1-2h} + 4h(B_{12} + h)\,, \\
 & \leq \frac{ \bar \varepsilon_2}{16} \,.
 \end{align*}
 Using Lemma \ref{lemma:adaptation_lemma2_poon} with $\delta=\frac{ \bar \varepsilon_2}{16}$, we deduce that for all $x \in \X_j^{\rm near}(r)$, $\eta_{\alpha,\beta}(x) \leq 1-\frac{15}{32}\bar \varepsilon_2 \mathfrak{d}_\mathfrak{g}(x,x_j^0)^2$.

 With the same reasoning, for all $x\in \mathcal{G}_{x_j^0}(\X_j^{\rm near}(r))$ we can obtain $\norm{D_2[\eta_{\alpha^j,\beta^j}](x)-K^{(02)}_{\rm norm}(x_j^0,x)}_{x}\leq \frac{ \bar \varepsilon_2}{16}$. So from Lemma \ref{lemma:adaptation_lemma2_poon}, for all $x \in \X_j^{\rm near}(r)$ we have $|1-\eta_{\alpha^j,\beta^j}(x)|\leq \frac{B_{02}+\bar \varepsilon_2/16}{2}\mathfrak{d}_\mathfrak{g}(x,x_j^0)^2$. We used that $B_{02}\geq \bar \varepsilon_2$.

 We also have, for $i\neq j$, and $x\in \mathcal{G}_{x_i^0}(\X_i^{\rm near}(r))$,
 \begin{align*}
 \norm{D_2[\eta_{\alpha^j,\beta^j}](x)}_x&=\norm{\alpha_i^j K^{(02)}_{\rm norm}(x_i^0,x) + \sum_{l \neq i}\alpha_l^j K^{(02)}_{\rm norm}(x_l^0,x) + [\beta_i^j] K^{(12)}_{\rm norm}(x_i^0,x)+\sum_{l \neq i}[\beta_l^j] K^{(12)}_{\rm norm}(x_l^0,x)}_x\,,\\
 &\leq |\alpha_i^j|B_{02}+ \norm{\alpha^j}_\infty h + \max_i\norm{\beta_i^j}_{x_i^0} (B_{12} + h)\,,\\ 
 & \leq \frac{2h}{1-2h}B_{02} + \frac{h}{1-2h} +4h(B_{12} + h) \,, \\
 & \leq \frac{ \bar \varepsilon_2}{16}\,.
 \end{align*}
 Lemma \ref{lemma:adaptation_lemma2_poon} ensures that for all $x \in \X_i^{\rm near}(r)$, $|\eta_{\alpha^j,\beta^j}(x)|\leq \frac{ \bar \varepsilon_2}{16} \mathfrak{d}_\mathfrak{g}(x,x_i^0)^2$.
 This concludes the proof. 
\end{proof}

\paragraph{Norm of the certificates} We use bounds on the norm of the certificates to get the controls in estimation and in prediction ($c_p$ in Assumption \ref{assumption:existence_non-degenerate_certif}). The construction of certificates presented in Lemma \ref{lemma:invertibility_Upsilon} allows us to obtain the bounds displayed in the following proposition.
\begin{proposition}\label{prop:norm_certif}
 The certificates constructed in Lemma \ref{lemma:invertibility_Upsilon} verify, for all $x\in \R^d \times [u_{\min},+\infty)^d$, for all $j~\in~\{1,\ldots,s\}$, \[\eta_{\alpha,\beta}(x)=\innerprod{\Psi\delta_x}{p_{\alpha,\beta}}_\L \quad \text{and} \quad 
 \eta_{\alpha^j,\beta^j}(x)=\innerprod{\Psi\delta_x}{p_{\alpha^j,\beta^j}}_\L \quad \text{with} \quad \norm{p_{\alpha,\beta}}_\L \leq \sqrt{2s}\,, \quad \norm{p_{\alpha^j,\beta^j}}_\L \leq \sqrt{2}\,.\]
\end{proposition}
\begin{proof}
 From Lemma \ref{lemma:eta_c_infty_im_psi_star}, for all $j~\in~\{1,\ldots,s\}$, \[p_{\alpha,\beta}=\sum_{i=1}^s \alpha_i \Psi\delta_{x_i^0}+ \sum_{i=1}^s \beta_i \nabla_{x}\left(\Psi\delta_{x_i^0}\right)\] and \[p_{\alpha^j,\beta^j}=\sum_{i=1}^s \alpha_i^j \Psi\delta_{x_i^0}+ \sum_{i=1}^s \beta_i^j \nabla_{x}\left(\Psi\delta_{x_i^0}\right)\,.\] So recalling \eqref{eq:Upsilon} and defining \[D_{\mathfrak{g}}\coloneq \begin{pmatrix}
 \operatorname{Id}_s & & & & \\
& & \mathfrak{g}_{x_1^0}^{-\frac{1}{2}} & & \\
& & & \ddots & \\
& & & & \mathfrak{g}_{x_s^0}^{-\frac{1}{2}} 
 \end{pmatrix}\in \R^{s(d+1)\times s(d+1)}\,, \quad \tilde\Upsilon\coloneq D_{\mathfrak{g}}\Upsilon D_{\mathfrak{g}}\,,\] using the results of \citep[p. 268]{off-the-grid_cs} we have $\norm{p_{\alpha,\beta}}_\L^2=\begin{pmatrix}
 \alpha \\\beta 
 \end{pmatrix}^T \Upsilon \begin{pmatrix}
 \alpha \\\beta 
 \end{pmatrix}= \bm u_s^T \tilde \Upsilon^{-1} \bm u_s$.
 We can apply \citep[Lemma 3]{off-the-grid_cs} that gives $\norm{\tilde\Upsilon^{-1}}_2\leq 2$. So $\norm{p_{\alpha,\beta}}_\L^2 \leq \norm{\tilde\Upsilon^{-1}}_2 \norm{\bm u_s}_2^2 \leq 2s$. 
 We repeat the same reasoning to bound $\norm{p_{\alpha^j,\beta^j}}_\L$: we have $\norm{p_{\alpha^j,\beta^j}}_\L^2=(\bm u_s^j)^T \tilde \Upsilon^{-1} \bm u_s^j \leq \norm{\tilde\Upsilon^{-1}}_2 \norm{\bm u_s^j}_2^2 \leq 2$.
\end{proof}

The combination of Theorem \ref{th:proof_non_degenerescence_certif} and Proposition \ref{prop:norm_certif} concludes the proof of Theorem \ref{th:existence_certificate_under_positive_curvature_assumption}.

\section{Proof of Theorem \ref{th:check_local_curvature_assumption}}\label{section:controls_kernel}
\begin{lemma}\label{lemma:simplified_expressions_operator_norms_kernel} 
Let $x,x' \in \R^d \times [u_{\min},+\infty)^d$.
We can establish that
\begin{align*}
\norm{K_{\rm norm}^{(00)}(x,x')}_{x,x'}&=|K_{\rm norm}(x,x')|\,, \\
 \norm{K_{\rm norm}^{(10)}(x,x')}_{x,x'}&=\norm{\mathfrak{g}_x^{-1/2}\nabla_1 K_{\rm norm}(x,x')}_2\,, \\
 \norm{K_{\rm norm}^{(11)}(x,x')}_{x,x'}&=\norm{\mathfrak{g}_x^{-1/2}\nabla_1 \nabla_2 K_{\rm norm}(x,x')\mathfrak{g}_{x'}^{-1/2}}_2 \,, \\
 \norm{K_{\rm norm}^{(02)}(x,x')}_{x,x'}&=\norm{\mathfrak{g}_{x'}^{-1/2}H_2^{\mathfrak{g}} K_{\rm norm}(x,x')\mathfrak{g}_{x'}^{-1/2}}_2 \,, \\
\norm{K_{\rm norm}^{(12)}(x,x')}_{x,x'}
 & \leq \sqrt{2d} \max_{k=1,\ldots d} \left\lbrace \norm{\mathfrak{g}_{t_kt_k}^{-1/2} \mathfrak{g}_{x'}^{-1/2} \partial_{t_k} H_2^{\mathfrak{g}} K_{\rm norm}(x,x') \mathfrak{g}_{x'}^{-1/2}}_2 \right. \,, \\ 
 &\qquad\qquad\qquad\quad \left. \norm{\mathfrak{g}_{u_ku_k}^{-1/2} \mathfrak{g}_{x'}^{-1/2} \partial_{u_k} H_2^{\mathfrak{g}} K_{\rm norm}(x,x') \mathfrak{g}_{x'}^{-1/2}}_2
  \right\rbrace \,.
\end{align*}
\end{lemma}
\begin{proof}
 To get the simplified expressions for the operator norms (Definition \ref{def:riemannian_derivatives}), we use that $\norm{v}_x^2\coloneq v^T \mathfrak{g}_x v$. The result for $\norm{K_{\rm norm}^{(10)}(x,x')}_{x,x'}, \norm{K_{\rm norm}^{(11)}(x,x')}_{x,x'}, \norm{K_{\rm norm}^{(02)}(x,x')}_{x,x'}$ is stated by \citep[Equation (27)]{off-the-grid_cs}. 
 
 To deal with $\norm{K_{\rm norm}^{(12)}(x,x')}_{x,x'}$, we also use that our metric is diagonal. Denoting $\tilde q= \sqrt{\mathfrak{g}_x} q$, $\tilde V_1=\sqrt{\mathfrak{g}_{x'}} V_1$, $\tilde V_2=\sqrt{\mathfrak{g}_{x'}} V_2$, we have 
 \begin{align*}
 \norm{K_{\rm norm}^{(12)}(x,x')}_{x,x'}&=\sup_{\substack{ \norm{V_1}_{x'}, \norm{V_2}_{x'} \leq 1 \\ \norm{q=(q_1,\ldots q_{2d})}_{x} \leq 1}} \sum_{k=1}^{d} q_k V_1^T \partial_{t_k} H_2^{\mathfrak{g}} K_{\rm norm}(x,x') V_2 + \sum_{k=1}^{d} q_{k+d} V_1^T \partial_{u_k} H_2^{\mathfrak{g}} K_{\rm norm}(x,x') V_2\,, \\
 &=\sup_{\norm{\tilde V_1}_2\, , \: \norm{\tilde V_2}_2 \leq 1\,, \: \norm{\tilde q}_2 \leq 1} \bigg(
 \sum_{k=1}^{d} \tilde q_k \tilde V_1^T \mathfrak{g}_{t_kt_k}^{-1/2} \mathfrak{g}_{x'}^{-1/2}\partial_{t_k} H_2^{\mathfrak{g}} K_{\rm norm}(x,x') \mathfrak{g}_{x'}^{-1/2} \tilde V_2 
 \\& \qquad\qquad\qquad\qquad\qquad\qquad + \sum_{k=1}^{d} \tilde q_{k+d} \tilde V_1^T \mathfrak{g}_{u_ku_k}^{-1/2}\mathfrak{g}_{x'}^{-1/2} \partial_{u_k} H_2^{\mathfrak{g}} K_{\rm norm}(x,x') \mathfrak{g}_{x'}^{-1/2} \tilde V_2
 \bigg) \,, \\
 &\leq\sqrt{2d}\max_{k=1,\ldots d} \left\lbrace \sup_{ \norm{\tilde V_1}_2, \norm{\tilde V_2}_2 \leq 1}\left| \tilde V_1^T \mathfrak{g}_{t_kt_k}^{-1/2} \mathfrak{g}_{x'}^{-1/2}\partial_{t_k} H_2^{\mathfrak{g}} K_{\rm norm}(x,x') \mathfrak{g}_{x'}^{-1/2} \tilde V_2 \right| \right. \,, \\ & \qquad\qquad\qquad\qquad \left.
 \sup_{ \norm{\tilde V_1}_2, \norm{\tilde V_2}_2 \leq 1}\left|\tilde V_1^T \mathfrak{g}_{u_ku_k}^{-1/2}\mathfrak{g}_{x'}^{-1/2} \partial_{u_k} H_2^{\mathfrak{g}} K_{\rm norm}(x,x') \mathfrak{g}_{x'}^{-1/2} \tilde V_2 \right|
 \right\rbrace\,,
 \end{align*} 
 giving the desired result.
\end{proof}

\paragraph{Global controls} To check the first condition of Definition \ref{def:positive_curvature}, we give global controls for the Riemannian derivatives of the kernel. We take advantage of the metric used, which allows us to have uniform bounds on $\R^d \times [u_{\min},+\infty)^d$.
\begin{lemma}[Global controls]\label{lemma:global_controls_kernel}
For $(i,j) \in \{0,1\} \times \{0,1,2\} $ we have \[\underset{x,x' \in \R^d\times [u_{\min},+\infty)^d}{\sup} \norm{K_{\rm norm}^{(ij)}(x,x')}_{x,x'} \leq B_{ij}\]
 with
 \begin{align*}
 B_{00}&=1 \,,\\
 B_{10}=B_{01}&= \sqrt{2d} \,, \\
 B_{11}&=2d\,, \\
 B_{02}=B_{20}&=\sqrt{4d^2+10d} \,, \\
 B_{12}=B_{21}&= \sqrt{2d}B_{02}\,. 
 \end{align*}
\end{lemma}
\begin{proof}
 The bounds are based on the expressions given in Lemma \ref{lemma:simplified_expressions_operator_norms_kernel}. We make use of the relation between the Frobenius norm and the $2$-norm: \[\forall \, M=(m_{ij})_{\substack{1 \leq i \leq n \\ 1\leq j \leq m}} \in \R^{n \times m}\,, \quad \norm{M}_2 \leq \sqrt{\sum_{\substack{1 \leq i \leq n \\ 1\leq j \leq m}} m_{ij}^2 } \leq \sqrt{nm \max_{i,j} m_{ij}^2} \,.\] 
 Let $x,x' \in \R^d \times [u_{\min},+\infty)^d$.
 \\\underline{$B_{00}$:} Using Cauchy-Schwarz and since the kernel is normalized, we have \[|K_{\rm norm}(x,x')| =|\innerprod{\Psi\delta_x}{\Psi \delta_{x'}}_\L| \leq \sup_{x \in \R^d \times [u_{\min},+\infty)^d} \norm{\Psi\delta_x}_\L^2 =\sup_{x \in \R^d \times [u_{\min},+\infty)^d} K_{\rm norm}(x,x)=1\,.\]
 \underline{$B_{10}$:} We use that $\mathfrak{g}_x$ is diagonal, and name its diagonal elements by $\diag(\mathfrak{g}_{t_1 t_1},\ldots, \mathfrak{g}_{t_d t_d},\mathfrak{g}_{u_1 u_1},\ldots, \mathfrak{g}_{u_d u_d})~=~\mathfrak{g}_x$. Remark that $\nabla_1 K_{\rm norm}(x,x')=\innerprod{\nabla \Psi\delta_x}{\Psi \delta_{x'}}_\L$ along with, for all $k\in\{1,\ldots d\}$, $\norm{\partial_{t_k} \Psi\delta_x}_\L=\mathfrak{g}_{t_k t_k}^{1/2}$ and $\norm{\partial_{u_k} \Psi\delta_x}_\L=\mathfrak{g}_{u_k u_k}^{1/2}$. We use Cauchy-Schwarz to write
 \begin{align*}
 \norm{\mathfrak{g}_x^{-1/2}\nabla_1 K_{\rm norm}(x,x')}_2 &= \sqrt{\sum_{k=1}^d \mathfrak{g}_{t_k t_k}^{-1} \innerprod{\partial_{t_k} \Psi\delta_x}{\Psi \delta_{x'}}_\L^2 +\sum_{k=1}^d \mathfrak{g}_{u_k u_k}^{-1} \innerprod{\partial_{u_k} \Psi\delta_x}{\Psi \delta_{x'}}_\L^2 } \,, \\
 & \leq \sqrt{\sum_{k=1}^d \norm{\Psi \delta_{x'}}_\L^2 +\sum_{k=1}^d \norm{\Psi \delta_{x'}}_\L^2 } \,, \\
 & \leq \sqrt{2d} \,.
 \end{align*}
 \underline{$B_{11}$:} We use the same reasoning. With $\mathfrak{g}_x^{-1/2}\nabla_1 \nabla_2 K_{\rm norm}(x,x')\mathfrak{g}_{x'}^{-1/2}\eqcolon (m_{ij})_{1 \leq i,j \leq 2d}$, we have $m_{ij}$ of the form $\mathfrak{g}_{b_k b_k}^{-1/2}\mathfrak{g}_{b_l b_l}^{-1/2}\innerprod{\partial_{b_k} \Psi\delta_x}{\partial_{b_l}\Psi \delta_{x'}}_\L$ where $b_m$ stands for $u_m$ or $t_m$, $1\leq m \leq d$.
 So \[m_{ij}^2 \leq \mathfrak{g}_{b_k b_k}^{-1}\mathfrak{g}_{b_l b_l}^{-1}\norm{\partial_{b_k} \Psi\delta_x}_\L^2 \norm{\partial_{b_l}\Psi \delta_{x'}}_\L^2 \leq 1\,,\] from which we conclude that $\norm{M}_2 \leq \sqrt{4d^2}=2d$.
 \\\underline{$B_{02}$:} Here, we denote $(m_{ij})_{1 \leq i,j \leq 2d}=\mathfrak{g}_{x'}^{-1/2}H_2^{\mathfrak{g}} K_{\rm norm}(x,x')\mathfrak{g}_{x'}^{-1/2}$. Recall that \[H_2^{\mathfrak{g}} K_{\rm norm}(x,x')= \nabla_2^2 K_{\rm norm}(x,x')-\sum_{k=1}^d\Gamma^{t_k'} \partial_{t_k'}K_{\rm norm}(x,x')-\sum_{k=1}^d\Gamma^{u_k'} \partial_{u_k'}K_{\rm norm}(x,x')\,.\] 
 Using that $\Gamma^{b_k}{}_{b_l b_m} \neq 0$ implies that $b_k=b_l=t_k, b_m=u_k$ or $b_k=b_m=t_k, b_l=u_k$ or $b_k=u_k, b_l=b_m=t_k$ or $b_k=b_l=b_m=u_k$, we can treat the $4d$ terms $m_{ii},m_{i 2i}, m_{2i i}, m_{2i 2i}$ ($i \in \{1,\ldots d\}$) separately from the rest of the matrix.
 
 The other terms are of the form \[\mathfrak{g}_{b_k'b_k'}^{-1/2}\mathfrak{g}_{b_l'b_l'}^{-1/2} \partial_{b_k'b_l'}K_{\rm norm}(x,x') \] where $k\neq l$ (\ie the derivatives are associated with different dimensions). 
 By abuse of notation, we denote $K_{\rm norm}(x_m,x_m')$ an evaluation of the kernel in dimension $d=1$ at points $x_m=(t_m,u_m),x_m'=(t_m',u_m')$, and we consider in a similar way $\Psi \delta_{x_m}$. Remark that $\mathfrak{g}_{b_m b_m}=\norm{\partial_{b_m}\Psi \delta_{x_m}}_\L^2$.
 As $K_{\rm norm}(x,x')= \prod_{m=1}^d K_{\rm norm}(x_m,x_m')$, for $k \neq l$ we have 
\begin{equation}\label{eq:decomposition_second_derivatives_diff_dim}
 \partial_{b_k'b_l'}K_{\rm norm}(x,x')=\partial_{b_k'}K_{\rm norm}(x_k,x_k')\partial_{b_l'}K_{\rm norm}(x_l,x_l')\prod_{m \notin \{k,l\}} K_{\rm norm}(x_m,x_m')\,.
 \end{equation}
 So using Cauchy-Schwarz, for $m_{ij}$ such that $i \neq j$ and $i \neq 2j$, $j\neq 2i$, \[m_{ij}^2 \leq \mathfrak{g}_{b_k'b_k'}^{-1}\mathfrak{g}_{b_l'b_l'}^{-1} \left(\partial_{b_k'}K_{\rm norm}(x_k,x_k')\right)^2\left(\partial_{b_l'}K_{\rm norm}(x_l,x_l')\right)^2\prod_{m \notin \{k,l\}} K_{\rm norm}(x_m,x_m')^2 \leq 1 \,.\] This quantity bounds $(2d)^2-4d$ squared coefficients of $M$.

 The terms $m_{ii}$, $m_{i 2i}$ and $m_{2i i}$, $m_{2i 2i}$ ($i \in \{1,\ldots d\}$) are of the form 
 \begin{align} & \mathfrak{g}_{t_k't_k'}^{-1}\left(\partial_{t_k't_k'} K_{\rm norm}(x,x') - \Gamma^{u_k'}{}_{t_k' t_k'}\partial_{u_k'}K_{\rm norm}(x,x')\right) \label{eq:B_02_1st_form}
 \\\text{or} \quad & \mathfrak{g}_{t_k't_k'}^{-1/2}\mathfrak{g}_{u_k'u_k'}^{-1/2}\left(\partial_{t_k'u_k'} K_{\rm norm}(x,x') - \Gamma^{t_k'}{}_{t_k' u_k'}\partial_{t_k'}K_{\rm norm}(x,x')\right) \label{eq:B_02_2nd_form} \\\text{or} \quad & \mathfrak{g}_{u_k'u_k'}^{-1}\left(\partial_{u_k'u_k'} K_{\rm norm}(x,x') - \Gamma^{u_k'}{}_{u_k' u_k'}\partial_{u_k'}K_{\rm norm}(x,x')\right)\,.\label{eq:B_02_3rd_form}
 \end{align} These forms concern respectively $d$, $2d$, $d$ coefficients. We can again reduce the problem to dimension 1 using the decomposition of the kernel: 
\begin{align}\label{eq:decomposition_second_derivatives_same_dim} 
 \begin{split} &\partial_{t_k't_k'} K_{\rm norm}(x,x') - \Gamma^{u_k'}{}_{t_k' t_k'}\partial_{u_k'}K_{\rm norm}(x,x') = \left(\partial_{t_k't_k'} K_{\rm norm}(x_k,x_k')- \Gamma^{u_k'}{}_{t_k' t_k'}\partial_{u_k'}K_{\rm norm}(x_k,x_k') \right)\prod_{l \neq k } K_{\rm norm}(x_l,x_l') \,, \\
 & \partial_{t_k'u_k'} K_{\rm norm}(x,x') - \Gamma^{t_k'}{}_{t_k' u_k'}\partial_{t_k'}K_{\rm norm}(x,x')= \left(\partial_{t_k'u_k'} K_{\rm norm}(x_k,x_k')- \Gamma^{t_k'}{}_{t_k' u_k'}\partial_{t_k'}K_{\rm norm}(x_k,x_k')\right)\prod_{l \neq k } K_{\rm norm}(x_l,x_l') \,,\\
 &\partial_{u_k'u_k'} K_{\rm norm}(x,x') - \Gamma^{u_k'}{}_{u_k' u_k'}\partial_{u_k'}K_{\rm norm}(x,x')=\left(\partial_{u_k'u_k'} K_{\rm norm}(x_k,x_k')- \Gamma^{u_k'}{}_{u_k' u_k'}\partial_{u_k'}K_{\rm norm}(x_k,x_k')\right)\prod_{l \neq k } K_{\rm norm}(x_l,x_l') \,. 
 \end{split}
 \end{align} 
 So for the first form \eqref{eq:B_02_1st_form},
 \begin{align*} m_{ii}^2&=\mathfrak{g}_{t_k't_k'}^{-2}\innerprod{\Psi \delta_x}{\partial_{t_k't_k'}\Psi \delta_{x'} - \Gamma^{u_k'}{}_{t_k' t_k'}\partial_{u_k'}\Psi \delta_{x'}}_\L^2 \,, \\
 & \leq \mathfrak{g}_{t_k't_k'}^{-2} \norm{\partial_{t_k't_k'}\Psi \delta_{x_k'} - \Gamma^{u_k'}{}_{t_k' t_k'}\partial_{u_k'}\Psi \delta_{x_k'}}_\L^2 \,, \\
 &= \mathfrak{g}_{t_k't_k'}^{-2} \left(\partial_{t_k t_kt_k't_k'} K_{\rm norm}(x_k',x_k') +\Gamma^{u_k'}{}_{t_k' t_k'}^2\partial_{u_ku_k'} K_{\rm norm}(x_k',x_k')-2\Gamma^{u_k'}{}_{t_k' t_k'} \partial_{t_k t_ku_k'}K_{\rm norm}(x_k',x_k')\right) \,, \\
 &=1
 \end{align*}
 where $\partial_{t_k}K_{\rm norm}$ (resp.~$\partial_{t_k'}K_{\rm norm}$) denotes a derivative \wrt the first (resp.\ the second) variable of the kernel. We can calculate this quantity, which is constant equal to $1$ (see \citep[Section IV]{notebook}).
 In a similar way, for the second form \eqref{eq:B_02_2nd_form} we have 
 \begin{align*} 
 m_{i2i}^2,m_{2ii}^2&\leq \mathfrak{g}_{t_k't_k'}^{-1}\mathfrak{g}_{u_k'u_k'}^{-1} \norm{\partial_{t_k'u_k'}\Psi \delta_{x_k'} - \Gamma^{u_k'}{}_{t_k' u_k'}\partial_{t_k'}\Psi \delta_{x_k'}}_\L^2 \,, \\
 & =\mathfrak{g}_{t_k't_k'}^{-1}\mathfrak{g}_{u_k'u_k'}^{-1} \left(\partial_{t_k u_kt_k'u_k'} K_{\rm norm}(x_k',x_k') +\Gamma^{t_k'}{}_{t_k' u_k'}^2\partial_{t_kt_k'} K_{\rm norm}(x_k',x_k')-2\Gamma^{t_k'}{}_{t_k' u_k'} \partial_{t_k u_kt_k'}K_{\rm norm}(x_k',x_k') \right) \,, \\
 &=3 \,.
 \end{align*}
 For the third one \eqref{eq:B_02_3rd_form}, 
 \begin{align*} 
 m_{2i2i}^2&\leq \mathfrak{g}_{u_k'u_k'}^{-2}\left(\partial_{u_k u_ku_k'u_k'} K_{\rm norm}(x_k',x_k') +\Gamma^{u_k'}{}_{u_k' u_k'}^2\partial_{u_ku_k'} K_{\rm norm}(x_k',x_k')-2\Gamma^{u_k'}{}_{u_k' u_k'} \partial_{u_k u_ku_k'}K_{\rm norm}(x_k',x_k') \right) \,, \\
 &=7 \,.
 \end{align*}
 Hence $\norm{M}_2 \leq \sqrt{4d^2 - 4d + d + 2\times 3d+ 7d}= \sqrt{4d^2 + 10d}$.
\\\underline{$B_{12}$:} Lemma \ref{lemma:simplified_expressions_operator_norms_kernel} gives \[\norm{K_{\rm norm}^{(12)}(x,x')}_{x,x'}\leq \sqrt{2d}\max_{t_k,u_k\,,\,k=1,\ldots d} \left\lbrace \norm{\mathfrak{g}_{t_kt_k}^{-1/2} \mathfrak{g}_{x'}^{-1/2} \partial_{t_k} H_2^{\mathfrak{g}} K_{\rm norm}(x,x') \mathfrak{g}_{x'}^{-1/2}}_2\,,\, \norm{\mathfrak{g}_{u_ku_k}^{-1/2} \mathfrak{g}_{x'}^{-1/2} \partial_{u_k} H_2^{\mathfrak{g}} K_{\rm norm}(x,x') \mathfrak{g}_{x'}^{-1/2}}_2\right\rbrace \,.\]
For any $k\in\{1,\ldots,d\}$, $b_k=t_k$ or $b_k=u_k$, the coefficients of the matrix $M_k=\mathfrak{g}_{b_k b_k}^{-1/2} \mathfrak{g}_{x'}^{-1/2} \partial_{b_k} H_2^{\mathfrak{g}} K_{\rm norm}(x,x') \mathfrak{g}_{x'}^{-1/2}$ are of the form 
\begin{align*}
 &\mathfrak{g}_{b_m b_m}^{-1/2}\mathfrak{g}_{b_l' b_l'}^{-1/2}\mathfrak{g}_{b_k' b_k'}^{-1/2} \partial_{b_m b_l' b_k'}K_{\rm norm}(x,x') \quad \text{where} \; \; k\neq l \\
 \text{or} \quad &\mathfrak{g}_{b_m b_m}^{-1/2}\mathfrak{g}_{t_k't_k'}^{-1} \left(\partial_{b_m t_k't_k'} K_{\rm norm}(x,x') - \Gamma^{u_k'}{}_{t_k' t_k'}\partial_{b_m u_k'}K_{\rm norm}(x,x')\right)
 \\\text{or} \quad & \mathfrak{g}_{b_m b_m}^{-1/2}\mathfrak{g}_{t_k't_k'}^{-1/2}\mathfrak{g}_{u_k'u_k'}^{-1/2}\left(\partial_{b_m t_k'u_k'} K_{\rm norm}(x,x') - \Gamma^{t_k'}{}_{t_k' u_k'}\partial_{b_m t_k'}K_{\rm norm}(x,x')\right) \\\text{or} \quad & \mathfrak{g}_{b_m b_m}^{-1/2}\mathfrak{g}_{u_k'u_k'}^{-1}\left(\partial_{b_m u_k'u_k'} K_{\rm norm}(x,x') - \Gamma^{u_k'}{}_{u_k' u_k'}\partial_{b_m u_k'}K_{\rm norm}(x,x')\right) \,.
\end{align*} These forms correspond respectively to $(2d)^2-4d$, $d$, $2d$, $d$ coefficients.
We can bound them with the same arguments as before. We obtain $\norm{M_k}_2 \leq B_{02}$, so $\norm{K_{\rm norm}^{(12)}(x,x')}_{x,x'}\leq \sqrt{2d}B_{02}$.
\end{proof}

\paragraph{Controls when $\mathpzc{d}(x,x')$ is small}
\begin{lemma}[Curvature constants in dimension $d=1$]\label{lemma:curvature_constants_dim_1}
Let $r>0$. Let $x,x' \in \R \times [u_{\min},+\infty)$. If $\mathpzc{d}(x,x')\geq r$, then $K_{\rm norm}(x,x')\leq 1-\bar \varepsilon_0(r)$ where $\bar \varepsilon_0(r)\leq 1-e^{-r^2/2}$. 
Moreover, if $\mathpzc{d}(x,x')\leq r$ where $r\leq 0.32$, then \[-v^T H^{\mathfrak{g}}_2 K_{\rm norm}(x,x')v\geq \bar \varepsilon_2(r) \norm{v}_{x'}^2 \quad \forall \, v\in \R^2 \] where $\bar \varepsilon_2(r)\leq e^{-r^2/2}|G(r)|$ ($G(r)$ defined by \eqref{eq:def_G(r)} below). 
\end{lemma}
\begin{proof}
 \underline{Obtaining $\bar \varepsilon_0$:} By definition of the semi-distance (see \eqref{eq:link_semi_dist_kernel}), $\mathpzc{d}(x,x')\geq r$ implies that $K_{\rm norm}(x,x')~\leq~e^{-r^2/2}$.
 \\\underline{Obtaining $\bar \varepsilon_2$, general overview:}
We use that \[-v^T H^{\mathfrak{g}}_2 K_{\rm norm}(x,x')v\geq \bar \varepsilon_2 \norm{v}_{x'}^2 \; \; \forall \, v\in \R^2 \quad \iff \quad -v^T \mathfrak{g}_{x'}^{-1/2}H^{\mathfrak{g}}_2 K_{\rm norm}(x,x') \mathfrak{g}_{x'}^{-1/2}v\geq \bar \varepsilon_2 \norm{v}_{2}^2 \; \; \forall \, v\in \R^2\,. \]
Defining $\tilde H^{02}(x,x') \coloneq K_{\rm norm}(x,x')^{-1}\mathfrak{g}_{x'}^{-1/2}H^{\mathfrak{g}}_2 K_{\rm norm}(x,x') \mathfrak{g}_{x'}^{-1/2}$, we have (see \citep[Section V.1]{notebook})
\begin{align}
 \begin{split}\label{eq:tilde_H_02}
 \tilde H^{02}(x,x') &=\begin{pmatrix}
 \tilde H_{t't'}^{02}(x,x') & \tilde H_{t'u'}^{02}(x,x') \\
 \tilde H_{t'u'}^{02}(x,x') & \tilde H_{u'u'}^{02}(x,x') 
 \end{pmatrix}\quad \text{where}\\
 \tilde H_{t't'}^{02}(x,x')&=-1 \,, \\
 \tilde H_{t'u'}^{02}(x,x')&=\frac{(t - t')^3 (2 u'^2 + \tau^2)^{3/2}}{\sqrt{2} (u^2 + u'^2 + \tau^2)^3}
- \frac{3 (t - t') (2 u'^2 + \tau^2)^{1/2} (u'^2 - u^2)}{\sqrt{2} (u^2 + u'^2 + \tau^2)^2} \,, \\
 \tilde H_{u'u'}^{02}(x,x')&=\frac{(t-t')^4 (2u'^2+\tau^2)^{2}}{2(u^2+u'^2+\tau^2)^4}+ \frac{3(t-t')^2 (2u'^2+\tau^2)(u^2-u'^2)}{(u^2+u'^2+\tau^2)^3}+\frac{(u^2-u'^2)^2}{2(u^2+u'^2+\tau^2)^2}-\frac{(2u^2+\tau^2)(2u'^2+\tau^2)}{(u^2+u'^2+\tau^2)^2}\,. 
 \end{split}
\end{align}
If the maximal eigenvalue $\lambda$ of $ \tilde H^{02}(x,x')$ is smaller than some $c<0$, as $K_{\rm norm}(x,x')\geq e^{-r^2/2}$, \[-v^T H^{\mathfrak{g}}_2 K_{\rm norm}(x,x')v\geq-\lambda K_{\rm norm}(x,x') \norm{v}_\L^2 \geq -c e^{-r^2/2}\norm{v}_\L^2\] and we can take $\bar\varepsilon_2(r)=-e^{-r^2/2}c$.
It remains to bound $\lambda$. We use that \[\lambda \leq \max\{\tilde H_{t't'}^{02}(x,x'),\tilde H_{u'u'}^{02}(x,x')\} + |\tilde H_{t'u'}^{02}(x,x')|\,.\]
We will provide control of the three terms in the right-hand side of the previous inequality.
\\\underline{Basic inequalities:} 
Using that $\mathpzc{d}(x,x') \leq r$, we have 
\begin{equation}\label{eq:ineq_eps_2_t}
 \frac{|t-t'|}{\sqrt{u^2+u'^2+\tau^2}} \leq r 
\end{equation}
and \[\frac{u^2+u'^2+\tau^2}{\sqrt{2u^2+\tau^2}\sqrt{2u'^2+\tau^2}}\leq e^{r^2}\,.\]
Using \eqref{eq:control_close_variances_2}, we get
\begin{equation}
\frac{|u^2-{u'}^2|}{u^2+u'^2+\tau^2} \leq \sqrt{e^{2r^2}-1} \label{eq:ineq_eps_2_u_1} \,.
\end{equation}
From \eqref{eq:control_close_variances_3} we also have
\begin{equation}\label{eq:ineq_eps_2_u_2}
 \frac{2u'^2+\tau^2}{u^2+u'^2+\tau^2} \leq e^{r^2}+\sqrt{e^{2r^2}-1} \,.
\end{equation}
\underline{Control of $\tilde H_{t'u'}^{02}$:} Using \eqref{eq:ineq_eps_2_t}, \eqref{eq:ineq_eps_2_u_1} and \eqref{eq:ineq_eps_2_u_2}, we get from \eqref{eq:tilde_H_02} that
\begin{align}
 |\tilde H_{t'u'}^{02}(x,x')| &\leq r^3 \frac{ (2u'^2+\tau^2)^{3/2}}{\sqrt{2}(u^2+u'^2+\tau^2)^{3/2}}+3r\frac{(2u'^2+\tau^2)^{1/2}|u^2-u'^2|}{\sqrt{2}(u^2+u'^2+\tau^2)^{3/2}} \,, \notag\\
 & \leq \frac{1}{\sqrt{2}} r^3(e^{r^2}+\sqrt{e^{2r^2}-1})^{3/2}+\frac{3}{\sqrt{2}}r \sqrt{e^{2r^2}-1}\sqrt{e^{r^2}+\sqrt{e^{2r^2}-1}}\,. \label{eq:bound_eps_2_dim_1_H_02_12}
\end{align}
\underline{Control of $\tilde H_{u'u'}^{02}$:} 
As $\frac{\sqrt{2u^2+\tau^2}\sqrt{2u'^2+\tau^2}}{u^2+u'^2+\tau^2} \geq e^{-r^2}$,
\begin{align}
 \tilde H_{u'u'}^{02}(x,x')&\leq r^4 \frac{(2u'^2+\tau^2)^{2}}{2(u^2+u'^2+\tau^2)^2}+ 3r^2\frac{ (2u'^2+\tau^2)|u^2-u'^2|}{(u^2+u'^2+\tau^2)^2}+\frac{(u^2-u'^2)^2}{2(u^2+u'^2+\tau^2)^2}-\frac{(2u^2+\tau^2)(2u'^2+\tau^2)}{(u^2+u'^2+\tau^2)^2}\,, \notag \\
 & \leq \frac{1}{2}r^4 (e^{r^2}+\sqrt{e^{2r^2}-1})^2 + 3r^2\sqrt{e^{2r^2}-1}(e^{r^2}+\sqrt{e^{2r^2}-1})+ \frac{1}{2}(e^{2r^2}-1)-e^{-2r^2} \label{eq:bound_eps_2_dim_1_H_02_22}
\end{align}
where we have used \eqref{eq:ineq_eps_2_t}, \eqref{eq:ineq_eps_2_u_1} and \eqref{eq:ineq_eps_2_u_2} again.
\\\underline{Conclusion:}
The previous bound is greater than $\tilde H_{t't'}^{02}(x,x')=-1$ (because $-e^{-2r^2}\geq -1)$. It comes that 
\begin{align}
 \begin{split} \label{eq:def_G(r)}
 \lambda &\leq \frac{1}{\sqrt{2}} r^3(e^{r^2}+\sqrt{e^{2r^2}-1})^{3/2}+\frac{3}{\sqrt{2}}r \sqrt{e^{2r^2}-1}\sqrt{e^{r^2}+\sqrt{e^{2r^2}-1}} \\
 & \quad +\frac{1}{2}r^4 (e^{r^2}+\sqrt{e^{2r^2}-1})^2 + 3r^2\sqrt{e^{2r^2}-1}(e^{r^2}+\sqrt{e^{2r^2}-1})+ \frac{1}{2}(e^{2r^2}-1)-e^{-2r^2} \eqcolon G(r) \,.
 \end{split} 
 \end{align}
The function $G$ is non-decreasing on $\R^+$ as a sum of non-decreasing functions. It is negative for $r\leq 0.32$ (see \citep[Section V.2]{notebook}).
Then $\bar\varepsilon_2(r)$ can be chosen as $-e^{-r^2/2}G(r)=e^{-r^2/2}|G(r)|$ (or smaller) for $r \in [0,0.32]$.
\end{proof}

\begin{lemma}[Curvature constants in dimension $d \geq 1$]\label{lemma:curvature_constants_dim_d}
 Let $r\geq 0$. Let $x,x' \in \R^d \times [u_{\min},+\infty)^d$. If $\mathpzc{d}(x,x')\geq r$, then $K_{\rm norm}(x,x')\leq 1-\bar \varepsilon_0(r)$ where $\bar \varepsilon_0(r)\leq 1-e^{-\frac{r^2}{2}}$. 
Moreover, if $\mathpzc{d}(x,x')\leq r$ where $r=\frac{r_0}{\sqrt{d}}$ with $0< r_0\leq 0.32$, then \[-v H^{\mathfrak{g}}_2 K_{\rm norm}(x,x')v\geq \bar \varepsilon_2(r)\norm{v}_{x'}^2 \quad \forall \, v\in \R^2 \] where $\bar \varepsilon_2(r)\leq e^{-\frac{r_0^2}{2d}}|G(r_0)|$ ($G(r_0)$ defined by \eqref{eq:def_G(r)}). 
\end{lemma}
\begin{proof}
 To get $\bar \varepsilon_0(r)$, remark again that $\mathpzc{d}(x,x')\geq r$ implies that $K_{\rm norm}(x,x')\leq e^{-r^2/2}$ (see \eqref{eq:link_semi_dist_kernel}). 
 \\\underline{Reduction to dimension $1$:} For $\bar \varepsilon_2(r)$, we use the same reasoning as in Lemma \ref{lemma:curvature_constants_dim_1}. \\We denote $\tilde H^{02}(x,x') \coloneq K_{\rm norm}(x,x')^{-1}\mathfrak{g}_{x'}^{-1/2}H^{\mathfrak{g}}_2 K_{\rm norm}(x,x') \mathfrak{g}_{x'}^{-1/2}$. Its maximum eigenvalue $\lambda$ is smaller than $\max_{1 \leq i \leq 2d}\{\tilde H_{ii}^{02}(x,x') +\sum_{j \neq i} |\tilde H_{ij}^{02}(x,x')| \}$.
 We denote $\tilde H_{ij}^{02}(x,x')=\tilde H_{b_k'b_l'}^{02}(x,x')$ to specify that this coefficient corresponds to the derivatives \wrt $b_k',b_l'$ where $1\leq k,l\leq d$ ($b$ can be $u$ or $t$).
 
 We then remark that we can reduce the problem to dimension 1. First, according to \eqref{eq:link_semi_dist_kernel}, $\mathpzc{d}(x,x')\leq r$ implies that $\mathpzc{d}(x_k,x_k')\leq r$ for all $k=1,\ldots,d$. Using \eqref{eq:decomposition_second_derivatives_diff_dim}, we get for $l\neq k$ 
 \begin{align}
 |\tilde H_{b_k'b_l'}^{02}(x,x')| &\leq |\mathfrak{g}_{b_k'b_k'}^{-1/2}\mathfrak{g}_{b_l'b_l'}^{-1/2}K_{\rm norm}(x_k,x_k')^{-1}K_{\rm norm}(x_l,x_l')^{-1}\partial_{b_k'}K_{\rm norm}(x_k,x_k')\partial_{b_l'}K_{\rm norm}(x_l,x_l')| \,, \notag \\
 &= |\mathfrak{g}_{b_k'b_k'}^{-1/2}K_{\rm norm}(x_k,x_k')^{-1}\partial_{b_k'}K_{\rm norm}(x_k,x_k')| |\mathfrak{g}_{b_l'b_l'}^{-1/2}K_{\rm norm}(x_l,x_l')^{-1}\partial_{b_l'}K_{\rm norm}(x_l,x_l')| \,.\label{eq:eps_2_bound_tilde_H_20_b_k_b_l}
 \end{align}
 For $l=k$, using \eqref{eq:decomposition_second_derivatives_same_dim} we have 
\begin{equation}\label{eq:eps_2_bound_tilde_H_20_b_k_b_k}
 \tilde H_{t_k't_k'}^{02}(x,x') \leq \tilde H_{t_k't_k'}^{02}(x_k,x_k')=-1\,, \quad \tilde H_{u_k'u_k'}^{02}(x,x') \leq \tilde H_{u_k'u_k'}^{02}(x_k,x_k')\,, \quad |\tilde H_{t_k'u_k'}^{02}(x,x')| \leq |\tilde H_{t_k'u_k'}^{02}(x_k,x_k')|\,. \end{equation}
 \underline{Bounds for $\tilde H_{b_k'b_k'}^{02}(x,x')$:} We use \eqref{eq:eps_2_bound_tilde_H_20_b_k_b_k}. From \eqref{eq:bound_eps_2_dim_1_H_02_12} and \eqref{eq:bound_eps_2_dim_1_H_02_22}, we have \[ \tilde H_{t_k' t_k'}^{02}(x,x') \vee \tilde H_{u_k' u_k'}^{02}(x,x') \leq \frac{1}{2}r^4 (e^{r^2}+\sqrt{e^{2r^2}-1})^2 + 3r^2\sqrt{e^{2r^2}-1}(e^{r^2}+\sqrt{e^{2r^2}-1})+ \frac{1}{2}(e^{2r^2}-1)-e^{-2r^2}\] and \[|\tilde H_{t_k'u_k'}^{02}(x,x')|\leq \frac{1}{\sqrt{2}} r^3(e^{r^2}+\sqrt{e^{2r^2}-1})^{3/2}+\frac{3}{\sqrt{2}}r \sqrt{e^{2r^2}-1}\sqrt{e^{r^2}+\sqrt{e^{2r^2}-1}}\,.\] 
 \underline{Bounds for $\tilde H_{b_k'b_l'}^{02}(x,x')$, $l\neq k$:} See \citep[Section V.3]{notebook}. We calculate \[\mathfrak{g}_{x_k}^{-1/2}K_{\rm norm}(x_k,x_k')^{-1}\nabla_1 K_{\rm norm}(x_k,x_k')=\begin{pmatrix}
 - \frac{(t_k-t_k') \sqrt{2u_k^2+\tau^2}}{u_k^2+u_k'^2+\tau^2}
 \\ \frac{(t_k-t_k')^2 (2u_k^2+\tau^2)}{\sqrt{2}(u_k^2+u_k'^2+\tau^2)^2}+ \frac{u_k'^2-u_k^2}{\sqrt{2}(u_k^2+u_k'^2+\tau^2)}
\end{pmatrix} \,.\]
With \eqref{eq:ineq_eps_2_t}, \eqref{eq:ineq_eps_2_u_1} and \eqref{eq:ineq_eps_2_u_2}, it follows using \eqref{eq:eps_2_bound_tilde_H_20_b_k_b_l} that for $l\neq k$,
\begin{align*}
 |\tilde H_{t_k't_l'}^{02}(x,x')| &\leq \frac{|t_k'-t_k| \sqrt{2u_k'^2+\tau^2}}{u_k^2+u_k'^2+\tau^2} \frac{|t_l-t_l'| \sqrt{2u_l'^2+\tau^2}}{u_l^2+u_l'^2+\tau^2} \,, \\
 & \leq r^2 (e^{r^2}+\sqrt{e^{2r^2} -1})
\end{align*}
and 
\begin{align*}
 |\tilde H_{u_k't_l'}^{02}(x,x')| &\leq \frac{|t_l'-t_l| \sqrt{2u_l'^2+\tau^2}}{u_l^2+u_l'^2+\tau^2} \left( \frac{(t_k-t_k')^2 (2u_k'^2+\tau^2)}{\sqrt{2}(u_k^2+u_k'^2+\tau^2)^2}+ \frac{|u_k^2-u_k'^2|}{\sqrt{2}(u_k^2+u_k'^2+\tau^2)}\right) \,, \\
 & \leq r \sqrt{e^{r^2}+\sqrt{e^{2r^2} -1}} \times \frac{1}{\sqrt{2}}\left(r^2(e^{r^2}+\sqrt{e^{2r^2} -1}) + \sqrt{e^{2r^2} -1}\right)
\end{align*}
along with 
 
\begin{align*}
 |\tilde H_{u_k'u_l'}^{02}(x,x')|
 & \leq \frac{1}{2}\left(r^2(e^{r^2}+\sqrt{e^{2r^2} -1}) + \sqrt{e^{2r^2} -1} \right)^2 \,.
\end{align*} 
\underline{Conclusion:} The bound found for $|\tilde H_{u_k'u_l'}^{02}(x,x')|$ is greater than the one found for $|\tilde H_{t_k't_l'}^{02}(x,x')|$. In fact, using that $e^{r^2}\geq 1+r^2$,
\begin{align*}
 \frac{1}{2}\left(r^2(e^{r^2}+\sqrt{e^{2r^2} -1}) + \sqrt{e^{2r^2} -1} \right)^2 
 & \geq r^2(e^{r^2}+\sqrt{e^{2r^2} -1})\sqrt{e^{2r^2} -1}+ \frac{1}{2}(e^{2r^2} -1) \,, \\
 & \geq r^2 \sqrt{e^{2r^2} -1} + \frac{1}{2}(2r^2+1)(e^{2r^2} -1) \,, \\
 & \geq r^2 \sqrt{e^{2r^2} -1} + \frac{1}{2}(2r^2+1)(e^{r^2}(1+r^2) -1) \,,\\
 & \geq r^2 \sqrt{e^{2r^2} -1} + r^2e^{r^2}+ \frac{1}{2}(r^2+1) e^{r^2}- \frac{1}{2}(2r^2+1) \,,\\
 &\geq r^2 \sqrt{e^{2r^2} -1} + r^2e^{r^2}\,.
\end{align*}
Hence
\begin{align*}
\begin{split}
 \lambda &\leq \frac{1}{2}r^4 (e^{r^2}+\sqrt{e^{2r^2}-1})^2 + 3r^2\sqrt{e^{2r^2}-1}(e^{r^2}+\sqrt{e^{2r^2}-1})+ \frac{1}{2}(e^{2r^2}-1)-e^{-2r^2} \\
 &\quad + (d-1)\frac{1}{2}\left(r^2(e^{r^2}+\sqrt{e^{2r^2} -1}) + \sqrt{e^{2r^2} -1} \right)^2 \\
 &\quad + (d-1) r \sqrt{e^{r^2}+\sqrt{e^{2r^2} -1}} \times \frac{1}{\sqrt{2}}\left(r^2(e^{r^2}+\sqrt{e^{2r^2} -1}) + \sqrt{e^{2r^2} -1}\right) \\
 &\quad + \frac{1}{\sqrt{2}} r^3(e^{r^2}+\sqrt{e^{2r^2}-1})^{3/2}+\frac{3}{\sqrt{2}}r \sqrt{e^{2r^2}-1}\sqrt{e^{r^2}+\sqrt{e^{2r^2}-1}} \,,
 \end{split}\\
 \begin{split}
 &= -e^{-2r^2}+2r^2\sqrt{e^{2r^2}-1}(e^{r^2}+\sqrt{e^{2r^2}-1}) + \frac{d}{2}\left(r^2(e^{r^2}+\sqrt{e^{2r^2} -1}) + \sqrt{e^{2r^2} -1} \right)^2 \\ &\quad + \frac{d r}{\sqrt{2}} \sqrt{e^{r^2}+\sqrt{e^{2r^2} -1}} \left(r^2(e^{r^2}+\sqrt{e^{2r^2} -1}) + \sqrt{e^{2r^2} -1}\right) \\ &\quad +\sqrt{2} r \sqrt{e^{2r^2}-1}\sqrt{e^{r^2}+\sqrt{e^{2r^2}-1}} \eqcolon G_d(r) \,.
 \end{split}
\end{align*}
Furthermore, $r\in \R^+ \mapsto \sqrt{e^{2r^2}-1}$ is convex. In fact, for $r>0$, using that $e^{2r^2}\geq 1+2r^2$, we have \[\frac{\partial^2}{\partial r^2}\sqrt{e^{2r^2}-1}= \frac{2e^{2r^2}(2r^2e^{2r^2}-4r^2+e^{2r^2}-1)}{(e^{2r^2}-1)^{3/2}} \geq \frac{2e^{2r^2}(2r^2-4r^2+1+2r^2-1)}{(e^{2r^2}-1)^{3/2}} \geq 0\,.\]
Let $r_0>0$. As $r\coloneq \frac{r_0}{\sqrt{d}}\leq r_0$, $\sqrt{e^{2r^2}-1}\leq \frac{\sqrt{e^{2r_0^2}-1}}{\sqrt{d}}$ and we deduce that $G_d(r)\leq G_1(r_0)=G(r_0)$ (see \eqref{eq:def_G(r)}.)
If $r_0 \leq 0.32$, this quantity is negative (proof of Lemma \ref{lemma:curvature_constants_dim_1}). We can take $\bar \varepsilon_2\left(\frac{r_0}{\sqrt{d}}\right)=-e^{-\frac{r_0^2}{2d}} G(r_0)$. The choice of the dependence on $d$ for $r=\frac{r_0}{\sqrt{d}}$ is intended to compensate for the term $d(e^{2r^2}-1)$ appearing in $G_d(r)$.
\end{proof}

\paragraph{Controls when $\mathpzc{d}(x,x')$ is large}
The constraint $\tau \leq u_{\min}$ is used in the following lemma.
\begin{lemma}[Bounds under a large separation, in dimension $d=1$]\label{lemma:control_kernel_d_large_dim1}
 Let $\Delta>0$. Let $x,x' \in \R \times [u_{\min},+\infty)$. Assume that $\tau \leq u_{\min}$. If $\mathpzc{d}(x,x')\geq \Delta$, then \[\max_{(i,j)\in \{0,1\} \times \{0,1,2\}} \norm{K_{\rm norm}^{(ij)}(x,x')}_{x,x'} \leq \sqrt{2}\,153.05 e^{-\frac{\Delta^2}{4}} \,.\] 
\end{lemma}

\begin{proof} The calculations in this proof are presented in \citep[Sections VI.1 and VI.2]{notebook}. Let $x,x' \in \R \times [u_{\min},+\infty)$ such that $\mathpzc{d}(x,x')\geq \Delta$. The following bounds on the operator norms are based on the simplified expressions from Lemma \ref{lemma:simplified_expressions_operator_norms_kernel}. We will use the inequalities
\begin{equation}\label{eq:large_sep_control_u}
 \frac{\sqrt{2u^2+\tau^2}\sqrt{2u'^2+\tau^2}}{u'^2+u^2+\tau^2} \vee \frac{|u'^2-u^2|}{u'^2+u^2+\tau^2} \leq 1\,, \quad \frac{2u^2+\tau^2}{u'^2+u^2+\tau^2}\leq 2\,, \quad
 \frac{2u^2+\tau^2}{u^2}\leq 3\,,
\end{equation}
which holds since $\tau\leq u_{\min}$. As it holds $y^q e^{-\frac{y^2}{4}} \leq \left( \frac{2q}{e}\right)^{q/2}$ for all $y \geq 0$, $q\geq 1$, we also have that for all $x,x' \in \R \times [u_{\min},+\infty)$,
\begin{align}
\forall \, q\geq 1\,, \quad \frac{|t-t'|^q}{(u'^2+u^2+\tau^2)^{\frac{q}{2}}}\sqrt{K_{\rm norm}(x,x')}&\leq \left( \frac{2q}{e}\right)^{q/2} \frac{(2u^2+\tau^2)^{1/8}(2u'^2+\tau^2)^{1/8}}{(u'^2+u^2+\tau^2)^{\frac{1}{4}}} \,, \notag\\
&\leq \left( \frac{2q}{e}\right)^{q/2}\,. \label{eq:basic_ineq_control_kernel_d_large}
\end{align}
We will also use that $\sqrt{K_{\rm norm}(x,x')}\leq e^{-\Delta^2/4}$ for $\mathpzc{d}(x,x')\geq \Delta$ (see \eqref{eq:link_semi_dist_kernel}). 

In what follows, we bound the $2$-norm of a matrix $M=(m_{ij})_{ij}$ by its Frobenius norm $\sqrt{\sum_{i,j}m_{ij}^2}$. 
\\\underline{$\norm{K_{\rm norm}^{(00)}(x,x')}$:} We have $|K_{\rm norm}(x,x')|\leq e^{-\Delta^2/2}$, which is an immediate consequence of \eqref{eq:link_semi_dist_kernel}.
\\\underline{$\norm{K_{\rm norm}^{(10)}(x,x')}_{x}$:} First, 
\begin{align*}
 \begin{pmatrix}
 H^{10}_t(x,x') \\ H^{10}_u(x,x')
 \end{pmatrix}=H^{10}(x,x')&\coloneq \sqrt{K_{\rm norm}(x,x')}^{-1}\mathfrak{g}_x^{-1/2}\nabla_1 K_{\rm norm}(x,x')\,,\\
 &=\sqrt{K_{\rm norm}(x,x')}\begin{pmatrix}
 \frac{(t-t') \sqrt{2u^2+\tau^2}}{u^2+u'^2+\tau^2}
 \\ \frac{(t-t')^2 (2u^2+\tau^2)}{\sqrt{2}(u^2+u'^2+\tau^2)^2}+ \frac{u'^2-u^2}{\sqrt{2}(u^2+u'^2+\tau^2)}
\end{pmatrix}\,. 
\end{align*}
Using \eqref{eq:basic_ineq_control_kernel_d_large} and \eqref{eq:large_sep_control_u} along with $\sqrt{K_{\rm norm}(x,x')}\leq 1$, it comes
\begin{align}\label{eq:control_kernel_d_large_H_10}
\begin{split}
 |H_{t}^{10}(x,x')| &\leq \frac{2}{\sqrt{e}} \\
 \text{and} \quad |H_{u}^{10}(x,x')| &\leq \left(\frac{4\sqrt{2}}{e}+\frac{1}{\sqrt{2}} \right) \,,
 \end{split}
\end{align}
hence $\norm{H^{10}(x,x')}_2 \leq \sqrt{\frac{4}{e} + \left(\frac{4\sqrt{2}}{e}+\frac{1}{\sqrt{2}} \right)^2}$ and \[\norm{K_{\rm norm}^{(10)}(x,x')}_{x}=\sqrt{K_{\rm norm}(x,x')}\norm{H^{10}(x,x')}_2\leq e^{-\Delta^2/4}\sqrt{\frac{4}{e} + \left(\frac{4\sqrt{2}}{e}+\frac{1}{\sqrt{2}} \right)^2}\,.\]
\underline{$\norm{K_{\rm norm}^{(11)}(x,x')}_{x,x'}$:} We have 
\[H^{11}(x,x')\coloneq \sqrt{K_{\rm norm}(x,x')}^{-1}\mathfrak{g}_x^{-1/2}\nabla_1\nabla_2 K_{\rm norm}(x,x')\mathfrak{g}_{x'}^{-1/2}=\sqrt{K_{\rm norm}(x,x')}\begin{pmatrix}
 \tilde H_{tt'}^{11}(x,x') & \tilde H_{tu'}^{11}(x,x') \\
 \tilde H_{tu'}^{11}(x,x') & \tilde H_{uu'}^{11}(x,x')
\end{pmatrix}\]
where 
\begin{align*}
 \tilde H_{tt'}^{11}(x,x')&=\frac{\sqrt{2u^2+\tau^2}\sqrt{2u'^2+\tau^2}}{u^2+u'^2+\tau^2}- \frac{(t-t')^2 \sqrt{2u^2+\tau^2}\sqrt{2u'^2+\tau^2}}{(u^2+u'^2+\tau^2)^2} \,, \\
 \tilde H_{tu'}^{11}(x,x')&=\frac{-(t-t')^3 u'(2u^2+\tau^2)\sqrt{2u'^2+\tau^2}}{\sqrt{2}u(u^2+u'^2+\tau^2)^3}+\frac{\sqrt{2}(t-t') u'(2u^2+\tau^2)\sqrt{2u'^2+\tau^2}}{u(u^2+u'^2+\tau^2)^2} \\
 & \quad +\frac{(t-t') u'(u'^2-u^2)(2u^2+\tau^2)}{\sqrt{2}u\sqrt{2u'^2+\tau^2}(u^2+u'^2+\tau^2)^2} \,,\\
 \tilde H_{tu'}^{11}(x,x')&= \tilde H_{tu'}^{11}(x',x) \,, \\
 \tilde H_{uu'}^{11}(x,x')&= \frac{(t-t')^4 (2u^2+\tau^2)(2u'^2+\tau^2)}{2(u^2+u'^2+\tau^2)^4}+\frac{-2(t-t')^2(2u^2+\tau^2)(2u'^2+\tau^2)}{(u^2+u'^2+\tau^2)^3}+\frac{(t-t')^2(u'^2-u^2)^2}{(u^2+u'^2+\tau^2)^3} \\ 
 & \quad + \frac{(2u^2+\tau^2)(2u'^2+\tau^2)}{(u^2+u'^2+\tau^2)^2} - \frac{(u^2-u'^2)^2}{2(u^2+u'^2+\tau^2)^2}\,.
\end{align*}
The constraint $\tau \leq u_{\min}$ is used here, to control $\frac{\sqrt{2u^2 + \tau^2}}{u}$ in $\tilde H_{tu'}^{11}, \tilde H_{tu'}^{11}$.
Using again \eqref{eq:basic_ineq_control_kernel_d_large}, \eqref{eq:large_sep_control_u} and the normalization $\sqrt{K_{\rm norm}(x,x')}\leq 1$, it comes 
\begin{align}\label{eq:control_kernel_d_large_H_11}
\begin{split}
|H_{tt'}^{11}(x,x')|&\leq \left(1+\frac{4}{e}\right) \,,\\
|H_{ut'}^{11}(x,x')| \vee |H_{tu'}^{11}(x,x')|&\leq \left(\frac{\sqrt{3}}{\sqrt{2}}\left(\frac{6}{e} \right)^{3/2}+\frac{2\sqrt{3}}{\sqrt{e}} + \frac{\sqrt{3}}{\sqrt{e}}\right) \,,\\
|H_{uu'}^{11}(x,x')|&\leq \left(\frac{1}{2}\left(\frac{8}{e} \right)^2+\frac{8}{e}+\frac{4}{e}+1+\frac{1}{2}\right) \,.
\end{split}
\end{align}
Hence $\norm{H^{(11)}(x,x')}_2 \leq \sqrt{\left(1+\frac{4}{e} \right)^2+ 2\left(\frac{\sqrt{3}}{\sqrt{2}}\left(\frac{6}{e} \right)^{3/2} + \frac{3\sqrt{3}}{\sqrt{e}} \right)^2+\left( \frac{32}{e^2} +\frac{12}{e}+\frac{3}{2}\right)^2 }$ and 
\[\norm{K_{\rm norm}^{(11)}(x,x')}_{x,x'} \leq e^{-\Delta^2/4} \sqrt{\left(1+\frac{4}{e} \right)^2+ 2\left(\frac{\sqrt{3}}{\sqrt{2}}\left(\frac{6}{e} \right)^{3/2} + \frac{3\sqrt{3}}{\sqrt{e}} \right)^2+\left( \frac{32}{e^2} +\frac{12}{e}+\frac{3}{2}\right)^2 }\,.\]
\underline{$\norm{K_{\rm norm}^{(02)}(x,x')}_{x'}$:}
We use \eqref{eq:tilde_H_02} to get the expression of \[H^{02}(x,x')\coloneq \sqrt{K_{\rm norm}(x,x')}\tilde H^{02}(x,x')=\sqrt{K_{\rm norm}(x,x')}^{-1}\mathfrak{g}_{x'}^{-1/2} H_2^{\mathfrak{g}} K_{\rm norm}(x,x')\mathfrak{g}_{x'}^{-1/2}\,.\] 
Using the same techniques as before, we find $\norm{H^{02}(x,x')}_2\leq \sqrt{1+2 \left( 3\frac{\sqrt{2}}{\sqrt{e}} +2 \left(\frac{6}{e} \right)^{3/2}\right)^2 + \left( \frac{128}{e^2}+\frac{24}{e} +\frac{3}{2} \right)^2}$ so \[\norm{K_{\rm norm}^{(02)}(x,x')}_{x'}\leq e^{-\Delta^2/4}\sqrt{1+2 \left( 3\frac{\sqrt{2}}{\sqrt{e}} +2 \left(\frac{6}{e} \right)^{3/2}\right)^2 + \left( \frac{128}{e^2}+\frac{24}{e} +\frac{3}{2} \right)^2}\,.\]
\underline{$\norm{K_{\rm norm}^{(12)}(x,x')}_{x,x'}$:} We denote \[H^{12,1}(x,x')=\sqrt{K_{\rm norm}(x,x')}^{-1}\mathfrak{g}_{tt}^{-1/2} \mathfrak{g}_{x'}^{-1/2} \partial_t H_2^{\mathfrak{g}} K_{\rm norm}(x,x') \mathfrak{g}_{x'}^{-1/2}\] and 
\[H^{12,2}(x,x')=\sqrt{K_{\rm norm}(x,x')}^{-1}\mathfrak{g}_{uu}^{-1/2} \mathfrak{g}_{x'}^{-1/2} \partial_u H_2^{\mathfrak{g}} K_{\rm norm}(x,x') \mathfrak{g}_{x'}^{-1/2}\,.\] 
We find 
$\tilde H^{12,1}=\sqrt{K_{\rm norm}(x,x')}\begin{pmatrix}
\tilde H_{t't'}^{12,1}(x,x') &\tilde H_{t'u'}^{12,1}(x,x') \\
\tilde H_{t'u'}^{12,1}(x,x')&\tilde H_{u'u'}^{12,1}(x,x')
\end{pmatrix}$
with 
\begin{align*}
 \tilde H_{t't'}^{12,1}(x,x')&=\frac{(t-t')\sqrt{2u^2+\tau^2}}{u^2+u'^2+\tau^2}\,, \\
 \tilde H_{t'u'}^{12,1}(x,x') &=\frac{3 (t-t')^2 \sqrt{2u^2+\tau^2}(2u'^2+\tau^2)^{3/2}}{\sqrt{2}(u^2+u'^2+\tau^2)^3} -\frac{ (t-t')^4 \sqrt{2u^2+\tau^2}(2u'^2+\tau^2)^{3/2}}{\sqrt{2}(u^2+u'^2+\tau^2)^4}
 \\& \quad +\frac{3 (t-t')^2 \sqrt{2u^2+\tau^2}\sqrt{2u'^2+\tau^2}(u'^2-u^2)}{\sqrt{2}(u^2+u'^2+\tau^2)^3} + \frac{ 3\sqrt{2u^2+\tau^2}\sqrt{2u'^2+\tau^2}(u^2-u'^2)}{\sqrt{2}(u^2+u'^2+\tau^2)^2}
 \end{align*} and
\begin{align*} 
 \tilde H_{u'u'}^{12,1}(x,x')&= \frac{-(t-t')^5(2u'^2+\tau^2)^2\sqrt{2u^2+\tau^2}}{2(u^2+u'^2+\tau^2)^5}+\frac{2(t-t')^3(2u'^2+\tau^2)^2\sqrt{2u^2+\tau^2}}{(u^2+u'^2+\tau^2)^4}\\
 &\quad +\frac{3(t-t')^3 (u'^2-u^2)(2u'^2+\tau^2)\sqrt{2u^2+\tau^2}}{(u^2+u'^2+\tau^2)^4} - \frac{(t-t')\sqrt{2u^2+\tau^2}(u^2-u'^2)^2}{2(u^2+u'^2+\tau^2)^3}\\
&\quad + \frac{(t-t')(2u'^2+\tau^2)(2u^2+\tau^2)^{3/2}}{(u^2+u'^2+\tau^2)^3} +\frac{6(t-t')(2u'^2+\tau^2)(u^2-u'^2)\sqrt{2u^2+\tau^2}}{(u^2+u'^2+\tau^2)^3} \,.
\end{align*}
Using the same ideas as before to bound the coefficients, we get \[\norm{H^{12,1}(x,x')}_2 \leq \sqrt{\frac{4}{e} + 2\left( \frac{18\sqrt{2}}{e}+\sqrt{2}\frac{64}{e^2}+ \frac{3}{\sqrt{2}}\right)^2+ \left(\sqrt{2}\left(\frac{10}{e} \right)^{5/2} + 7 \sqrt{2}\left(\frac{6}{e} \right)^{3/2} +\frac{15}{\sqrt{e}}\right)^2 }\,.\]

We also have $\tilde H^{12,2}=\sqrt{K_{\rm norm}(x,x')}\begin{pmatrix}
\tilde H_{t't'}^{12,1}(x,x') &\tilde H_{t'u'}^{12,1}(x,x') \\
\tilde H_{t'u'}^{12,1}(x,x')&\tilde H_{u'u'}^{12,1}(x,x')
\end{pmatrix}$
with 
\begin{align*}
 \tilde H_{t't'}^{12,1}(x,x')&=\frac{-(t-t')^2(2u^2+\tau^2)}{\sqrt{2}(u^2+u'^2+\tau^2)^2}+\frac{u^2-u'^2}{\sqrt{2}(u^2+u'^2+\tau^2)}\,, \\ 
 \tilde H_{t'u'}^{12,1}(x,x')&=\frac{(t-t')^5(2u'^2+\tau^2)^{3/2}(2u^2+\tau^2)}{2(u^2+u'^2+\tau^2)^5} +\frac{(t-t')^3(2u'^2+\tau^2)^{1/2}(2u^2+\tau^2)(u^2-u'^2)}{(u^2+u'^2+\tau^2)^4} \\
 &\quad -\frac{4(t-t')^3(2u'^2+\tau^2)^{3/2}(2u^2+\tau^2)}{(u^2+u'^2+\tau^2)^4}+\frac{(t-t')^3(2u'^2+\tau^2)^{1/2}}{(u^2+u'^2+\tau^2)^2} \\
& \quad +\frac{3(t-t')(2u'^2+\tau^2)^{1/2}(2u^2+\tau^2)^2}{(u^2+u'^2+\tau^2)^3}+\frac{21(t-t')(2u'^2+\tau^2)^{1/2}(2u^2+\tau^2)(u'^2-u^2)}{2(u^2+u'^2+\tau^2)^3} \\ &\quad +\frac{3(t-t')(2u'^2+\tau^2)^{1/2}(u^2-u'^2)}{2(u^2+u'^2+\tau^2)^2}
\end{align*}
and
\begin{align*}
\tilde H_{u'u'}^{12,1}(x,x')= &\frac{\sqrt{2}(t-t')^6(2u'^2+\tau^2)^{2}(2u^2+\tau^2)}{4(u^2+u'^2+\tau^2)^6} +\frac{5\sqrt{2}(t-t')^4(2u'^2+\tau^2)(2u^2+\tau^2)(u^2-u'^2)}{4(u^2+u'^2+\tau^2)^5} \\ & \quad -\frac{2\sqrt{2}(t-t')^4(2u'^2+\tau^2)^2(2u^2+\tau^2)}{(u^2+u'^2+\tau^2)^5}+\frac{\sqrt{2}(t-t')^4(2u'^2+\tau^2)(u^2-u'^2)^2}{2(u^2+u'^2+\tau^2)^5}\\ \quad
&-\frac{5\sqrt{2}(t-t')^2(2u'^2+\tau^2)(u^2-u'^2)^2}{4(u^2+u'^2+\tau^2)^4}+\frac{5\sqrt{2}(t-t')^2(2u^2+\tau^2)(2u'^2+\tau^2)^2}{2(u^2+u'^2+\tau^2)^4} \\ & \quad +\frac{\sqrt{2}(t-t')^2(u^2-u'^2)^3}{2(u^2+u'^2+\tau^2)^4}-\frac{7\sqrt{2}(t-t')^2(2u^2+\tau^2)(2u'^2+\tau^2)(u^2-u'^2)}{(u^2+u'^2+\tau^2)^4} \\
& \quad +\frac{7\sqrt{2}(2u'^2+\tau^2)(2u^2+\tau^2)(u^2-u'^2)}{2(u^2+u'^2+\tau^2)^3}-\frac{\sqrt{2}(u^2-u'^2)^3}{4(u^2+u'^2+\tau^2)^3} \,.
\end{align*}
We get 
\[ \norm{H^{12,2}(x,x')}_2 \leq \sqrt{ \begin{aligned}
 & \left( \frac{4\sqrt{2}}{e}+1\right)^2+ 2\left(\frac{1}{\sqrt{2}}\left(\frac{10}{e} \right)^{5/2}+ 6\sqrt{2}\left(\frac{6}{e} \right)^{3/2}+\frac{36}{\sqrt{e}} \right)^2 \\
 & \: \: + \left(\frac{432 \sqrt{2}}{e^3} + \frac{272\sqrt{2}}{e^2} + \frac{60\sqrt{2}}{e}+ \frac{15\sqrt{2}}{4}\right)^2
\end{aligned}}\leq 153.05 \,. \]
As this bound is greater than the one found for $\norm{ H^{12,1}(x,x')}_2 $, we deduce that
\[\norm{K_{\rm norm}^{(12)}(x,x')}_{x,x'}\leq \sqrt{2}\,153.05 e^{-\Delta^2/4}\,.\]
This bound is the largest we obtained for $\norm{K_{\rm norm}^{(ij)}(x,x')}_{x,x'}$, hence it is an upper bound for all the derivatives investigated by this lemma.
\end{proof}

\begin{lemma}[Bounds under a large separation, in dimension $d\geq 1$]\label{lemma:control_kernel_d_large_dim_d}
 Let $\Delta>0$. Let $x,x' \in \R^d \times [u_{\min},+\infty)^d$. Assume that $\tau \leq u_{\min}$. If $\mathpzc{d}(x,x')\geq \Delta$, then \[\max_{(i,j)\in \{0,1\} \times \{0,1,2\}} \norm{K_{\rm norm}^{(ij)}(x,x')}_{x,x'} \leq \sqrt{2d}(170.5+ 25.78d)e^{-\Delta^2/4} \,.\] 
\end{lemma}
\begin{proof}
 Let $x,x' \in \R^d \times [u_{\min},+\infty)^d$ such that $\mathpzc{d}(x,x')\geq \Delta$. We use the decomposition of the kernel $K_{\rm norm}(x,x')=\prod_{k=1}^d K_{\rm norm}(x_k,x_k')$ to reuse the results proven in dimension 1 (see Lemma \ref{lemma:control_kernel_d_large_dim1}). As in Lemma \ref{lemma:curvature_constants_dim_d}, we use the expressions for the operator norms given in Lemma \ref{lemma:simplified_expressions_operator_norms_kernel} and we bound the $2$-norm of a matrix $M=(m_{ij})_{ij}$ by $\sqrt{\sum_{i,j}m_{ij}^2}$.
The calculations in this proof are presented in \citep[Section VI.3]{notebook}.
\\\underline{$\norm{K_{\rm norm}^{(00)}(x,x')}$:} According to \eqref{eq:link_semi_dist_kernel}, we still have $|K_{\rm norm}^{(00)}(x,x')|\leq e^{-\Delta^2/2}$.
\\\underline{$\norm{K_{\rm norm}^{(10)}(x,x')}_{x}$:} We recall that the notation $\partial_{b_k}K_{\rm norm}$ (resp.\ $\partial_{b_k'}K_{\rm norm}$) refers to a derivative \wrt to the first (resp.\ the second) variable of the kernel, where $b_k$ is $t_k$ or $u_k$. The decomposition of the kernel gives $\partial_{b_k}K_{\rm norm}(x,x') = \partial_{b_k}K_{\rm norm}(x_k,x_k')\prod_{l\neq k}^d K_{\rm norm}(x_l,x_l')$. Using \eqref{eq:basic_ineq_control_kernel_d_large} for $K_{\rm norm}(x_k,x_k')$, we get that \[\forall \, q\geq 1\,, \quad \frac{|t_k-t_k'|^q}{(u_k'^2+u_k^2+\tau^2)^{\frac{q}{2}}}\sqrt{K_{\rm norm}(x_k,x_k')}\leq \left( \frac{2q}{e}\right)^{q/2}\,.\]
We have 
\begin{align*}
|\mathfrak{g}_{b_kb_k}^{-1/2}\partial_{b_k}K_{\rm norm}(x,x')| &\leq \sqrt{ \prod_{l=1}^d K_{\rm norm}(x_l,x_l')} \left|\sqrt{K_{\rm norm}(x_k,x_k')}^{-1}\mathfrak{g}_{b_kb_k}^{-1/2}\partial_{b_k}K_{\rm norm}(x_k,x_k') \right| \,.
\end{align*}
The term in the right-hand side has already been dealt with in Lemma \ref{lemma:control_kernel_d_large_dim1}: \eqref{eq:control_kernel_d_large_H_10} gives a bound for \[\left|\sqrt{K_{\rm norm}(x_k,x_k')}^{-1}\mathfrak{g}_{b_kb_k}^{-1/2}\partial_{b_k}K_{\rm norm}(x_k,x_k') \right|=|H_{b_k}^{10}(x_k,x_k')|\,,\] that we denote by $B_{H_{b_k}^{10}}$ (note that this bound does not depend on $\Delta$). So \[|\mathfrak{g}_{b_kb_k}^{-1/2}\partial_{b_k}K_{\rm norm}(x,x')| \leq e^{-\Delta^2/4}B_{H_{b_k}^{10}}\,.\]
Hence $\norm{K_{\rm norm}^{(10)}(x,x')}_{x}\leq \sqrt{d} \sqrt{ B_{H_{t_k}^{10}}^2+ B_{H_{u_k}^{10}}^2} e^{-\Delta^2/4} \leq 3.05\sqrt{d}e^{-\Delta^2/4}$.
\\\underline{$\norm{K_{\rm norm}^{(11)}(x,x')}_{x,x'}$:} We use in the same way as for $ B_{H_{b_k}^{(10)}}$ the notation $B_{H_{b_k,b_k'}^{(11)}}$, denoting the bound obtained in Lemma \ref{lemma:control_kernel_d_large_dim1} for $\mathfrak{g}_{x_k}^{-1/2}\partial_{b_k,b_k'} K_{\rm norm}(x_k,x_k')\mathfrak{g}_{x_k'}^{-1/2}\sqrt{K_{\rm norm}(x_k,x_k')}^{-1}=| H_{b_k b_k'}^{11}(x_k,x_k')| $ (see \eqref{eq:control_kernel_d_large_H_11}). 
For $k\neq l$, 
\begin{align*}
|\mathfrak{g}_{b_kb_k}^{-1/2}\partial_{b_kb_l'}K_{\rm norm}(x_k,x_k')\mathfrak{g}_{b_l'b_l'}^{-1/2} | & \leq \sqrt{ \prod_{m=1}^d K_{\rm norm}(x_m,x_m')} \left|\sqrt{K_{\rm norm}(x_k,x_k')}^{-1}\mathfrak{g}_{b_kb_k}^{-1/2}\partial_{b_k}K_{\rm norm}(x_k,x_k') \right| \\
& \quad \times \left|\sqrt{K_{\rm norm}(x_l,x_l')}^{-1}\mathfrak{g}_{b_l'b_l'}^{-1/2}\partial_{b_l'}K_{\rm norm}(x_l,x_l') \right| \,, \\
& \leq e^{-\Delta^2/4} B_{H_{b_k}^{10}} B_{H_{b_l'}^{01}}=e^{-\Delta^2/4} B_{H_{b_k}^{10}} B_{H_{b_l}^{10}} \,.
\end{align*}
For $k=l$, we have in the same way 
\begin{align*}
|\mathfrak{g}_{t_kt_k}^{-1/2}\partial_{t_kt_k'}K_{\rm norm}(x_k,x_k')\mathfrak{g}_{t_k't_k'}^{-1/2} |&\leq e^{-\Delta^2/4} B_{H_{t_k,t_k'}^{11}} \,,\\ 
|\mathfrak{g}_{u_ku_k}^{-1/2}\partial_{u_kt_k'}K_{\rm norm}(x_k,x_k')\mathfrak{g}_{t_k't_k'}^{-1/2} |&\leq e^{-\Delta^2/4} B_{H_{u_k,t_k'}^{11}}\,,\\ 
|\mathfrak{g}_{u_ku_k}^{-1/2}\partial_{u_ku_k'}K_{\rm norm}(x_k,x_k')\mathfrak{g}_{u_k'u_k'}^{-1/2} |&\leq e^{-\Delta^2/4} B_{H_{u_k,u_k'}^{11}} \,. 
\end{align*}
So 
\begin{align*}
 \norm{K_{\rm norm}^{(11)}(x,x')}_{x,x'} &\leq e^{-\Delta^2/4} \sqrt{
 \begin{aligned}
 &d(d-1) B_{H_{t_k}^{10}}^4+ d(d-1) B_{H_{u_k}^{10}}^4 + 2d(d-1) B_{H_{t_k}^{10}}^2 B_{H_{u_k}^{10}}^2 \\& \: \: +2d B_{H_{t_k,u_k'}^{11}}^2+ d B_{H_{t_k,t_k'}^{11}}^2+ d B_{H_{u_k,u_k'}^{11}}^2 
 \end{aligned}
 } \,,\\
 &\leq(9.25d+6.95)e^{-\Delta^2/4}\,.
\end{align*}
\underline{$\norm{K_{\rm norm}^{(02)}(x,x')}_{x'}$:}
With the same arguments, the terms of $\mathfrak{g}_{x'}^{-1/2}H_2^{\mathfrak{g}} K_{\rm norm}(x,x')\mathfrak{g}_{x'}^{-1/2}$ corresponding to derivatives taken in $b_k',b_l'$ with $k\neq l$ can be bounded by $e^{-\Delta^2/4} B_{H_{b_k}^{10}} B_{H_{b_l'}^{10}}$.
We bound the terms corresponding to $k=l$ with $e^{-\Delta^2/4} B_{H_{t_k't_k'}^{02}}$ or $e^{-\Delta^2/4} B_{H_{t_k'u_k'}^{02}}$ or $e^{-\Delta^2/4} B_{H_{u_k'u_k'}^{02}}$, where $ B_{H_{b_k'b_k'}^{02}}$ denotes again the bound obtained in Lemma \ref{lemma:control_kernel_d_large_dim1}.
We get 
\begin{align*}
 \norm{K_{\rm norm}^{(02)}(x,x')}_{x'}& \leq e^{-\Delta^2/4} \sqrt{
 \begin{aligned}
 &d(d-1) B_{H_{t_k}^{10}}^4+d(d-1) B_{H_{u_k}^{10}}^4+2d(d-1) B_{H_{t_k}^{10}}^2B_{H_{u_k}^{10}}^2 \\
 & \:\: + 2dB_{H_{t_k'u_k'}^{02}}^2+d B_{H_{t_k't_k'}^{02}}^2+d B_{H_{u_k'u_k'}^{02}}^2 
 \end{aligned}} \,, \\
 &\leq (9.25d+45.9)e^{-\Delta^2/4} \,.
\end{align*}
\underline{$\norm{K_{\rm norm}^{(12)}(x,x')}_{x,x'}$:} We denote $M_{b_k}=\mathfrak{g}_{b_kb_k}^{-1/2} \mathfrak{g}_{x'}^{-1/2} \partial_{b_k} H_2^{\mathfrak{g}} K_{\rm norm}(x,x') \mathfrak{g}_{x'}^{-1/2}$. The terms at positions $b_l',b_m'$ where $l\neq m$ and $l,m \neq k$ can be bounded by $e^{-\Delta^2/4} B_{H_{b_k}^{10}} B_{H_{b_l}^{10}} B_{H_{b_m}^{10}}$.
The terms at positions $b_l',b_k'$ or $b_k',b_l'$ where $l\neq k$ can be bounded by $e^{-\Delta^2/4} B_{H_{b_k,t_k'}^{11}} B_{H_{b_l}^{10}}$ or $e^{-\Delta^2/4} B_{H_{b_k,u_k'}^{11}} B_{H_{b_l}^{10}}$. 
The terms at positions $b_l',b_m'$ where $l=m \neq k$ can be bounded by $e^{-\Delta^2/4} B_{H_{b_k}^{10}} B_{H_{b_l'b_m'}^{02}}$. 
The terms at positions $b_l',b_m'$ where $l=m=k$ can be bounded by $e^{-\Delta^2/4} B_{H_{t_k',t_k'}^{12,b_k}}$ or $e^{-\Delta^2/4} B_{H_{t_k',u_k'}^{12,b_k}}$ or $e^{-\Delta^2/4} B_{H_{u_k',u_k'}^{12,b_k}}$.
Writing $M_{b_k}=(m_{b_l'b_m'})_{b_l',b_m'\in \{t_1',\ldots,t_d',u_1',\ldots u_d'\}}$, we use the following bound for its $2$-norm: $\norm{M_{b_k}}_2 \leq \max_l \sqrt{m_{t_l't_l'}^2+ m_{t_l'u_l'}^2 +m_{u_l't_l'}^2+ m_{u_l'u_l'}^2 }+ \sqrt{\sum_{l\neq m} m_{b_k' b_l'}^2}$.
It comes that
\begin{align*}
 \norm{M_{t_k}}_2 & \leq e^{-\Delta^2/4} 
 \left(  \max\left\{\sqrt{B_{H_{t_k}^{10}}^2B_{H_{t_k't_k'}^{02}}^2 +2B_{H_{t_k}^{10}}^2B_{H_{t_k'u_k'}^{02}}^2 + B_{H_{t_k}^{10}}^2B_{H_{u_k'u_k'}^{02}}^2} \,,\, \sqrt{2B_{H_{t_k',u_k'}^{12,1}}^2+ B_{H_{t_k',t_k'}^{12,1}}^2+ B_{H_{u_k',u_k'}^{12,1}}^2}\right\} \right.
 \\& \qquad\qquad\qquad\qquad\left. + \sqrt{
 \begin{aligned}
 &(d^2-3d+2)B_{H_{t_k}^{10}}^2\left(B_{H_{t_k}^{10}}^4+2B_{H_{t_k}^{10}}^2B_{H_{u_k}^{10}}^2+ B_{H_{u_k}^{10}}^4 \right) \\
 & \: \: +2(d-1)\left( B_{H_{t_kt_k'}^{11}}^2B_{H_{t_k}^{10}}^2 + B_{H_{t_kt_k'}^{11}}^2B_{H_{u_k}^{10}}^2 + B_{H_{t_ku_k'}^{11}}^2B_{H_{t_k}^{10}}^2+ B_{H_{t_ku_k'}^{11}}^2B_{H_{u_k}^{10}}^2 \right) 
 \end{aligned}
 } \qquad
 \right)
\end{align*}
and 
\begin{align*}
 \norm{M_{u_k}}_2 & \leq e^{-\Delta^2/4} 
 \left( \max\left\{\sqrt{B_{H_{u_k}^{10}}^2B_{H_{t_k't_k'}^{02}}^2 +2B_{H_{u_k}^{10}}^2B_{H_{t_k'u_k'}^{02}}^2 + B_{H_{u_k}^{10}}^2B_{H_{u_k'u_k'}^{02}}^2} \,,\, \sqrt{2B_{H_{t_k',u_k'}^{12,2}}^2+ B_{H_{t_k',t_k'}^{12,2}}^2+ B_{H_{u_k',u_k'}^{12,2}}^2}\right\} \right.
 \\& \qquad\qquad\qquad\qquad + \left. \sqrt{
 \begin{aligned}
 &(d^2-3d+2)B_{H_{u_k}^{10}}^2\left(B_{H_{t_k}^{10}}^4+2B_{H_{t_k}^{10}}^2B_{H_{u_k}^{10}}^2+ B_{H_{u_k}^{10}}^4 \right) \\
 & \: \: +2(d-1)\left( B_{H_{u_kt_k'}^{11}}^2B_{H_{t_k}^{10}}^2 + B_{H_{u_kt_k'}^{11}}^2B_{H_{u_k}^{10}}^2 + B_{H_{u_ku_k'}^{11}}^2B_{H_{t_k}^{10}}^2+ B_{H_{u_ku_k'}^{11}}^2B_{H_{u_k}^{10}}^2\right) 
 \end{aligned}
 } \qquad
 \right)\,,\\
 &\leq e^{-\Delta^2/4}(170.5+ 25.78d)\,.
\end{align*}
As this last bound is greater than the previous one,
$\norm{K_{\rm norm}^{(12)}(x,x')}_{x,x'} \leq \sqrt{2d}(170.5+ 25.78d)e^{-\Delta^2/4}$.
\end{proof}

\paragraph{Choice of $r,\Delta$ for the LPC}
\begin{proposition}[$K_{\rm norm}$ satisfies the LPC, $d=1$]\label{prop:proof_K_norm_verify_positive_curvature_dim_1}
Let $s \geq 2$. Assume that $\X \subset \R \times [u_{\min},u_{\max}]$ and that $\tau \leq u_{\min}$.
Then $K_{\rm norm}$ satisfies the LPC (see Definition \ref{def:positive_curvature}) with parameters $s$, $\Delta(s)=2\sqrt{13.88 + \ln(s-1)}$, $r=0.3025$, $\bar\varepsilon_2(0.3025)=0.13139$, $\bar\varepsilon_0(0.3025)=0.04472$.
\end{proposition}
\begin{proof}
 To establish that $K_{\rm norm}$ satisfies the LPC, we first determine the size of the near regions $r$, giving the constraint on the minimal separation $\Delta$ (see Definition \ref{def:positive_curvature}).
 
 We want to pick $r$ such that $\bar \varepsilon_0(r),\bar \varepsilon_2(r)$ exist and $\frac{1}{64}\min \left\{\frac{\bar \varepsilon_0(r)}{B_0} \, ,\, \frac{\bar \varepsilon_2(r)}{B_2}\right\}$ is maximal, using Lemmas \ref{lemma:curvature_constants_dim_1} and \ref{lemma:global_controls_kernel}. This allows us to get the smallest $\Delta$ (see item 3 of Definition \ref{def:positive_curvature}). Graphically, we choose $r=0.3025$ (see \citep[Section VII.1]{notebook}). We can take $\bar\varepsilon_2(0.3025)=0.13139$ and $\bar \varepsilon_0(0.3025)=0.04472$.
 Then \[\frac{1}{64}\min\left(\frac{0.04472}{B_0},\frac{0.13139}{B_2}\right) \geq 0.000204618 \eqcolon c_r\] (see \citep[Section VII.1]{notebook}).
 
 Let $s\geq 2$. The minimal separation $\Delta(s)$ must satisfy item 3 of Definition \ref{def:positive_curvature}. Using Lemma \ref{lemma:control_kernel_d_large_dim1}\,, it suffices that \[(s-1) \sqrt{2}\,153.05 e^{-\frac{\Delta(s)^2}{4}} \leq c_r\,,\] \ie that $\Delta(s)\geq 2\sqrt{c_\Delta + \ln(s-1)}$ where $c_\Delta=13.88 \geq \ln \left( \frac{\sqrt{2}153.05}{c_r}\right)$ (see \citep[Section VII.1]{notebook}).
\end{proof}

\begin{proposition}[$K_{\rm norm}$ satisfies the LPC, $d\geq 1$] Let $s \geq 2$. Assume that $\X \subset \R^d \times [u_{\min},u_{\max}]^d$ and that $\tau \leq u_{\min}$.
Then $K_{\rm norm}$ satisfies the LPC with parameters $s$, $r=\frac{0.3025}{\sqrt{d}}$, $\bar\varepsilon_2(r)=0.13139$, $\bar\varepsilon_0(r)=\frac{0.0894}{2d}$, $\Delta(s)= 2\sqrt{ 11.9+3\ln(d+6.62)+\ln(s-1)}$.
\end{proposition}
\begin{proof}
 We take Proposition \ref{prop:proof_K_norm_verify_positive_curvature_dim_1} as a starting point. We make use of Lemmas \ref{lemma:global_controls_kernel}, \ref{lemma:curvature_constants_dim_d} and \ref{lemma:control_kernel_d_large_dim_d}. 
 Setting $r=\frac{0.3025}{\sqrt{d}}$, we can take $\bar\varepsilon_2(r)=0.13139$ (as in dimension $d=1$) because $d \in \mathbb{N}^* \mapsto e^{-\frac{r_0^2}{2d}}|G(r_0)|$ with $r_0=0.3025$ is non-decreasing.
 
 Furthermore, for $r\leq 0.3025$ we have $1-e^{-r^2/2} \geq 0.977 \frac{r^2}{2}$. In fact, for $0<c<1$, denoting \[\square_c: r\in \R^+ \mapsto 1-e^{-r^2/2}- c \frac{r^2}{2}\,,\] we have $\frac{\partial}{\partial r}\square_c(r)=r(e^{-r^2/2}-c)$ so $\square_c$ is increasing then decreasing. Then remark that $\square_c(0)=0$ and $\square_{0.977}(0.3025)\geq 0$. 
 So we can take $\bar\varepsilon_0(r)= \frac{0.0894}{2d} \leq 0.977 \frac{0.3025^2}{2d}$ (see \citep[Section VII.2]{notebook}).

 Moreover, $\frac{\bar \varepsilon_0(r)}{1+B_{00}+B_{10}} \leq \frac{\bar \varepsilon_2(r)}{1+B_{02}+B_{12}}$ (see \citep[Section VII.2]{notebook}). 
 So \[\frac{1}{64}\min\left(\frac{\bar \varepsilon_0(r)}{1+B_{00}+B_{10}},\frac{\bar \varepsilon_2(r)}{1+B_{02}+B_{12}}\right) \geq \frac{1}{64}\frac{\varepsilon_0(r)}{1+B_{00}+B_{10}} \geq \frac{1}{64} \frac{0.0894}{2d(2+\sqrt{2d})} \eqcolon c_{d,r}\,.\]
 
 The minimal separation should verify $(s-1) \sqrt{2d}(170.5+ 25.78d)e^{-\Delta(s)^2/4} \leq c_{d,r}$, \ie \[\Delta(s)\geq 2\sqrt{\ln\left( \frac{64}{0.0894 }\right)+ \ln\left((170.5+ 25.78d)2d(2+\sqrt{2d})\sqrt{2d}\right)+\ln(s-1)}\,.\] 
 It suffices that $\Delta(s)\geq 2\sqrt{ 11.9+3\ln(d+6.62)+\ln(s-1)}$ (see \citep[Section VII.2]{notebook}).
\end{proof}

\section{Proofs of Section \ref{section:NDSC}}

\subsection{Non degenerate source condition}\label{section:proof_lemma_eta_ndsc}
Theorem \ref{th:main_theorem_certif} gives the existence of a non-degenerate global certificate $\eta$ under a separation condition. We show that it satisfies the properties given by \eqref{eq:NDSC}. 

\begin{lemma}\label{lemma:eta_ndsc}
    Under the assumptions of Theorem \ref{th:main_theorem_certif}, the global non-degenerate certificate $\eta$ verifies
    \begin{enumerate}
    \item  $\eta(x) \leq 1$ for all $x\in \X$,
    \item  For $x\in\X$, $\eta(x)=1 \iff x=x_j^0$ for some $j=1,\ldots,s$,
        \item $\eta(x) \leq 1- \varepsilon_0$ for all $x \in \X^{\rm far}(r)$,
        \item $v^T H^{\mathfrak{g}}\eta(x) v \leq -2 \varepsilon_2 \norm{v}_x^2$ for all $v\in \R^{2d}$, $x \in \mathcal{G}_{x_j^0}(\X_j^{\rm near}(r))$, $j=1,\ldots,s$. 
    \end{enumerate}
    where we recall that $r=\frac{0.3025}{\sqrt{d}}$, $\varepsilon_0=\frac{0.03911}{d}$, $\varepsilon_2=0.06158$.
\end{lemma}
\begin{proof}
Items 1 to 3 directly arise from Definition \ref{def:non_degenerate_certificate}.
We show item 4. We recall the control on the near regions obtained in the proof of Theorem \ref{th:proof_non_degenerescence_certif}. For $v\in \R^{2d}$, we deduce that 
\begin{align*}
   v^T  H^{\mathfrak{g}}\eta(x) v &\leq K_{\rm norm}^{(02)}(x_j^0, x)[v,v]+ \norm{D_2[\eta](x)- K_{\rm norm}^{(02)}(x_j^0, x)}_{x} \norm{v}_x^2 \,,\\
   &\leq \left(-\bar \varepsilon_2+ \frac{\bar \varepsilon_2}{16}\right)\norm{v}_x^2 \,,\\
   &= -\frac{15\bar \varepsilon_2}{16} \norm{v}_x^2\,. 
\end{align*}
We conclude recalling that we defined $\varepsilon_2 =\frac{15\bar \varepsilon_2}{32}$.
\end{proof}

As a side note, remark that the previous lemma also implies that $\nabla^2 \eta(x_j^0)$ is negative-definite, as $\nabla^2 \eta(x_j^0)=H^{\mathfrak{g}}\eta (x_j^0)$ because $\nabla \eta(x_j^0)=0$ (equality of the Riemannian and Euclidean hessians, see Definition \ref{def:riemannian_derivatives}).

\subsection{Proof of Theorem \ref{th:NDSC}}\label{section:proof_result_ndsc}
We begin with a preliminary result.
\begin{lemma}\label{lemma:check_hyp_ndsc}
Under the hypothesis of Theorem \ref{th:main_theorem_certif},
 $G_{S_0}$ is full-rank, where 
 \begin{equation}\label{eq:def_G_S_0}
 G_{S_0}:(\alpha_1,\ldots,\alpha_s),(\beta_1,\ldots,\beta_s) \in \R^s \times \R^{2d \times s} \mapsto \sum_{j=1}^{s}\alpha_j \Psi\delta_{x_j^0} + \sum_{j=1}^{s} \beta_j^T \nabla_x (\Psi\delta_{x_j^0}) \, . 
 \end{equation}
Moreover, $\max_{i=0,1,2} \sup_{x\in \R^d \times [u_{\min}, +\infty)^d} \norm{D_i \Psi(\delta_{x}) }_{x} \leq 2\sqrt{7}d \eqcolon B$ ($D_i$ and the associated norm defined in Definition \ref{def:riemannian_derivatives}).
\end{lemma}
\begin{proof}
\underline{$G_{S_0}$ is full-rank:} We used the invertibility of $\Upsilon$ (see \eqref{eq:Upsilon}) to construct $\eta$: we already know that \[\sum_{j=1}^s \alpha_j K_{\rm norm}(x_j^0,\dotp)+ \sum_{j=1}^s \beta_j^T \nabla_1K_{\rm norm}(x_j^0,\dotp)=0 \quad \implies \quad \alpha_j=0\,, \; \beta_j=0_{\R^2} \; \forall j=1,\ldots s \,. \]
\underline{Bound on $D_i \Psi$:}
Remark that $\sup_{x\in \R^d \times [u_{\min}, +\infty)^d} \norm{D_i \Psi(\delta_{x}) }_{x} \leq \sup_{x\in \R^d \times [u_{\min}, +\infty)^d} \sqrt{\norm{K_{\rm norm}^{(ii)}(x,x)}_{x,x}}$. To deal with $i=0,1$, remark that $\sqrt{\norm{K_{\rm norm}^{(ii)}(x,x)}_{x,x}} \leq \sqrt{B_{ii}}$ with $B_{00}, B_{11}$ defined in Lemma \ref{lemma:global_controls_kernel}. 

We follow the same approach as in the proof of Lemma \ref{lemma:global_controls_kernel} to obtain an upper bound $B_{22}$ on $\norm{K_{\rm norm}^{(22)}(x,x')}_{x,x'}$: with $(\mathfrak{g}_x^{-1/2}H^{\mathfrak{g}}\Psi \delta_x \mathfrak{g}_x^{-1/2})_{k,l}$ representing the coordinate at position $(k,l)$, we have
\begin{align*}
   \norm{K_{\rm norm}^{(22)}(x,x')}_{x,x'} &= \sup_{\substack{\norm{\tilde V}_2 \leq 1, \norm{\tilde Q}_2 \leq 1 \\\norm{\tilde V'}_2 \leq 1, \norm{\tilde Q'}_2 \leq 1 }}  \innerprod{\tilde V^T \mathfrak{g}_x^{-1/2}H^{\mathfrak{g}}\Psi \delta_x \mathfrak{g}_x^{-1/2} \tilde Q}{\tilde {V'}^T \mathfrak{g}_{x'}^{-1/2}H^{\mathfrak{g}}\Psi \delta_{x'} \mathfrak{g}_{x'}^{-1/2} \tilde Q'}_\L \,, \\
   &= \sup_{\substack{\norm{\tilde V}_2 \leq 1, \norm{\tilde Q}_2 \leq 1 \\ \norm{\tilde V'}_2 \leq 1, \norm{\tilde Q'}_2 \leq 1 }}  \sum_{k,l,k',l'} V_k V_{k'}' Q_l Q_{l'}'\innerprod{( \mathfrak{g}_x^{-1/2}H^{\mathfrak{g}}\Psi \delta_x \mathfrak{g}_x^{-1/2})_{k,l}}{(\mathfrak{g}_{x'}^{-1/2}H^{\mathfrak{g}}\Psi \delta_{x'} \mathfrak{g}_{x'}^{-1/2})_{k',l'}}_\L \,,  \\
   & \leq \sqrt{ \sum_{k,l,k',l'} \innerprod{( \mathfrak{g}_x^{-1/2}H^{\mathfrak{g}}\Psi \delta_x \mathfrak{g}_x^{-1/2})_{k,l}}{(\mathfrak{g}_{x'}^{-1/2}H^{\mathfrak{g}}\Psi \delta_{x'} \mathfrak{g}_{x'}^{-1/2})_{k',l'}}_\L^2} \,, \\
   & \leq 4d^2 \sup_{k,l,k',l'} \left|\innerprod{( \mathfrak{g}_x^{-1/2}H^{\mathfrak{g}}\Psi \delta_x \mathfrak{g}_x^{-1/2})_{k,l}}{(\mathfrak{g}_{x'}^{-1/2}H^{\mathfrak{g}}\Psi \delta_{x'} \mathfrak{g}_{x'}^{-1/2})_{k',l'}}_\L\right| \,, \\
   & \leq 4d^2 \sup_{\substack{k,l \\ x \in \R^d \times [u_{\min},+\infty)^d}} \norm{( \mathfrak{g}_x^{-1/2}H^{\mathfrak{g}}\Psi \delta_x \mathfrak{g}_x^{-1/2})_{k,l}}_\L^2 \,, \\
   & \leq 4d^2 \sup_{\substack{k,l \\ x \in \R^d \times [u_{\min},+\infty)^d}} \mathfrak{g}_{b_k b_k}^{-1} \mathfrak{g}_{b_l b_l}^{-1}\norm{(H^{\mathfrak{g}}\Psi \delta_x)_{b_k b_l}}_\L^2 \,.
\end{align*}
The terms $\mathfrak{g}_{b_k b_k}^{-1} \mathfrak{g}_{b_l b_l}^{-1}\norm{(H^{\mathfrak{g}}\Psi \delta_x)_{b_k b_l}}_\L^2$ take one of the following forms ($b$ represents $t$ or $u$, while $k,l$ refer to the dimension in the following):
\begin{align*}
    &\mathfrak{g}_{b_k b_k}^{-1} \mathfrak{g}_{b_l b_l}^{-1}\norm{\partial_{b_k b_l}\Psi \delta_x}_\L^2 \quad \text{with} \; k \neq l\,, \\
    &\mathfrak{g}_{t_k t_k}^{-2} \norm{\partial_{t_k t_k}\Psi \delta_x - \Gamma^{t_k}{}_{t_k u_k} \partial_{t_k}\Psi \delta_x}_\L^2 \,, \\
    &\mathfrak{g}_{t_k t_k}^{-1}\mathfrak{g}_{u_k u_k}^{-1} \norm{\partial_{t_k u_k}\Psi \delta_x - \Gamma^{t_k}{}_{t_k u_k} \partial_{t_k}\Psi \delta_x}_\L^2\,, \\
    &\mathfrak{g}_{u_k u_k}^{-2} \norm{\partial_{u_k u_k}\Psi \delta_x - \Gamma^{u_k}{}_{u_k u_k} \partial_{t_k}\Psi \delta_x}_\L^2\,.
\end{align*}
From the proof of Lemma \ref{lemma:global_controls_kernel}, we know that these terms are respectively bounded by $1,1,3,7$. Hence $B_{22} \leq 28 d^2$, and $\max_{i=0,1,2} \sup_{x\in \R^d \times [u_{\min}, +\infty)^d} \norm{D_i \Psi(\delta_{x}) }_{x} \leq 2\sqrt{7}d$.
\end{proof}

Below we present the main steps of the proof of Theorem \ref{th:NDSC}, adapted from \citep[Lemma 2, Proposition 7, Theorem 2]{blasso_duval_peyre} and \citep[Section 4.3]{blasso_poon}.
We work under the assumptions of Theorem \ref{th:main_theorem_certif}: there exists $\eta$ of the form \eqref{eq:eta_general_form} verifying the properties given in Lemma \ref{lemma:eta_ndsc}.

\paragraph{Dual, noisy and noiseless problems} We denote $\mathcal{D}_{\kappa,b}$ the dual problem with regularization $\kappa$ and noise $b$. 
More precisely: 
\begin{equation}\label{eq:pb_D_00}
\underset{p\in\L, \Psi^*p \leq 1}{\arg\max} \,\innerprod{\Psi\mu_\omega^0}{p}_{\L} \tag{$\mathcal{D}_{0,0}$}
\end{equation} is the dual problem of \begin{equation}\label{eq:pb_P_00}
\underset{\mu\in\M(\X)^+}{\arg\min} \,\norm{\mu}_{\rm TV} \;\text{such that}\; \Psi\mu=\Psi\mu_\omega^0\,. \tag{$\mathcal{P}_{0,0}$}
\end{equation}
For $\kappa>0$, the dual problem is 
\begin{equation}
 \label{eq:pb_D_kappa_b}
 \underset{p\in\L, \Psi^*p \leq 1}{\arg\min} \,\norm{\frac{y}{\kappa}-p}_{\L}^2 \tag{$\mathcal{D}_{\kappa,b}$}\,,
\end{equation}
where we observe $y=\Psi\mu_\omega^0+b$ (\ie $b=\Gamma_n$). Its solution is unique. It is the dual problem of
\begin{equation}
 \label{eq:pb_P_kappa_b} 
 \underset{\mu \in \M(\X)^+}{\arg\min} \; \frac{1}{2}\norm{y-\Psi\mu}_{\L}^2+\kappa\norm{\mu}_{\rm TV}\,.
 \tag{$\mathcal{P}_{\kappa,b}$}
\end{equation}
These dual problems differ slightly from those in \citep{blasso_duval_peyre}. Because we restrict attention to nonnegative measures, the total variation norm $\norm{\mu}_{\rm TV}$ in \eqref{eq:pb_P_00} and \eqref{eq:pb_P_kappa_b} reduces to $\int_\X \d\mu$. As a consequence, in the dual formulation the condition $\norm{\Psi^*p}_{\infty} \leq 1$ is relaxed to $\Psi^*p \leq 1$.

We use the notation $p_{\kappa,b}$ for the solution of the dual problem, and $\eta_{\kappa,b}= \Psi^*p_{\kappa,b}$. 
For $\kappa=0$, we choose the particular dual solution defined by \eqref{eq:def_p_00}.

The dual and primal solutions are connected. For $\mu \in \M(\X)^+$, we denote \[ \partial^+ \norm{\mu}_{\rm TV} \coloneq \lbrace \eta \in \C(\X) \: : \: \eta \leq 1\, , \:\supp(\mu) \subset \lbrace \eta=1 \rbrace \rbrace\,.\] 
Strong duality ensures that for $\mu_{\kappa,b}$ a solution of the primal problem, $\eta_{\kappa,b}\in \partial^+ \norm{\mu_{\kappa,b}}_{\rm TV}$. Moreover, for $\kappa>0$ we have $p_{\kappa,b}= -\frac{1}{\kappa}(\Psi\mu_{\kappa,b}-\Psi\mu_\omega^0 -b)$.
\begin{lemma}\label{lemma:unique_sol_minimal_certif_ndsc}
 Under the assumptions of Theorem \ref{th:main_theorem_certif}, $\mu_\omega^0$ is the unique solution of \eqref{eq:pb_P_00}. Moreover, $\eta_{0,0}\coloneq \restr{\eta}{\X}=\Psi^* p_{0,0}$ where $p_{0,0}$ is the solution of \eqref{eq:pb_D_00} with minimal norm, \ie \[p_{0,0}= \underset{p \in \L}{\arg\min} \left\lbrace \norm{p}_{\L}\; : \; \Psi^*p \in \partial^+\norm{\mu_\omega^0}_{\rm TV}\right\rbrace \,.\] 
\end{lemma}
\begin{proof}
As $\restr{\eta}{\X} \in \partial^+\norm{\mu_\omega^0}_{\rm TV}$, we deduce that $p_{0,0}$ is a solution of \eqref{eq:pb_D_00}, linked to any solution $\mu_{0,0}$ of \eqref{eq:pb_P_00} by $\Psi^*p_{0,0}=\restr{\eta}{\X} \in \partial^+\norm{\mu_{0,0}}_{\rm TV}$. 
So $\supp(\mu_{0,0}) \subset \lbrace \restr{\eta}{\X}=1 \rbrace=\{x_j^0\}_{j=1}^s$. Using the injectivity of $G_{S_0}$ (Lemma \ref{lemma:check_hyp_ndsc}), we have $\mu_{0,0}=\mu_\omega^0$.
\\Moreover, as $\eta$ (of the form \eqref{eq:eta_general_form}) satisfies the NDSC, \citep[Proposition 7]{blasso_duval_peyre} implies that $p_{0,0}$ is of minimal norm. In our setting, working with nonnegative measures allows the condition $|\Psi^* p| \leq 1$ in the proof of \citep[Proposition 7]{blasso_duval_peyre} to be replaced by $\Psi^* p \leq 1$, without altering the conclusion.
\end{proof}

\begin{lemma}[Convergence of the dual solutions]\label{lemma:convergence_dual_sol_ndsc}
It holds that $\norm{p_{\kappa,0}- p_{0,0}}_{\L} \underset{\kappa \to 0}{\to} 0$.
\\Moreover, $\norm{p_{\kappa,b}- p_{\kappa,0}}_{\L} \leq \frac{\norm{b}_{\L}}{\kappa}$.
\end{lemma}
\begin{proof}
\underline{$\norm{p_{\kappa,0}- p_{0,0}}_{\L} \underset{\kappa \to 0}{\to} 0$:}
 Using that $p_{\kappa,0}$ is the solution of \eqref{eq:pb_D_kappa_b} for $b=0$ and $p_{0,0}$ is a solution of \eqref{eq:pb_D_00}, writing $y=\Psi\mu_\omega^0$ we have \[\innerprod{y}{p_{\kappa,0}}_{\L}-\frac{\kappa}{2}\norm{p_{\kappa,0}}^2 \geq \innerprod{y}{p_{0,0}}_{\L}-\frac{\kappa}{2}\norm{p_{0,0}}^2 \quad \text{and} \quad \innerprod{y}{p_{0,0}}_{\L} \geq \innerprod{y}{p_{\kappa,0}}_{\L}\,.\] We deduce that $\norm{p_{\kappa,0}}_\L$ is bounded by $\norm{p_{0,0}}_\L$. Given $\kappa_n \to 0$, as the closed unit ball of a Hilbert space is weakly sequentially compact, we can extract $p_{\kappa_{n_k},0}$ that converges weakly towards $p^* \in \L$. 
\\We show that $p^*=p_{0,0}$. We have
$\innerprod{y}{p^*}_{\L}\geq\innerprod{y}{p_{0,0}}_{\L}$ so \[\innerprod{\mu_\omega^0}{\Psi^*p^*}\geq\innerprod{\mu_\omega^0}{\Psi^*p_{0,0}}=\norm{\mu_\omega^0}_{\rm TV}\,.\]
We also have that $\Psi^*$ is weakly continuous from $\L$ to $\C(\X)$ so \[\sup_\X \Psi^*p^*\leq \sup_\X \Psi^*p_{\kappa_{n_k},0}=1\,,\] hence $\supp(\mu^0)\subset\{\Psi^*p*=1\}$ and $p^* \in \partial^+\norm{\mu_\omega^0}_{\rm TV}$. Finally, \[\norm{p_{0,0}}_{\L} \geq \lim\inf \norm{p_{\kappa_{n_k},0}}_{\L}\geq\norm{p^*}_{\L}\] as if $h_n \rightharpoonup h$ in $\L$, then $\norm{h}_{\L}^2\leq \lim \inf \norm{h_n}_{\L}^2$. Hence $p^*=p_{0,0}$.
\\Then $p_{\kappa_{n_k},0}\rightarrow p_{0,0}$ strongly because $\norm{p_{\kappa_{n_k},0}}_{\L}\rightarrow\norm{p_{0,0}}_{\L}$ and $p_{\kappa_{n_k},0}\rightharpoonup p_{0,0}$. 
Finally, $p_{\kappa,0} \underset{n_k\rightarrow \infty}{\longrightarrow} p_{0,0}$ strongly because any convergent subsequence of $(p_{\kappa,0})_\kappa$ has limit $p_{0,0}$.
\\\underline{$\norm{p_{\kappa,b}- p_{\kappa,0}}_{\L} \leq \frac{\norm{b}_{\L}}{\kappa}$:} The mapping $P: \frac{y}{\kappa} \rightarrow p_{\kappa,b}$ is the projection onto a close convex set so it is non expansive: for $h_1, h_2 \in \L$, \[\innerprod{h_1-P(h_1)}{P(h_2)-P(h_1)}\leq 0\quad \text{and} \quad \innerprod{h_2-P(h_2)}{P(h_1)-P(h_2)}\leq 0\,.\]
 Subtracting the two and applying Cauchy-Schwarz leads to \[\innerprod{P(h_1)-P(h_2)}{P(h_1)-P(h_2)}\leq\innerprod{h_1-h_2}{P(h_1)-P(h_2)}\leq \norm{h_1-h_2}\norm{P(h_1)-P(h_2)}\,.\] Hence \[\norm{p_{\kappa,b}-p_{\kappa,0}}\leq\norm{\frac{\Psi \mu_\omega^0+b-\Psi \mu_\omega^0}{\kappa}}=\frac{\norm{b}}{\kappa}\,.\]
\end{proof}

\begin{proposition}[Sparsity of the solution]\label{prop:at_most_1_spike_near_region_ndsc}
Under the assumptions of Theorem \ref{th:main_theorem_certif},
 there exists $\tilde \kappa_0>0$ (depending on $\mu^0,\X, \tau$), $\gamma_0$ (depending on $d$) such that for all $\kappa \leq \tilde \kappa_0$ and $b$ such that $\norm{b}_\L \leq \gamma_0 \kappa$, the solution $\mu_{\kappa,b}$ of \eqref{eq:pb_P_kappa_b} is a discrete measure and has at most 1 particle in each $\X_j^{\rm near}(r)$.
\end{proposition}
\begin{proof}
We recall the definitions of the covariant derivatives and the associated norms in Definition \ref{def:riemannian_derivatives}.

Using Lemma \ref{lemma:eta_ndsc}, we deduce that $\eta_{0,0}(x)\leq 1-\varepsilon_0$ for all $x \notin \bigcup \X_j^{\rm near}(r)$ and $v^T H^{\mathfrak{g}} \eta_{0,0}(x) v \leq -2\varepsilon_2 \norm{v}_x^2 $ for all $v\in \R^{2d}$ and $x \in \bigcup \mathcal{G}_{x_j^0}(\X_j^{\rm near}(r))$.
\\In light of Lemma \ref{lemma:convergence_dual_sol_ndsc}, we can choose $\tilde \kappa_0$ such that for all $\kappa \leq \tilde \kappa_0$, $B\norm{p_{\kappa,0}- p_{0,0}}_{\L} \leq \frac{\varepsilon_0}{4}$ (where $B$ is defined in Lemma \ref{lemma:check_hyp_ndsc}).
We assume that $\norm{b}_{\L} \leq \gamma_0 \kappa$ with $\gamma_0= \frac{\varepsilon_0}{4B}=\frac{0.03911}{8 \sqrt{7}d}$. Using Cauchy-Schwarz we get that for all $x \in \R^d \times [u_{\min}, +\infty)^d$, $i=0,1,2$,
\begin{align*}
   \norm{D_i\eta_{\kappa,b}(x)- D_i\eta_{0,0}(x)}_x & \leq  \norm{p_{\kappa,b}-p_{0,0}}_{\L} \norm{D_i \Psi\delta_x}_x \,, \\
   &\leq B\left( \frac{\norm{b}_{\L}}{\kappa}+ \norm{p_{\kappa,0}-p_{0,0}}_{\L} \right) \,,\\
 & \leq \frac{\varepsilon_0}{2} \,.
\end{align*}

So $\eta_{\kappa,b}(x)<1$ for all $x \notin \bigcup \X_j^{\rm near}(r)$ and $H^\mathfrak{g}\eta_{\kappa,b}(x)$ is negative-definite for all $x \in \bigcup \mathcal{G}_{x_j^0}(\X_j^{\rm near}(r))$ (in fact, remark that as $\varepsilon_0 \leq 2\varepsilon_2$, $v^T H^{\mathfrak{g}} \eta_{\kappa,b}(x) v \leq - \varepsilon_2 \norm{v}_x^2$ for all $v\in \R^{2d}$). Hence there exists at most 1 point $\bar x$ in each $\X_j^{\rm near}(r)$ such that $\eta_{\kappa,b}(\bar x)=1$. In fact for such a point $\bar x$, $\nabla \eta_{\kappa,b}(\bar x)=0$ because $\eta_{\kappa,b}\leq 1$. So for $x\in \X_j^{\rm near}(r) \setminus \{\bar x\}$, denoting $\gamma$ the geodesic connecting $\bar x$ to $x$, we have \[\eta_{\kappa,b}(x)=1 + \int_0^1 (1-y) \dot \gamma(y)^T H^\mathfrak{g}\eta_{\kappa,b}(\gamma(y)) \dot \gamma(y) \, \d y <1\,.\]
As $\eta_{\kappa,b} \in \partial^+ \norm{\mu_{\kappa,b}}_{\rm TV}$, we deduce that $\mu_{\kappa,b}$ is a discrete measure and that it has at most 1 particle in each $\X_j^{\rm near}(r)$.
\end{proof}

\begin{proposition}
 Under the assumptions of Theorem \ref{th:main_theorem_certif},
 there exists $\kappa_0>0$ (depending on $\mu^0,\X, \tau$), $\gamma_0$ (depending on $d$) such that for all $\kappa \leq \kappa_0$ and $b$ such that $\norm{b}_\L \leq \gamma_0 \kappa$, $\mu_{\kappa,b}$ is a discrete measure and has exactly 1 particle in each $\X_j^{\rm near}(r)$, and no particle on the far region.
 
 Moreover, writing $\mu_{\kappa,b}=\sum_{j=1}^s \omega_j^{\kappa,b} \delta_{x_j^{\kappa,b}}$ with $x_j^{\kappa,b} \in \X_j^{\rm near}(r)$, for $j=1,\ldots,s$ we have $|\omega^{\kappa,b}-\omega^0| \lesssim_{\X, \mu^0, \tau} \kappa$ and $\mathpzc{d}(x_j^0, x_j^{\kappa,b}) \lesssim_{\X, \mu^0, \tau} \kappa$.
\end{proposition}
\begin{proof}
\underline{$\mu_{\kappa,b}(\X_j^{\rm near}(r))>0$:}
We take Proposition \ref{prop:at_most_1_spike_near_region_ndsc} as a starting point.
 We show that $\mu_{\kappa,b}(\X_j^{\rm near}(r))>0$ for $\norm{b}_{\L} \leq \gamma_0 \kappa$ and $\kappa\leq \kappa_0$ with $\kappa_0 \leq \tilde \kappa_0$ small enough. 
 
It suffices to show the weak* convergence of $\mu_{\kappa,b}$ towards $\mu_\omega^0$ as $\kappa \to 0$ and for $\norm{b}_{\L} \leq \gamma_0 \kappa$ (this result is also stated in \citep[Proposition 4]{blasso_duval_peyre}). 
We have \[\frac{1}{2}\norm{\Psi\mu_\omega^0+b-\Psi\mu_{\kappa,b}}_{\L}^2+\kappa\norm{\mu_{\kappa,b}}_{\rm TV}\leq \frac{1}{2}\norm{b}_{\L}^2+\kappa \norm{\mu_\omega^0}_{\rm TV}\] since $\mu_{\kappa,b}$ minimizes \[J_{\kappa,b}:\mu \mapsto \frac{1}{2}\norm{\Psi\mu_\omega^0+b-\Psi\mu}_{\L}^2+\kappa\norm{\mu}_{\rm TV}\,.\] Dividing by $\kappa$, it follows that $\norm{\mu_{\kappa,b}}_{\rm TV}$ is bounded by $\frac{\gamma_0^2}{2} \kappa+ \norm{\mu_\omega^0}_{\rm TV}$. So we can extract $\mu_k$ a weak* convergent subsequence of $\mu_{\kappa,b}$ as $\kappa,b \rightarrow 0$ with $\norm{b}\leq \gamma_0 \kappa$. Let $\mu^*$ be its limit. By the lower semi-continuity of the TV norm for the weak* convergence, $\norm{\mu^*}_{\rm TV}\leq \lim\inf \norm{\mu_k}_{\rm TV} \leq \norm{\mu_\omega^0}_{\rm TV}$. To establish $\mu^*=\mu_\omega^0$, it remains to show that $\Psi \mu^*=\Psi\mu_\omega^0$ (because $\mu_\omega^0$ is the only solution of \eqref{eq:pb_P_00}, see Lemma \ref{lemma:unique_sol_minimal_certif_ndsc}).
Recall that $\norm{p_{\kappa,b}-p_{\kappa,0}}_\L\leq \frac{\norm{b}_\L}{\kappa}$ and that $p_{\kappa,b}=-\frac{1}{\kappa}(\Psi\mu_{\kappa,b}-\Psi\mu_\omega^0-b)$, giving $\norm{\Psi\mu_{\kappa,b}-b-\Psi\hat{\mu}_{\kappa,0}}_\L\leq \norm{b}_\L$. As $\norm{\Psi\mu_{\kappa,0}-\Psi\mu_\omega^0}_\L^2\leq 2 \kappa \norm{\mu_\omega^0}_{\rm TV}$ (because $J_{\kappa,0}(\mu_{\kappa,0})\leq J_{\kappa,0}(\mu_\omega^0)$), it comes 
\begin{align*}
\norm{\Psi\mu_{\kappa,b}-\Psi\mu_\omega^0}_\L& \leq \norm{\Psi\mu_{\kappa,b}-\Psi\hat{\mu}_{\kappa,0}}_\L+ \norm{\Psi\mu_{\kappa,0}-\Psi\mu_\omega^0}_\L \,, \\
& \leq 2\norm{b}_\L + \sqrt{2 \kappa \norm{\mu_\omega^0}_{\rm TV}} \,.
\end{align*}
As $\Psi$ is weak* to weak continuous from $\M(\X)$ to $\L$, for all $p\in\L$ we have 
\[|\innerprod{p}{\Psi(\mu_k-\mu_\omega^0)}_\L| \to |\innerprod{p}{\Psi(\mu^*-\mu_\omega^0)}_\L|\,.\] Applying Cauchy-Schwarz, we have $|\innerprod{p}{\Psi(\mu_k-\mu_\omega^0)}_\L| \to 0$ from which we conclude that $\innerprod{p}{\Psi(\mu^*-\mu_\omega^0)}_\L=0$. So $\Psi \mu^*=\Psi \mu_\omega^0$, and $\mu_{\kappa,b}$ converges towards $\mu_\omega^0$ as $\kappa \to 0$ and $\norm{b}_\L \leq \gamma_0 \kappa$ for the weak* topology. This concludes this part of the proof.
\\\underline{Continuity of the solution \wrt $\kappa,b$:}
Following the proof of \citep[Theorem 2]{blasso_duval_peyre}, as $G_{S^0}$ defined by \eqref{eq:def_G_S_0} is full-rank, we can make use of the implicit function theorem. 
We define \[D_{\gamma_0,\kappa_0} \coloneq \{\kappa\geq 0\,, b\in\L \: :\: \kappa\leq \kappa_0 \: \text{and} \: \norm{b}_L \leq \gamma_0\kappa\}\,.\]
Possibly taking $\kappa_0$ smaller than in the first part of this proof, there exists a $C^1$ function $g: D_{\gamma_0,\kappa_0}\to \R^s \times \R^{2d \times s}$ such that $g(\kappa,b)=(\omega_1^{\kappa,b},\ldots,\omega_s^{\kappa,b},x_1^{\kappa,b},\ldots,x_s^{\kappa,b})$ for all $(\kappa,b) \in D_{\gamma_0,\kappa_0}$.
So for all $j=1,\ldots,s$, $|\omega_j^{\kappa,b}-\omega_j^0|= O(\max\{\kappa,\norm{b}_\L\})$ and $\norm{x_j^{\kappa,b}-x_j^0}_2= O(\max\{\kappa,\norm{b}_\L\})$. The constant hiding in $O$ depends on $\X,\mu^0,\tau$. As $\norm{b}_\L \leq \gamma_0 \kappa$ and as $\gamma_0$ only depends on $\X$ via $d$, we get \[|\omega_j^{\kappa,b}-\omega_j^0| \lesssim_{\X, \mu^0, \tau} \kappa \quad \text{and} \quad \norm{x_j^{\kappa,b}-x_j^0}_2\lesssim_{\X, \mu^0, \tau} \kappa\,.\]
\underline{Control of $\mathpzc{d}(x_j^{\kappa,b},x_j^0)$:} We then show that we can replace the Euclidean norm by the semi-distance in the last inequality.
Indeed, for~$x,x'\in \R^d \times [u_{\min},+\infty)^d$ we have 
\begin{align*}
\mathpzc{d}(x,x')^2&=\sum_{i=1}^d \left(\frac{(t_i-t_i')^2}{u_i^2+{u_i'}^2+\tau^2} + \ln\left(\frac{u_i^2+{u_i'}^2+\tau^2}{\sqrt{2u_i^2+\tau^2}\sqrt{2{u_i'}^2+\tau^2}} \right) \right) \,, \\
&= \sum_{i=1}^d \left(\frac{(t_i-t_i')^2}{u_i^2+{u_i'}^2+\tau^2} + \ln\left(1+\frac{\left(\sqrt{u_i^2+\frac{\tau^2}{2}} -\sqrt{{u_i'}^2+\frac{\tau^2}{2}}\right)^2}{2\sqrt{u_i^2+\frac{\tau^2}{2}}\sqrt{{u_i'}^2+\frac{\tau^2}{2}}} \right) \right) \,, \\
&\leq \sum_{i=1}^d \left(\frac{(t_i-t_i')^2}{u_i^2+{u_i'}^2+\tau^2} + \frac{\left(\sqrt{u_i^2+\frac{\tau^2}{2}} -\sqrt{{u_i'}^2+\frac{\tau^2}{2}}\right)^2}{2\sqrt{u_i^2+\frac{\tau^2}{2}}\sqrt{{u_i'}^2+\frac{\tau^2}{2}}} \right) \,, \\
&\leq \sum_{i=1}^d \left(\frac{(t_i-t_i')^2}{2u_{\min}^2+\tau^2} + \frac{\left(\sqrt{u_i^2+\frac{\tau^2}{2}} -\sqrt{{u_i'}^2+\frac{\tau^2}{2}}\right)^2}{2u_{\min}^2+\tau^2} \right) \,, \\
&\leq \sum_{i=1}^d \left(\frac{(t_i-t_i')^2}{2u_{\min}^2+\tau^2} + \frac{(u_i -u_i')^2}{2u_{\min}^2+\tau^2} \right) \,, \\
&= \frac{1}{2u_{\min}^2+\tau^2}\norm{x-x'}_2^2 \,, \\
&\lesssim_{\X,\tau} \norm{x-x'}_2^2 \,.
\end{align*}
\underline{Control of $|a_j^{\kappa,b}-a_j^0|$:} To deal with the reparameterized weights, we can reuse the calculations given in Section \ref{section:proof_cor_effective_near}. Denoting $a_j^{\kappa,b}=\frac{\omega_j^{\kappa,b}}{W(x_j^{\kappa,b})}$, we showed that if $r_e \coloneq \mathpzc{d}(x_j^0,x_j^{\kappa,b}) \leq r$, then
\[
    |a_j^{\kappa,b}-a_j^0|\leq |\omega_j^{\kappa,b} - \omega_j^0| W(x_j^0)^{-1} (1+H(r) r_e) + a_j^0 H(r) r_e \,,
\]
where $H(r)$ is defined by \eqref{eq:def_H(r)_effective}. As we can make the upper bound $r$ on $\mathpzc{d}(x_j^0,x_j^{\kappa,b})$ only depend on $\X$ and $\tau$, it holds that
\begin{align*}
    |a_j^{\kappa,b}-a_j^0| &\lesssim_{\X,\mu^0,\tau} |\omega_j^{\kappa,b} - \omega_j^0| + \mathpzc{d}(x_j^0,x_j^{\kappa,b}) \,, \\
    &\lesssim_{\X,\mu^0,\tau} \kappa \,.
\end{align*}
\end{proof}

\subsection{Proof of Corollary \ref{cor:ndsc}}\label{section:proof_cor_ndsc}
 If $n \geq \frac{c_\kappa^2}{\kappa_0^2 (2\pi)^{d/2} \tau^d}$, choosing $\kappa=\frac{c_\kappa}{(2\pi)^{d/4} \tau^{d/2} \sqrt{n}}$ we have $\kappa \leq \kappa_0$. Moreover, it holds that $\norm{\Gamma_n} \leq \gamma_0 \kappa$ with probability greater than $1-C_\Gamma e^{-\left(\frac{\gamma_0 c_\kappa}{C_\Gamma}\right)^2}$ (\cf Lemma \ref{lemma:control_noise} where we took $\rho=\frac{\gamma_0^2 c_\kappa^2}{C_\Gamma^2}$). Theorem \ref{th:NDSC} applies: $\mu_{n,\omega}^\star=\sum_{j=1}^s \omega_j^\star \delta_{x_j^{\star}}$ where $\omega_j^\star>0$ and $x_j^{\star} \in \X_j^{\rm near}(r)$ for all $j=1,\ldots, s$.
 
 Moreover, $|a_j^\star-a_j^0|\lesssim_{\X,\mu^0,\tau} \frac{c_\kappa}{\sqrt{n}}$ and $\mathpzc{d}(x_j^\star,x_j^0) \lesssim_{\X,\mu^0,\tau} \frac{c_\kappa}{\sqrt{n}}$.

\end{document}